\documentclass[a4paper]{article}
\usepackage{amsmath,amsthm,mathtools, amsfonts,amssymb}
\usepackage[utf8]{inputenc}
\usepackage[dvipsnames]{xcolor}
\usepackage{graphicx}
\usepackage{graphicx}
\usepackage{color}
\usepackage{tikz}
\usepackage{tikz-cd}
\usepackage[all]{xy}
\usetikzlibrary{cd,positioning}
\usepackage{hyperref}
\usepackage{extarrows}
\usepackage{ytableau}
\usepackage{float}

\usepackage[mathscr]{euscript}

\usepackage{amsmath}
\allowdisplaybreaks[4] 
\numberwithin{equation}{section}

\title{
Type $B$- and $C$- 
trigonometric weight functions
}
\author{Kohei Motegi}
\date{\today}

\newcommand{\Inv}{\operatorname{Inv}}

\newtheorem{thm}{Theorem}[section]
\newtheorem{lem}[thm]{Lemma}
\theoremstyle{definition}
\newtheorem{Definition}[thm]{Definition}
\newtheorem{prop}[thm]{Proposition}
\newtheorem{corollary}[thm]{Corollary}

\begin{document}

\maketitle
\begin{abstract}
We introduce and study trigonometric weight functions of type $B$
and type $C$ arising from integrable lattice models.
These functions give $q$-deformations of the factorial
$K$-theoretic Schur $P$- and $Q$-functions of Ikeda and Naruse,
with the $P$-functions specialized to the type $B$ case.
We realize them as specializations of partition functions
of a six-vertex model with triangular boundary conditions
and derive their algebraic properties using the Yang--Baxter
equation and properties of the $R$- and $K$-matrices.
We introduce normalized weight functions and establish
diagonal evaluation formulas, divisibility properties,
and bounds on Newton polytopes analogous to those in type $A$,
together with recursive relations and stability.
We further derive branching formulas relating these functions
to type $A$ weight functions, and prove equal-cardinality
orthogonality and all-cardinality biorthogonality relations.
\end{abstract}

\section{Introduction}

Trigonometric weight functions arise in quantum integrable systems as
off-shell Bethe wavefunctions and as ingredients of integral solutions of
the quantum Knizhnik--Zamolodchikov equations
\cite{Mimachi,TV,TVNested,FR}.
They admit realizations as partition functions of vertex models, in which
the Yang--Baxter equation provides a method for deriving their algebraic
properties \cite{Bethe,Baxterbook,KBI,Slavnov}.
For type $A$, these constructions are based on the $R$-matrices of quantum
affine algebras \cite{Drinfeld,Jimbo,RTF}.
The resulting functions also connect integrable models with Schubert
calculus: suitable $q=0$ degenerations give Grothendieck polynomials
\cite{LS,FK,RTV}, and their lattice-model realizations degenerate to
five-vertex models \cite{MS,GK,WZ}.

A geometric interpretation of type $A$ weight functions is provided by
the theory of stable envelopes \cite{MO,Okounkov,AO}.
Rim\'anyi--Tarasov--Varchenko \cite{RTV} identified normalized
trigonometric weight functions with restrictions of $K$-theoretic stable
envelopes for cotangent bundles of type $A$ partial flag varieties.
Their characterization involves diagonal values, divisibility conditions,
and bounds on Newton polytopes.
Thus, weight functions provide explicit formulas whose algebraic
properties reflect geometric structures.
For elliptic extensions and their relations to trigonometric weight
functions, see \cite{FelderVarchenko,RTV2,Konno2017}.

For orthogonal and symplectic Grassmannians, the factorial
$K$-theoretic Schur $P$- and $Q$-functions introduced by
Ikeda and Naruse \cite{IkedaNaruse}, called the $GP$-
and $GQ$-functions for short, represent equivariant
$K$-theoretic Schubert classes.
The type $B$ specialization of the $GP$-functions is
denoted by $GB$.
The type $A$ theory suggests seeking trigonometric deformations
of these functions, together with suitable normalizations
that yield diagonal evaluations expressed as products over roots,
divisibility properties, and bounds on Newton polytopes.
The trigonometric lattice model provides a natural algebraic
construction of such deformations through its partition functions
and their symmetrization formulas.
Such a construction also requires understanding the algebraic
relations arising from the boundary conditions.

In this paper, we introduce and study type $B$ and type $C$
trigonometric weight functions as $q$-deformations of the
factorial $GB$- and $GQ$-functions, respectively.
Our construction starts from the six-vertex model with triangular
boundary conditions studied in \cite{Motegi}.
We denote its partition function by
$F_{N,n,r}(\mathbf u|\mathbf w|\mathbf I)$.
An explicit symmetrization formula for this partition function was
obtained in \cite{Motegi}.
Two choices of boundary parameters, together with a boundary scalar
factor in type $B$, give the weight functions considered here.
We introduce normalizations adapted to types $B$ and $C$ and establish
their algebraic properties using the Yang--Baxter equation,
properties of the $R$- and $K$-matrices, and equivariant $K$-theory.
For related boundary partition functions and their applications to
alternating sign matrices and probability, see \cite{BBCW,ABF,BFK,GGMW}.

Our main results are as follows.
\begin{itemize}
\item
We establish triangular vanishing, type $B$ and type $C$ diagonal
evaluation formulas, divisibility, and Newton smallness for the normalized weight functions
at equal-cardinality specializations.
Together with exchange congruences, these properties characterize the
equal-cardinality specializations uniquely.

\item
We derive recursive relations and stability.
These recursive relations include relations between weight
functions with different numbers of spectral variables
and different cardinalities of the indexing subsets.

\item
We obtain explicit branching formulas for specializations of arbitrary
cardinality in terms of normalized type $A$ weight functions.
We prove support conditions, Laurent polynomiality, divisibility, and
Newton degree bounds for these specializations.

\item
We derive equal-cardinality orthogonality and all-cardinality
biorthogonality relations, as well as inversion formulas for the
branching coefficients.
\end{itemize}

The relation with type $A$ weight functions provides a starting
point for the equal-cardinality results.
To relate different cardinalities, we derive additional branching
and inversion formulas, which lead to the all-cardinality
biorthogonality relation.
A detailed study of the relation to motivic Chern classes and
stable bases is reserved for subsequent work.

The paper is organized as follows.
In Section~2, we introduce the $R$- and $K$-matrices, recall the
partition function and its symmetrization formula, and define the
type $B$ and type $C$ weight functions and their normalized versions.
We state the equal-cardinality properties and their uniqueness
characterizations, derive algebraic relations from the lattice model,
and establish the relation with type $A$ weight functions.
In Section~3, we prove the stated properties of the normalized weight
functions and derive their recursive relations, stability, and
two-variable reduction formulas.
In Section~4, we study specializations of arbitrary cardinality,
prove the branching formulas and their consequences, and establish
orthogonality, all-cardinality biorthogonality, and branching inversion.
Section~5 contains concluding remarks.
Appendix~A recalls the type $A$ results used in the proofs.
We generally present the type $C$ arguments in detail, followed by
their type $B$ counterparts.

\section{Six-vertex model and type $B$ and $C$ trigonometric weight functions}

In this section, we introduce trigonometric weight functions of type \(B\) and \(C\) via six-vertex model partition functions and derive their properties.

\subsection{$R$-matrix, $K$-matrix and partition functions}

Let $V$ be a two-dimensional complex vector space with basis $\{e_0,e_1\}$.
For $n \ge 1$, we write $V_i \cong V$ for the $i$-th copy of $V$, and consider
$V_1 \otimes \cdots \otimes V_n$.

We define the $R$-matrix $R(u,v)$ (see Figure~\ref{fig:rmatrix} for a graphical description) as a linear operator on $V \otimes V$ by
\begin{equation}\label{eq:R-def}
R(u,v)(e_i \otimes e_j)
= \sum_{k,l=0}^1 [R(u,v)]_{ij}^{kl} \, e_k \otimes e_l.
\end{equation}
The matrix elements $[R(u,v)]_{ij}^{kl}$ are explicitly given by
(Figure~\ref{fig:sixvertex})
\begin{equation}\label{eq:R-components}
[R(u,v)]_{ij}^{kl} =
\begin{cases}
\;1-qv/u & \text{if } (i,j,k,l) = (0,0,0,0), \\
\;q(1-v/u) & \text{if } (i,j,k,l) = (0,1,0,1), \\
\;(1-q)v/u & \text{if } (i,j,k,l) = (0,1,1,0), \\
\;1-q & \text{if } (i,j,k,l) = (1,0,0,1), \\
\;1-v/u & \text{if } (i,j,k,l) = (1,0,1,0), \\
\;1-qv/u & \text{if } (i,j,k,l) = (1,1,1,1), \\
\;0 & \text{otherwise}.
\end{cases}
\end{equation}

For tensor products of $n$ factors ($n \ge 2$), we define
$
R_{ab}(u,v) \in \operatorname{End}(V_1 \otimes \cdots \otimes V_n)
$
to be the operator acting as $R(u,v)$
on the $a$-th and $b$-th components and as the identity on the others.

\begin{figure}[!htbp]
  \centering
  \includegraphics[width=0.5\linewidth]{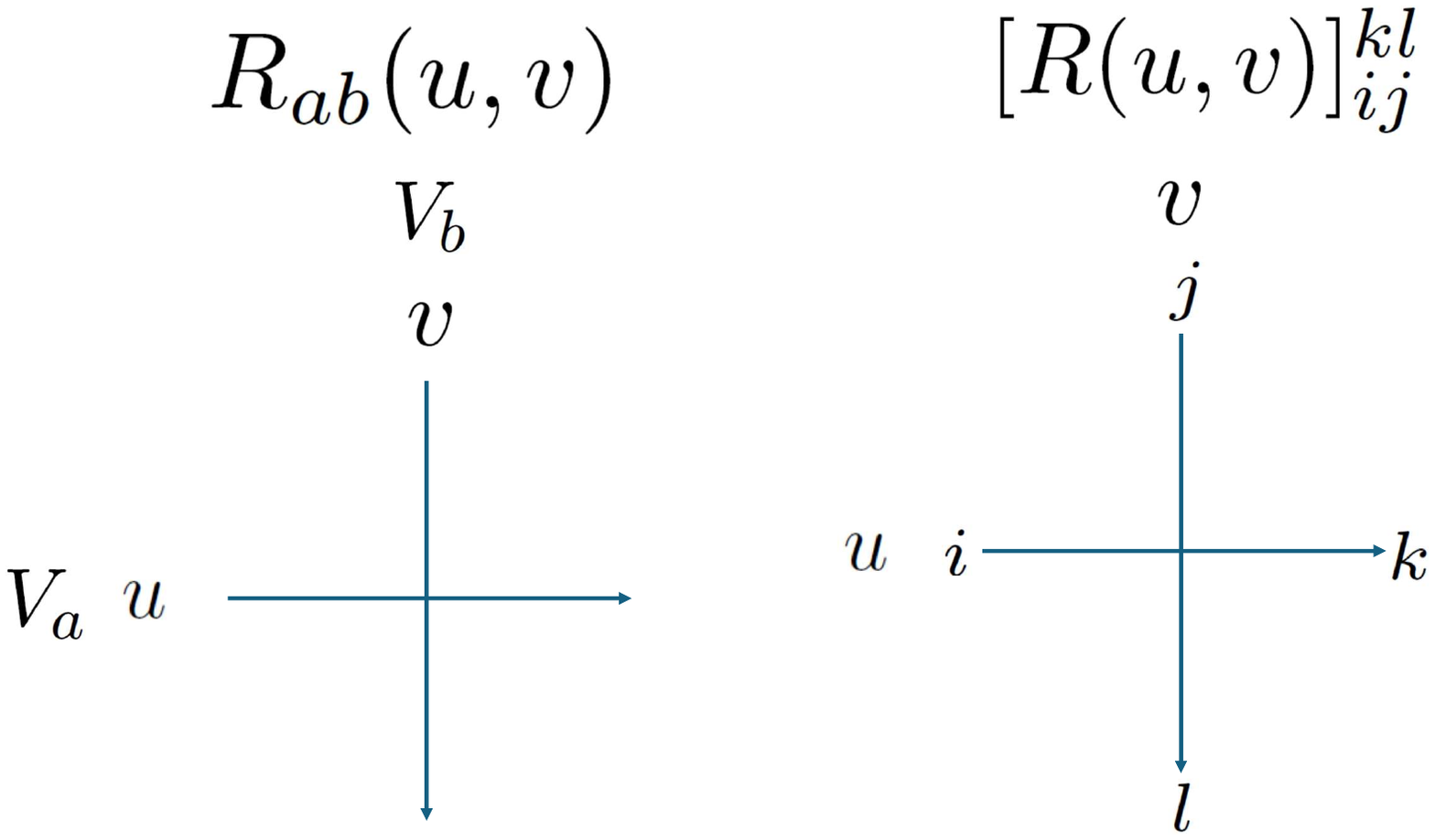}
  \caption{
  A graphical description of the $R$-matrix $R_{ab}(u,v)$, built from
$R(u,v)$ whose action is given in \eqref{eq:R-def} and whose matrix elements
$[R(u,v)]_{ij}^{kl}$ are given in \eqref{eq:R-components}.
Each line represents a vector space. The spectral parameters $u$ and $v$
are associated with the lines corresponding to $V_a$ and $V_b$, respectively.}
  \label{fig:rmatrix}
\end{figure}

\begin{figure}[!htbp]
  \centering
  \includegraphics[width=0.5\linewidth]{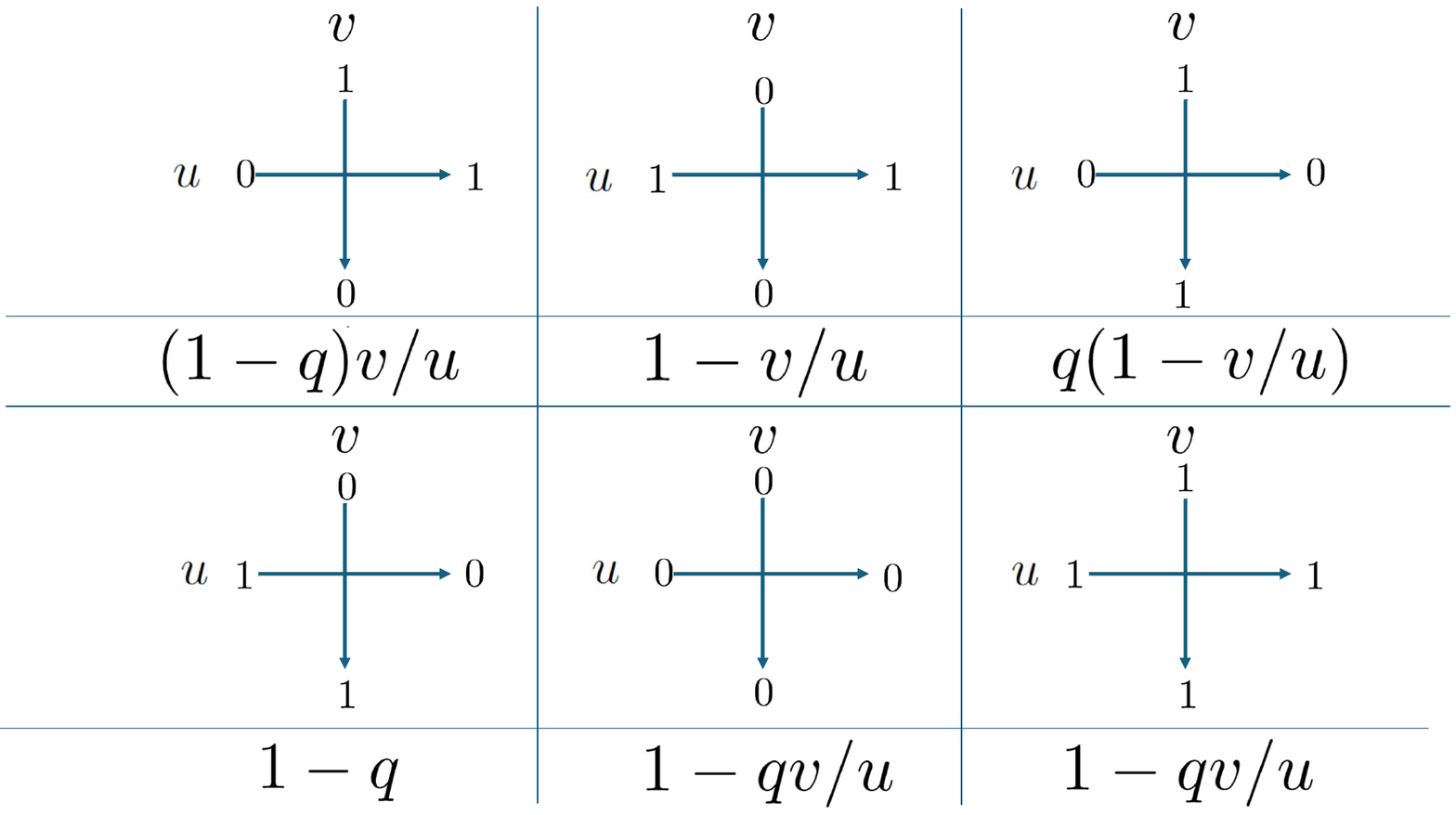}
  \caption{Nonzero configurations and corresponding elements of the $R$-matrix $R(u,v)$ for the six-vertex model.}
  \label{fig:sixvertex}
\end{figure}

The $R$-matrix $R(u,v)$ used here is essentially the so-called stochastic $R$-matrix
(see \cite{BP} for example),
obtained as a gauge transformation of the standard quantum $R$-matrix \cite{Drinfeld,Jimbo,RTF}. 
We adopt this normalization because the degeneration at $q=0$ is particularly simple.

The $R$-matrix satisfies the Yang--Baxter equation \cite{Yang,Baxter} (Figure~\ref{fig:yangbaxter})
\begin{equation}\label{eq:YBE}
R_{12}(u,v)\, R_{13}(u,w)\, R_{23}(v,w)
=
R_{23}(v,w)\, R_{13}(u,w)\, R_{12}(u,v),
\end{equation}
as an identity in $\operatorname{End}(V_1 \otimes V_2 \otimes V_3)$.

\begin{figure}[!htbp]
  \centering
  \includegraphics[width=0.5\linewidth]{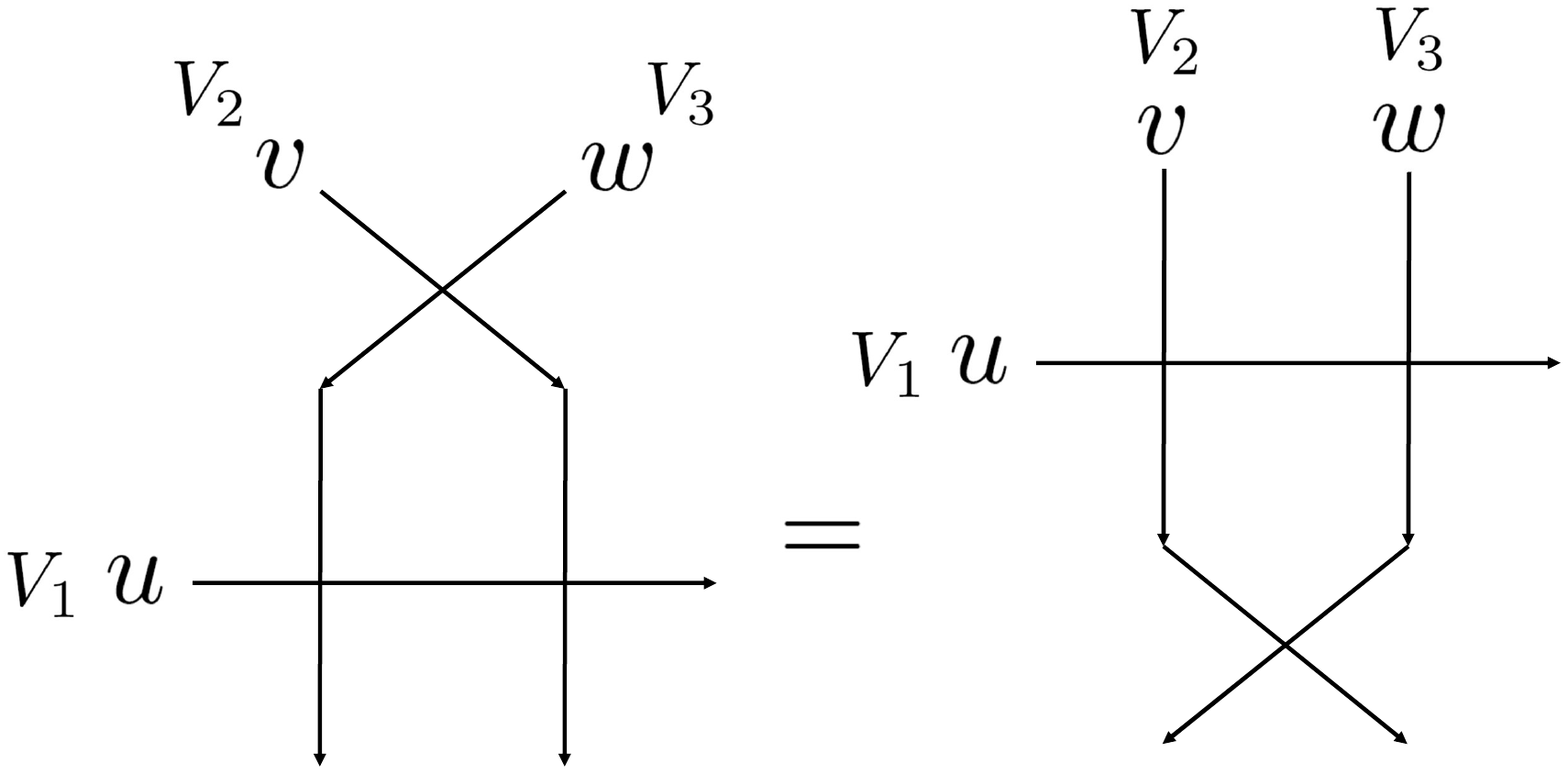}
  \caption{The Yang--Baxter equation \eqref{eq:YBE}.}
  \label{fig:yangbaxter}
\end{figure}

\begin{figure}[!htbp]
  \centering
  \includegraphics[width=0.5\linewidth]{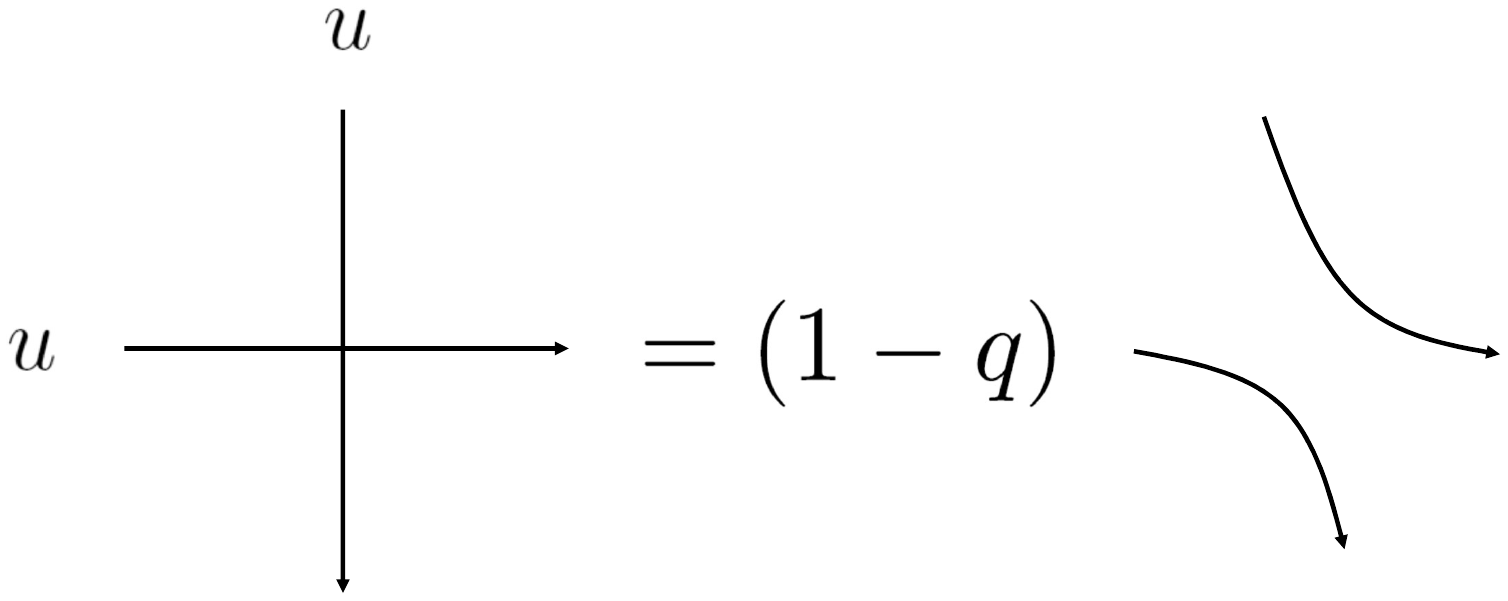}
  \caption{Reduction to the permutation operator \eqref{eq:Ruu}.}
  \label{fig:reducepermutation}
\end{figure}

We also list other properties used in this paper.
At $u = v$, the $R$-matrix reduces to
\begin{equation}\label{eq:Ruu}
R(u,u) = (1 - q)\, P,
\end{equation}
where $P \in \operatorname{End}(V \otimes V)$ is the permutation operator defined by
$
P(e_i \otimes e_j) = e_j \otimes e_i
$, $i,j=0,1$.
See Figure~\ref{fig:reducepermutation}.

The unitarity relation is also a fundamental property (Figure~\ref{fig:unitarity})
\begin{equation} \label{eq:unitarity}
    R_{12}(u,v)R_{21}(v,u)=(1-qu/v)(1-qv/u) \mathrm{Id}.
\end{equation}

\begin{figure}[!htbp]
  \centering
  \includegraphics[width=0.5\linewidth]{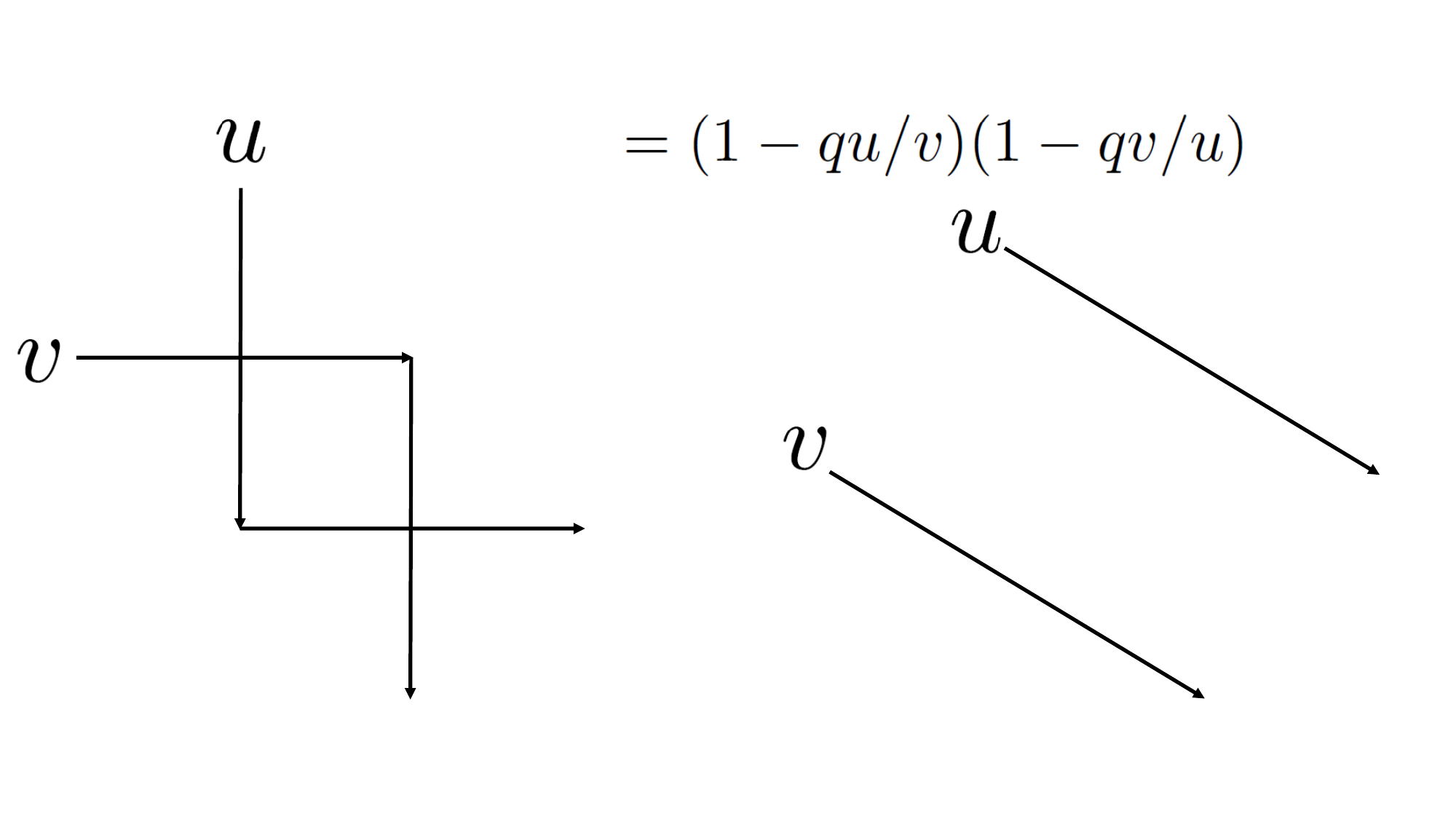}
  \caption{The unitarity relation \eqref{eq:unitarity}.}
  \label{fig:unitarity}
\end{figure}

Moreover, the $R$-matrix satisfies the ice rule, namely,
\begin{equation}\label{eq:ice-rule}
[R(u,v)]_{ij}^{kl} = 0 \quad \text{unless } i + j = k + l.
\end{equation}

\begin{figure}[!htbp]
  \centering
  \includegraphics[width=0.5\linewidth]{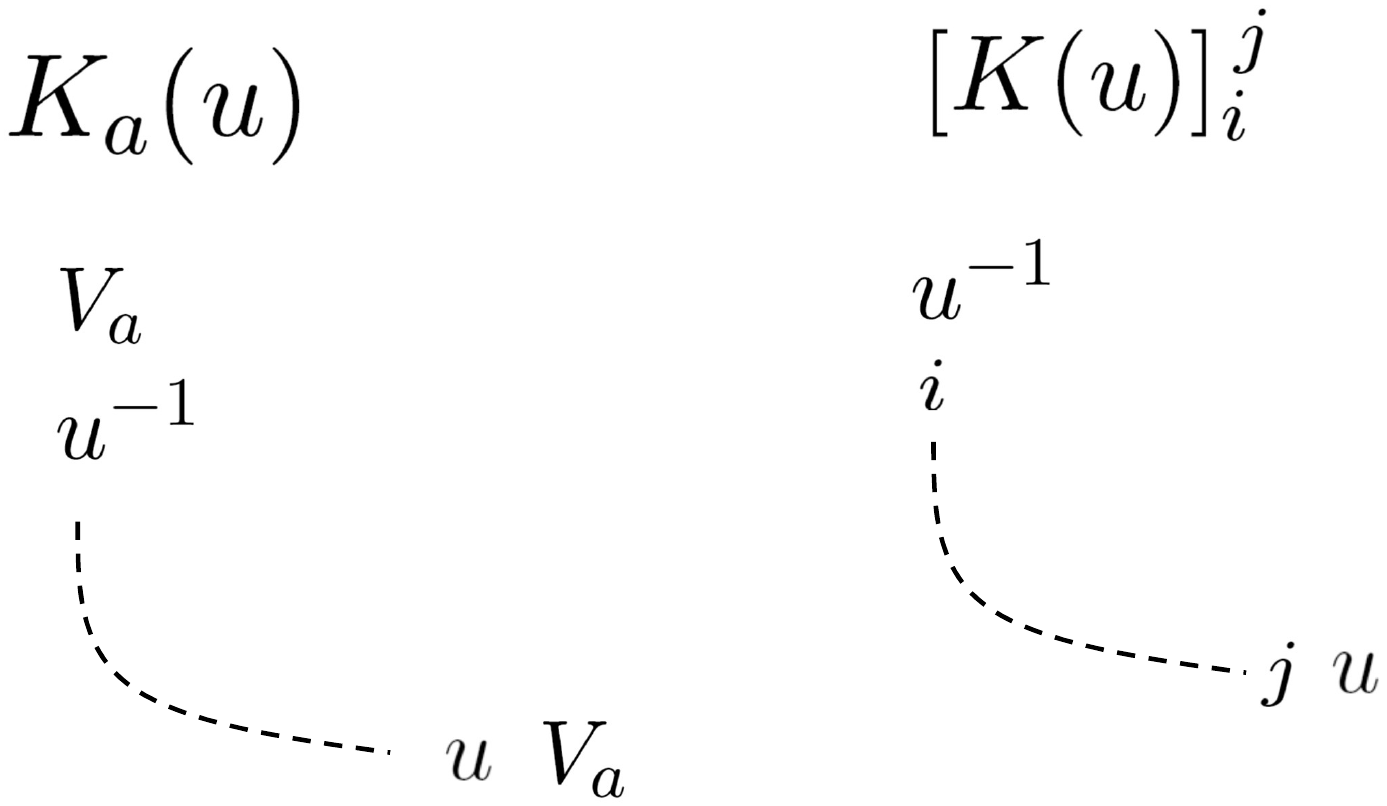}
  \caption{A graphical description of the operator $K_a(u)$ induced by the
local $K$-matrix $K(u)$. The action of $K(u)$ is given in
\eqref{eq:K-def}, and its matrix elements $[K(u)]_i^{\,j}$ are given
explicitly in \eqref{eq:Ktri-components}.}
  \label{fig:kmatrix}
\end{figure}

\begin{figure}[!htbp]
  \centering
  \includegraphics[width=0.5\linewidth]{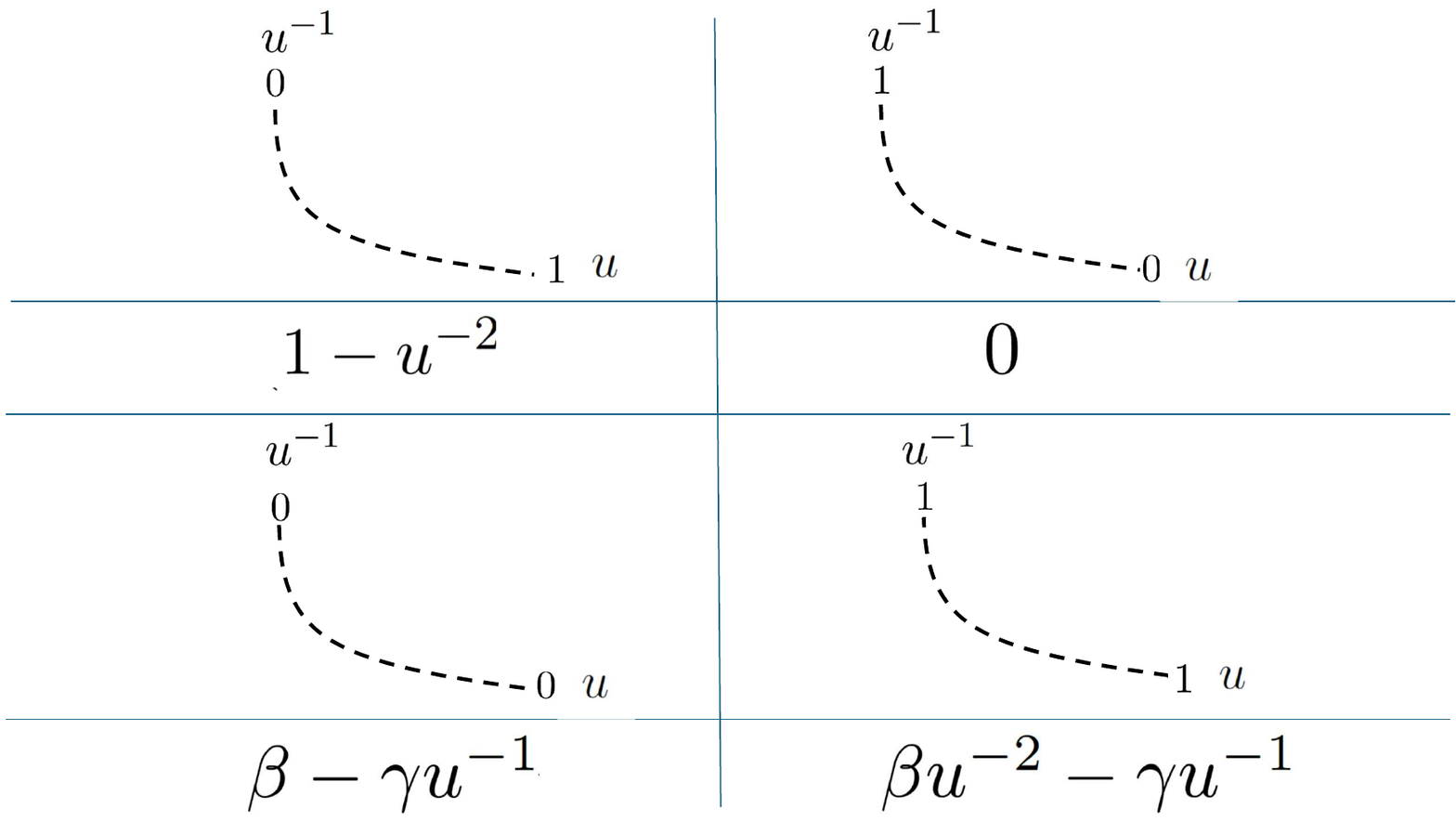}
  \caption{The triangular $K$-matrix $K(u)=K^{\mathrm{tri}}(u)$ \eqref{eq:Ktri-components}.}
  \label{fig:triangularkmatrix}
\end{figure}

We introduce the triangular $K$-matrix $K(u)=K^{\mathrm{tri}}(u)$, together with two specializations $K^C(u)$ and $K^B(u)$.
The reason for the superscripts $C$ and $B$ will become clear later from the correspondence with the \(GQ\)- and \(GB\)-functions.

More precisely, $K(u)$ is the linear operator on $V$ defined by (Figure~\ref{fig:kmatrix})
\begin{equation}\label{eq:K-def}
K(u)e_i = \sum_{j =0}^1 [K(u)]_i^{\,j} \, e_j.
\end{equation}
Its matrix elements are given by (Figure~\ref{fig:triangularkmatrix})
\begin{equation}\label{eq:Ktri-components}
[K(u)]_i^{\,j} =
\begin{cases}
\;\beta-\gamma u^{-1} & \text{if } (i,j) = (0,0), \\
\;1-u^{-2} & \text{if } (i,j) = (0,1), \\
\;0 & \text{if } (i,j) = (1,0), \\
\; \beta u^{-2}-\gamma u^{-1} & \text{if } (i,j) = (1,1),
\end{cases}
\end{equation}
where $\gamma,\beta$ are complex parameters.
For comparison, the generic stochastic boundary $K$-matrix in \cite[(2.12)--(2.15)]{GGMW} becomes triangular in the limit $c\to\infty$.  After setting $a=\beta/\gamma$ and changing the overall scalar normalization, its diagonal weights agree with those of \eqref{eq:Ktri-components}; the remaining off-diagonal normalization differs, in general, by a constant diagonal gauge.

The matrix elements of $K^C(u)$ are given by
\begin{equation}\label{eq:KC-components}
[K^C(u)]_i^{\,j} =
\begin{cases}
\; 1 & \text{if } (i,j) = (0,0), \\
\; 1-u^{-2} & \text{if } (i,j) = (0,1), \\
\;0 & \text{if } (i,j) = (1,0), \\
\; u^{-2} & \text{if } (i,j) = (1,1).
\end{cases}
\end{equation}

Similarly, the matrix elements of $K^B(u)$ are given by
\begin{equation}\label{eq:KB-components}
[K^B(u)]_i^{\,j} =
\begin{cases}
\;1 & \text{if } (i,j) = (0,0), \\
\;1-u^{-1} & \text{if } (i,j) = (0,1), \\
\;0 & \text{if } (i,j) = (1,0), \\
\;u^{-1} & \text{if } (i,j) = (1,1).
\end{cases}
\end{equation}

Note the following relation
\begin{align}
K^C(u)&=K(u)|_{\gamma=0,\beta=1}, \label{triandc} \\
K^B(u)&=(1+u^{-1})^{-1} K(u)|_{\gamma=-1,\beta=1}, \label{triandb}
\end{align}
between the $K$-matrices.

\begin{figure}[!htbp]
  \centering
  \includegraphics[width=0.5\linewidth]{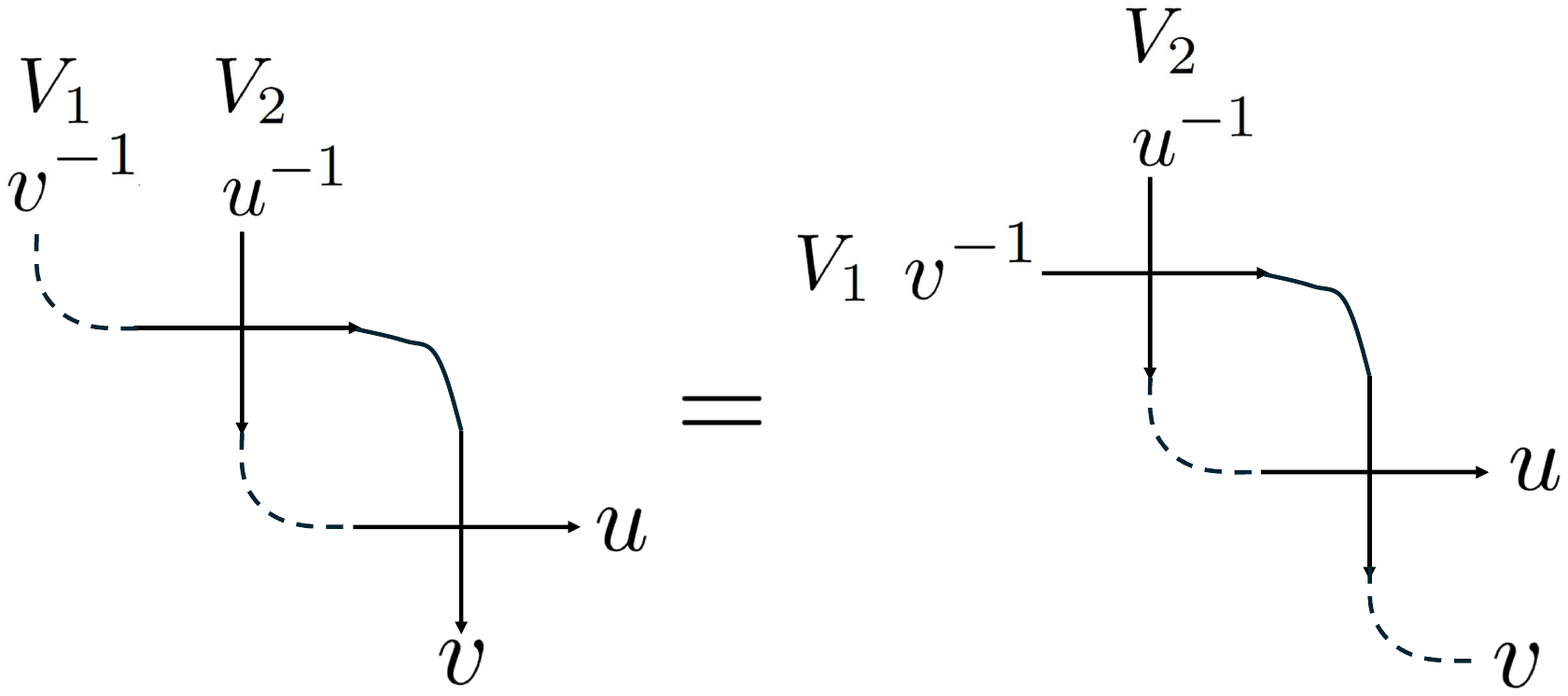}
  \caption{The reflection equation \eqref{eq:reflection}.}
  \label{fig:reflection}
\end{figure}

The $K$-matrix $K(u)$ satisfies the reflection equation \cite{Cherednik,Sklyanin} (Figure~\ref{fig:reflection})
\begin{equation}\label{eq:reflection}
R_{21}(u,v)\, K_2(u)\, R_{12}(v,u^{-1})\, K_1(v)
=
K_1(v)\, R_{21}(u,v^{-1})\, K_2(u)\, R_{12}(v^{-1},u^{-1}),
\end{equation}
as an identity in $\operatorname{End}(V_1 \otimes V_2)$.
Here $K_a(u)$ denotes the operator acting non-trivially on the $a$-th component as $K(u)$.
By \eqref{triandc} and \eqref{triandb}, the specializations $K^C(u)$ and $K^B(u)$ satisfy the same reflection equation.

We note the symmetry
$
R(u,v) = R(v^{-1},u^{-1}),
$
under which \eqref{eq:reflection} can be rewritten as
\begin{equation*}
R_{21}(u,v)\, K_2(u)\, R_{12}(u,v^{-1})\, K_1(v)
=
K_1(v)\, R_{21}(u,v^{-1})\, K_2(u)\, R_{12}(u,v).
\end{equation*}

We next introduce the partition functions of the six-vertex model,
constructed from the $R$-matrices $R(u_i,w_j)$ $(1 \le i \le n,\ 1 \le j \le N)$,
$R(u_i,u_j^{-1})$ $(1 \le i < j \le n)$,
and the $K$-matrices $K(u_j)$ $(1 \le j \le n)$,
as depicted in Figure~\ref{fig:partitionfunction}.
Here, $N,n,r$ are nonnegative integers with
$r\le\min(n,N)$.
We use the empty-product convention and set
$F_{N,0,0}=WC_{N,0,0}=WB_{N,0,0}=1$.
The normalized functions for $n=0$ are defined by the same
normalization formulas as below.

\begin{figure}[!htbp]
  \centering
  \includegraphics[width=0.7\linewidth]{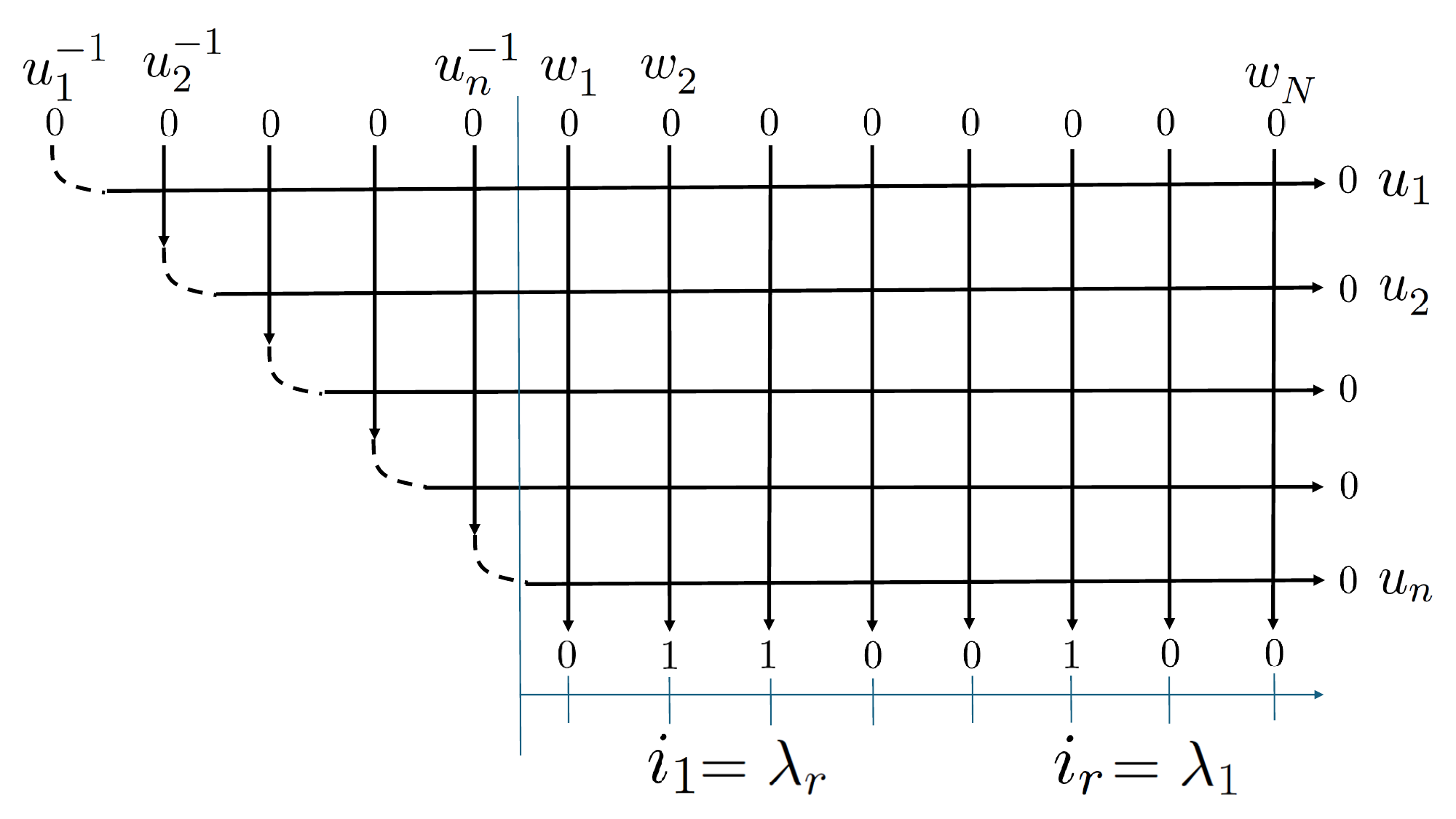}
 \caption{Partition functions corresponding to the $K$-matrices 
$K(u)=K^{\mathrm{tri}}(u)$, $K^{C}(u)$, and $K^{B}(u)$. 
In this paper, they are denoted by 
$F_{N,n,r}(\mathbf{u}|\mathbf{w}|\mathbf{I})$,
$WC_{N,n,r}(\mathbf{u}|\mathbf{w}|\mathbf{I})$, and
$WB_{N,n,r}(\mathbf{u}|\mathbf{w}|\mathbf{I})$, respectively.}
  \label{fig:partitionfunction}
\end{figure}

We impose boundary conditions such that the top and right boundaries consist entirely of \(0\)'s.
We label a binary sequence of \(0\)'s and \(1\)'s on the bottom boundary by the positions of the \(1\)'s counted from the left:
\[
N \ge i_r > \cdots > i_2 > i_1 \ge 1.
\]
Alternatively, it can be encoded by a strict partition \(\lambda = (\lambda_1, \lambda_2, \dots, \lambda_r)\) such that
\[
N \ge \lambda_1 > \lambda_2 > \cdots > \lambda_r \ge 1,
\]
where the two parametrizations are related by
$
i_j = \lambda_{r+1-j} \ (j=1,\dots,r).
$

We denote by $F_{N,n,r}(\mathbf{u}|\mathbf{w}|\mathbf{I})$ 
the partition function corresponding to the triangular $K$-matrix $K(u)$ \eqref{eq:Ktri-components}.

Here,
\[
\mathbf{u} = (u_1,u_2,\dots,u_n), \quad 
\mathbf{w} = (w_1,w_2,\dots,w_N), \quad
\mathbf{I} = (i_1,i_2,\dots,i_r)
\]
are ordered sequences (not sets). 
We remark that $F_{N,n,r}(\mathbf{u}|\mathbf{w}|\mathbf{I})$
is symmetric in the variables $\mathbf{u}$,
which can be proved by using the Yang--Baxter equation \eqref{eq:YBE} and the reflection equation \eqref{eq:reflection}.

\subsection{Type $B$ and $C$ trigonometric weight functions and their normalized versions}
\label{subsec:weight-functions-and-normalizations}

\subsubsection{Unnormalized weight functions and specializations}

The partition function $F_{N,n,r}(\mathbf{u}|\mathbf{w}|\mathbf{I})$ admits an explicit symmetrization formula.

\begin{thm}[\cite{Motegi}, Theorem 3.3]\label{explicitFformula}
We have
\begin{align}
&F_{N,n,r}(\mathbf{u}|\mathbf{w}|\mathbf{I} ) \nonumber \\
 =& \frac{\prod_{k=1}^{n} u_{k}^{-(N+n)}}{(n-r)!} \sum_{\sigma \in S_{n}} \sigma  \Biggl[ \prod_{j=1}^{r} \Bigg\{ (1 - q)(u_{j}^{2} - 1) \prod_{k=i_{j}+1}^{N} (u_{j} - q w_{k}) \prod_{k=1}^{i_{j}-1} (u_{j} - w_{k}) \Bigg\} \nonumber \\
&\quad \times \prod_{1 \le j < k \le r} \frac{q u_{j} - u_{k}}{u_{j} - u_{k}} (u_{j} u_{k} - 1) \prod_{j=r+1}^{n} \prod_{k=1}^{N} (u_{j} - q w_{k}) \nonumber \\
&\quad \times \prod_{j=r+1}^{n} \prod_{k=1}^{r} \frac{q u_{j} - u_{k}}{u_{j} - u_{k}} (u_{j} u_{k} - 1) \prod_{r+1 \le j < k \le n} (u_{j} u_{k} - q) \prod_{j=r+1}^{n} (\beta u_{j} - \gamma) \Biggr]. \label{Motherfunction}
\end{align}
Here, \(\sigma \in S_n\) acts by permuting the variables \(\mathbf{u}\), and \(S_n\) is the symmetric group.
\end{thm}

The equation \eqref{Motherfunction} without the overall factor 
$\prod_{k=1}^{n} u_{k}^{-(N+n)}$ was introduced in \cite[(3.19)]{Motegi}.

We introduce two specializations of $F_{N,n,r}(\mathbf{u}|\mathbf{w}|\mathbf{I} )$.
\begin{Definition}
We define \(WC_{N,n,r}(\mathbf{u}|\mathbf{w}|\mathbf{I})\)
and \(WB_{N,n,r}(\mathbf{u}|\mathbf{w}|\mathbf{I})\) by:
\begin{align}
&WC_{N,n,r}(\mathbf{u}|\mathbf{w}|\mathbf{I} )
 :=  F_{N,n,r}(\mathbf{u}|\mathbf{w}|\mathbf{I} )|_{\gamma=0,\beta=1}
\label{WCfunction}
\end{align}

and
\begin{align}
&WB_{N,n,r}(\mathbf{u}|\mathbf{w}|\mathbf{I} )
:= \left( \prod_{k=1}^{n}(1+u_k^{-1})^{-1} \right) F_{N,n,r}(\mathbf{u}|\mathbf{w}|\mathbf{I} )|_{\gamma=-1,\beta=1}
\label{WBfunction}
\end{align}

\end{Definition}

By \eqref{triandc} and \eqref{triandb}, these specializations have direct lattice-model realizations.  Namely,
$WC_{N,n,r}(\mathbf{u}|\mathbf{w}|\mathbf{I})$ is the partition function obtained by replacing $K(u)$ with $K^C(u)$, while
$WB_{N,n,r}(\mathbf{u}|\mathbf{w}|\mathbf{I})$ is the partition function obtained by replacing $K(u)$ with $K^B(u)$.
The factor $\prod_{k=1}^{n}(1+u_k^{-1})^{-1}$ in the definition of $WB_{N,n,r}$ comes from the scalar normalization in \eqref{triandb}.

We call
\(WC_{N,n,r}(\mathbf{u}|\mathbf{w}|\mathbf{I})\)
and \(WB_{N,n,r}(\mathbf{u}|\mathbf{w}|\mathbf{I})\)
\emph{trigonometric weight functions of type $C$ and type $B$}, respectively.
One of the reasons for this naming is that, in the type \(A\) case,
trigonometric weight functions reduce at \(q=0\) to Grothendieck polynomials.
Moreover, the type \(C\) and type \(B\) trigonometric weight functions reduce at \(q=0\) to
the factorial \(GQ\)- and \(GB\)-functions, respectively \cite{IkedaNaruse},
where the factorial \(GB\)-functions form a specialization of the factorial \(GP\)-functions.

To state this degeneration precisely, we recall the definitions that will be used below.
Let \(\lambda=(\lambda_1,\dots,\lambda_r)\) be the strict partition corresponding to \(\mathbf{I}\) as above.
We use the definitions in \cite{IkedaNaruse} with the $K$-theoretic parameter specialized to $\beta=-1$.
Here the $K$-theoretic parameter $\beta$ is unrelated to the parameter $\beta$ used for the $K$-matrix in this paper.
Thus, we define
\begin{align}
x \oplus y:=x+y-xy, \ \ x \ominus y:=\frac{x-y}{1-y}. \notag
\end{align}

\begin{Definition} \cite[Definition 2.1]{IkedaNaruse}
The factorial $K$-theoretic Schur $Q$- and $P$-functions are defined by
\begin{align}
GQ_\lambda(\mathbf{x}| \mathbf{b}) &:= \frac{1}{(n-r)!} \sum_{\sigma \in S_n} \sigma \left[ \prod_{i=1}^r (x_i \oplus x_i) \prod_{k=1}^{\lambda_i-1} (x_i \oplus b_k) \prod_{i=1}^r \prod_{j=i+1}^n \frac{x_i \oplus x_j}{x_i \ominus x_j} \right], \label{eq:GQexp} \\
GP_\lambda(\mathbf{x}| \mathbf{b}) &:= \frac{1}{(n-r)!} \sum_{\sigma \in S_n} \sigma \left[ \prod_{i=1}^r  \prod_{k=1}^{\lambda_i} (x_i \oplus b_k) \prod_{i=1}^r \prod_{j=i+1}^n \frac{x_i \oplus x_j}{x_i \ominus x_j} \right]. \label{eq:GPexp}
\end{align}
\end{Definition}

\begin{Definition} \cite[(6.1), (6.2)]{IkedaNaruse}
We set
\begin{align}
GC_\lambda(\mathbf{x}|\mathbf{b}) &:=GQ_\lambda(\mathbf{x}|\mathbf{b}), \notag \\
GB_\lambda(\mathbf{x}, \mathbf{b}) &:=GP_\lambda(\mathbf{x},\mathbf{b})|_{b_1=0,b_i \mapsto b_{i-1} (i \ge 2)}
\nonumber \\
&=
\frac{1}{(n-r)!} \sum_{\sigma \in S_n} \sigma \left[ \prod_{i=1}^r x_i  \prod_{k=1}^{\lambda_i-1} (x_i \oplus b_k) \prod_{i=1}^r \prod_{j=i+1}^n \frac{x_i \oplus x_j}{x_i \ominus x_j} \right]. \label{eq:GBexp}
\end{align}
\end{Definition}

A direct computation gives the following relations.

\begin{prop} \label{reductionwbwc}
We have the following relations:
    \begin{align}
    &WC_{N,n,r}(\mathbf{u}|\mathbf{w}|\mathbf{I})|_{q=0}=GC_\lambda(\mathbf{x}|\mathbf{b})=GQ_\lambda(\mathbf{x}|\mathbf{b}), \\
      &WB_{N,n,r}(\mathbf{u}|\mathbf{w}|\mathbf{I})|_{q=0}=GB_\lambda(\mathbf{x}|\mathbf{b})
      =GP_\lambda(\mathbf{x},\mathbf{b})|_{b_1=0,b_i \to b_{i-1} (i \ge 2)}
      .
    \end{align}
    Here, the correspondence between the variables and labels is given by
    $x_i = 1 - u_i^{-1}$ ($i=1,\dots,n$), $b_k = 1 - w_k$ ($k=1,\dots,N$), and $i_j = \lambda_{r+1-j}$ ($j=1,\dots,r$).
\end{prop}

We introduce a sequence of integers
\[
\mathbf{M}=(m_1,\ldots,m_s),\qquad
0\le s\le \min(n,N),\qquad
1\le m_1<\cdots<m_s\le N,
\]
and define the substitution of variables $\mathbf{u}=\mathbf{w}_{\mathbf{M}}$ by
\[
u_j=
\begin{cases}
w_{m_j},&1\le j\le s,\\
1,&s+1\le j\le n.
\end{cases}
\]
For $s=0$, this means $u_1=\cdots=u_n=1$.

We define a partial order on increasing sequences of integers, possibly of different lengths, by
\[
\mathbf{M}\ge \mathbf{I}
\quad\Longleftrightarrow\quad
s\ge r
\ \text{and}\
m_{s-r+j}\ge i_j
\quad\text{for all }1\le j\le r.
\]
When $s=r$, this reduces to the usual componentwise order
\[
\mathbf{M}\ge \mathbf{I}
\quad\Longleftrightarrow\quad
m_j\ge i_j
\quad\text{for all }1\le j\le r.
\]

We show the following properties
which are analogous to those of the type $A$ trigonometric weight functions,
also known as type $A$ $K$-theoretic stable bases \cite{RTV,RTV2}.

\begin{thm} \label{ffunctionproperties}
$\mathrm{(i)}$ We have
\begin{align}
F_{N,n,r}(\mathbf{w}_{\mathbf{M}}|\mathbf{w}|\mathbf{I})=0,
\end{align}
unless $\mathbf{M}\ge \mathbf{I}$, and
\begin{align}
&F_{N,n,r}(\mathbf{w}_{\mathbf{I}}|\mathbf{w}|\mathbf{I} )
= (1 - q)^{\frac{(n-r)(n-r-1)}{2}} (\beta - \gamma)^{n-r} \nonumber \\
& \times \left( \prod_{j=1}^{r} w_{i_j}^{-(N+n)} \right) \prod_{j=1}^{r}
\left\{ (1 - q)(w_{i_j}^{2} - 1) \prod_{k=i_{j}+1}^{N} (w_{i_j} - q w_{k}) \prod_{k=1}^{i_{j}-1} (w_{i_j} - w_{k}) \right\} \nonumber \\
& \times \prod_{1 \le j < k \le r} \frac{q w_{i_j} - w_{i_k}}{w_{i_j} - w_{i_k}} (w_{i_j} w_{i_k} - 1)  \prod_{k=1}^N (1 - q w_k)^{n-r}  \prod_{k=1}^r (w_{i_k} - q) ^{n-r}.
\label{diagonalFfunction}
\end{align}

$\mathrm{(ii)}$
$(1-q)^{-r} F_{N,n,r}(\mathbf{w}_{\mathbf{M}}|\mathbf{w}|\mathbf{I})$ is divisible by $(1-q)$ for every $\mathbf{M}\neq\mathbf{I}$, with no restriction that $|\mathbf{M}|=|\mathbf{I}|$.
\end{thm}

\begin{proof}
We first prove $\mathrm{(i)}$.
Suppose first that $s\ge r$.
From the factor
$\displaystyle \sigma \Bigg( \prod_{j=1}^{r}(u_j^2-1) \Bigg)$
in the explicit formula \eqref{Motherfunction} for
$F_{N,n,r}(\mathbf{u}|\mathbf{w}|\mathbf{I})$,
the summation after the substitution
$\mathbf{u}=\mathbf{w}_{\mathbf{M}}$ can be restricted to permutations
$\sigma\in S_n$ satisfying
$\{\sigma(1),\ldots,\sigma(r)\}\subseteq\{1,\ldots,s\}$.
Indeed, if $\sigma(j)>s$ for some $1\le j\le r$, then
$u_{\sigma(j)}=1$, and hence $u_{\sigma(j)}^2-1=0$.

For such a permutation, let
$K=(k_1,\ldots,k_r)$, where
$1\le k_1<\cdots<k_r\le N$, be determined by
$
\{u_{\sigma(1)},\ldots,u_{\sigma(r)}\}
=
\{w_{k_1},\ldots,w_{k_r}\}.
$
Then $K$ is an $r$-element subsequence of $\mathbf{M}$.
There exists a permutation $\tau\in S_r$ such that
$u_{\sigma(j)}=w_{k_{\tau(j)}}$ for $1\le j\le r$.
The factor
$\displaystyle
\sigma\left(
\prod_{j=1}^{r}\prod_{a=1}^{i_j-1}(u_j-w_a)
\right)
=
\prod_{j=1}^{r}\prod_{a=1}^{i_j-1}
(u_{\sigma(j)}-w_a)
$
shows that a nonzero summand must satisfy
$k_{\tau(j)}\ge i_j$ for $1\le j\le r$.

For each $j$, the $r-j+1$ distinct integers
$k_{\tau(j)},k_{\tau(j+1)},\ldots,k_{\tau(r)}$
are all greater than or equal to $i_j$, since
$i_j\le i_{j+1}\le\cdots\le i_r$.
Thus at most $j-1$ elements of $K$ are smaller than $i_j$,
and hence $k_j\ge i_j$.
Since $K$ is an $r$-element subsequence of
$\mathbf{M}=(m_1,\ldots,m_s)$, we have
$k_j\le m_{s-r+j}$ for $1\le j\le r$.
Consequently,
$
m_{s-r+j}\ge i_j
\qquad (1\le j\le r),
$
and hence $\mathbf{M}\ge\mathbf{I}$. Therefore,
$F_{N,n,r}(\mathbf{w}_{\mathbf{M}}|\mathbf{w}|\mathbf{I})=0$
unless $\mathbf{M}\ge\mathbf{I}$.

We next evaluate the diagonal specialization $\mathbf{M}=\mathbf{I}$.
In this case $s=r$. A surviving summand must satisfy
$u_{\sigma(j)}=w_{i_{\tau(j)}}$ and
$i_{\tau(j)}\ge i_j$ for $1\le j\le r$, for some
$\tau\in S_r$.
Since $i_1<\cdots<i_r$, these inequalities force
$\tau=\mathrm{id}$. It follows that
$\sigma(j)=j$ for $1\le j\le r$, while
$
\{\sigma(k):r+1\le k\le n\}
=
\{r+1,\ldots,n\}.
$
Since $u_{r+1}=\cdots=u_n=1$, the $(n-r)!$ permutations
of these variables give the same contribution.
Thus the factor $(n-r)!$ produced by the sum cancels the factor
$1/(n-r)!$ in \eqref{Motherfunction}.
To make the remaining evaluation explicit, for each $b$ with
$r+1\le b\le n$ we have
$
\displaystyle
\left.\prod_{k=1}^N(u_b-qw_k)\right|_{u_b=1}
=
\prod_{k=1}^N(1-qw_k).
$
For $1\le a\le r$ and $r+1\le b\le n$,
$\displaystyle
\left.
\frac{qu_b-u_a}{u_b-u_a}(u_bu_a-1)
\right|_{u_b=1,\,u_a=w_{i_a}}
=
\frac{q-w_{i_a}}{1-w_{i_a}}(w_{i_a}-1)
=w_{i_a}-q.
$
For every pair $r+1\le a<b\le n$, the factor $u_au_b-q$
becomes $1-q$, and each factor $\beta u_b-\gamma$ becomes $\beta-\gamma$.
Therefore the factors involving $u_{r+1},\ldots,u_n$ give
$\displaystyle
(1-q)^{\frac{(n-r)(n-r-1)}{2}}(\beta-\gamma)^{n-r}
\prod_{k=1}^N(1-qw_k)^{n-r}
\prod_{a=1}^r(w_{i_a}-q)^{n-r}.
$
Multiplying this by the remaining factors in the identity summand,
with $u_a=w_{i_a}$ for $1\le a\le r$, gives exactly
\eqref{diagonalFfunction}. Thus no additional factor is hidden in the
cancellation of the $(n-r)!$ identical summands with the prefactor
$1/(n-r)!$.

Finally, suppose that $s<r$.
Repeated application of Lemma~\ref{onevariablesp}
which will be shown later
gives
$
F_{N,n,r}(\mathbf{w}_{\mathbf{M}}|\mathbf{w}|\mathbf{I})
=
C_{N,n,r,\mathbf{M}}
F_{N,r,r}
(
w_{m_1},\ldots,w_{m_s},
\underbrace{1,\ldots,1}_{r-s}
|\mathbf{w}|\mathbf{I}
)
$
where $C_{N,n,r,\mathbf{M}}$ is the product of the factors obtained
from Lemma~\ref{onevariablesp}.

For $n=r$, the factor
$\displaystyle \sigma\Biggl(\prod_{j=1}^{r}(u_j^2-1)\Biggr)
=\allowbreak\prod_{j=1}^{r}(u_j^2-1)$
appears in every summand.
Since $s<r$, the factorization in Lemma~\ref{factorizationlemma}
shows that the specialization vanishes:
$F_{N,r,r}\bigl(w_{m_1},\ldots,w_{m_s},\allowbreak
1,\ldots,1|\allowbreak\mathbf w|\mathbf I\bigr)=0$.
It follows that
$F_{N,n,r}(\mathbf{w}_{\mathbf{M}}|\mathbf{w}|\mathbf{I})=0$.
Since $\mathbf{M}\not\ge\mathbf{I}$ when $s<r$, this completes the proof of
$\mathrm{(i)}$. \\
\\
 We next prove $\mathrm{(ii)}$.
The factor $(1-q)^r$ is explicit in \eqref{Motherfunction}, so
$(1-q)^{-r}F_{N,n,r}$ is regular as a polynomial in $q$.
It is therefore enough to set $q=1$ before making the
specialization $\mathbf{u}=\mathbf{w}_{\mathbf{M}}$. At $q=1$, all quotients
$(qu_j-u_k)/(u_j-u_k)$ become $1$, and every summand contains the factor
$\displaystyle
\prod_{j=1}^{r}
\Bigg\{
(u_{\sigma(j)}^2-1)
\prod_{\substack{k=1\\k\ne i_j}}^N
(u_{\sigma(j)}-w_k)
\Bigg\}.
$
After the specialization, this factor can be nonzero only if
$u_{\sigma(j)}=w_{i_j}$ for every $j=1,\ldots,r$.
Hence every element of $\mathbf{I}$ must occur in $\mathbf{M}$.
If $s<r$, or if $\mathbf{I}\not\subseteq\mathbf{M}$, every summand is zero.
Suppose next that $\mathbf{I}\subseteq\mathbf{M}$ and $s>r$.
Then one of the values $w_m$ with
$m\in\mathbf{M}\setminus\mathbf{I}$ must occur among
$u_{\sigma(r+1)},\ldots,u_{\sigma(n)}$.
For that variable, say $u_{\sigma(b)}=w_m$ with $r+1\le b\le n$, the factor
$\displaystyle
\prod_{k=1}^N(u_{\sigma(b)}-w_k)
$
contains $w_m-w_m=0$.
Thus the value at $q=1$ is zero for every
$\mathbf{M}\ne\mathbf{I}$, independently of the cardinality of
$\mathbf{M}$. This proves divisibility by $(1-q)$.
\end{proof}

$WC_{N,n,r}(\mathbf{u}|\mathbf{w}|\mathbf{I} )$ and
$WB_{N,n,r}(\mathbf{u}|\mathbf{w}|\mathbf{I} )$ are specializations
of $F_{N,n,r}(\mathbf{u}|\mathbf{w}|\mathbf{I})$ and hence satisfy the
same properties as Theorem~\ref{ffunctionproperties}.
Theorem~\ref{ffunctionproperties} (i) is a generalization of
\cite[Proposition 7.1]{IkedaNaruse}.

\subsubsection{Common notation for the normalized weight functions}

We next introduce common notation and root-theoretic conventions used for the
normalized type $C$ and type $B$ weight functions and their main vanishing and
evaluation properties under the specializations $\mathbf{u}=\mathbf{w}_{\mathbf{M}}$. Their proofs use the lattice-model identities
established in Subsection~2.3 together with the relation with the type $A$ weight
functions established in Subsection~2.4, and are given in
Section~\ref{sec:normalized-algebraic-properties}. Specializations of
arbitrary cardinality are studied in
Section~\ref{sec:normalized-specializations}.

The following notation and root-theoretic interpretation will be used for the
normalized type $C$ and type $B$ weight functions.
Let
$\displaystyle
L_N=\bigoplus_{j=1}^{N}\mathbb{Z}t_j
$
be the lattice in which the corresponding root systems are realized, and
introduce the multiplicative variables
$
w_j=e^{t_j}
\
(1\leq j\leq N).
$
Throughout this section, we write
\[
\mathbf{I}=(i_1<i_2<\cdots<i_r),
\qquad
1\leq i_1<i_2<\cdots<i_r\leq N,
\]
as an index set. We also use $\mathbf{I}$ to denote its
underlying subset of $\{1,\dots,N\}$ when no confusion can arise.
This correspondence was introduced above, but we recall it here because it will be important below.
The strict partition corresponding to $\mathbf{I}$ is
\[
\lambda=\lambda(\mathbf{I})
=(\lambda_1>\lambda_2>\cdots>\lambda_r>0),
\qquad
\lambda_j=i_{r+1-j}.
\]
Thus,
$
\lambda=(i_r>i_{r-1}>\cdots>i_1)
$
and
$
\{\lambda_1,\dots,\lambda_r\} = \{i_1,\dots,i_r\}.
$
We use $\mathbf{I}$ when writing the weight functions and their
specializations, and use $\lambda$ when describing the Weyl
group action and the Bruhat--Chevalley order.

For $X=C,B$, let $W(X_N)$ be the Weyl group of type $X_N$.
As a Coxeter group, $W(X_N)$ is generated by
$
s_0,s_1,\dots,s_{N-1},
$
subject to the relations
\begin{align}
s_i^2&=1
&& (0\leq i\leq N-1), \notag \\
s_is_{i+1}s_i&=s_{i+1}s_is_{i+1}
&& (1\leq i\leq N-2), \notag \\
s_is_j&=s_js_i
&& (1\leq i,j\leq N-1,\ |i-j|\geq2), \notag \\
s_0s_1s_0s_1&=s_1s_0s_1s_0, \notag \\
s_0s_i&=s_is_0
&& (2\leq i\leq N-1). \notag
\end{align}

Let
\[
\mathcal{SP}_N = \{\emptyset\} \cup \left\{ (\mu_1>\mu_2>\cdots>\mu_{l}>0) \mathrel{}\middle|\mathrel{} \mu_1\leq N \right\}.
\]
For $\mu\in\mathcal{SP}_N$, its shifted Young diagram is
\[
S(\mu) = \left\{ (a,b)\in\mathbb{Z}_{\geq1}^{2} \mathrel{}\middle|\mathrel{} 1\leq a\leq\ell(\mu),\ a\leq b<a+\mu_a \right\}.
\]
We set $S(\emptyset)=\emptyset$.
The content of a box $\mathfrak b=(a,b)$ is defined by
$
c(\mathfrak b)=b-a.
$
Let $0\leq i\leq N-1$.
A box $\mathfrak b\in S(\mu)$ is called \emph{$i$-removable} if
$
c(\mathfrak b)=i
$
and $S(\mu)\setminus\{\mathfrak b\}$ is the shifted Young diagram of a strict
partition in $\mathcal{SP}_N$.
A box $\mathfrak b\notin S(\mu)$ is called \emph{$i$-addable} if
$
c(\mathfrak b)=i
$
and $S(\mu)\cup\{\mathfrak b\}$ is the shifted Young diagram of a strict
partition in $\mathcal{SP}_N$.
We define $s_i\mu$ as follows.
If $\mu$ has an $i$-removable box $\mathfrak b$, then $s_i\mu$ is determined
by
$
S(s_i\mu)=S(\mu)\setminus\{\mathfrak b\}.
$
If $\mu$ has an $i$-addable box $\mathfrak b$, then $s_i\mu$ is determined by
$
S(s_i\mu)=S(\mu)\cup\{\mathfrak b\}.
$
If $\mu$ has neither an $i$-removable nor an $i$-addable box, then
we set
$
s_i\mu=\mu.
$
These operators define the usual action of the Weyl group $W(X_N)$,
where $X=C$ or $B$, on $\mathcal{SP}_N$.
Moreover, with respect to the Bruhat--Chevalley order,
$
s_i\mu<\mu
$
if and only if $\mu$ has an $i$-removable box, whereas
$
s_i\mu>\mu
$
if and only if $\mu$ has an $i$-addable box \cite[Proposition~4.3]{IkedaNaruse}.
All inequalities between strict partitions below are understood with
respect to this order.

For a Laurent polynomial
$\displaystyle f(\mathbf{w})=\sum_{\nu\in\mathbb{Z}^{N}}c_{\nu}(q) w_{1}^{\nu_{1}}\cdots w_{N}^{\nu_{N}} $
where $c_{\nu}(q) \in \mathbb{C}(q)$,
we define its Newton polytope by
\[ \mathcal{N}(f)=\operatorname{Conv}\{\nu\in\mathbb{Z}^{N} \mid c_{\nu}(q)\ne0\}\subset\mathbb{R}^{N}. \]
We set $\mathcal N(0)=\emptyset$.
For a subset
$  \mathbf{M}=\{m_{1}<\cdots<m_{s}\}\subset\{1,\ldots,N\} $
of arbitrary cardinality $s$, define the linear functional
$ \varphi_{\mathbf{M}}:\mathbb{R}^{N}\longrightarrow\mathbb{R} $
by
\[
\varphi_{\mathbf{M}}(\nu)
=\sum_{j\in \mathbf{M}}\nu_{j}+2\sum_{j\notin \mathbf{M}}\nu_{j},
\qquad
\nu=(\nu_{1},\ldots,\nu_{N})\in\mathbb{R}^{N}.
\]
For $f\neq0$, the image $\varphi_{\mathbf M}(\mathcal N(f))$ is a closed interval. We call it the \emph{degree interval} of $f$ with respect to $\varphi_{\mathbf M}$, and call its lower and upper endpoints the \emph{lower Newton degree} and \emph{upper Newton degree} of $f$, respectively.

\subsubsection{Normalized type $C$ weight function}

The positive root system of the type $C_N$ is 
\[
\Delta^{+}
=
\{\, t_k \pm t_j \mid 1 \le j < k \le N \,\}
\;\cup\;
\{\, 2t_j \mid 1 \le j \le N \,\}.
\]

The simple roots are
\[
\alpha_0=2t_1,
\qquad
\alpha_i=t_{i+1}-t_i
\quad
(1\leq i\leq N-1).
\]

For each positive root $\alpha\in\Delta^+$, let $s_\alpha\in W(X_N)$
denote the reflection corresponding to $\alpha$. If
$
\alpha=w(\alpha_i)
$
for some $w\in W(X_N)$ and some simple root $\alpha_i$, then
$
s_\alpha=ws_iw^{-1}.
$
We let $s_\alpha$ act on
$\lambda\in\mathcal{SP}_N$ through the Weyl group action defined above.

The inversion set
$
\Inv(\lambda):= \{ \alpha \in \Delta^+ \mid s_\alpha \lambda < \lambda \},
$
is explicitly written as follows:
\begin{align}
\Inv(\lambda) 
&= \{\, 2t_k \mid k \in \mathbf{I} \,\} \cup \{\, t_j + t_k \mid 1 \le j < k \le N, \, j \in \mathbf{I}, \, k \in 
\mathbf{I} \,\} \nonumber \\
&\quad \cup \{\, t_k - t_j \mid 1 \le j < k \le N, \, j \notin \mathbf{I}, \, k \in \mathbf{I} \,\}. \notag
\end{align}
Therefore,
\begin{align}
\prod_{\alpha \in \Inv(\lambda)} (1-e^{-\alpha}) 
=& \prod_{j=1}^r (1 - w_{\lambda_j}^{-2}) \prod_{1 \le j < k \le r} (1 - w_{\lambda_j}^{-1} w_{\lambda_k}^{-1}) \prod_{j=1}^r \prod_{k \in \{1,\dots,\lambda_j-1\} \setminus \mathbf{I}} (1 - w_{\lambda_j}^{-1} w_k) \nonumber \\
=& \prod_{1 \le j \le k \le r} (1-w_{i_j}^{-1} w_{i_k}^{-1})  \prod_{j \in \{1, 2, \dots, N\} \setminus \mathbf{I}} \prod_{l: i_l > j} (1-w_{i_l}^{-1} w_j).
\label{typecfirstblock}
\end{align}

On the other hand, the ascent set
$
T(\lambda) := \{ \alpha \in \Delta^+ \mid s_\alpha \lambda > \lambda \},
$
is
\begin{align}
T(\lambda) 
=& \{\, 2t_k \mid k \in \{1, 2, \dots, N\} \setminus \mathbf{I} \,\} \cup \{\, t_j + t_k \mid 1 \le j < k \le N, \, j \notin \mathbf{I}, \, k \notin \mathbf{I} \,\} \nonumber \\
& \cup \{\, t_k - t_j \mid 1 \le j < k \le N, \, j \in \mathbf{I}, \, k \notin \mathbf{I} \,\}, \notag
\end{align}
Similarly,
\begin{align}
\prod_{\alpha \in T(\lambda)} (1 - q e^{+\alpha}) 
&= \prod_{k \in \{1, 2, \dots, N\} \setminus \mathbf{I}} (1 - q w_k^2) \prod_{\substack{j < k \\ j, k \in \{1, 2, \dots, 
N\} \setminus \mathbf{I}}} (1 - q w_j w_k) \nonumber \\
&\quad \times \prod_{k \in \{1, 2, \dots, N\} \setminus \mathbf{I}} \prod_{\substack{l=1 \\ \lambda_l < k}}^r (1 - q w_k w_{\lambda_l}^{-1}) \nonumber \\
&= \prod_{\substack{j \le k \\ j, k \in \{1, 2, \dots, N\} \setminus \mathbf{I}}} (1 - q w_j w_{k})  \prod_{j \in \{1, 2, \dots, N\} \setminus \mathbf{I}} \prod_{l: i_l < j} (1-q w_{i_l}^{-1} w_j).
\label{typecsecondblock}
\end{align}

\begin{Definition}
We define type $C$ $\mathbf{M}$-smallness only in the equal-cardinality case $s=r$.  Thus, for an $r$-element subset $\mathbf{M}=\{m_{1}<\cdots<m_{r}\}\subset\{1,\ldots,N\}$, we define the Laurent polynomial $f$ to be type $C$ \emph{$\mathbf{M}$-small} if
\begin{align}
\varphi_{\mathbf{M}}(\mathcal{N}(f)) \subset [-r(r+1), (N-r)(2N-r+2)-1].
\end{align}
\end{Definition}

\begin{Definition}

We introduce the normalization factor $\mathcal{E}^{C}_{N,n,r}(\mathbf{u}, \mathbf{w})$ and define the normalized weight function as follows:
\begin{equation}
\widetilde{WC}_{N,n,r}(\mathbf{u}|\mathbf{w}|\mathbf{I}) := WC_{N,n,r}(\mathbf{u}|\mathbf{w}|\mathbf{I}) \cdot \mathcal{E}^C_{N,n,r}(\mathbf{u}, \mathbf{w})^{-1}.
\label{typeCnormalizeddefinition}
\end{equation}
Here, the normalization factor $\mathcal{E}_{N,n,r}^C(\mathbf{u}, \mathbf{w})$, which is symmetric with respect to $\mathbf{u}$ and independent of $\mathbf{I}$, is defined as follows:
\begin{equation}
\mathcal{E}_{N,n,r}^C(\mathbf{u}, \mathbf{w}) =  \frac{1}{(1-q)^{n-r}} \frac{ \prod_{j,k=1}^n (1 - q u_j^{-1} u_k) \prod_{j=1}^n \prod_{k=1}^N (1 - q u_j w_k) }{ \prod_{1 \le j < k \le n} (1 - q u_j u_k) \prod_{1 \le j \le k \le N} (1 - q w_j w_k) }.
\label{typeCnormalizationfactor}
\end{equation}

\end{Definition}

\begin{thm} \label{typeCmainproperties}
Let $|\mathbf{M}|=|\mathbf{I}|=r$ and $\mathbf{M}=\{m_1<\cdots<m_r\}\subset\{1,\dots,N\}$. Then
$\widetilde{WC}_{N,n,r}
(\mathbf{w}_{\mathbf{M}}|\mathbf{w}|\mathbf{I})$ satisfies the following properties. 

$\mathrm{(i)}$
We have
\begin{align}
\widetilde{WC}_{N,n,r}(\mathbf{w}_{\mathbf{M}}|\mathbf{w}|\mathbf{I})=0,
\end{align}
unless $\mathbf{M} \ge \mathbf{I}$, and
\begin{equation}
\widetilde{WC}_{N,n,r}(\mathbf{w}_{\mathbf{I}}|\mathbf{w}|\mathbf{I}) = \prod_{\alpha \in \Inv(\lambda)} (1 - e^{-\alpha}) \prod_{\alpha \in T(\lambda)} (1 - q e^{+\alpha}). \label{normalizationWC}
\end{equation}

$\mathrm{(ii)}$
$\widetilde{WC}_{N,n,r}
(\mathbf{w}_{\mathbf{M}}|\mathbf{w}|\mathbf{I})$
is divisible by
\begin{align}
&\prod_{1\le j\le k\le r}
(1-w_{m_j}^{-1}w_{m_k}^{-1})
\prod_{\substack{j \le k \\ j, k \in \{1, 2, \dots, N\} \setminus \mathbf{M}}} (1 - q w_j w_{k}) 
\prod_{j \in \{1, 2, \dots, N\} \setminus \mathbf{M}} \prod_{l: m_l < j} (1-q w_{m_l}^{-1} w_j).
\label{typeCfullDivisibility}
\end{align}

$\mathrm{(iii)}$
$\widetilde{WC}_{N,n,r}
(\mathbf{w}_{\mathbf{M}}|\mathbf{w}|\mathbf{I})$ is type $C$ $\mathbf{M}$-small if $\mathbf{M} \neq \mathbf{I}$.
\end{thm}

The proof of Theorem~\ref{typeCmainproperties} is given in Section~\ref{sec:normalized-algebraic-properties}. At the end of the type $C$ part of Section~\ref{sec:normalized-specializations}, we explain explicitly how the specialization results reduce at $s=r$ to recover $\mathrm{(i)}$--$\mathrm{(iii)}$.

The following uniqueness statement concerns only the case $|\mathbf{M}|=r$.  It shows that these specializations are characterized by the properties in Theorem~\ref{typeCmainproperties}, together with the exchange congruence in condition $\mathrm{(iii)}$ below.

\begin{prop}
Fix $\mathbf{I}$, and suppose that a collection of Laurent polynomials
$(F_{\mathbf{M}})_{|\mathbf{M}|=r}$ satisfies the following conditions.
\begin{enumerate}
\item[(i)] If $\mathbf{M}\not\ge \mathbf{I}$, then $F_{\mathbf{M}}=0$, and
\begin{align*}
F_{\mathbf{I}}
=&
\prod_{1\le j\le k\le r}
(1-w_{i_j}^{-1}w_{i_k}^{-1})
\prod_{j\in \mathbf{I}^{\mathrm c}}
\prod_{l:i_l>j}
(1-w_{i_l}^{-1}w_j) \notag \\
&\times\prod_{\substack{j\le k\\j,k\in \mathbf{I}^{\mathrm c}}}
(1-qw_jw_k)
\prod_{j\in \mathbf{I}^{\mathrm c}}
\prod_{l:i_l<j}
(1-qw_{i_l}^{-1}w_j).
\end{align*}

\item[(ii)] The Laurent polynomial $F_{\mathbf{M}}$ is divisible by
\[
\prod_{1\le j\le k\le r}
(1-w_{m_j}^{-1}w_{m_k}^{-1})
\prod_{\substack{j\le k\\j,k\in \mathbf{M}^{\mathrm c}}}
(1-qw_jw_k)
\prod_{j\in \mathbf{M}^{\mathrm c}}
\prod_{l:m_l<j}
(1-qw_{m_l}^{-1}w_j).
\]

\item[(iii)] For $k\in \mathbf{M}$ and $j\notin \mathbf{M}$,
\[
F_{\mathbf{M}}-F_{(\mathbf{M}\setminus\{k\})\cup\{j\}}
\in(1-w_j/w_k).
\]

\item[(iv)] If $\mathbf{M}\ne \mathbf{I}$, then
\[
\varphi_{\mathbf{M}}\bigl(\mathcal N(F_{\mathbf{M}})\bigr)
\subset
[-r(r+1),(N-r)(2N-r+2)-1].
\]
\end{enumerate}
Then there is at most one such collection
$(F_{\mathbf{M}})_{|\mathbf{M}|=r}$.
\end{prop}

\begin{proof}
The proof is essentially the same as in \cite{RTV}.
Let $(F_{\mathbf{M}})_{|\mathbf{M}|=r}$ and
$(G_{\mathbf{M}})_{|\mathbf{M}|=r}$ be two collections satisfying
conditions~(i)--(iv), and set
$
\Delta_{\mathbf{M}}=F_{\mathbf{M}}-G_{\mathbf{M}}.
$
Choose a minimal $\mathbf{M}$ in the componentwise order among the indices for
which $\Delta_{\mathbf{M}}\ne0$.
The diagonal values agree, so $\mathbf{M}\ne \mathbf{I}$.
If $j\notin \mathbf{M}$, $k\in \mathbf{M}$, and $j<k$, set
$
\mathbf{M}'=(\mathbf{M}\setminus\{k\})\cup\{j\}.
$
Then $\mathbf{M}'<\mathbf{M}$, and hence
$\Delta_{\mathbf{M}'}=0$ by the minimality of $\mathbf{M}$.
Applying condition~(iii) to the two collections and subtracting the resulting
congruences gives
$\displaystyle
\Delta_{\mathbf{M}}-\Delta_{\mathbf{M}'}
\in(1-w_j/w_k).
$
Therefore $\Delta_{\mathbf{M}}$ is divisible by $1-w_k^{-1}w_j$.
Multiplying over all such pairs shows that $\Delta_{\mathbf{M}}$ is divisible by
$\displaystyle
\prod_{j\in \mathbf{M}^{\mathrm c}}
\prod_{l:m_l>j}
(1-w_{m_l}^{-1}w_j).
$
The coefficient ring
$\mathbb{C}(q)[w_1^{\pm1},\ldots,w_N^{\pm1}]$ is a unique factorization domain,
and the factors $1-w_{m_l}^{-1}w_j$ obtained here are relatively prime to the
factors in condition~(ii). Therefore $\Delta_{\mathbf{M}}$ is divisible by
\begin{align*}
&\prod_{1\le j\le k\le r}
(1-w_{m_j}^{-1}w_{m_k}^{-1})
\prod_{j\in \mathbf{M}^{\mathrm c}}
\prod_{l:m_l>j}
(1-w_{m_l}^{-1}w_j)
\prod_{\substack{j\le k\\j,k\in \mathbf{M}^{\mathrm c}}}
(1-qw_jw_k)
\prod_{j\in \mathbf{M}^{\mathrm c}}
\prod_{l:m_l<j}
(1-qw_{m_l}^{-1}w_j).
\end{align*}

The degree interval of the first product is
$
[-r(r+1),0].
$
The degree interval of the second product is
$
\left[
0,\,
\left|\{(j,l): j\in\mathbf{M}^{\mathrm c},\ m_l>j\}\right|
\right].
$
The degree interval of the third product is
$
[0,2(N-r)(N-r+1)].
$
The degree interval of the fourth product is
$
\left[
0,\,
\left|\{(j,l): j\in\mathbf{M}^{\mathrm c},\ m_l<j\}\right|
\right].
$
The sum of the upper endpoints of the second and fourth intervals is $r(N-r)$.
Consequently, the product of all four factors has degree interval
$
[-r(r+1),(N-r)(2N-r+2)].
$
Thus $\Delta_{\mathbf{M}}$ is the displayed product times a Laurent polynomial $g$.
Suppose that $g\ne0$ and write
$
\varphi_{\mathbf{M}}\bigl(\mathcal N(g)\bigr)=[a,b].
$
Then $a\le b$. The additivity of the lower and upper Newton degrees under
multiplication gives
$
\varphi_{\mathbf{M}}\bigl(\mathcal N(\Delta_{\mathbf{M}})\bigr)
=
[-r(r+1)+a,(N-r)(2N-r+2)+b].
$
Since $\mathbf{M}\ne\mathbf{I}$, condition~(iv) applies to both candidates and
therefore also to their difference. Hence
$
\varphi_{\mathbf{M}}\bigl(\mathcal N(\Delta_{\mathbf{M}})\bigr)
\subset
[-r(r+1),(N-r)(2N-r+2)-1].
$
Comparing the two endpoints yields $a\ge0$ and $b\le-1$, contradicting
$a\le b$. Thus $g=0$, and hence $\Delta_{\mathbf{M}}=0$, contrary to the
choice of $\mathbf{M}$.
\end{proof}

\begin{prop}
Fix $\mathbf{I}$ and set
\[
F_{\mathbf{M}}
=
\widetilde{WC}_{N,n,r}
(\mathbf{w}_{\mathbf{M}}|\mathbf{w}|\mathbf{I})
\qquad (|\mathbf{M}|=r).
\]
This collection satisfies all conditions (i)--(iv) in the preceding uniqueness proposition.
Consequently, these specialization values are uniquely determined by these conditions.
\end{prop}

\begin{proof}
Conditions (i), (ii), and (iv) follow from Theorem~\ref{typeCmainproperties}.

Condition (iii) follows directly from symmetry in the $\mathbf{u}$-variables. Let $k\in\mathbf{M}$ and $j\notin\mathbf{M}$, and set
$
\mathbf{M}'=(\mathbf{M}\setminus\{k\})\cup\{j\}.
$
When $w_j=w_k$, the sequences $\mathbf{w}_{\mathbf{M}}$ and
$\mathbf{w}_{\mathbf{M}'}$ agree up to permutation. Therefore
$F_{\mathbf{M}}-F_{\mathbf{M}'}$ vanishes when $w_j=w_k$. Since both terms are Laurent polynomials,
$
F_{\mathbf{M}}-F_{\mathbf{M}'}
\in (1-w_j/w_k).
$
\end{proof}

\subsubsection{Normalized type $B$ weight function}

The proofs of the type $B$ statements are essentially the same as the corresponding type $C$ arguments, with the type $B$ specialization and normalization factors in place of the type $C$ ones. For brevity, we omit these proofs.

The positive root system of the type $B_N$ is
\[
\Delta^{+}
=
\{\, t_k \pm t_j \mid 1 \le j < k \le N \,\}
\;\cup\;
\{\, t_j \mid 1 \le j \le N \,\}.
\]

The simple roots are
\[
\alpha_0=t_1,
\qquad
\alpha_i=t_{i+1}-t_i
\quad
(1\leq i\leq N-1).
\]

The inversion set
$
\Inv(\lambda):= \{ \alpha \in \Delta^+ \mid s_\alpha \lambda < \lambda \},
$
is explicitly written as follows:
\begin{align}
\Inv(\lambda) 
&= \{\, t_k \mid k \in \mathbf{I} \,\} \cup \{\, t_j + t_k \mid 1 \le j < k \le N, \, j \in \mathbf{I}, \, k \in \mathbf{I} \,\} \nonumber \\
&\quad \cup \{\, t_k - t_j \mid 1 \le j < k \le N, \, j \notin \mathbf{I}, \, k \in \mathbf{I} \,\}. \notag
\end{align}
Therefore,
\begin{align}
\prod_{\alpha \in \Inv(\lambda)} (1-e^{-\alpha})
=& \prod_{j=1}^r (1 - w_{\lambda_j}^{-1}) \prod_{1 \le j < k \le r} (1 - w_{\lambda_j}^{-1} w_{\lambda_k}^{-1}) \prod_{j=1}^r \prod_{k \in \{1,\dots,\lambda_j-1\} \setminus \mathbf{I}} (1 - w_{\lambda_j}^{-1} w_k) \nonumber \\
=& \prod_{j=1}^r (1 - w_{i_j}^{-1}) \prod_{1 \le j < k \le r} (1-w_{i_j}^{-1} w_{i_k}^{-1})  \prod_{j \in \{1, 2, \dots, N\} \setminus \mathbf{I}} \prod_{l: i_l > j} (1-w_{i_l}^{-1} w_j).
\end{align}

On the other hand, the ascent set
$
T(\lambda) := \{ \alpha \in \Delta^+ \mid s_\alpha \lambda > \lambda \},
$
is
\begin{align}
T(\lambda) 
=& \{\, t_k \mid k \in \{1, 2, \dots, N\} \setminus \mathbf{I} \,\} \cup \{\, t_j + t_k \mid 1 \le j < k \le N, \, j \notin \mathbf{I}, \, k \notin \mathbf{I} \,\} \nonumber \\
& \cup \{\, t_k - t_j \mid 1 \le j < k \le N, \, j \in \mathbf{I}, \, k \notin \mathbf{I} \,\}, \notag
\end{align}
Similarly,
\begin{align}
\prod_{\alpha \in T(\lambda)} (1 - q e^{+\alpha}) 
&= \prod_{k \in \{1, 2, \dots, N\} \setminus \mathbf{I}} (1 - q w_k) \prod_{\substack{j < k \\ j, k \in \{1, 2, \dots, N\} \setminus \mathbf{I}}} (1 - q w_j w_k) \nonumber \\
&\quad \times \prod_{k \in \{1, 2, \dots, N\} \setminus \mathbf{I}} \prod_{\substack{l=1 \\ \lambda_l < k}}^r (1 - q w_k w_{\lambda_l}^{-1}) \nonumber \\
&= \prod_{j \in \{1, 2, \dots, N\} \setminus \mathbf{I}} (1 - q w_j) \prod_{\substack{j < k \\ j, k \in \{1, 2, \dots, N\} \setminus \mathbf{I}}} (1 - q w_j w_{k}) \text{\, \,} \nonumber \\
&\quad
 \times \prod_{j \in \{1, 2, \dots, N\} \setminus \mathbf{I}} \prod_{l: i_l < j} (1-q w_{i_l}^{-1} w_j).
\end{align}

\begin{Definition}
We define type $B$ $\mathbf{M}$-smallness only in the equal-cardinality case $s=r$.  Thus, for an $r$-element subset $\mathbf{M}=\{m_{1}<\cdots<m_{r}\}\subset\{1,\ldots,N\}$, we define the Laurent polynomial $f$ to be type $B$ \emph{$\mathbf{M}$-small} if
\begin{align}
\varphi_{\mathbf{M}}(\mathcal{N}(f)) \subset [-r^2, (N-r)(2N-r)-1].
\end{align}
\end{Definition}

\begin{Definition}

We introduce the normalization factor $\mathcal{E}^B_{N,n,r}(\mathbf{u}, \mathbf{w})$ and define the normalized weight function as follows:
\begin{equation}
\widetilde{WB}_{N,n,r}(\mathbf{u}|\mathbf{w}|\mathbf{I}) := WB_{N,n,r}(\mathbf{u}|\mathbf{w}|\mathbf{I}) \cdot \mathcal{E}^B_{N,n,r}(\mathbf{u}, \mathbf{w})^{-1}.
\label{typeBnormalizeddefinition}
\end{equation}
Here, the normalization factor $\mathcal{E}^B_{N,n,r}(\mathbf{u}, \mathbf{w})$, which is symmetric with respect to the variables $\mathbf{u}$ and independent of $\mathbf{I}$, is defined as follows:
\begin{align}
&\mathcal{E}^B_{N,n,r}(\mathbf{u}, \mathbf{w}) 
:= (1-q)^{-(n-r)}
\nonumber \\
&\times\frac{ \prod_{j,k=1}^n (1 - q u_j^{-1} u_k) \prod_{j=1}^n \prod_{k=1}^N (1 - q u_j w_k) \prod_{j=1}^n (1 - q u_j) }{ \prod_{1 \le j < k \le n} (1 - q u_j u_k) \prod_{j=1}^n (1 - q u_j^2) \prod_{1 \le j < k \le N} (1 - q w_j w_k) \prod_{k=1}^N (1 - q w_k) }.
\label{typeBnormalizationfactor}
\end{align}

\end{Definition}

\begin{thm} \label{typeBmainproperties}
Let $|\mathbf{M}|=|\mathbf{I}|=r$ and $\mathbf{M}=\{m_1<\cdots<m_r\}\subset\{1,\dots,N\}$. Then
$\widetilde{WB}_{N,n,r}
(\mathbf{w}_{\mathbf{M}}|\mathbf{w}|\mathbf{I})$ satisfies the following properties. 

$\mathrm{(i)}$
We have
\begin{align}
\widetilde{WB}_{N,n,r}(\mathbf{w}_{\mathbf{M}}|\mathbf{w}|\mathbf{I})=0,
\end{align}
unless $\mathbf{M} \ge \mathbf{I}$, and
\begin{equation}
\widetilde{WB}_{N,n,r}(\mathbf{w}_{\mathbf{I}}|\mathbf{w}|\mathbf{I}) = \prod_{\alpha \in \Inv(\lambda)} (1 - e^{-\alpha}) \prod_{\alpha \in T(\lambda)} (1 - q e^{+\alpha}).
\label{normalizationWB}
\end{equation}

$\mathrm{(ii)}$
$\widetilde{WB}_{N,n,r}
(\mathbf{w}_{\mathbf{M}}|\mathbf{w}|\mathbf{I})$
is divisible by
\begin{align}
&\prod_{j=1}^{r}(1-w_{m_j}^{-1})
\prod_{1\le j<k\le r}(1-w_{m_j}^{-1}w_{m_k}^{-1})
\prod_{j \in \{1, 2, \dots, N\} \setminus \mathbf{M}} (1 - q w_j) \nonumber\\
&\quad\times
\prod_{\substack{j < k \\ j, k \in \{1, 2, \dots, N\} \setminus \mathbf{M}}} (1 - q w_j w_{k})
 \prod_{j \in \{1, 2, \dots, N\} \setminus \mathbf{M}} \prod_{l: m_l < j} (1-q w_{m_l}^{-1} w_j).
\label{typeBfullDivisibility}
\end{align}

$\mathrm{(iii)}$
$\widetilde{WB}_{N,n,r}
(\mathbf{w}_{\mathbf{M}}|\mathbf{w}|\mathbf{I})$ is type $B$ $\mathbf{M}$-small if $\mathbf{M} \neq \mathbf{I}$.

\end{thm}

The proof of Theorem~\ref{typeBmainproperties} is omitted, since it is essentially the same as the proof in the type $C$ case. At the end of the type $B$ part of Section~\ref{sec:normalized-specializations}, we explain explicitly how the specialization results reduce at $s=r$ to recover $\mathrm{(i)}$--$\mathrm{(iii)}$.

The corresponding uniqueness statement for type $B$ is as follows.

\begin{prop}
Fix $\mathbf{I}$, and suppose that a collection of Laurent polynomials
$(F_{\mathbf{M}})_{|\mathbf{M}|=r}$ satisfies the following conditions.
\begin{enumerate}
\item[(i)] If $\mathbf{M}\not\ge\mathbf{I}$, then $F_{\mathbf{M}}=0$, and
\begin{align*}
F_{\mathbf{I}}
={}&
\prod_{j=1}^{r}(1-w_{i_j}^{-1})
\prod_{1\le j<k\le r}(1-w_{i_j}^{-1}w_{i_k}^{-1})
\prod_{j\in\mathbf{I}^{\mathrm c}}
\prod_{l:i_l>j}(1-w_{i_l}^{-1}w_j)
\notag\\
&\times
\prod_{j\in\mathbf{I}^{\mathrm c}}(1-qw_j)
\prod_{\substack{j<k\\j,k\in\mathbf{I}^{\mathrm c}}}(1-qw_jw_k)
\prod_{j\in\mathbf{I}^{\mathrm c}}
\prod_{l:i_l<j}(1-qw_{i_l}^{-1}w_j).
\end{align*}

\item[(ii)] The Laurent polynomial $F_{\mathbf{M}}$ is divisible by
\begin{align*}
&\prod_{j=1}^{r}(1-w_{m_j}^{-1})
\prod_{1\le j<k\le r}(1-w_{m_j}^{-1}w_{m_k}^{-1})
\prod_{j\in\mathbf{M}^{\mathrm c}}(1-qw_j)
\notag\\
&\qquad\times
\prod_{\substack{j<k\\j,k\in\mathbf{M}^{\mathrm c}}}(1-qw_jw_k)
\prod_{j\in\mathbf{M}^{\mathrm c}}
\prod_{l:m_l<j}(1-qw_{m_l}^{-1}w_j).
\end{align*}

\item[(iii)] For $k\in\mathbf{M}$ and $j\notin\mathbf{M}$,
\[
F_{\mathbf{M}}-F_{(\mathbf{M}\setminus\{k\})\cup\{j\}}
\in(1-w_j/w_k).
\]

\item[(iv)] If $\mathbf{M}\ne\mathbf{I}$, then
\[
\varphi_{\mathbf{M}}\bigl(\mathcal N(F_{\mathbf{M}})\bigr)
\subset[-r^2,(N-r)(2N-r)-1].
\]
\end{enumerate}
Then there is at most one such collection
$(F_{\mathbf{M}})_{|\mathbf{M}|=r}$.
\end{prop}

\begin{prop}
Fix $\mathbf{I}$ and set
\[
F_{\mathbf{M}}
=
\widetilde{WB}_{N,n,r}(\mathbf{w}_{\mathbf{M}}|\mathbf{w}|\mathbf{I})
\qquad(|\mathbf{M}|=r).
\]
This collection satisfies all conditions (i)--(iv) in the preceding uniqueness proposition.
Consequently, these specialization values are uniquely determined by these conditions.
\end{prop}

\subsection{Algebraic properties from integrability}

In this subsection, we derive various algebraic properties of the symmetric functions
$F_{N,n,r}(\mathbf{u}|\mathbf{w}|\mathbf{I})$ using their realization as partition functions of the integrable lattice model.

For a sequence of variables $\mathbf{w}=(w_1, \dots, w_N)$, let $s_i$ ($1 \le i \le N-1$) be the simple transposition that swaps $w_i$ and $w_{i+1}$:
\[
s_i \mathbf{w} = ( w_1, \dots, w_{i-1}, w_{i+1}, w_i, \dots, w_N ).
\]

\begin{prop} \label{diddiffbeforenorm}  Let $i \in \{1, \dots, N-1\}$. \\
   (i)  If $i+1 = i_k \in \mathbf{I}$ and $i \notin \mathbf{I}$, we have
    \begin{align}
    &F_{N,n,r}(\mathbf{u}|\mathbf{w}|\mathbf{I}^{(i)}) \nonumber \\
    =&\frac{1}{w_i-w_{i+1}}
    \Big(
    (w_i-q w_{i+1}) F_{N,n,r}(\mathbf{u}|s_i \mathbf{w}|\mathbf{I})
    -(1-q)w_{i+1} F_{N,n,r}(\mathbf{u}|\mathbf{w}|\mathbf{I})
    \Big), \label{recrelone} \\
    &F_{N,n,r}(\mathbf{u}|\mathbf{w}|\mathbf{I}) \nonumber \\
    =&\frac{1}{q(w_i-w_{i+1})}
    \Big(
    (w_i-q w_{i+1}) F_{N,n,r}(\mathbf{u}|s_i \mathbf{w}|\mathbf{I}^{(i)})
    -(1-q)w_{i} F_{N,n,r}(\mathbf{u}|\mathbf{w}|\mathbf{I}^{(i)})
    \Big), \label{recreltwo}
    \end{align}
    where 
$\mathbf{I}^{(i)} = \{ i_1, \dots, i_{k-1}, i, i_{k+1}, \dots, i_r \}.
$ \\
\\
(ii) If $i,i+1 \in \mathbf{I}$ or $i,i+1 \not\in \mathbf{I}$
, we have
\begin{align}
F_{N,n,r}(\mathbf{u}|\mathbf{w}|\mathbf{I})
=F_{N,n,r}(\mathbf{u}|s_i \mathbf{w}|\mathbf{I}). \label{recrelthree}
\end{align}

\begin{figure}[!htbp]
  \centering
  \includegraphics[width=0.7\linewidth]{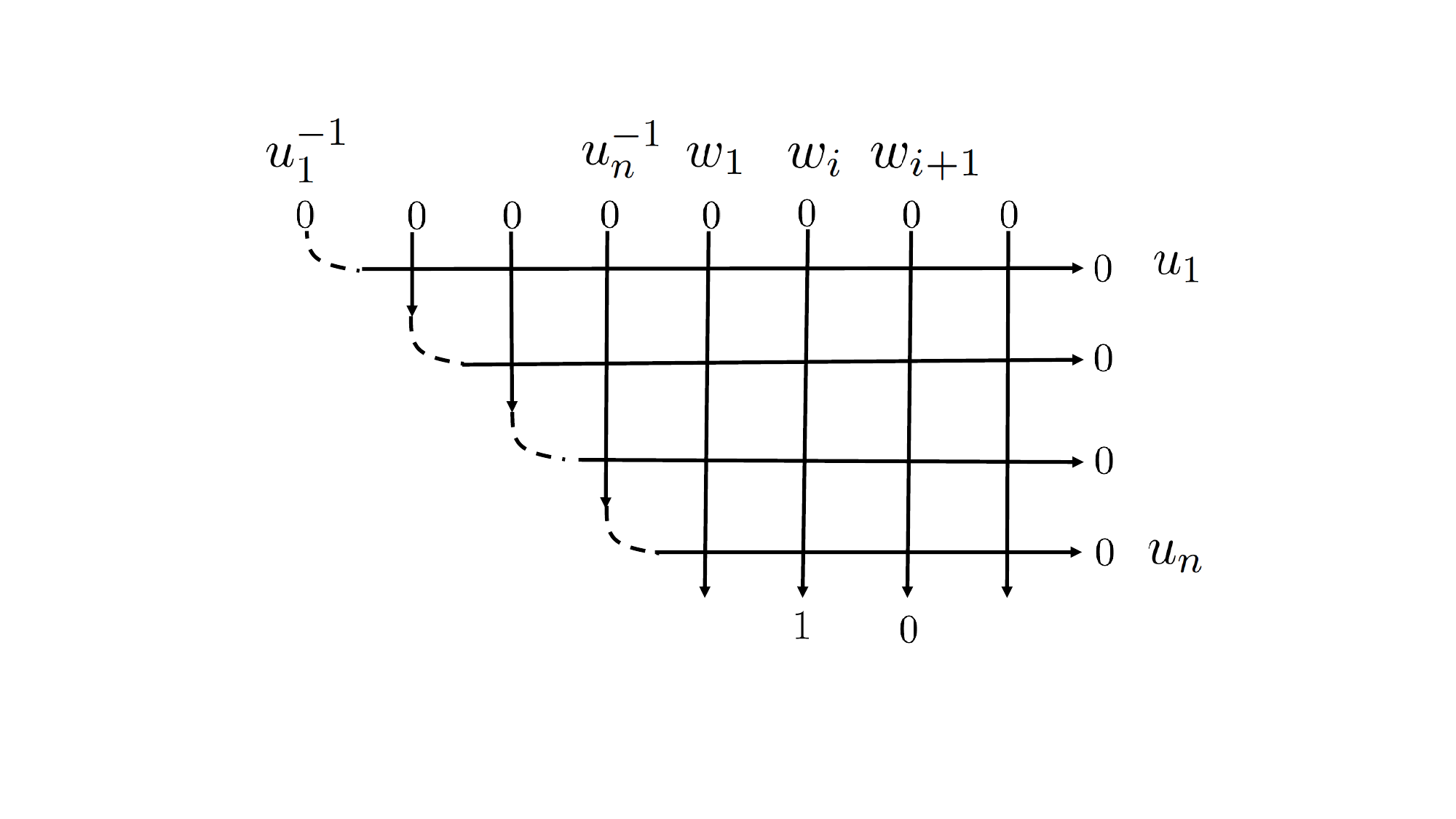}
\includegraphics[width=0.45\linewidth]{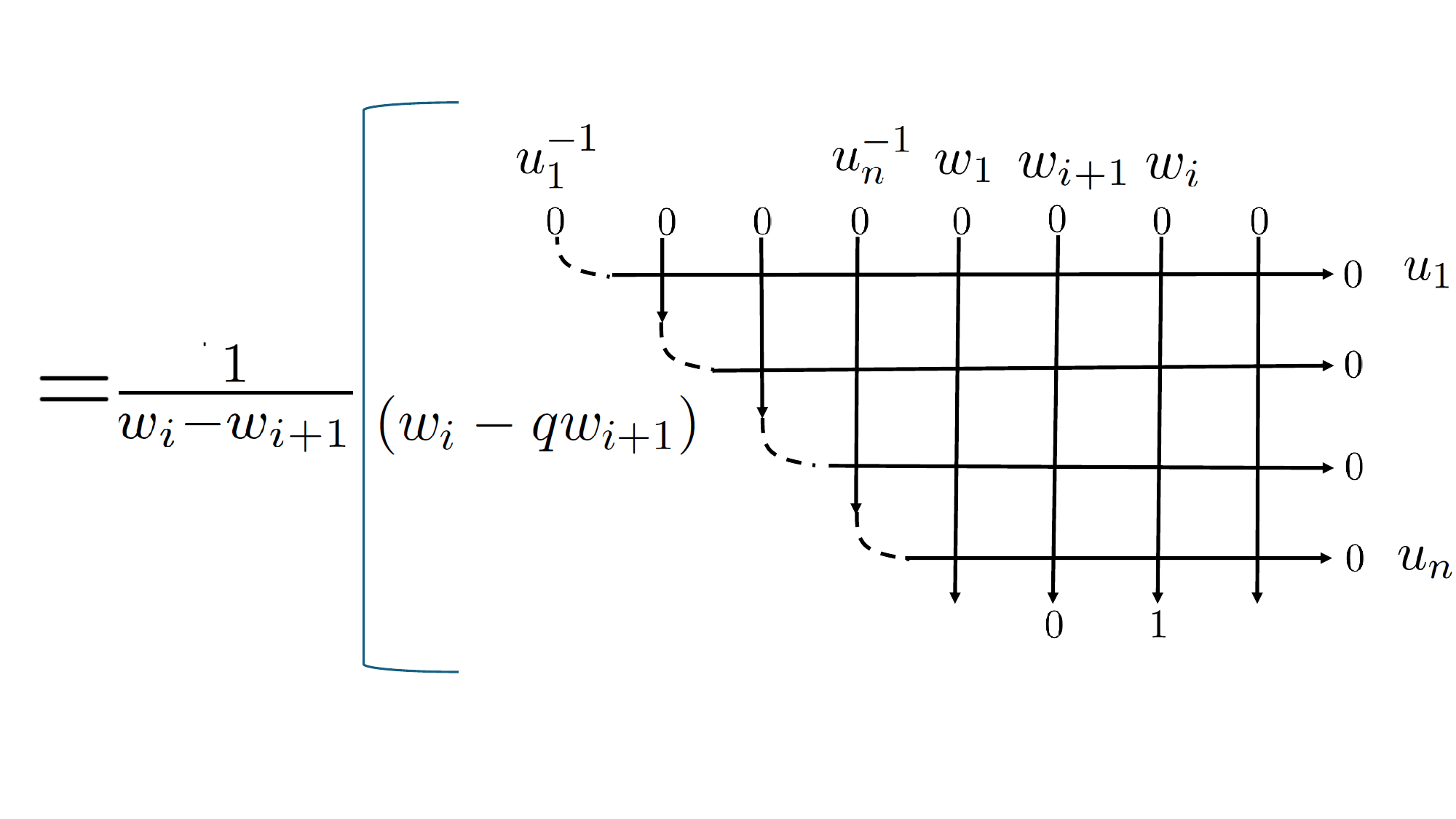}
\includegraphics[width=0.45\linewidth]{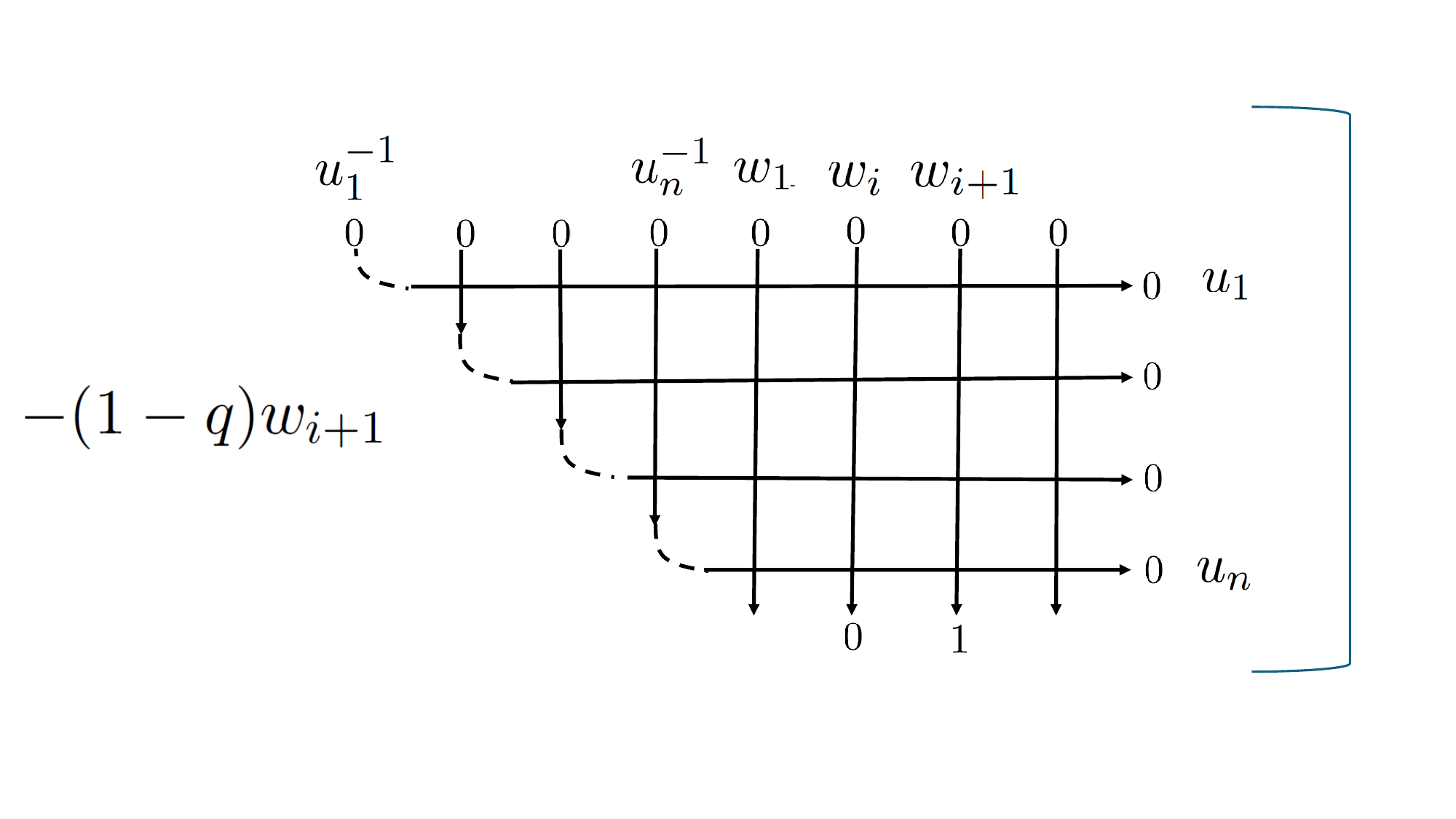}
  \caption{Recursive relation \eqref{recrelone}.}
  \label{fig:basicrecursion}
\end{figure}

\begin{figure}[!htbp]
  \centering
  \includegraphics[width=0.7\linewidth]{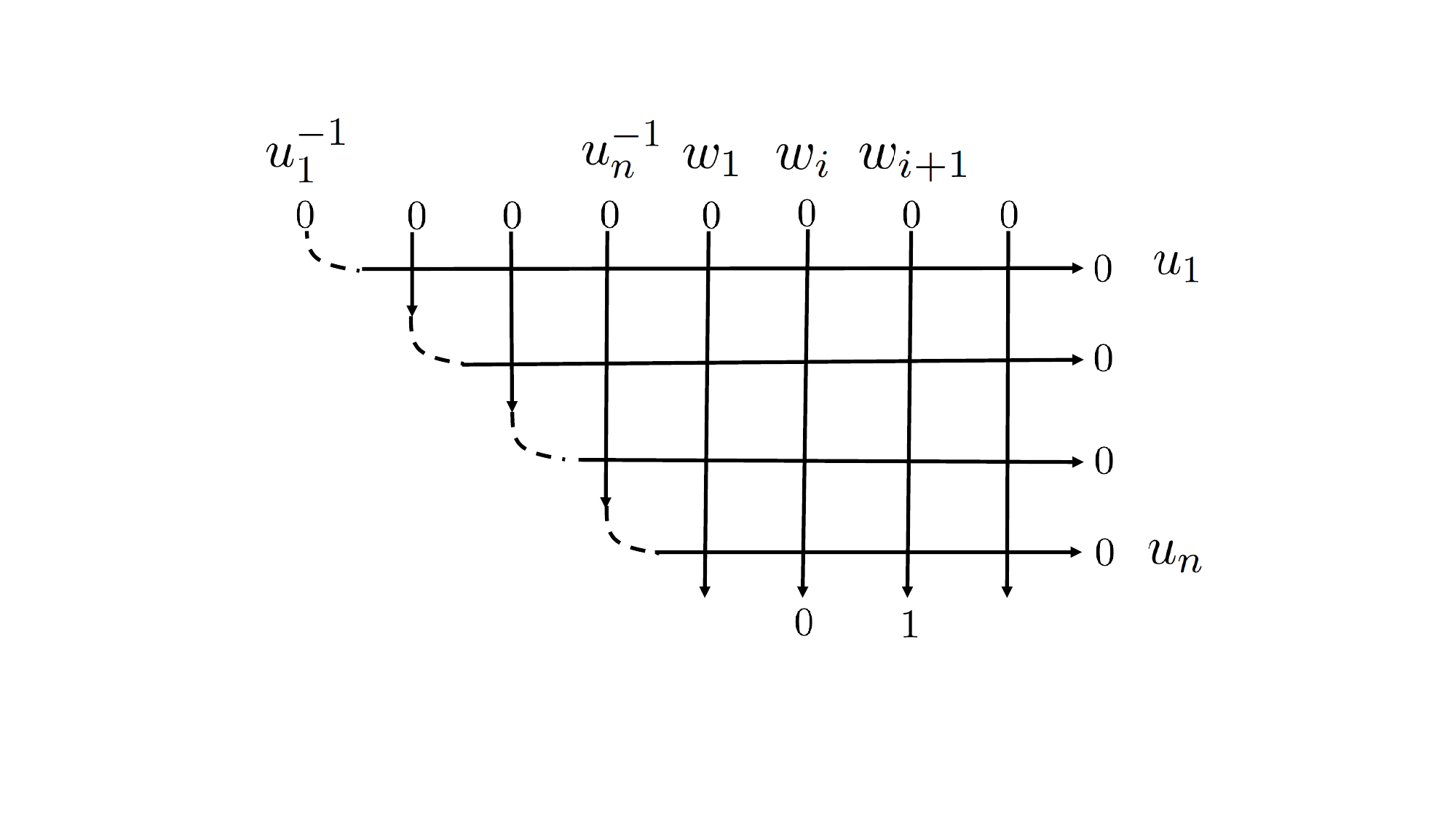}
\includegraphics[width=0.45\linewidth]{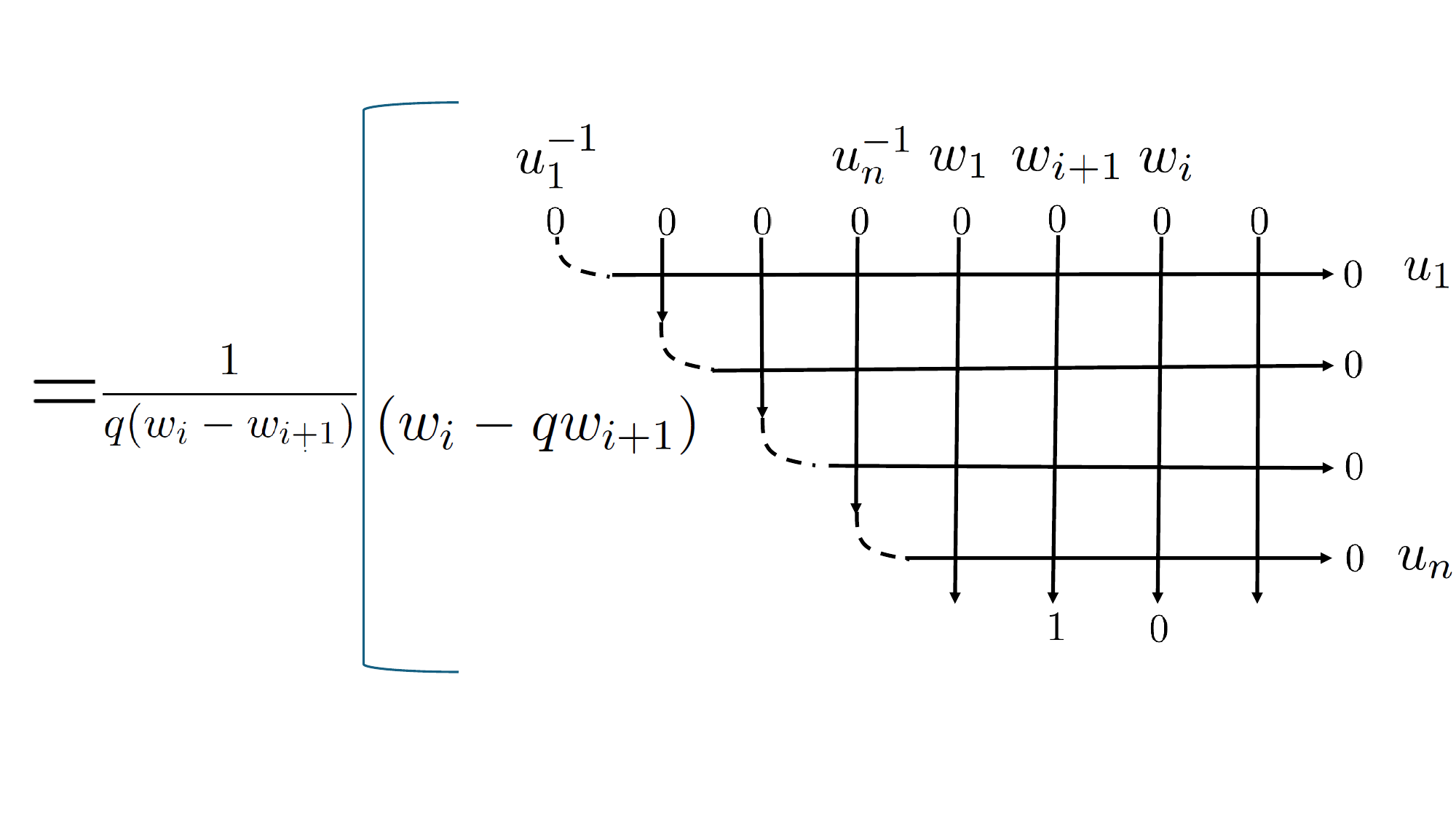}
\includegraphics[width=0.45\linewidth]{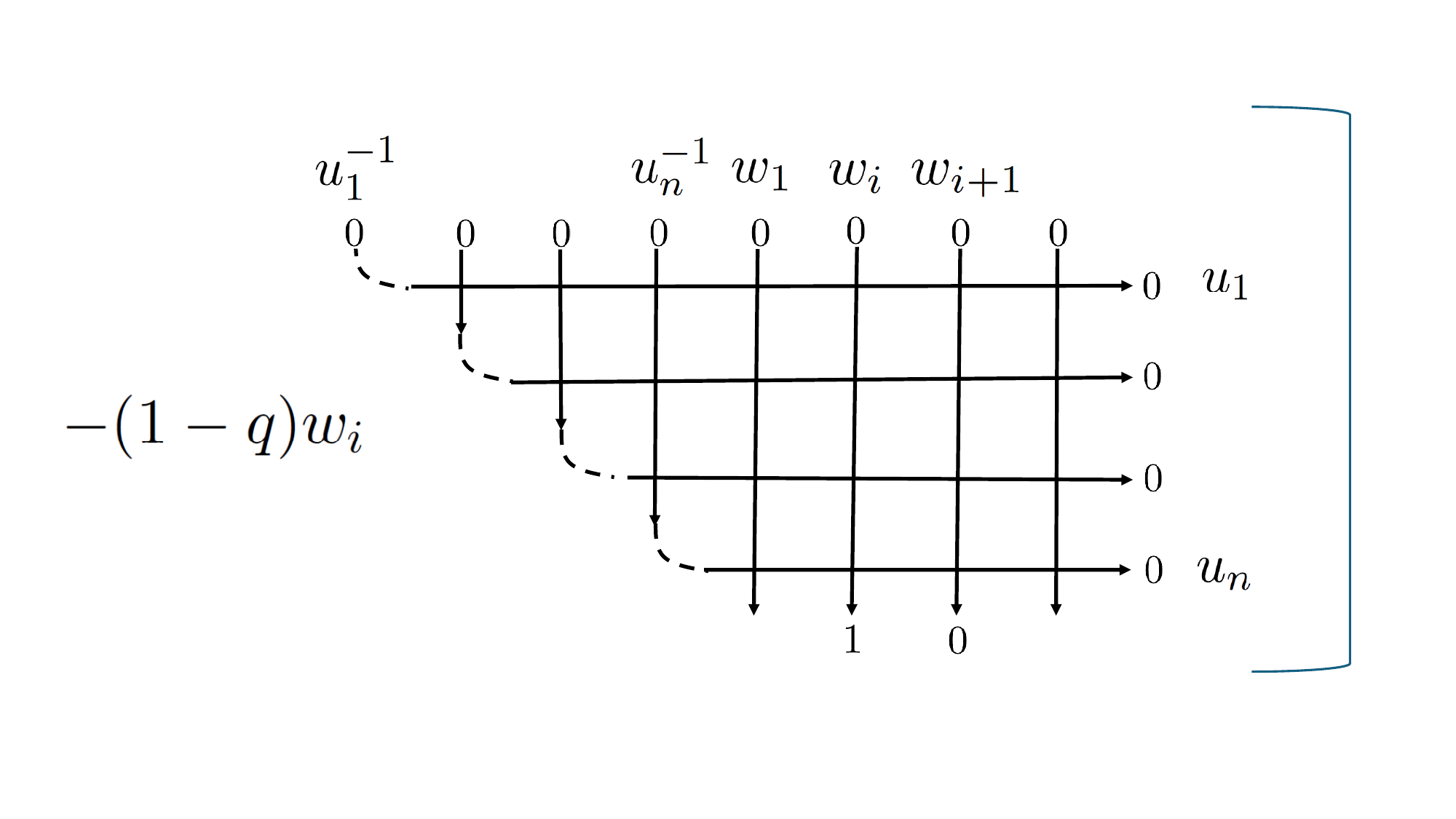}
  \caption{Recursive relation \eqref{recreltwo}.}
  \label{fig:basicrecursiontwo}
\end{figure}

\begin{figure}[!htbp]
  \centering
\includegraphics[width=0.45\linewidth]{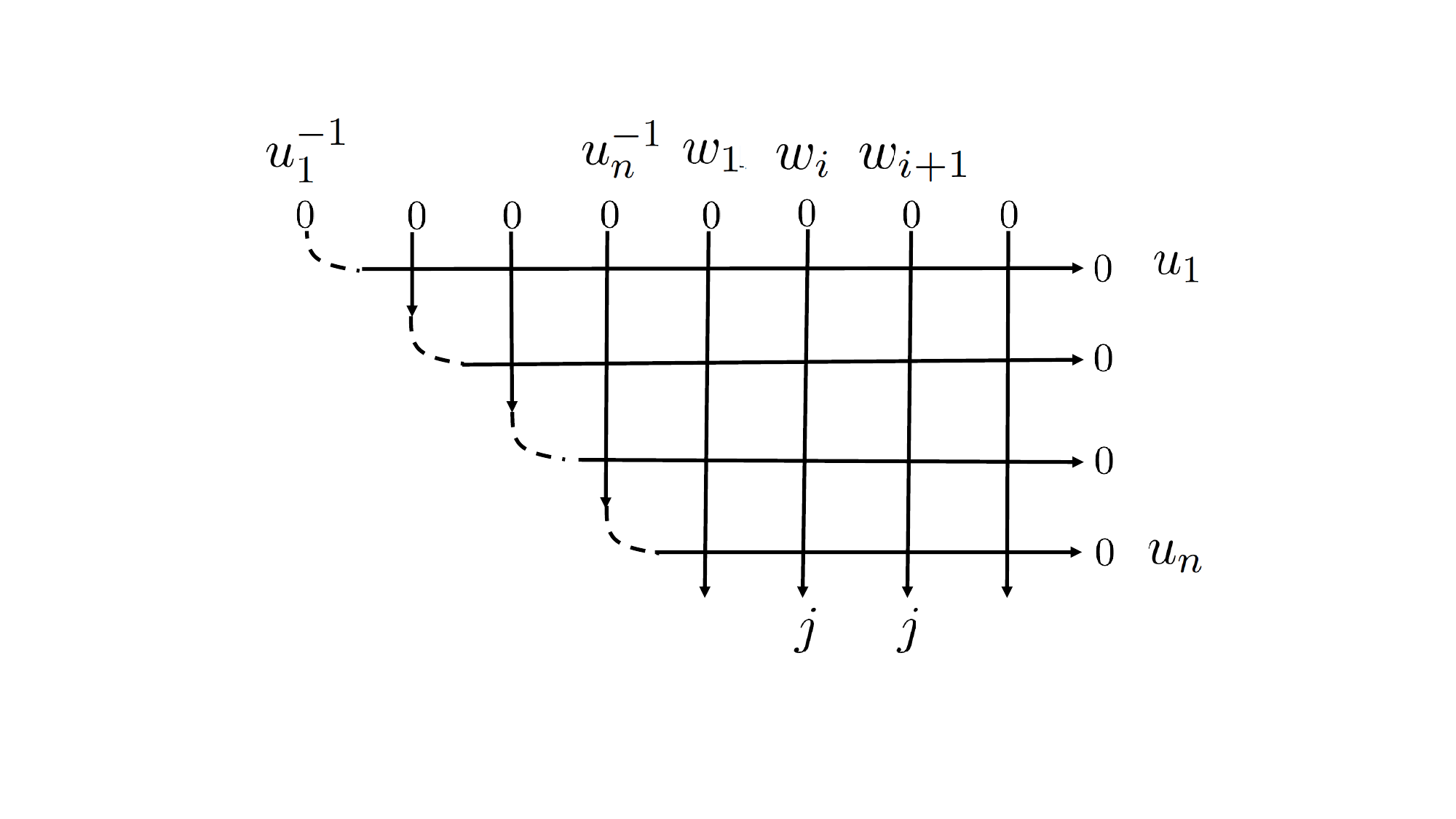}
\includegraphics[width=0.45\linewidth]{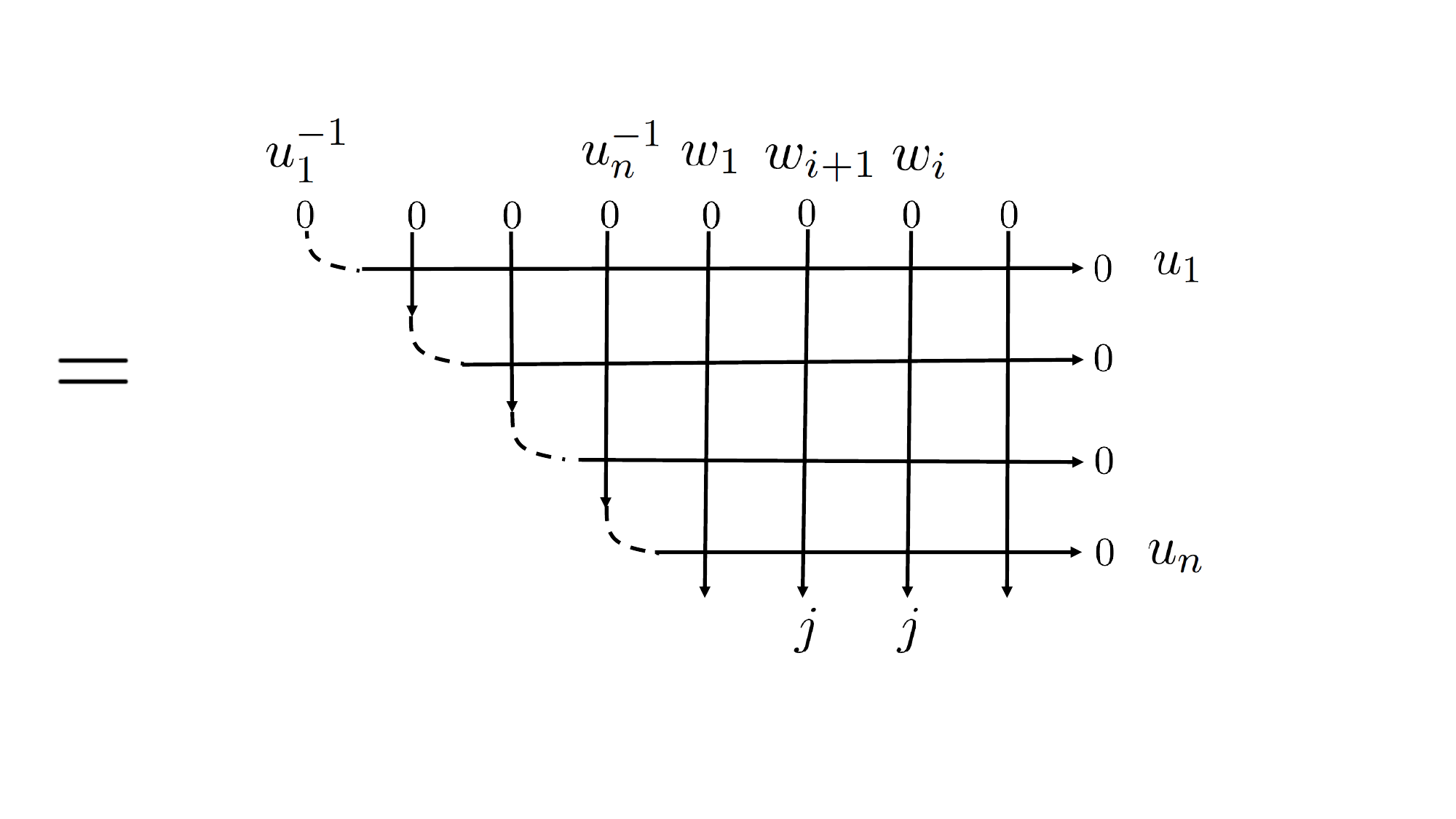}
  \caption{Commutativity \eqref{recrelthree}.}
  \label{fig:basicrecursionthree}
\end{figure}

\end{prop}

\begin{proof}
The derivation of relations of this type by the railroad argument
using the Yang--Baxter equation \eqref{eq:YBE} is well known
(cf.~\cite{GGMW}).
Rearranging the resulting relations gives
\eqref{recrelone}--\eqref{recrelthree}, whose graphical forms
are shown in Figures~\ref{fig:basicrecursion},
\ref{fig:basicrecursiontwo}, and
\ref{fig:basicrecursionthree}, respectively.
\end{proof}

\eqref{recrelone}, \eqref{recreltwo}, and \eqref{recrelthree},
with \(F_{N,n,r}(\mathbf{u}|\mathbf{w}|\mathbf{I})\) replaced by
\(WC_{N,n,r}(\mathbf{u}|\mathbf{w}|\mathbf{I})\) and
\(WB_{N,n,r}(\mathbf{u}|\mathbf{w}|\mathbf{I})\),
also hold.

\begin{prop}
(i) If $1 \in \mathbf{I}$, we have
\begin{align}
    &F_{N,n,r-1}(\mathbf{u}|\mathbf{w}|\mathbf{I}^{(0)}) \nonumber \\
    =&\frac{1}{1-w_1^2}
    \Bigg(
    (1-q)^{-1} \prod_{j=2}^{N} (1-q w_1w_j)^{-1} \prod_{j=1}^n (1-q w_1^{-1} u_j^{-1})^{-1}  \nonumber \\
    &\times
    F_{N,n+1,r}(\mathbf{u},w_1^{-1}|w_1^{-1},w_2,\dotsm,w_N|\mathbf{I})
    - (\beta w_1^2-\gamma w_1) F_{N,n,r}(\mathbf{u}|\mathbf{w}|\mathbf{I}) \Bigg),
    \label{pizerorecursionone}
\end{align}
where $\mathbf{I}^{(0)}=\{ i_2,\dotsm,i_r \}$.
\\
(ii) If $1 \not\in \mathbf{I}$, we have
\begin{align}
 &F_{N,n+1,r}(\mathbf{u},w_1^{-1}|w_1^{-1},w_2,\dotsm,w_N|\mathbf{I})
 \nonumber \\
 =& (1-q) (\beta-\gamma w_1) \prod_{j=2}^{N} (1-q w_1w_j) \prod_{j=1}^n (1-q w_1^{-1} u_j^{-1})
 F_{N,n,r}(\mathbf{u}|\mathbf{w}|\mathbf{I}).
 \label{pizerorecursiontwo}
\end{align}

\begin{figure}[!htbp]
  \centering
  \includegraphics[width=0.7\linewidth]{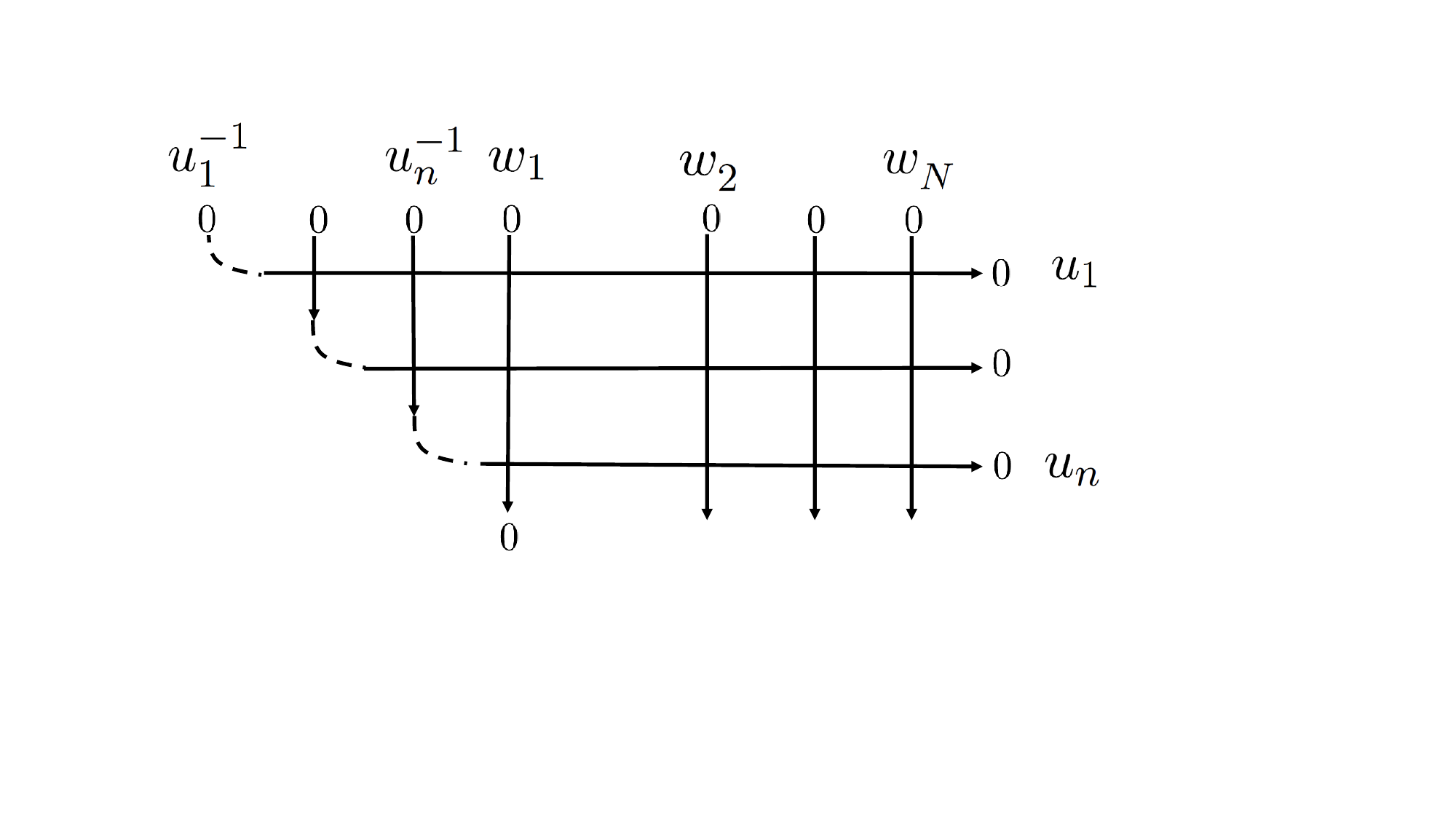}
\includegraphics[width=0.45\linewidth]{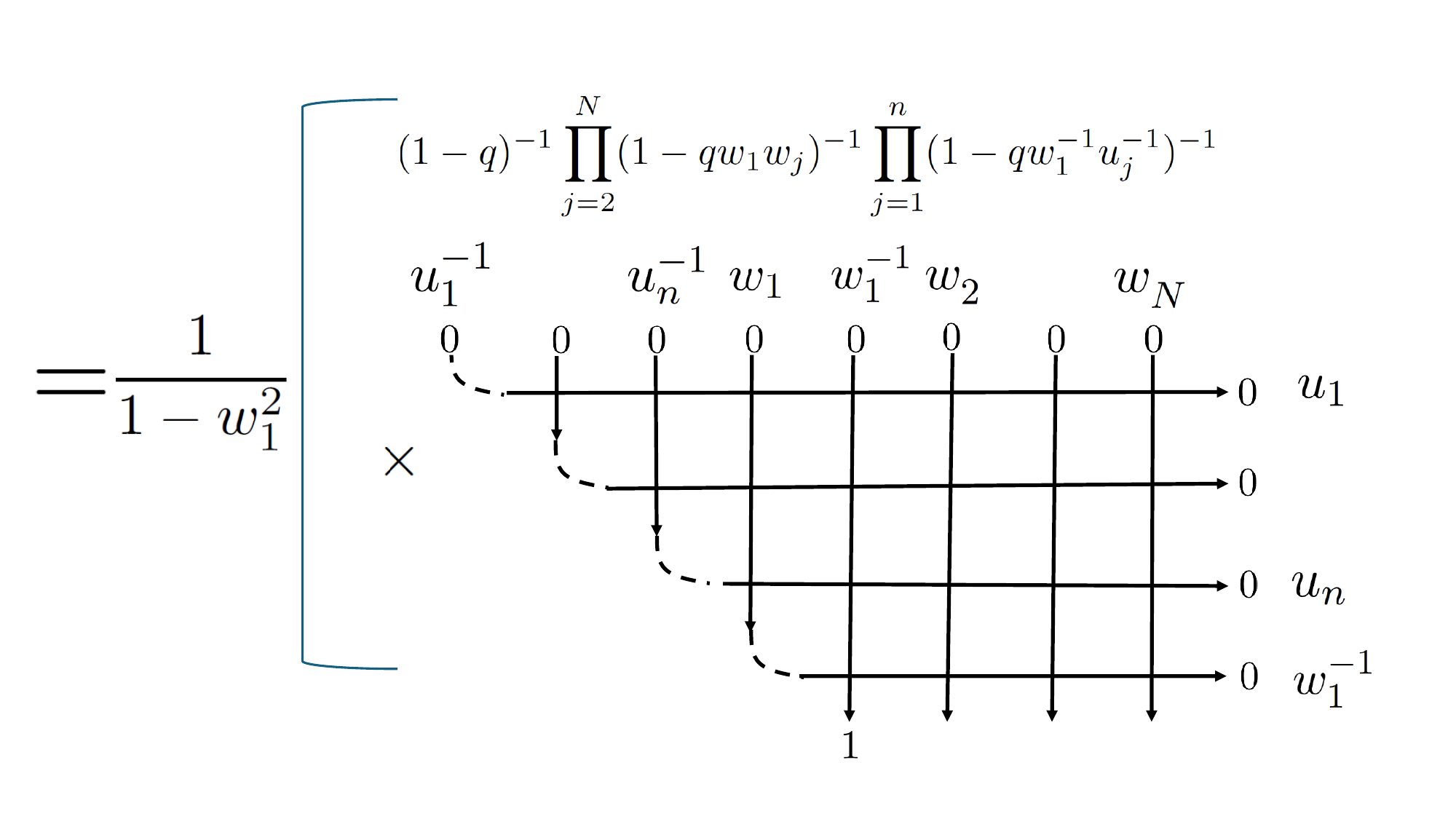}
\includegraphics[width=0.45\linewidth]{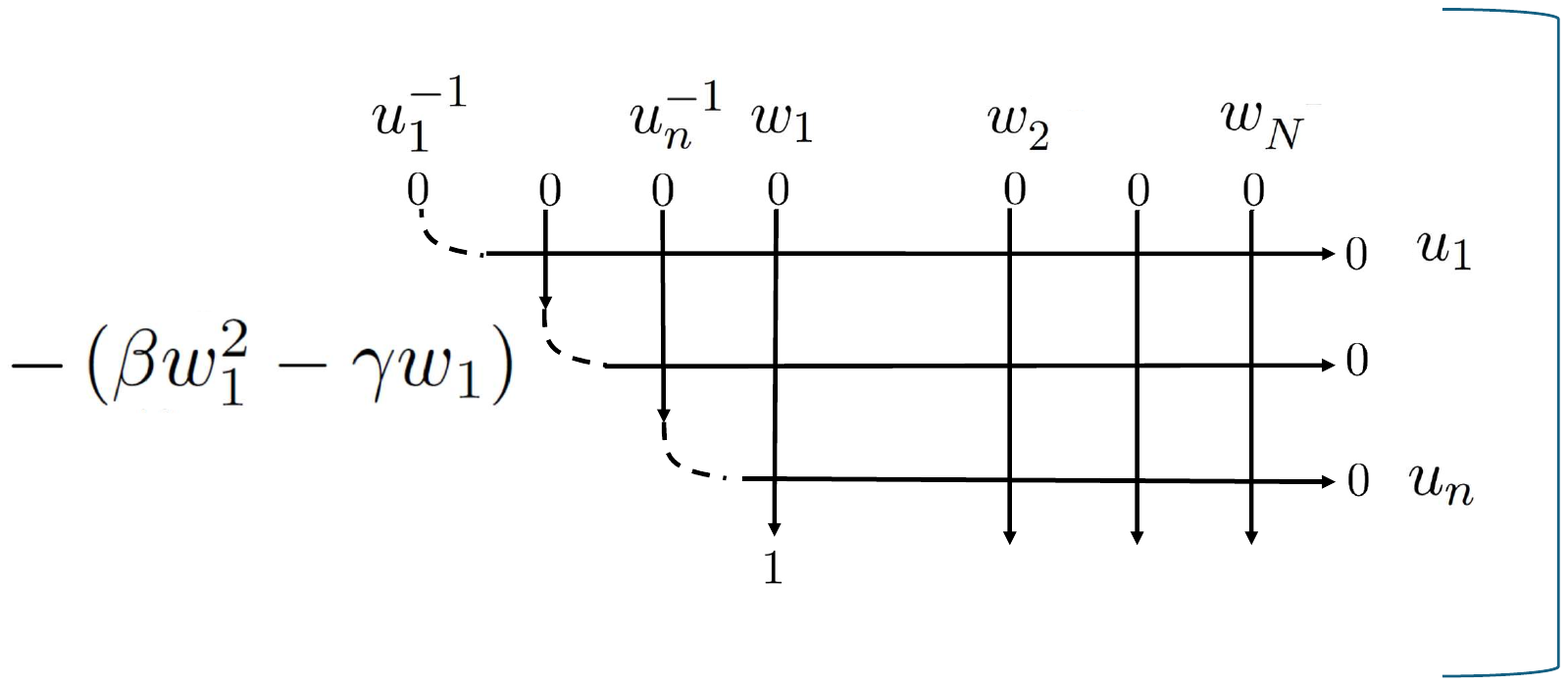}
  \caption{Recursive relation \eqref{pizerorecursionone}.}
  \label{fig:pizerorecursionone}
\end{figure}

\begin{figure}[!htbp]
  \centering
\includegraphics[width=0.45\linewidth]{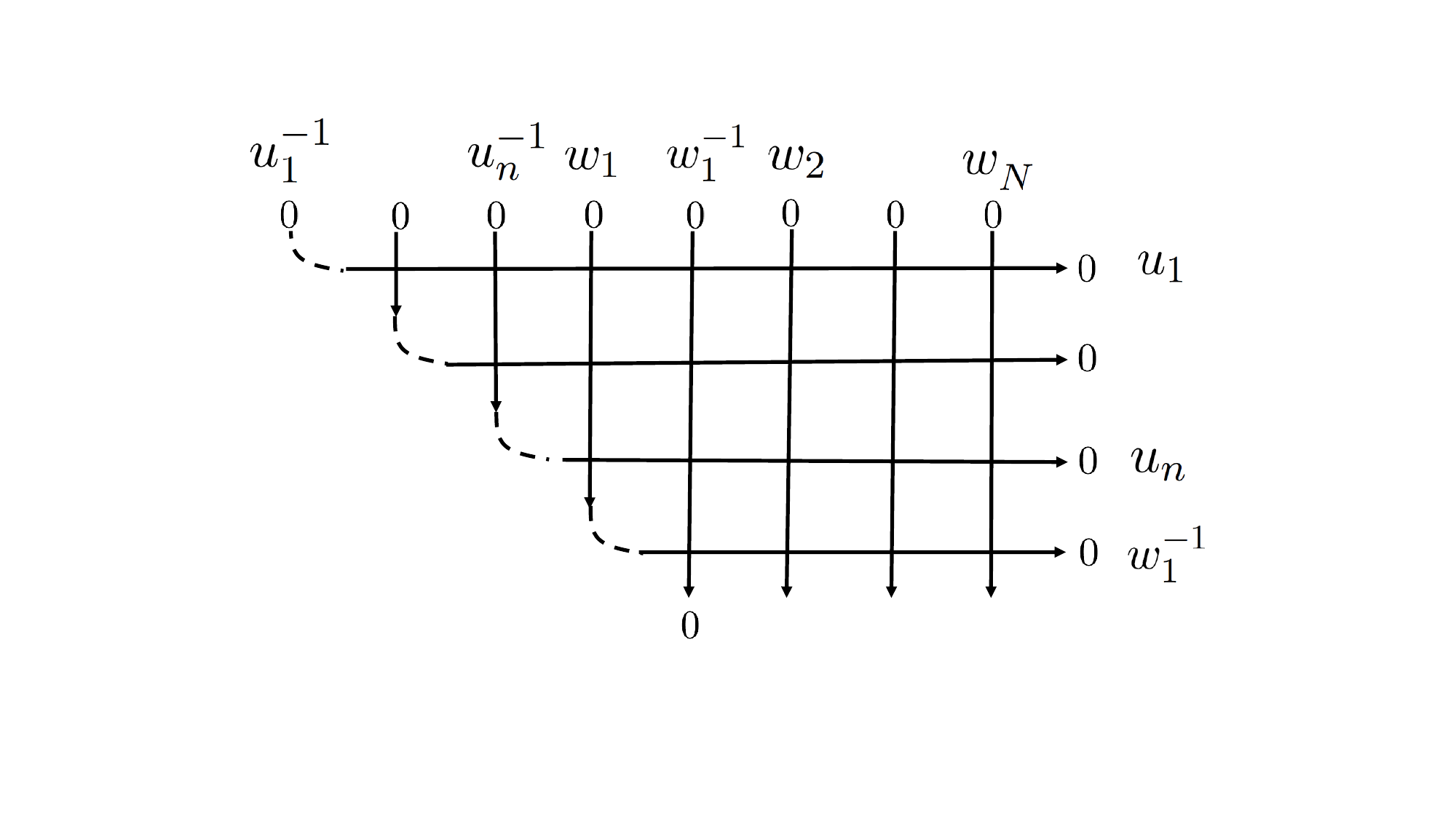}
\includegraphics[width=0.45\linewidth]{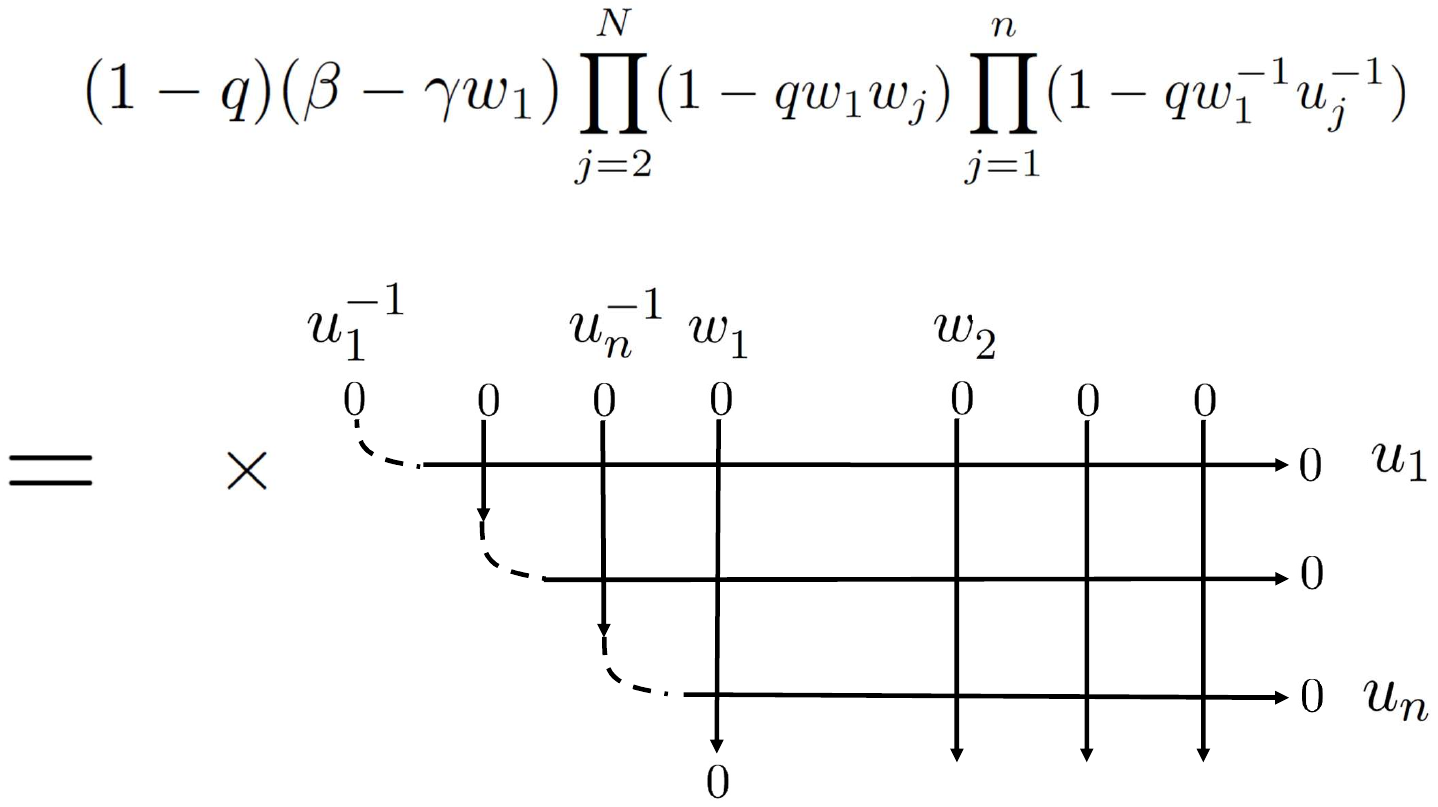}
  \caption{Reduction \eqref{pizerorecursiontwo}.}
  \label{fig:pizerorecursiontwo}
\end{figure}

\end{prop}

\begin{proof}
Graphical descriptions of
\eqref{pizerorecursionone} and \eqref{pizerorecursiontwo}
are given in Figures~\ref{fig:pizerorecursionone} and \ref{fig:pizerorecursiontwo}, respectively.
The two identities are proved in the same way, so we prove \eqref{pizerorecursionone}.

The graphical deformation process is depicted in Figure~\ref{fig:pizerorecursionthree}. 
We begin with \\
$F_{N,n+1,r}(\mathbf{u}, w_1^{-1}|w_1^{-1}, w_2, \dots, w_N|\mathbf{I})$. 
Here, we observe an $R$-matrix that essentially reduces to the permutation operator \eqref{eq:Ruu} $R(w_1^{-1}, w_1^{-1}) = (1-q)P$. 
By applying the Yang--Baxter equation \eqref{eq:YBE} repeatedly, we move the line to which the spectral parameter $w_1^{-1}$ is assigned toward the upper-right region. 
Due to the ice rule \eqref{eq:ice-rule}, the configuration involving this line is uniquely determined; thus, we can remove the line and instead multiply by the overall factor 
$\displaystyle
\prod_{j=2}^{N} (1-q w_1 w_j) \prod_{j=1}^n (1-q w_1^{-1} u_j^{-1}),
$
which consists of the $R$-matrix elements \eqref{eq:R-components}. Finally, we decompose the partition function into two terms based on the $K$-matrix elements \eqref{eq:Ktri-components}, leading to the final equation
\begin{align}
&(1-q)
\prod_{j=2}^{N} (1-q w_1 w_j) \prod_{j=1}^n (1-q w_1^{-1} u_j^{-1})
\nonumber \\
&\times
 \Big(
 (\beta w_1^2-\gamma w_1) F_{N,n,r}(\mathbf{u}|\mathbf{w}|\mathbf{I}) 
 +
 (1-w_1^2)
 F_{N,n,r-1}(\mathbf{u}|\mathbf{w}|\mathbf{I}^{(0)}) 
     \Big). \notag
\end{align}
The graphical identity in Figure~\ref{fig:pizerorecursionthree} is equivalent to Figure~\ref{fig:pizerorecursionone}, and algebraically gives \eqref{pizerorecursionone}.

\begin{figure}[!htbp]
  \centering
  \includegraphics[width=0.7\linewidth]{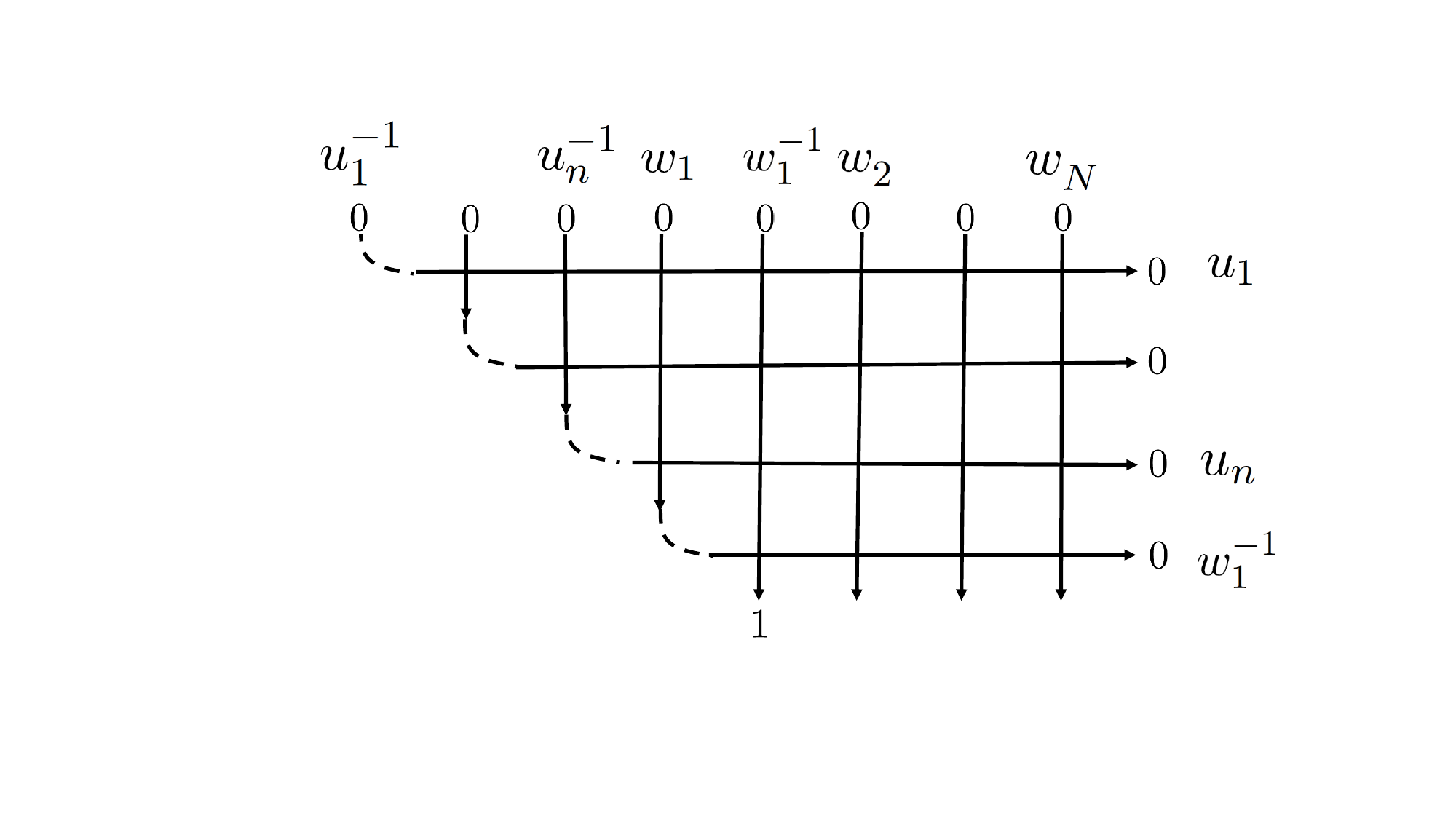}
\includegraphics[width=0.45\linewidth]{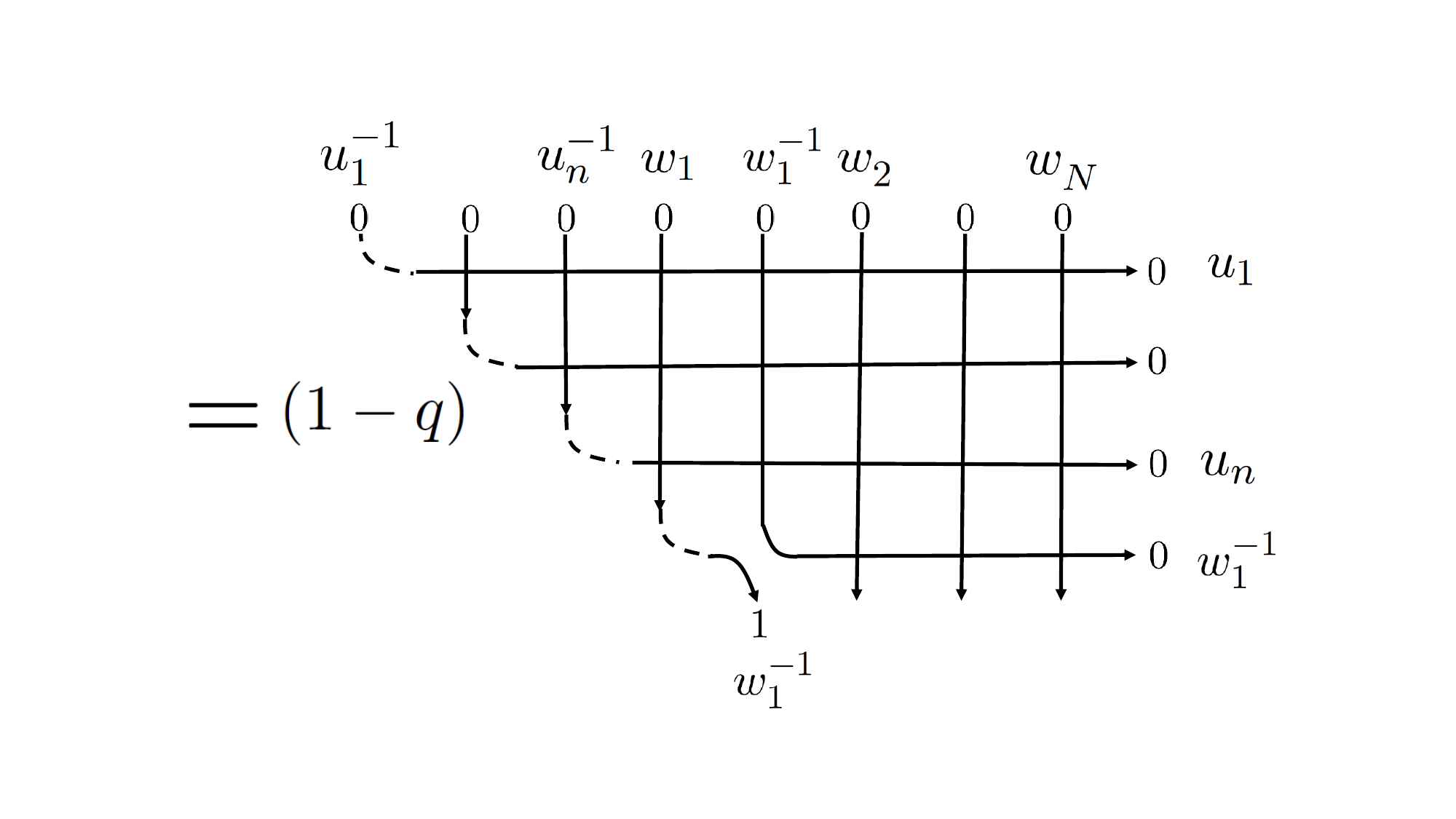}
\includegraphics[width=0.45\linewidth]{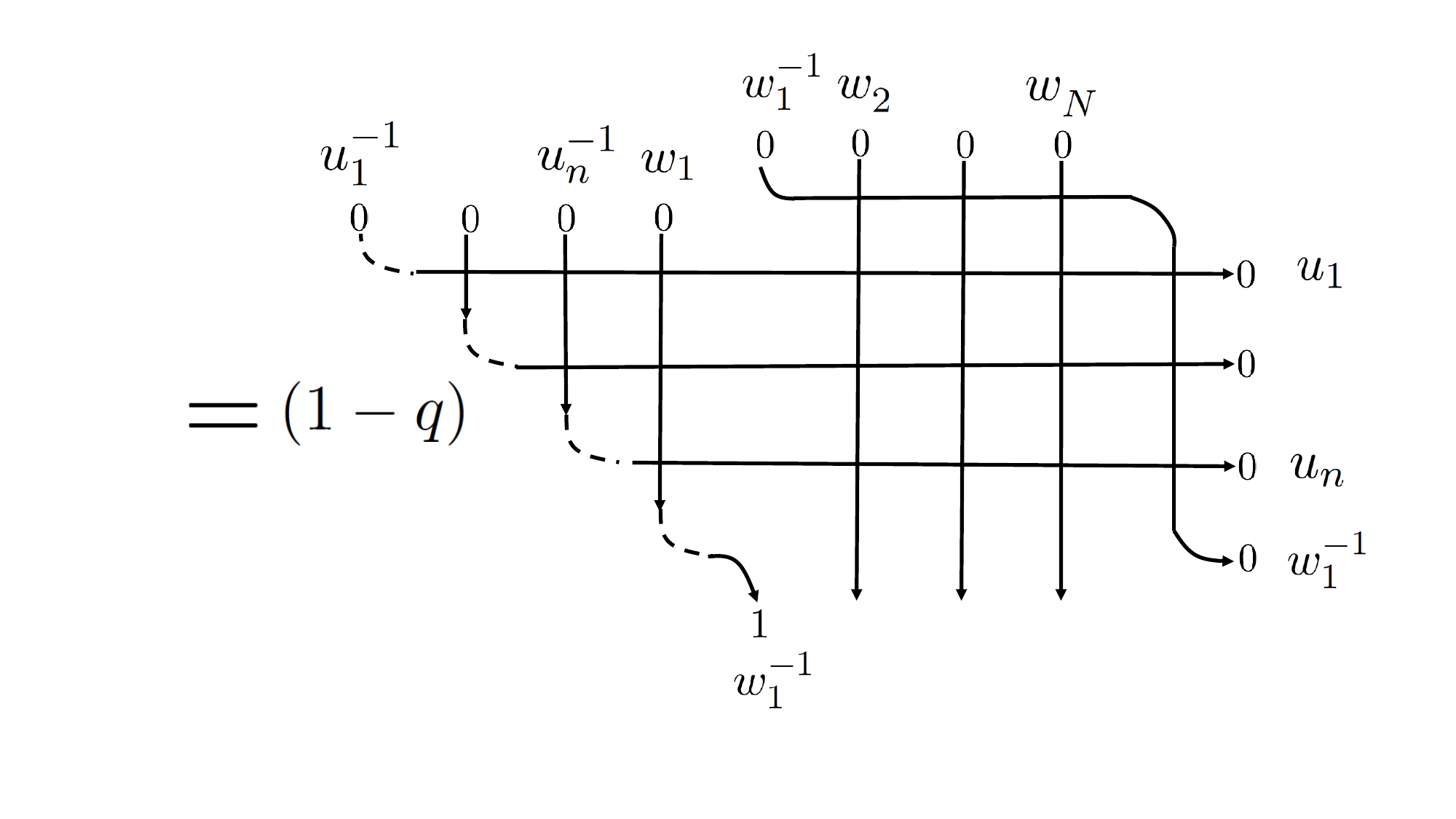}
 \includegraphics[width=0.45\linewidth]{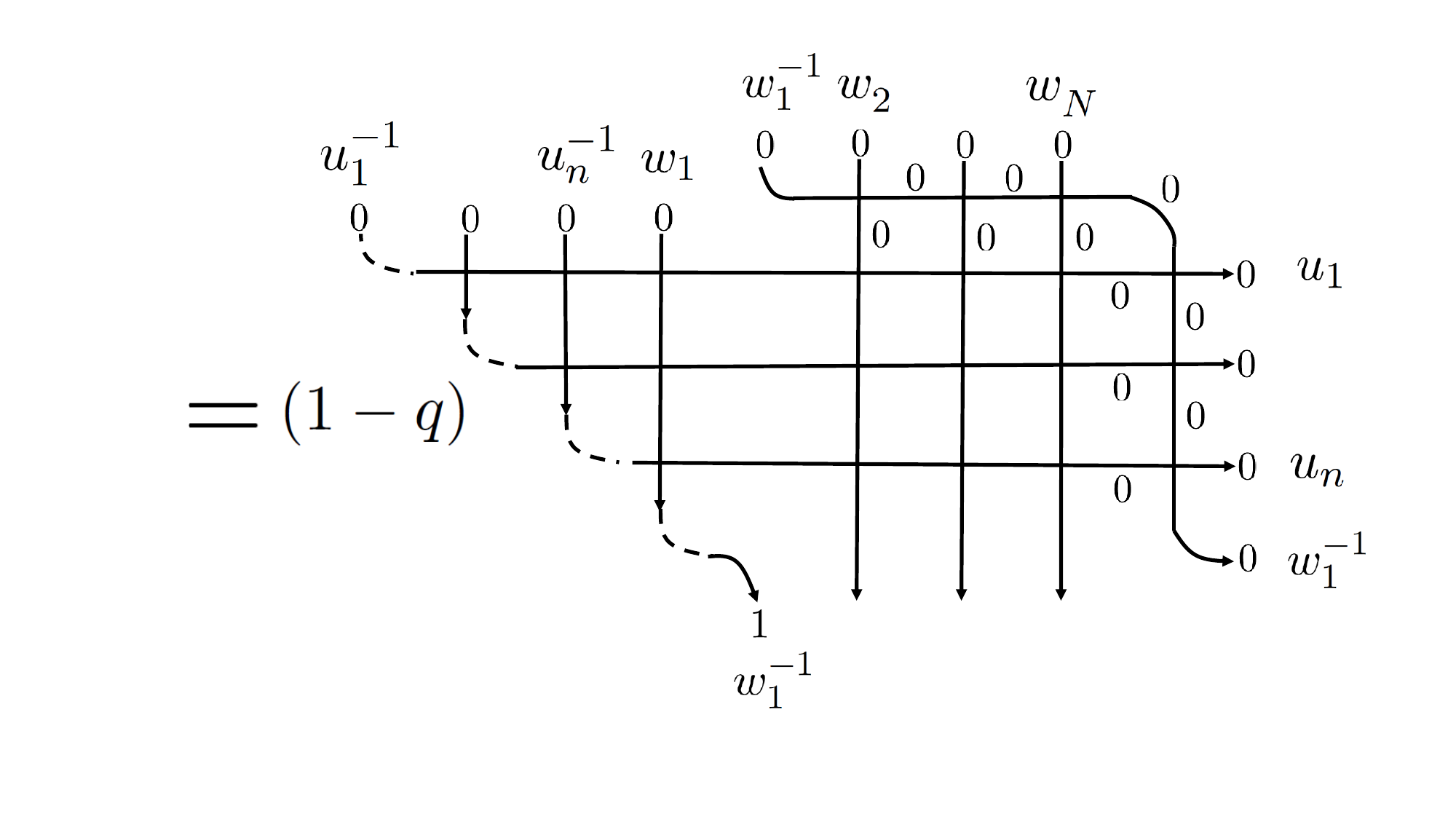}
\includegraphics[width=0.45\linewidth]{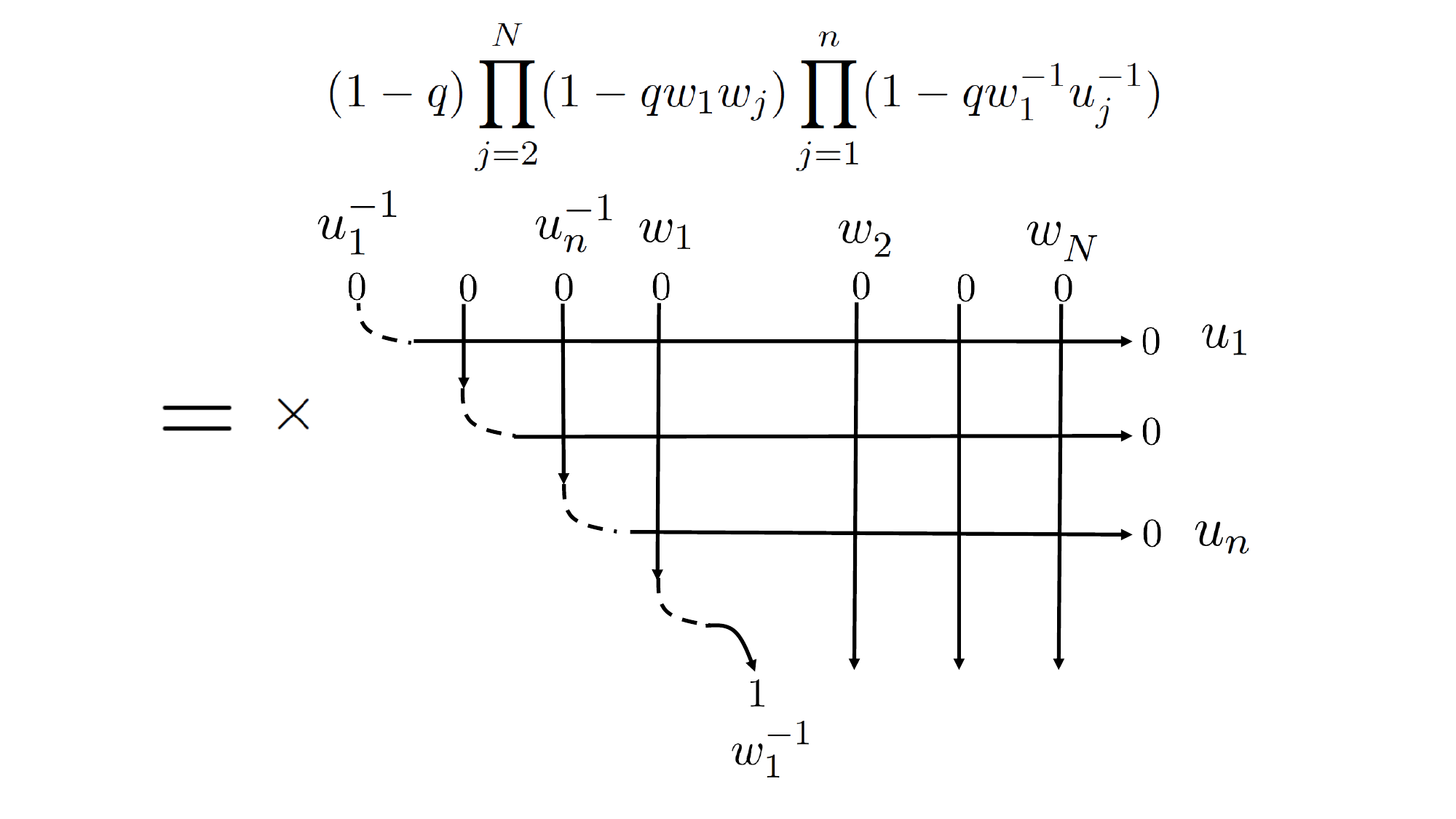}
\includegraphics[width=0.45\linewidth]{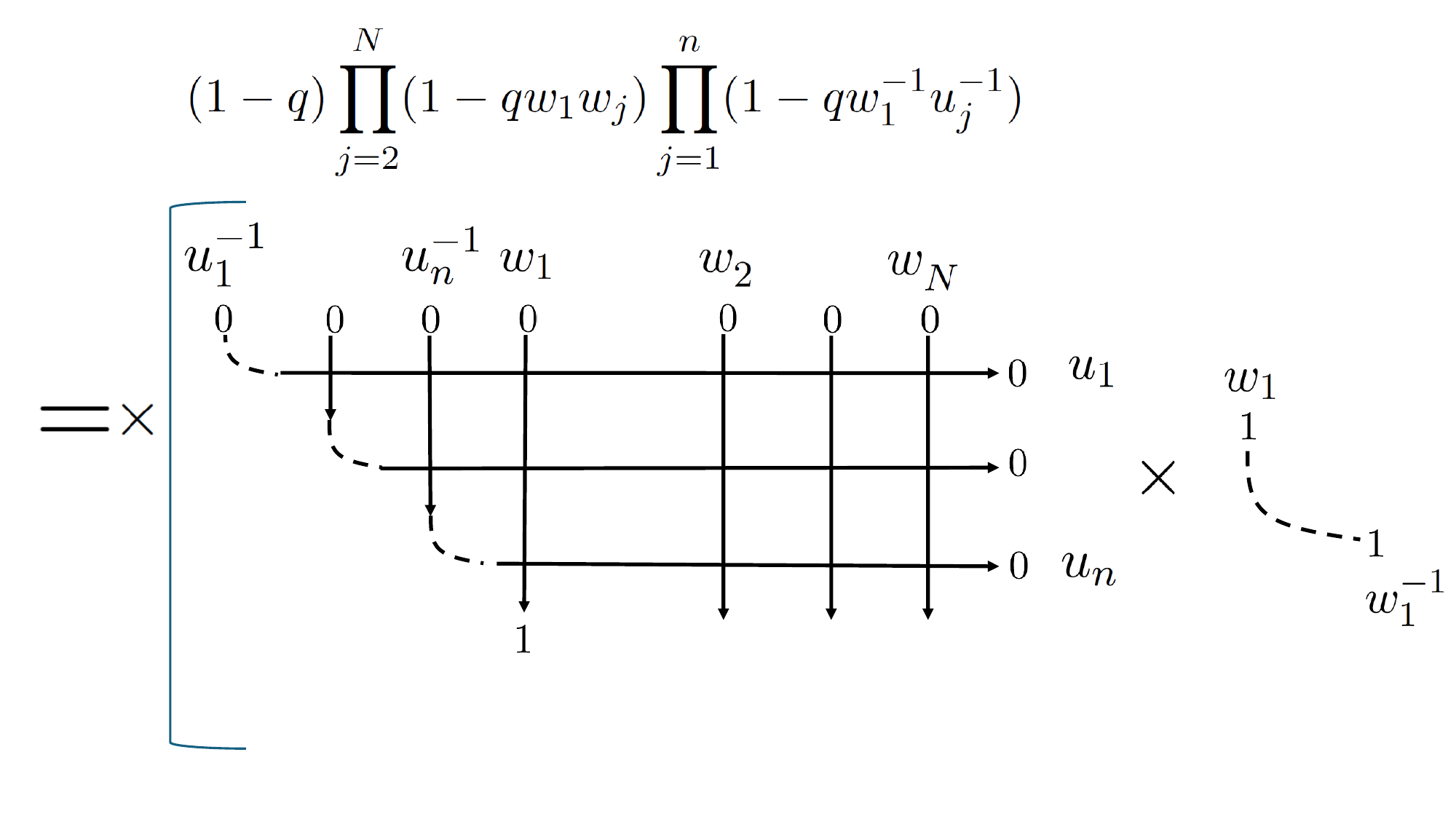}
 \includegraphics[width=0.45\linewidth]{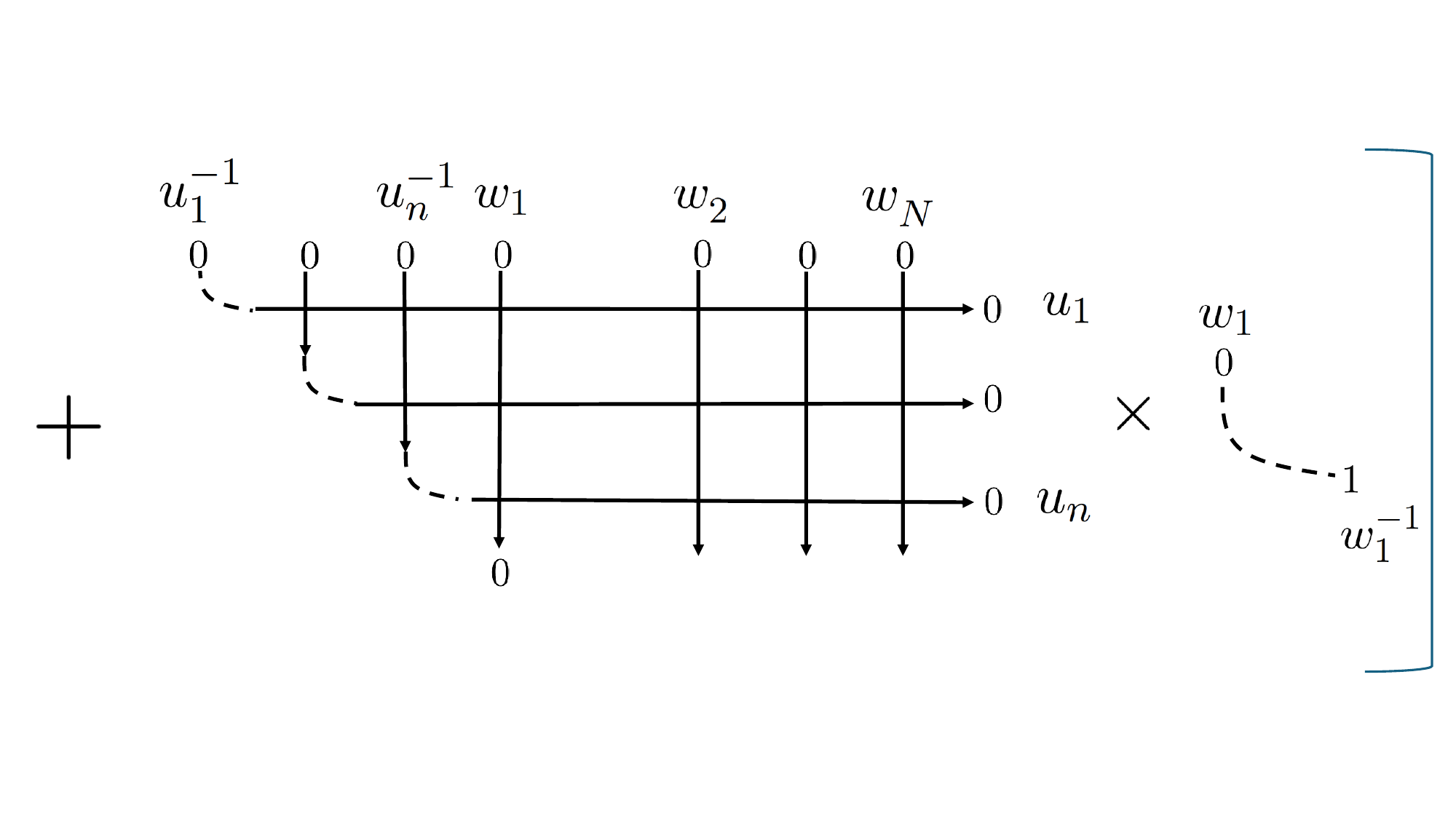}
  \caption{The graphical deformation process to get \eqref{pizerorecursionone}.}
  \label{fig:pizerorecursionthree}
\end{figure}
\end{proof}

The relations \eqref{pizerorecursionone} and \eqref{pizerorecursiontwo} yield the following identities for the $WC$- and $WB$-functions: 
\begin{corollary}

(i) If $1 \in \mathbf{I}$, we have
\begin{align}
    &WC_{N,n,r-1}(\mathbf{u}|\mathbf{w}|\mathbf{I}^{(0)}) \nonumber \\
    =&\frac{1}{1-w_1^2}
    \Bigg(
    (1-q)^{-1} \prod_{j=2}^{N} (1-q w_1w_j)^{-1} \prod_{j=1}^n (1-q w_1^{-1} u_j^{-1})^{-1}  \nonumber \\
    &\times
WC_{N,n+1,r}(\mathbf{u},w_1^{-1}|w_1^{-1},w_2,\dotsm,w_N|\mathbf{I})
    - w_1^2 WC_{N,n,r}(\mathbf{u}|\mathbf{w}|\mathbf{I}) \Bigg),
    \label{pizerorecursiononeWC} \\
    &WB_{N,n,r-1}(\mathbf{u}|\mathbf{w}|\mathbf{I}^{(0)}) \nonumber \\
    =&\frac{1}{1-w_1}
    \Bigg(
    (1-q)^{-1} \prod_{j=2}^{N} (1-q w_1w_j)^{-1} \prod_{j=1}^n (1-q w_1^{-1} u_j^{-1})^{-1}  \nonumber \\
    &\times
WB_{N,n+1,r}(\mathbf{u},w_1^{-1}|w_1^{-1},w_2,\dotsm,w_N|\mathbf{I})
    - w_1 WB_{N,n,r}(\mathbf{u}|\mathbf{w}|\mathbf{I}) \Bigg),
    \label{pizerorecursiononeWB}
\end{align}
where $\mathbf{I}^{(0)}=\{ i_2,\dotsm,i_r \}$.
\\
(ii) If $1 \not\in \mathbf{I}$, we have
\begin{align}
 &WC_{N,n+1,r}(\mathbf{u},w_1^{-1}|w_1^{-1},w_2,\dotsm,w_N|\mathbf{I})
 \nonumber \\
 =& (1-q)  \prod_{j=2}^{N} (1-q w_1w_j) \prod_{j=1}^n (1-q w_1^{-1} u_j^{-1})
 WC_{N,n,r}(\mathbf{u}|\mathbf{w}|\mathbf{I}),
 \label{pizerorecursiontwoWC} \\
 &WB_{N,n+1,r}(\mathbf{u},w_1^{-1}|w_1^{-1},w_2,\dotsm,w_N|\mathbf{I})
 \nonumber \\
 =& (1-q)  \prod_{j=2}^{N} (1-q w_1w_j) \prod_{j=1}^n (1-q w_1^{-1} u_j^{-1})
 WB_{N,n,r}(\mathbf{u}|\mathbf{w}|\mathbf{I}).
 \label{pizerorecursiontwoWB}
\end{align}

\end{corollary}

\begin{lem} \label{deltalem} We have
\begin{align}
    F_{N,n,r}(\mathbf{u}|\mathbf{w}|\mathbf{I})|_{u_1=\cdots=u_n=1}
    =(\beta-\gamma)^n (1-q)^{\frac{n(n-1)}{2}} \prod_{j=1}^N (1-q w_j)^n \delta_{\mathbf{I},\emptyset}. \label{specializetoone}
    \end{align}
    Here $\delta_{\mathbf{I},\emptyset}$ means 1 if $\mathbf{I}$ is an empty set $(r=0)$, and 0 otherwise.

\begin{figure}[!htbp]
  \centering
  \includegraphics[width=0.7\linewidth]{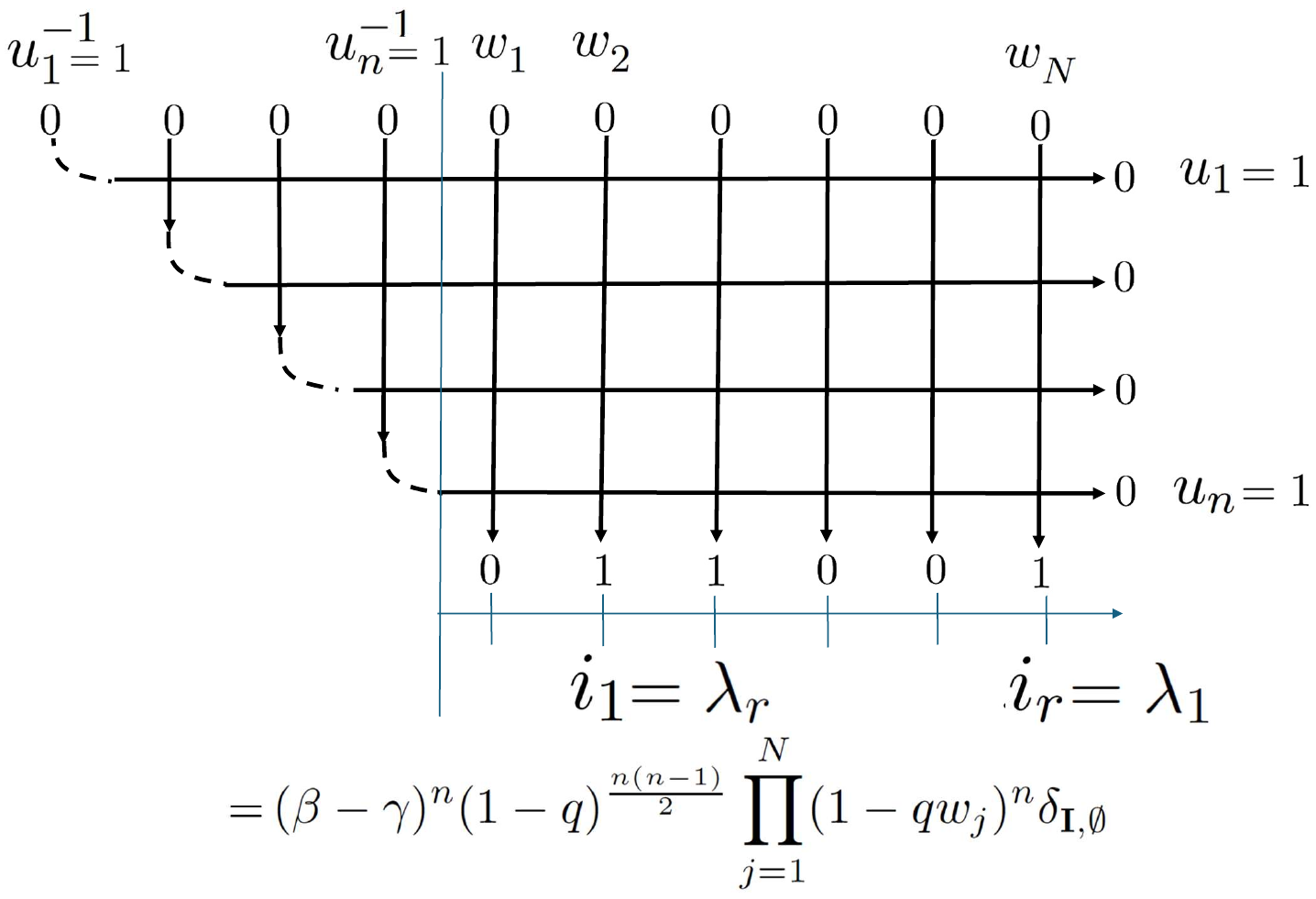}
  \caption{Specializing all $u$-variables to 1 \eqref{specializetoone}.}
  \label{fig:spone}
\end{figure}

\end{lem}

\begin{proof}

Consider a diagram where the configuration of $0$'s and $1$'s at the bottom boundary is not fixed
(Figure~\ref{fig:delta}, top left panel).
Since the $K$-matrix \eqref{eq:Ktri-components} reduces essentially to the identity matrix $K(1)=(\beta-\gamma) \mathrm{Id}$ and the $R$-matrix satisfies the ice rule \eqref{eq:ice-rule}, the configuration is uniquely determined row by row starting from the top, as illustrated in Figure~\ref{fig:delta}. This implies that the partition function is nonzero if the bottom row consists entirely of $0$'s (i.e., $\mathbf{I}=\emptyset$), and vanishes otherwise.
If $\mathbf{I} = \emptyset$, a direct calculation shows that
the product of the elements of the $K$-matrices and $R$-matrices
for the unique configuration is
$(\beta-\gamma)^n (1-q)^{\frac{n(n-1)}{2}} \prod_{j=1}^N (1-q w_j)^n$,
leading to \eqref{specializetoone}
(Figure~\ref{fig:spone}).

\begin{figure}[!htbp]
  \centering
 \includegraphics[width=0.45\linewidth]{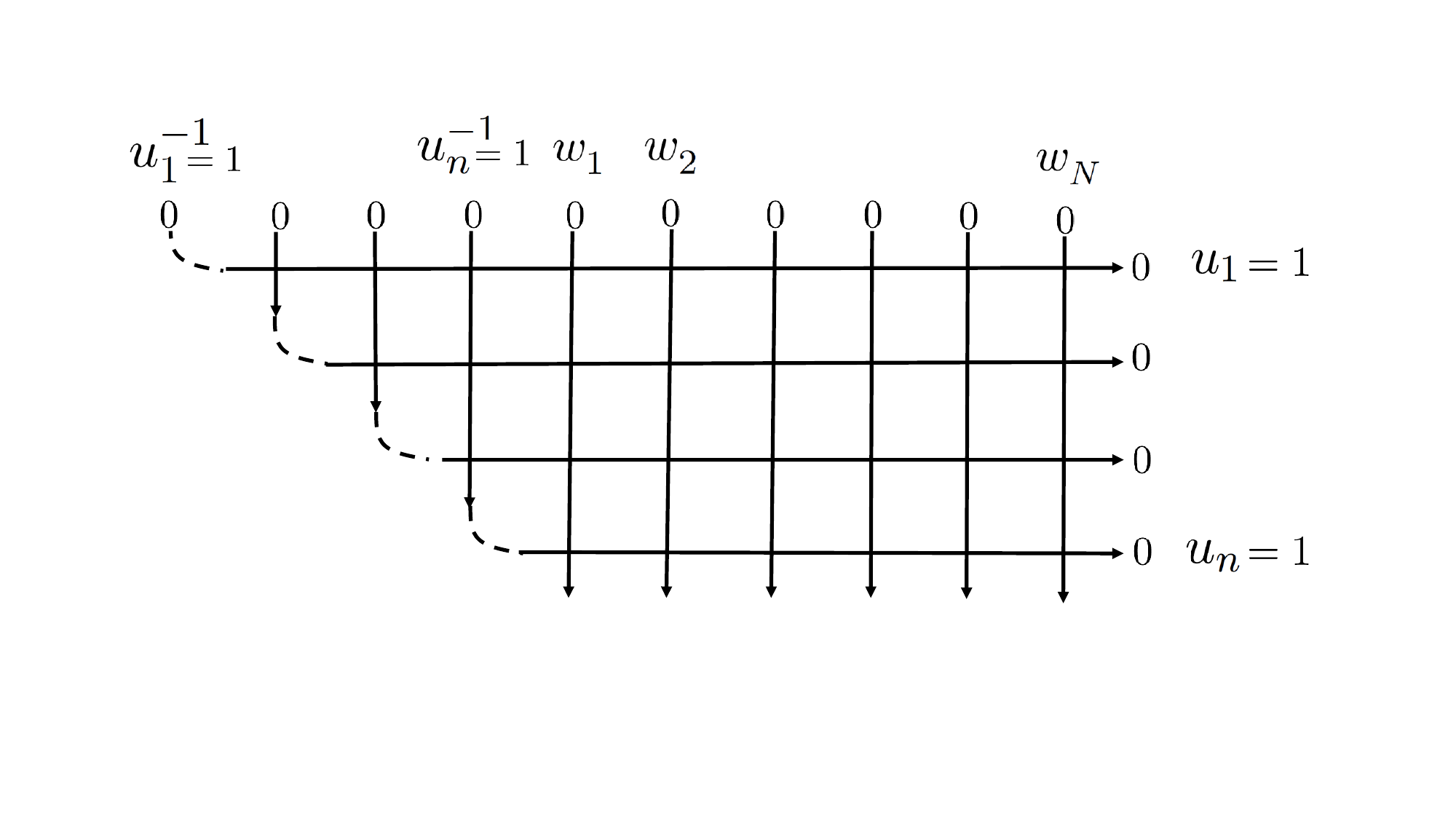}
\includegraphics[width=0.45\linewidth]{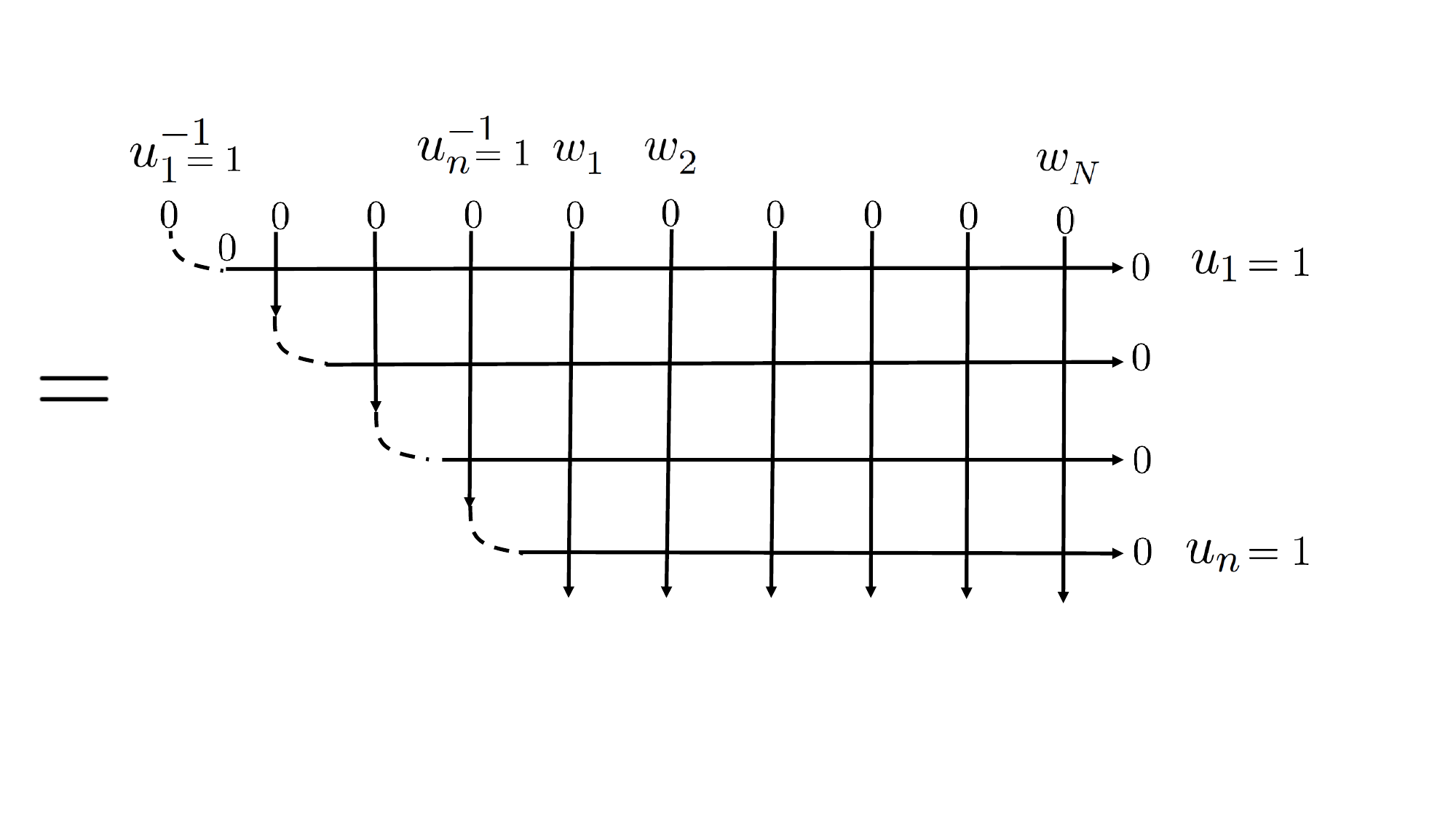}
\includegraphics[width=0.45\linewidth]{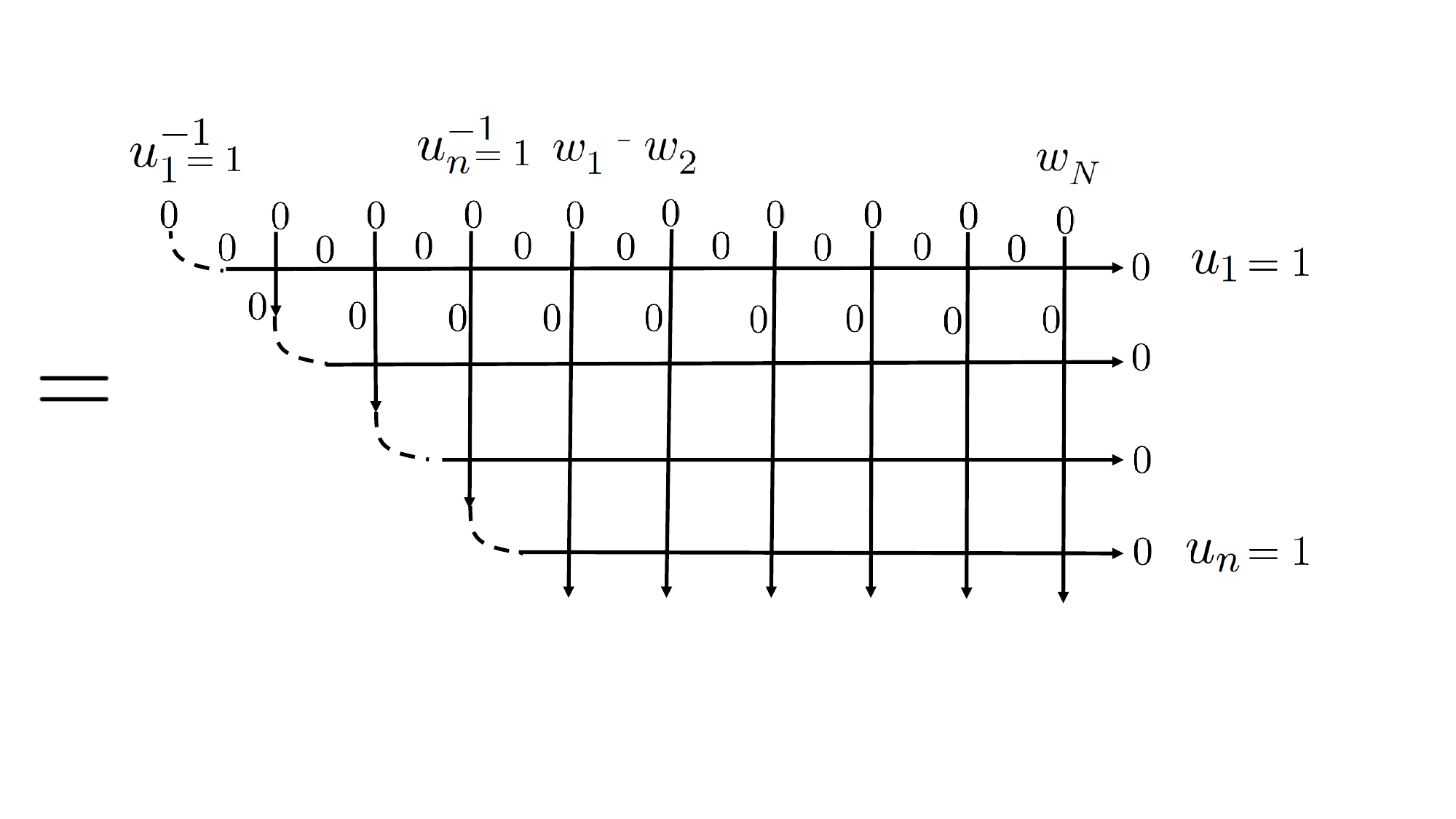}
 \includegraphics[width=0.45\linewidth]{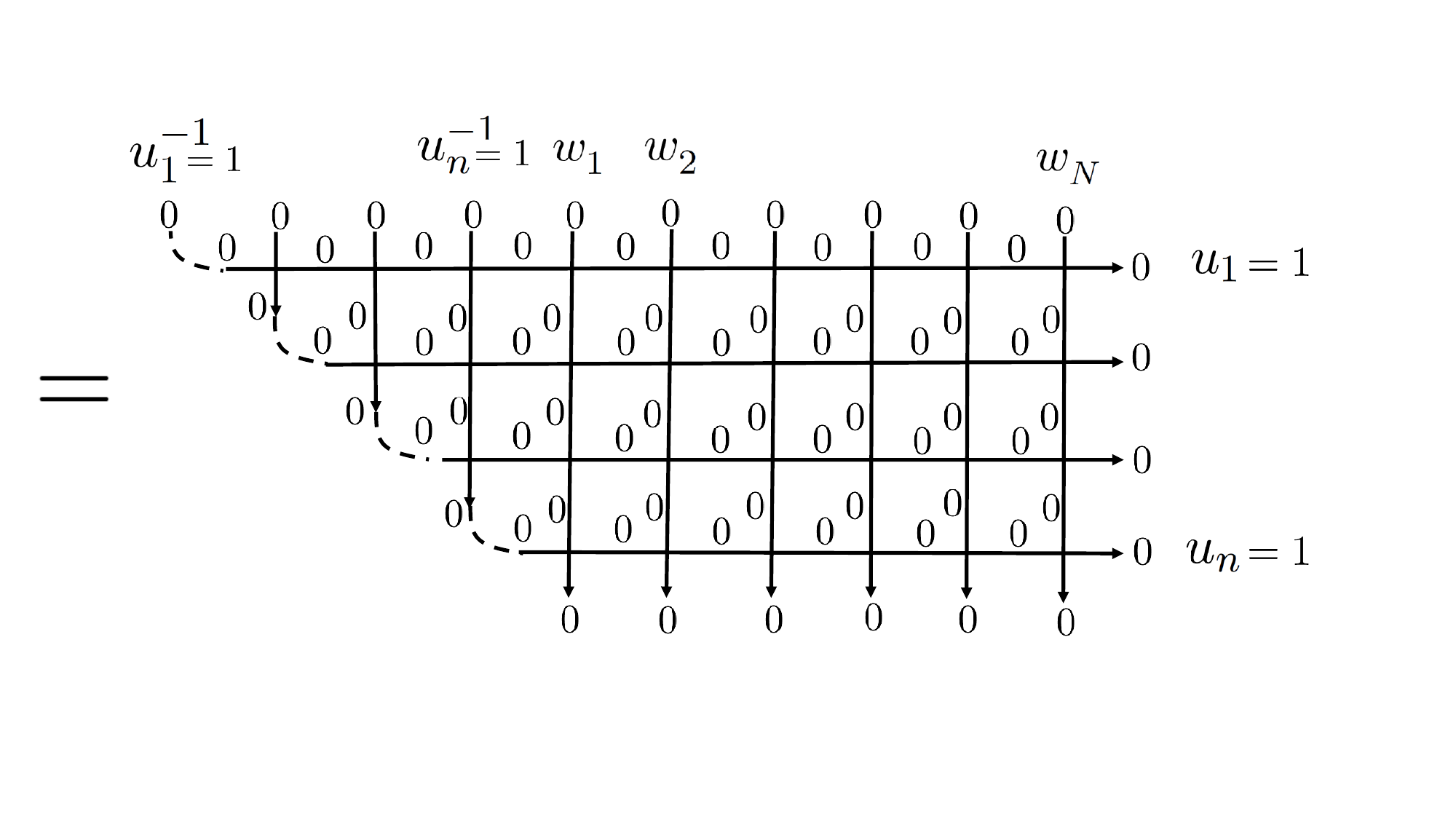}
  \caption{The freezing process starting from the top row.}
  \label{fig:delta}
\end{figure}

\end{proof}

Lemma~\ref{deltalem} gives the following corollary.
\begin{corollary} We have
    \begin{align}
    &WC_{N,n,r}(\mathbf{u}|\mathbf{w}|\mathbf{I})|_{u_1=\cdots=u_n=1}
    =
     WB_{N,n,r}(\mathbf{u}|\mathbf{w}|\mathbf{I})|_{u_1=\cdots=u_n=1}
     \nonumber \\
    =& (1-q)^{\frac{n(n-1)}{2}} \prod_{j=1}^N (1-q w_j)^n \delta_{\mathbf{I},\emptyset}. \label{whenalltoone}
    \end{align}
\end{corollary}

\begin{lem} \label{onevariablesp}
Assume $r\le n-1$. We have
\begin{align}
&F_{N,n,r}(\mathbf{u}|\mathbf{w}|\mathbf{I})|_{u_n=1}
=(\beta-\gamma) \prod_{k=1}^{n-1} (1-qu_k^{-1}) \prod_{j=1}^{N} (1-q w_j)
F_{N,n-1,r}(u_1,\dots,u_{n-1}|\mathbf{w}|\mathbf{I}).
\label{Fonevariablesp}
\end{align}
\end{lem}

\begin{proof}
The proof proceeds in the same way as that of Lemma~\ref{deltalem}, using the ice rule \eqref{eq:ice-rule}. Consider the first panel of Figure~\ref{fig:forstability}. When $u_n=1$, the configuration of the first row is uniquely determined, as shown in the second panel. Hence we may remove the first row at the cost of multiplying by
$\displaystyle (\beta-\gamma)\prod_{k=1}^{n-1}(1-qu_k^{-1}) \prod_{j=1}^{N}(1-qw_j)$. After removing the first row, the remaining partition function is
$F_{N,n-1,r}(u_1,\dots,u_{n-1}|\mathbf{w}|\mathbf{I})$. Therefore, we obtain \eqref{Fonevariablesp}.

\begin{figure}[!htbp]
  \centering
  \includegraphics[width=0.7\linewidth]{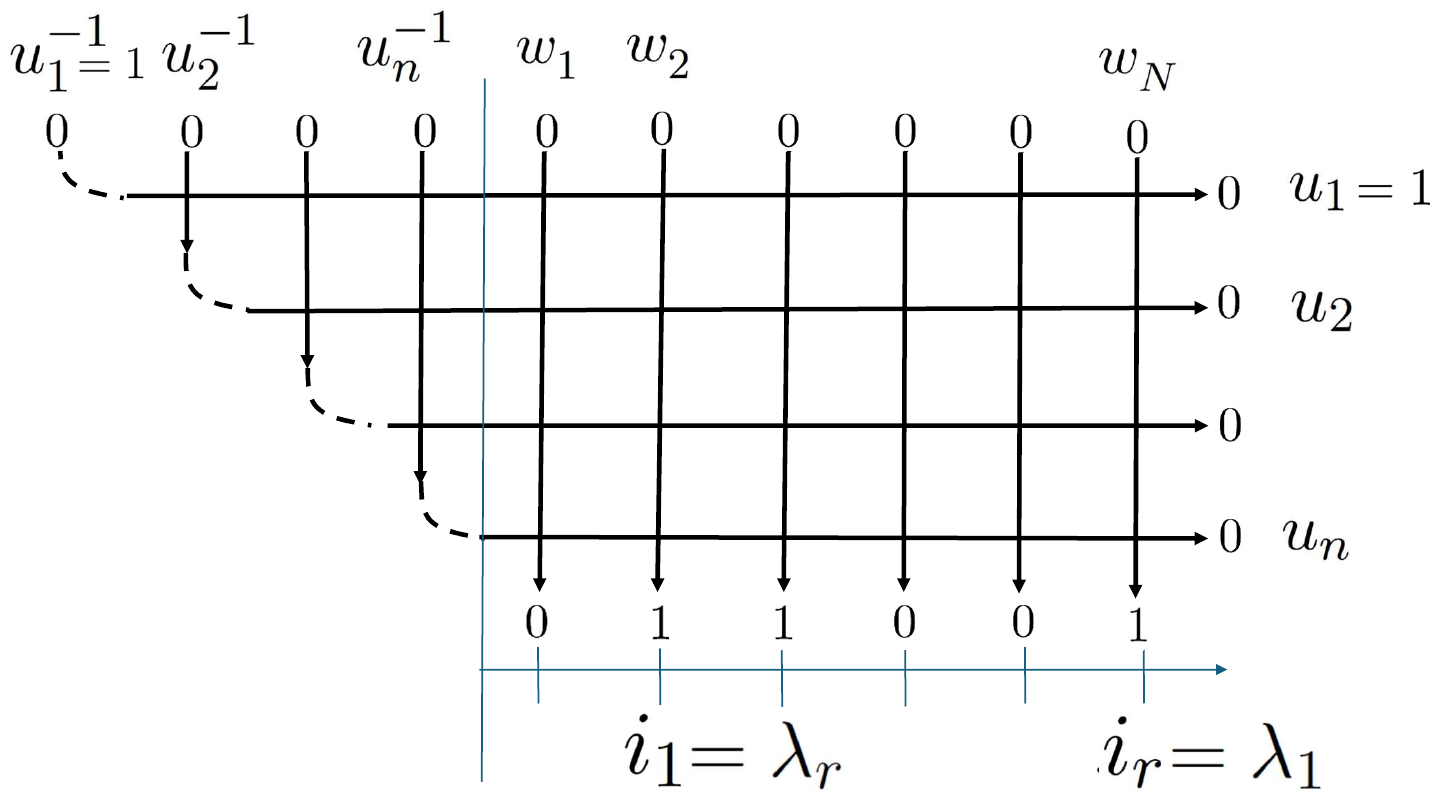}
 \includegraphics[width=0.45\linewidth]{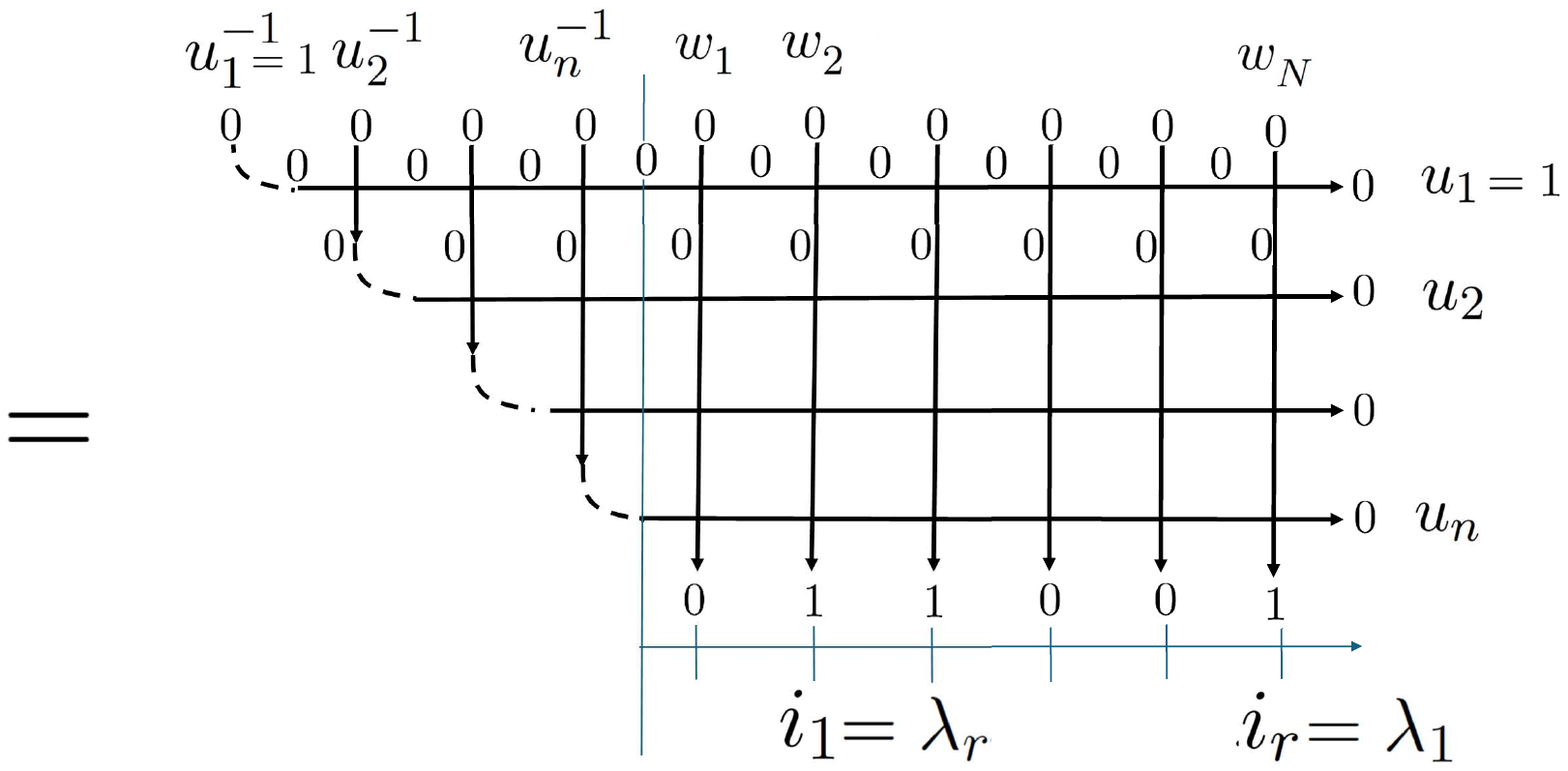}
\includegraphics[width=0.45\linewidth]{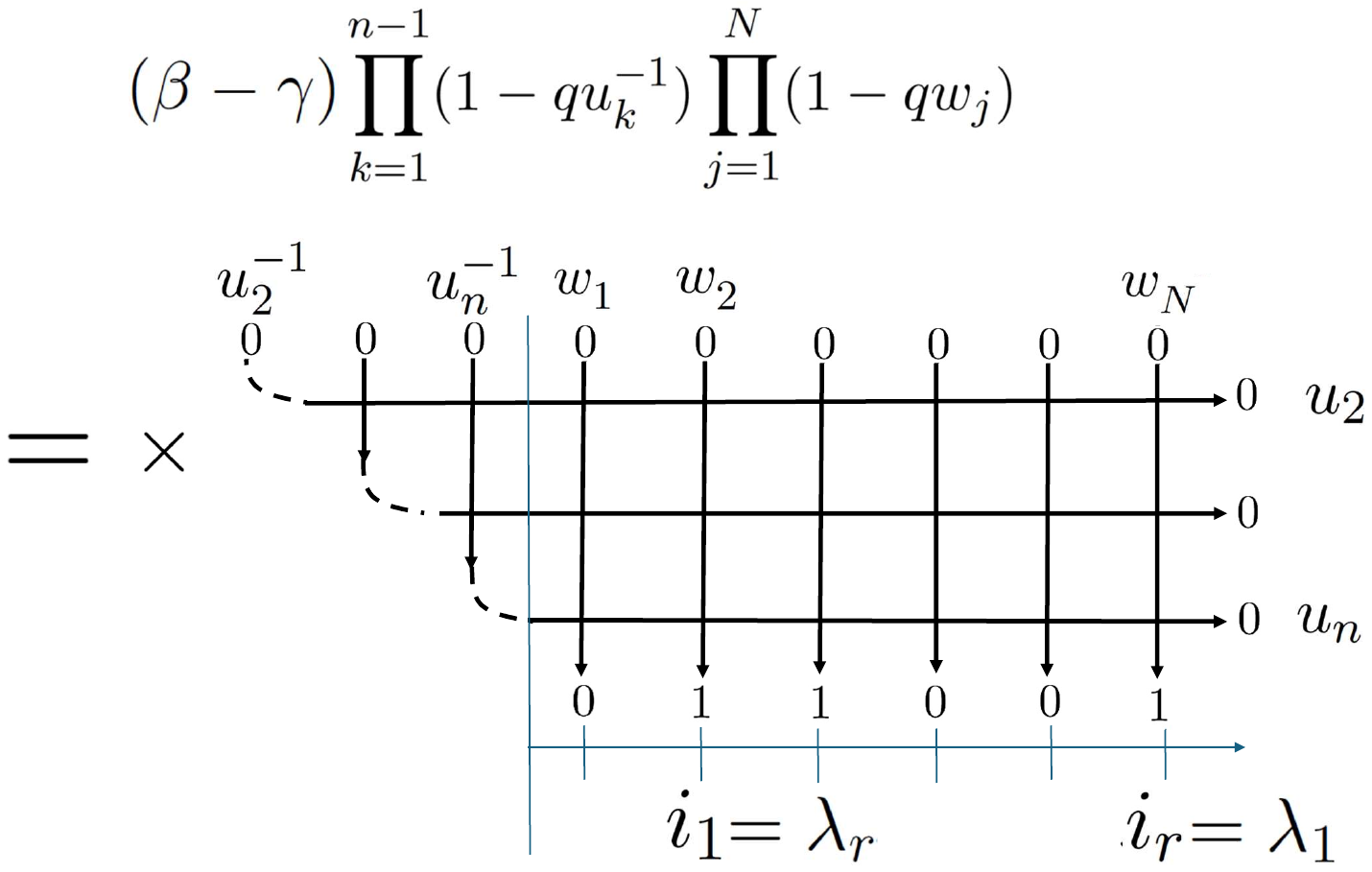}
\caption{A graphical derivation of \eqref{Fonevariablesp}.}
  \label{fig:forstability}
\end{figure}

\end{proof}

For $r=n$,  $F_{N,n,r}(\mathbf{u}|\mathbf{w}|\mathbf{I})$,
$WC_{N,n,r}(\mathbf{u}|\mathbf{w}|\mathbf{I})$, $WB_{N,n,r}(\mathbf{u}|\mathbf{w}|\mathbf{I})$ reduce to the
$K$-theoretic weight functions for type $A$ Grassmannian.

\begin{Definition}
We introduce type $A$ trigonometric weight functions:
\begin{align}
    W_{N,n}(\mathbf{u}|\mathbf{w}|\mathbf{I})
    :=&\prod_{k=1}^n u_k^{-N}
    \sum_{\sigma \in S_n} \sigma
    \Bigg[
    \prod_{j=1}^n \Bigg\{ (1-q)u_j \prod_{k=i_j+1}^N (u_j-qw_k) \prod_{k=1}^{i_j-1} (u_j-w_k) \Bigg\} \nonumber \\
    &\times
    \prod_{1 \le j < k \le n} \frac{qu_j-u_k}{u_j-u_k}
\Bigg].
\end{align}

\end{Definition}

The same functions admit the lattice-model realization shown in Figure~\ref{fig:basepartition}.
We use the same notation
$W_{N,n}(\mathbf{u}|\mathbf{w}|\mathbf{I})$
for these partition functions, which correspond to the standard off-shell Bethe wavefunctions.

\begin{figure}[!htbp]
  \centering
 \includegraphics[width=0.7\linewidth]{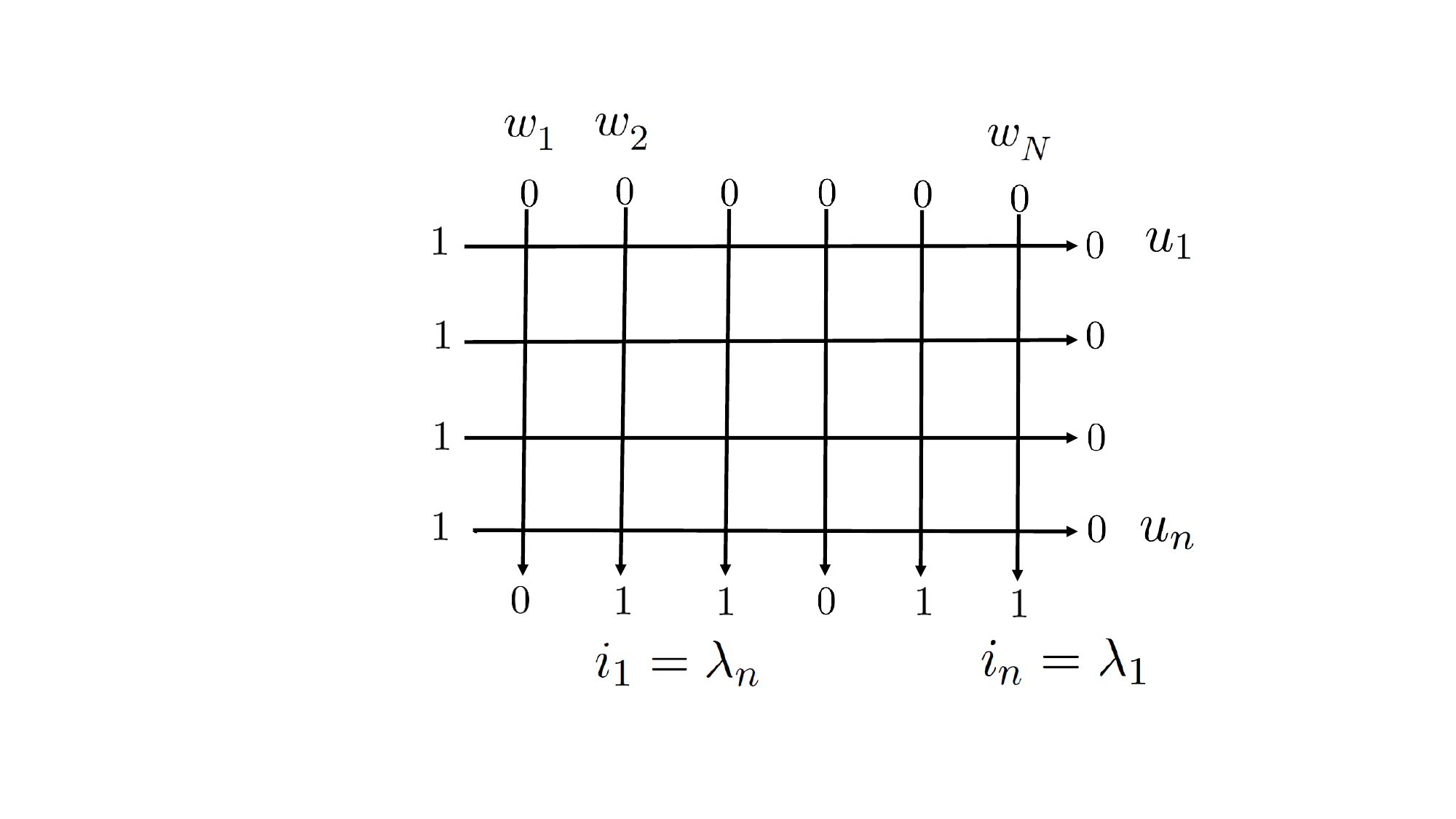}
  \caption{The partition function realization of
  $W_{N,n}(\mathbf{u}|\mathbf{w}|\mathbf{I})$ corresponding to the Bethe wavefunctions.}
  \label{fig:basepartition}
\end{figure}

The identification of this partition function with the symmetrization formula above
is well known and can be proved in several ways.
In standard textbooks on quantum integrable systems,
where the quantum \(R\)-matrix is used instead of the stochastic \(R\)-matrix,
it is discussed in \cite{Slavnov}, for example.
For higher-rank versions formulated in terms of nested Bethe vectors, see, for example, \cite{TVNested}.
Higher-spin versions appear, for example, in \cite{FelderVarchenko,BP}.
At the $q=0$ degeneration of the stochastic $R$-matrix, closely related five-vertex partition functions give Grothendieck polynomials; see, for example, \cite{MS,GK,WZ}.

Using this lattice-model realization of $W_{N,n}$,
we can graphically show the following factorization.

\begin{lem} \label{factorizationlemma}
    If $r=n$, we have
    \begin{align}
        F_{N,n,n}(\mathbf{u}|\mathbf{w}|\mathbf{I})
        =\prod_{j=1}^n (1-u_j^{-2}) \prod_{1 \le i < j \le n} (1-u_i^{-1} u_j^{-1})  W_{N,n}(\mathbf{u}|\mathbf{w}|\mathbf{I}).
        \label{factorizationlemmaidentity}
    \end{align}
\end{lem}

\begin{proof}

Starting from the top-left panel of Figure~\ref{fig:factorization}, we observe that for $r=n$, the ice rule \eqref{eq:ice-rule} restricts all states on the edges between the triangular and rectangular parts to be $1$. This restriction further fixes the boundary $K$-matrix elements to $[K(u_j)]_{0}^1$ in \eqref{eq:Ktri-components}, uniquely determining the configuration of the triangular part (see the middle-right panel). Removing this triangular part introduces the factor 
$\displaystyle
    \prod_{j=1}^n (1-u_j^{-2}) \prod_{1 \le i < j \le n} (1-u_i^{-1} u_j^{-1}).
$
The remaining rectangular part is the partition function $W_{N,n}(\mathbf{u}|\mathbf{w}|\mathbf{I})$. Hence we arrive at the factorization identity \eqref{factorizationlemmaidentity}.

\begin{figure}[!htbp]
  \centering
   \includegraphics[width=0.45\linewidth]{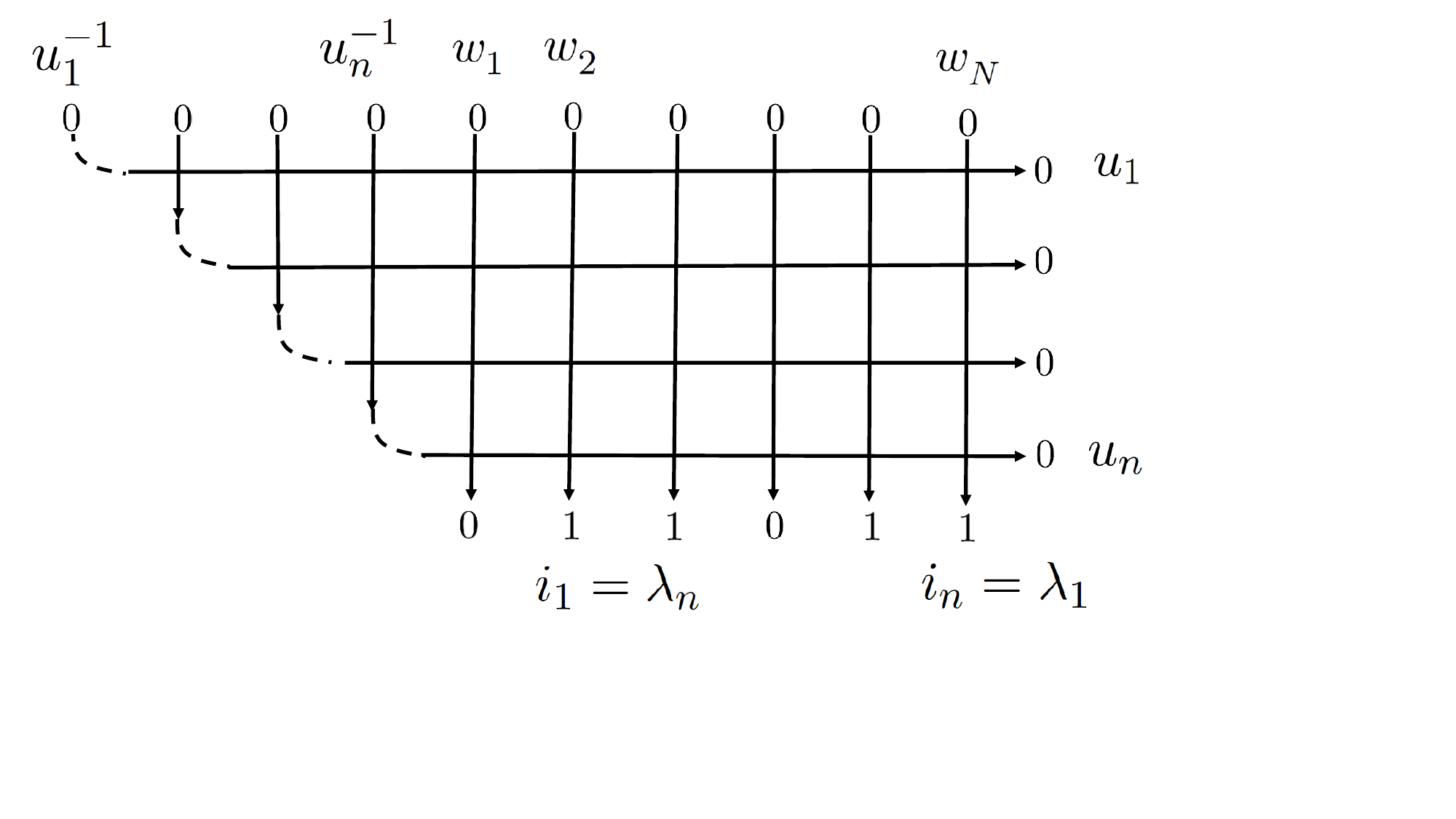}
 \includegraphics[width=0.45\linewidth]{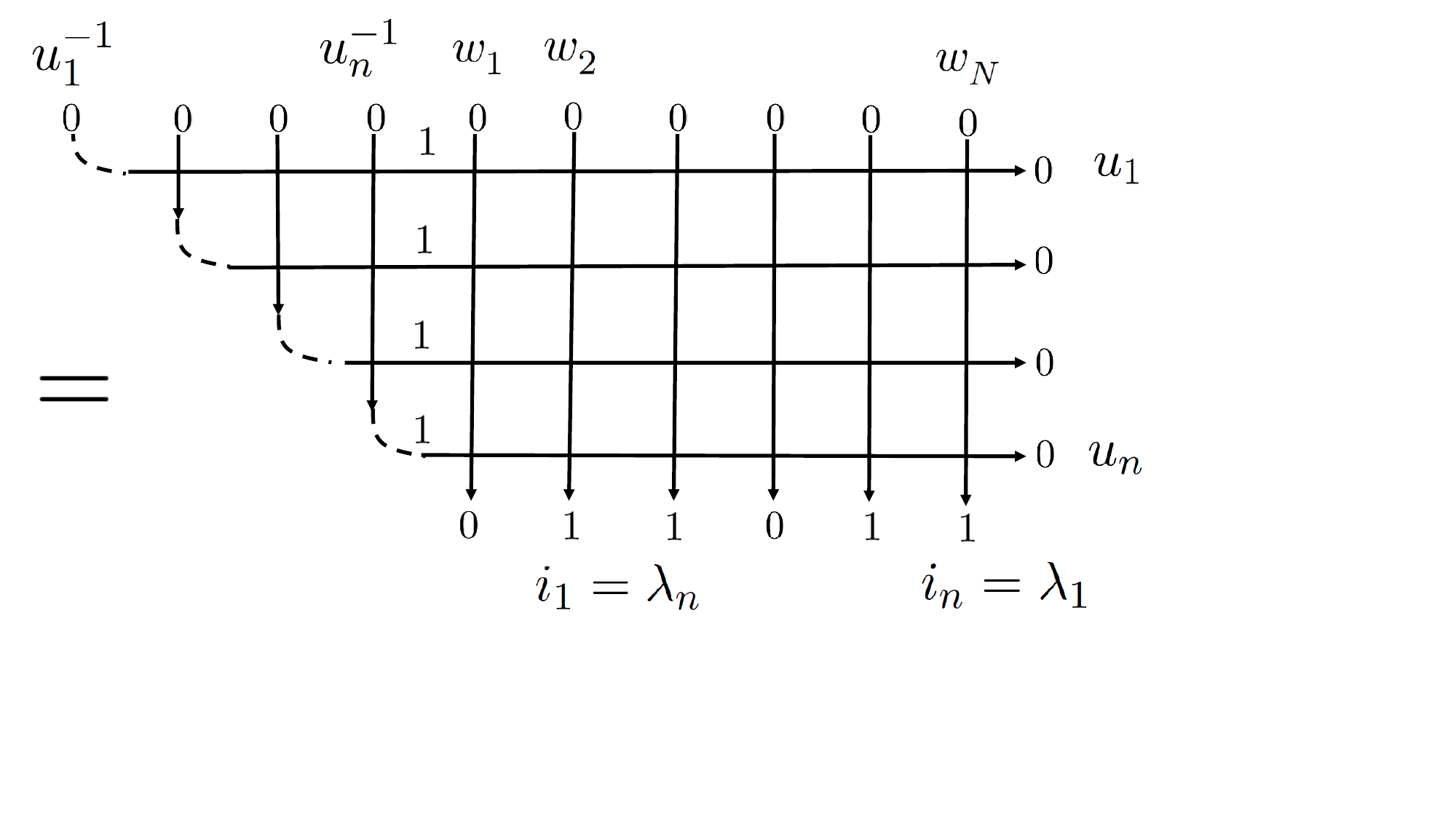}
\includegraphics[width=0.45\linewidth]{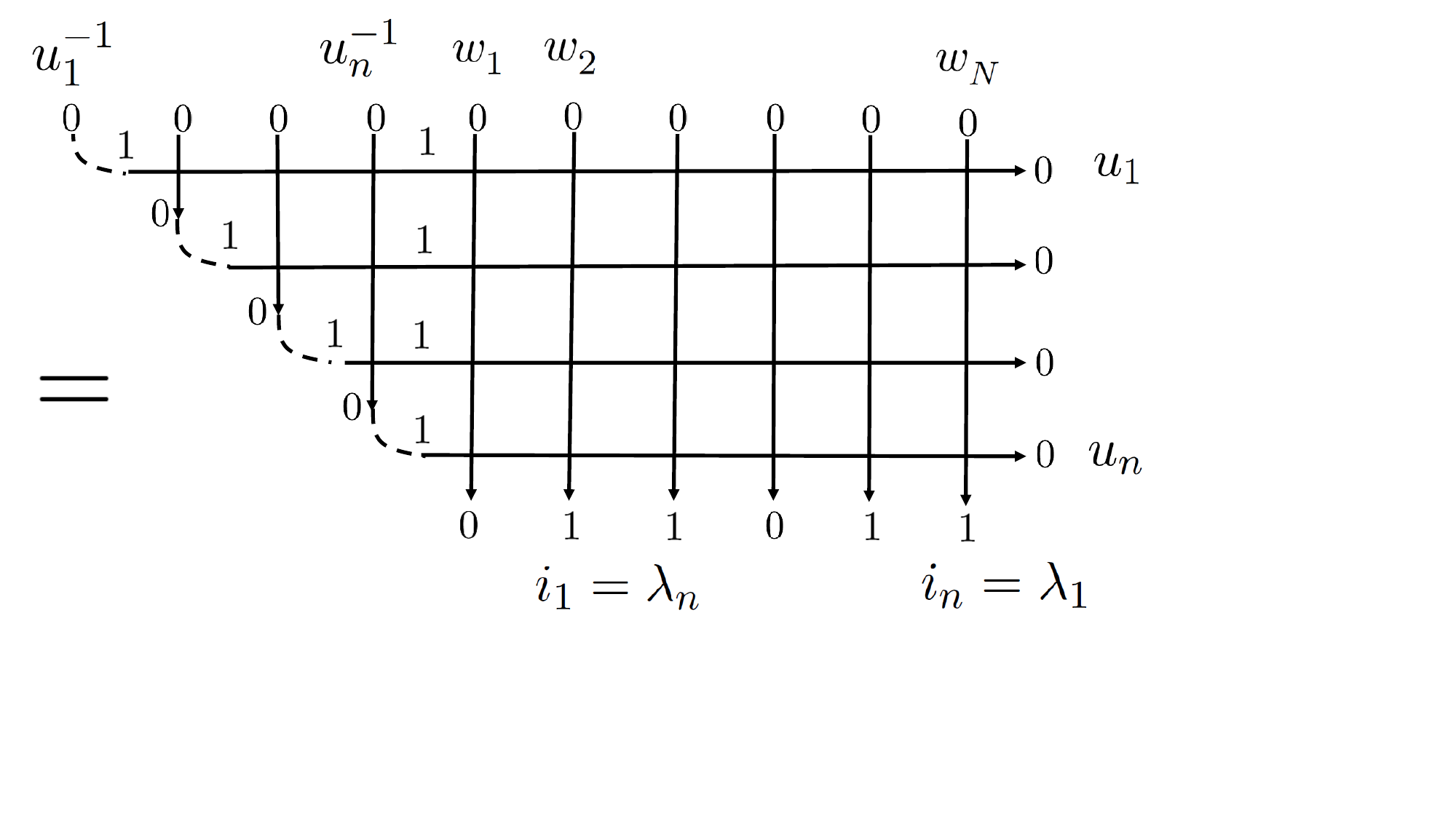}
\includegraphics[width=0.45\linewidth]{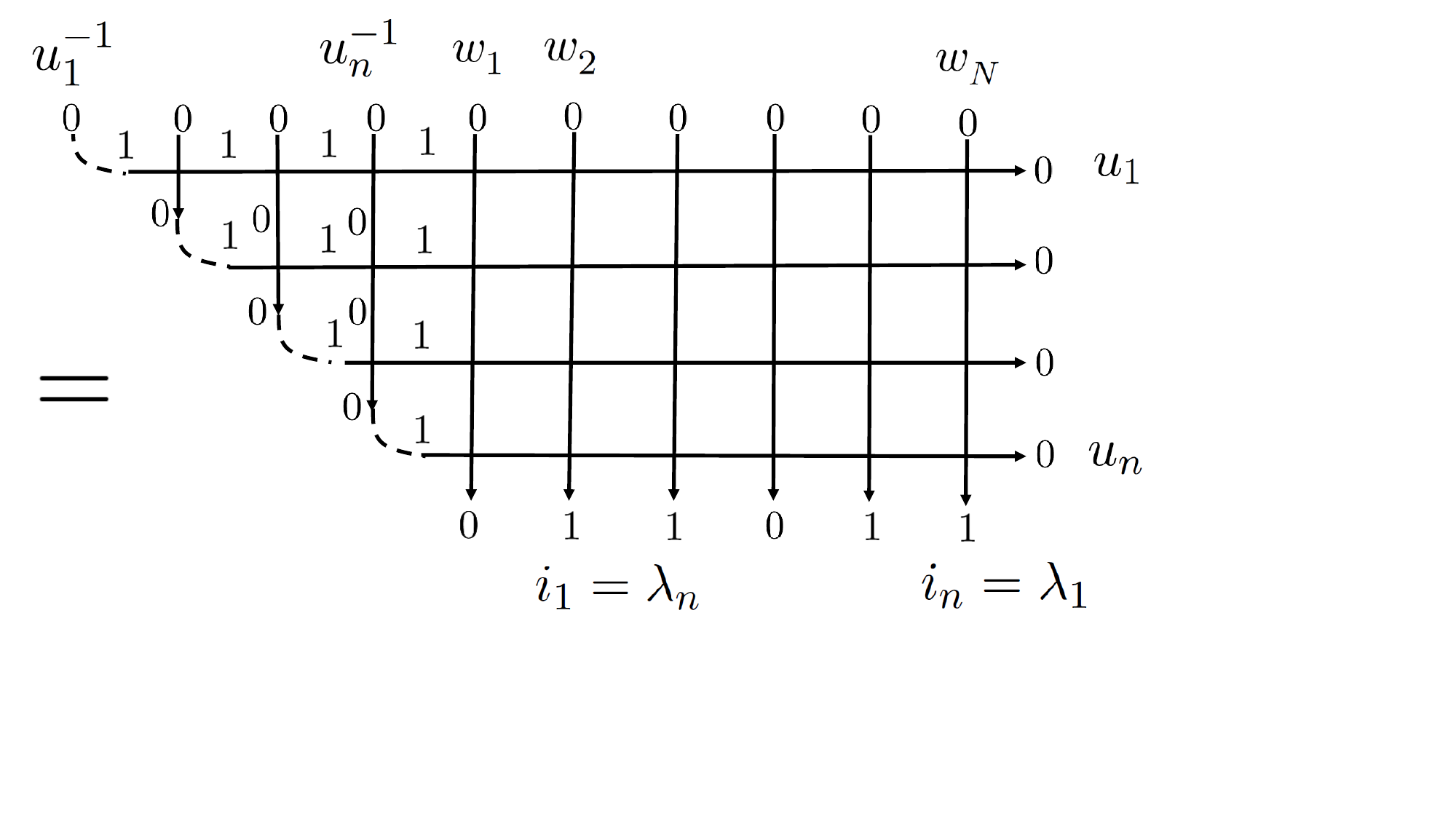}
 \includegraphics[width=0.7\linewidth]{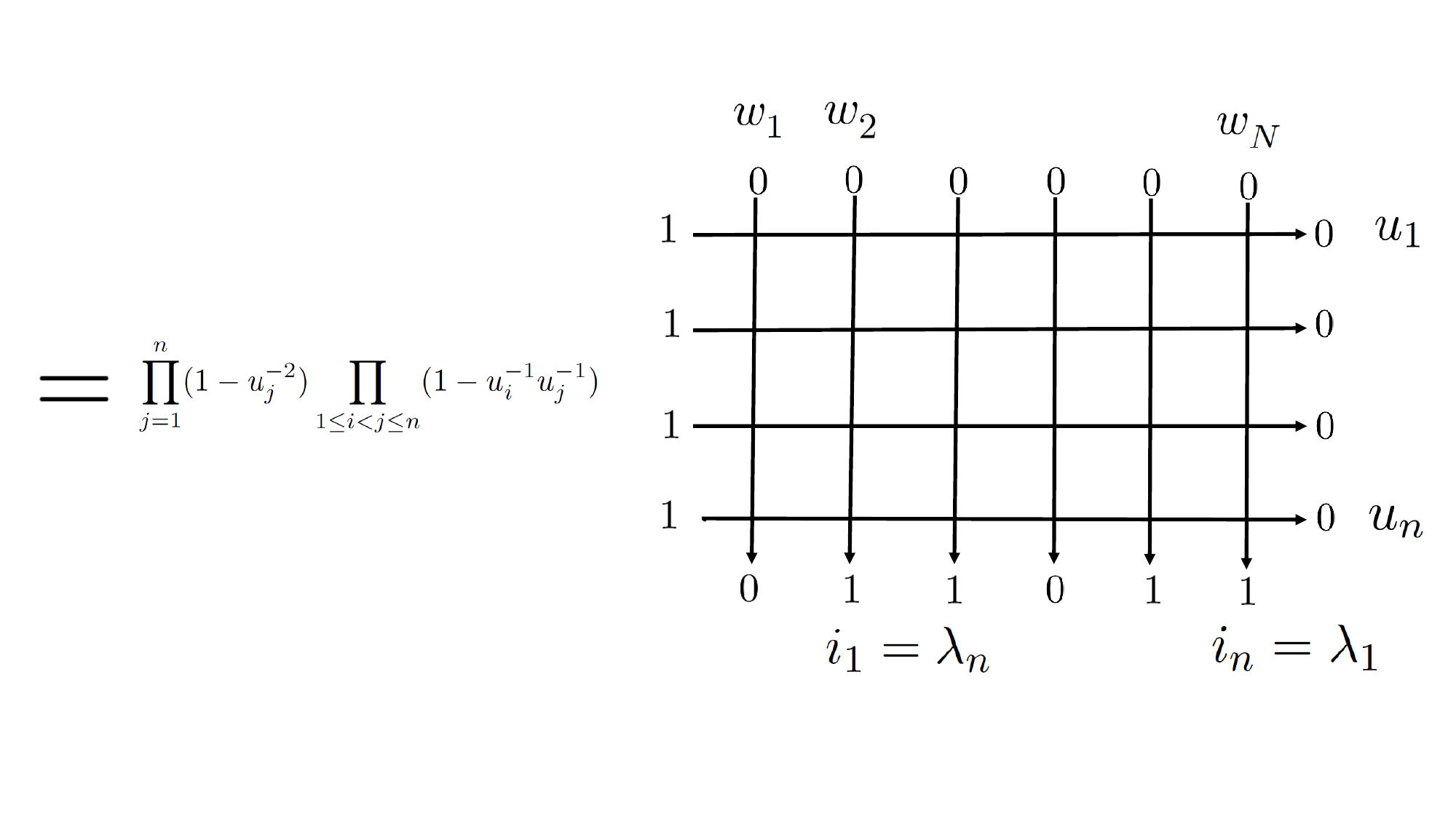}
  \caption{The factorization \eqref{factorizationlemmaidentity} when $r=n$.}
  \label{fig:factorization}
\end{figure}

\end{proof}

  \eqref{factorizationlemmaidentity} leads to the following,
  which is an extension of \cite[Proposition 2.3]{IkedaNaruse}.
\begin{corollary}
       If $r=n$, we have
    \begin{align}
        WC_{N,n,n}(\mathbf{u}|\mathbf{w}|\mathbf{I})
        &=\prod_{j=1}^n (1-u_j^{-2}) \prod_{1 \le i < j \le n} (1-u_i^{-1} u_j^{-1})  W_{N,n}(\mathbf{u}|\mathbf{w}|\mathbf{I}), \label{typeCnrFactorization}\\
           WB_{N,n,n}(\mathbf{u}|\mathbf{w}|\mathbf{I})
        &=\prod_{j=1}^n (1-u_j^{-1}) \prod_{1 \le i < j \le n} (1-u_i^{-1} u_j^{-1})  W_{N,n}(\mathbf{u}|\mathbf{w}|\mathbf{I}). \label{typeBnrFactorization}
    \end{align}

\end{corollary}

The following corresponds to the supersymmetry property.

\begin{prop}

Assume $r\le n-2$. If $u_i u_j=1$, we have
\begin{align}
     F_{N,n,r}(\mathbf{u}|\mathbf{w}|\mathbf{I})|_{u_i u_j=1} 
     &=(1-q)(\beta-\gamma u_i )  (\beta-\gamma u_{j}) 
     \prod_{k=1}^N \{ (1-qu_i w_k)(1-qu_j w_k)  \} \nonumber  \\
&  \times   \prod_{\substack{k=1 \\ k \neq i,j}}^n \{\
     (1-qu_i u_k^{-1})(1-q u_j u_k^{-1})
     \}
     F_{N,n-2,r}(\mathbf{u} \backslash \{ u_i, u_j \}|\mathbf{w}|\mathbf{I}). \label{reductionrelation}
\end{align}

\end{prop}

\begin{proof}

Since $F_{N,n,r}(\mathbf{u}|\mathbf{w}|\mathbf{I})$ is symmetric with respect to $\mathbf{u}$, it suffices to verify \eqref{reductionrelation} for $i=1$ and $j=2$, as illustrated in Figure~\ref{fig:super}. First, we note that the $R$-matrix at the upper-left corner reduces essentially to the permutation operator \eqref{eq:Ruu}. Together with the ice rule \eqref{eq:ice-rule}, this allows the configuration of the first row (excluding the corner site) to be uniquely fixed, as shown in the middle-left panel.

To proceed, we apply the auxiliary relation depicted in Figure~\ref{fig:auxsuper}, which implies that the boundary $K$-matrix elements in the first two rows must both be $[K]_{0}^0$. Consequently, the configuration of the second row is also uniquely determined. Removing these two rows yields the smaller weight function $F_{N,n-2,r}(\mathbf{u} \setminus \{ u_1, u_2 \}|\mathbf{w}|\mathbf{I})$. After multiplying by the contribution from the removed rows:
\begin{align}
    &(1-q)(\beta-\gamma u_1)(\beta-\gamma u_2) \prod_{k=1}^N \{ (1-qu_1 w_k)(1-qu_2 w_k) \}  \prod_{k=3}^n \{ (1-qu_1 u_k^{-1})(1-q u_2 u_k^{-1}) \}, \notag
\end{align}
we arrive at the reduction relation \eqref{reductionrelation}.

\begin{figure}[!htbp]
  \centering
   \includegraphics[width=0.45\linewidth]{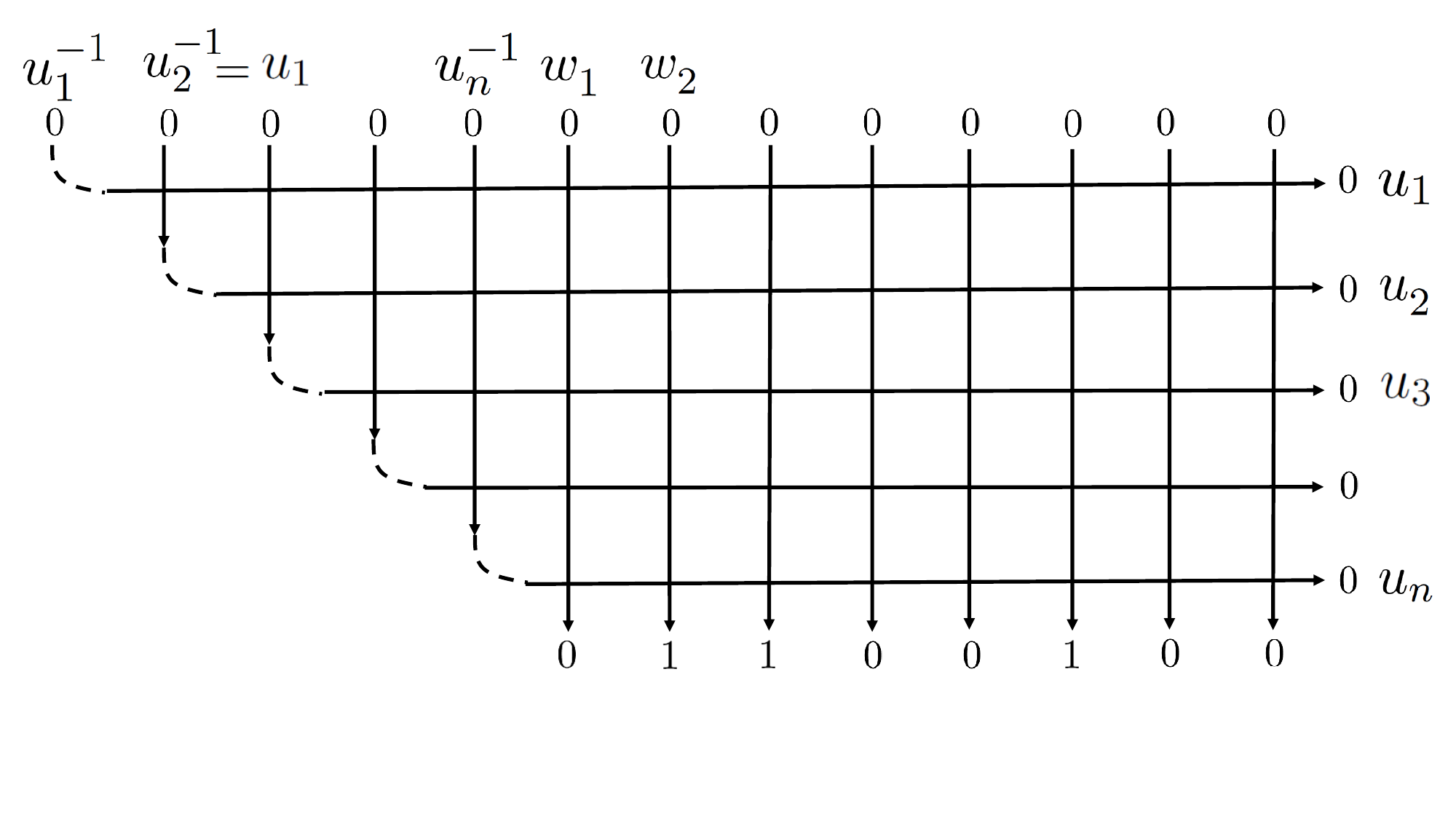}
 \includegraphics[width=0.45\linewidth]{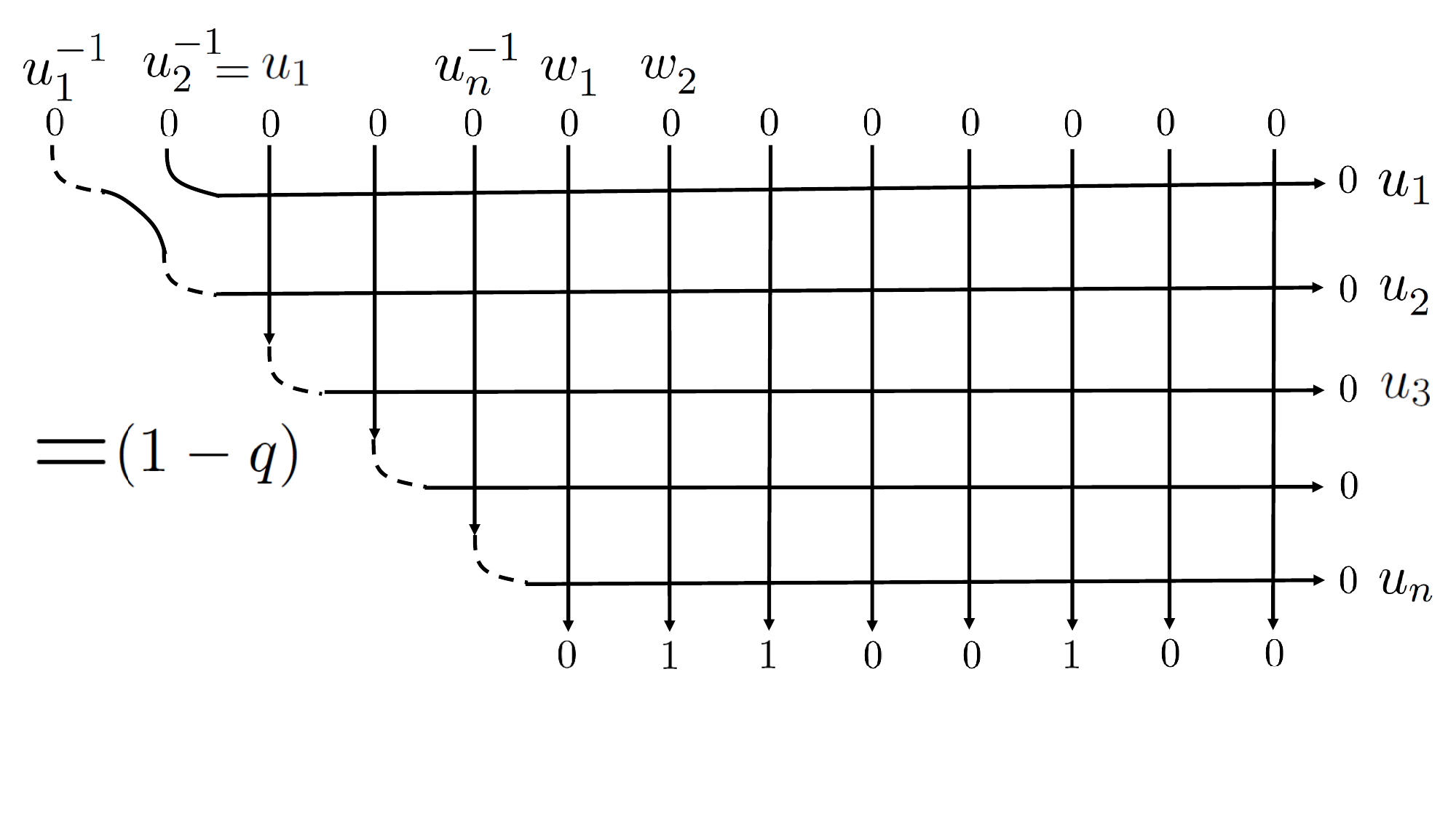}
\includegraphics[width=0.45\linewidth]{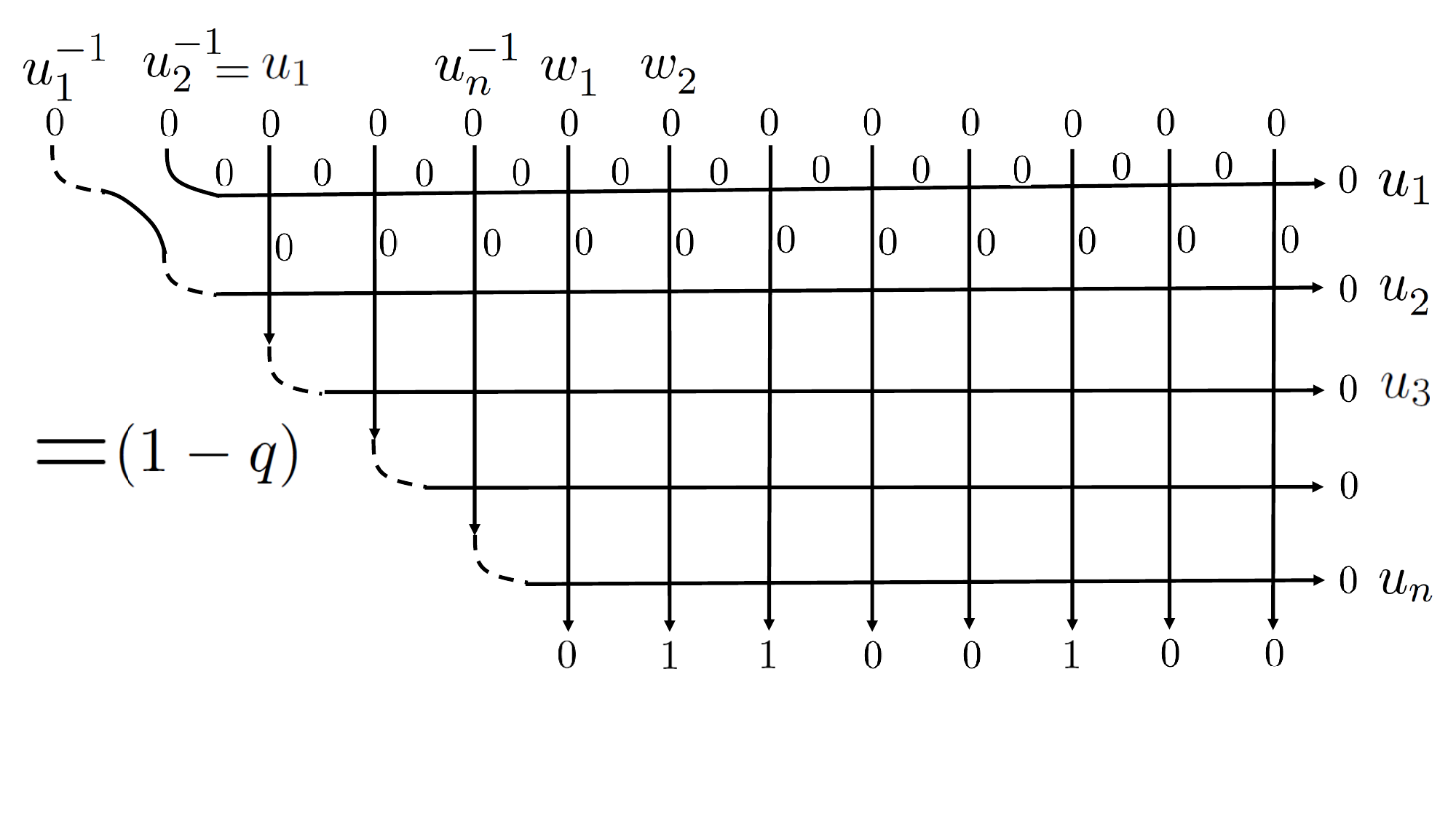}
\includegraphics[width=0.45\linewidth]{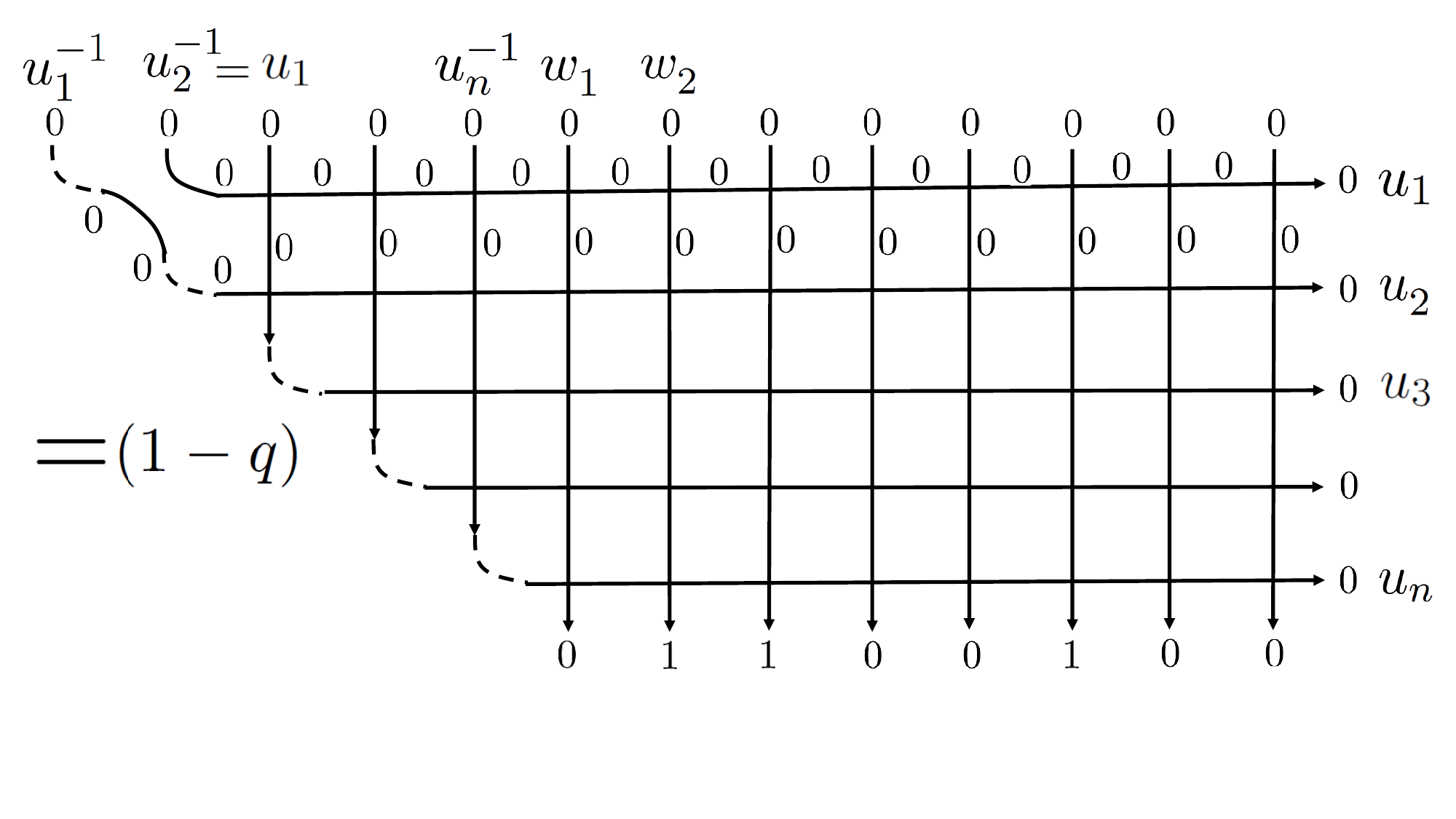}
 \includegraphics[width=0.45\linewidth]{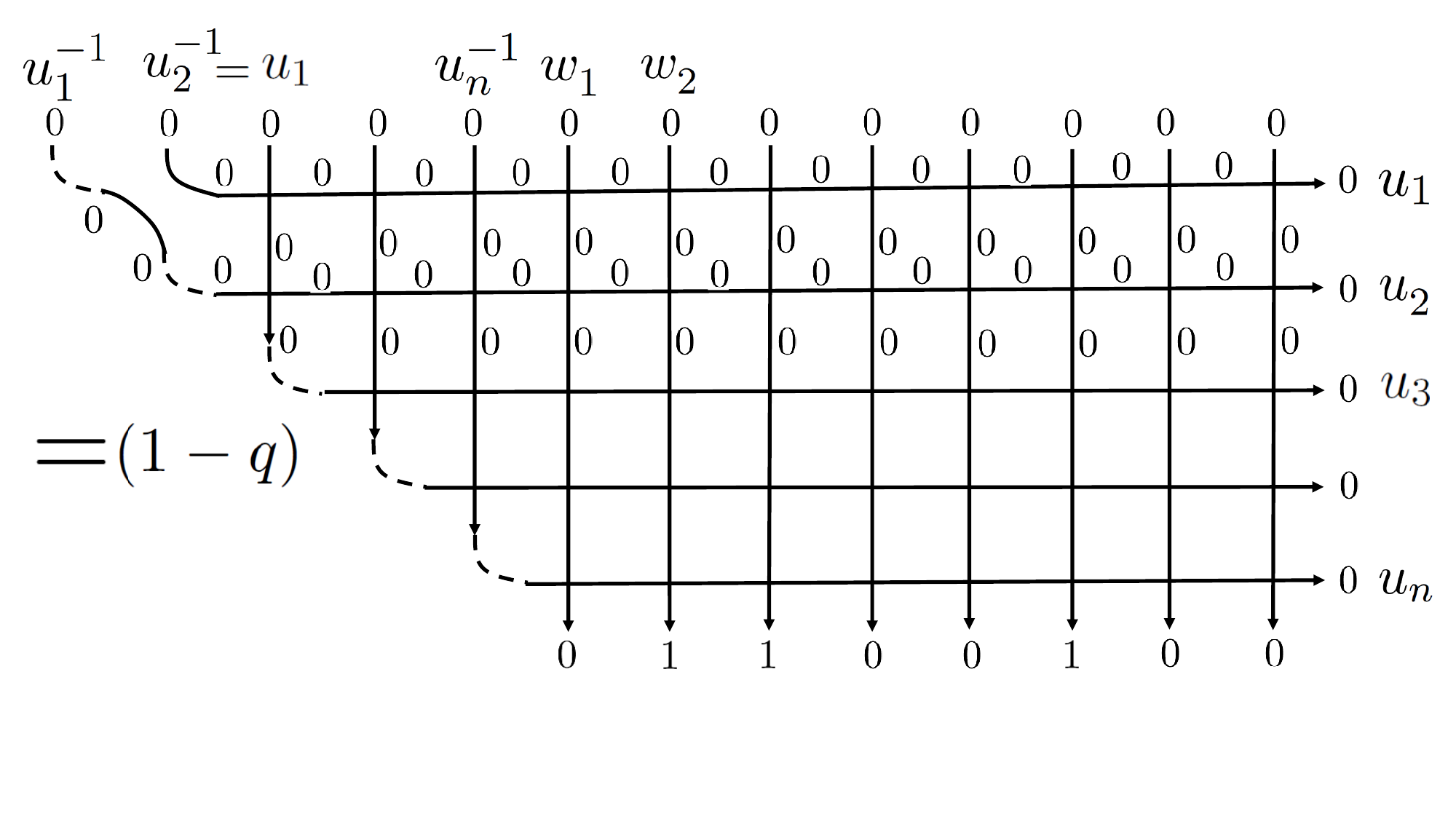}
 \includegraphics[width=0.45\linewidth]{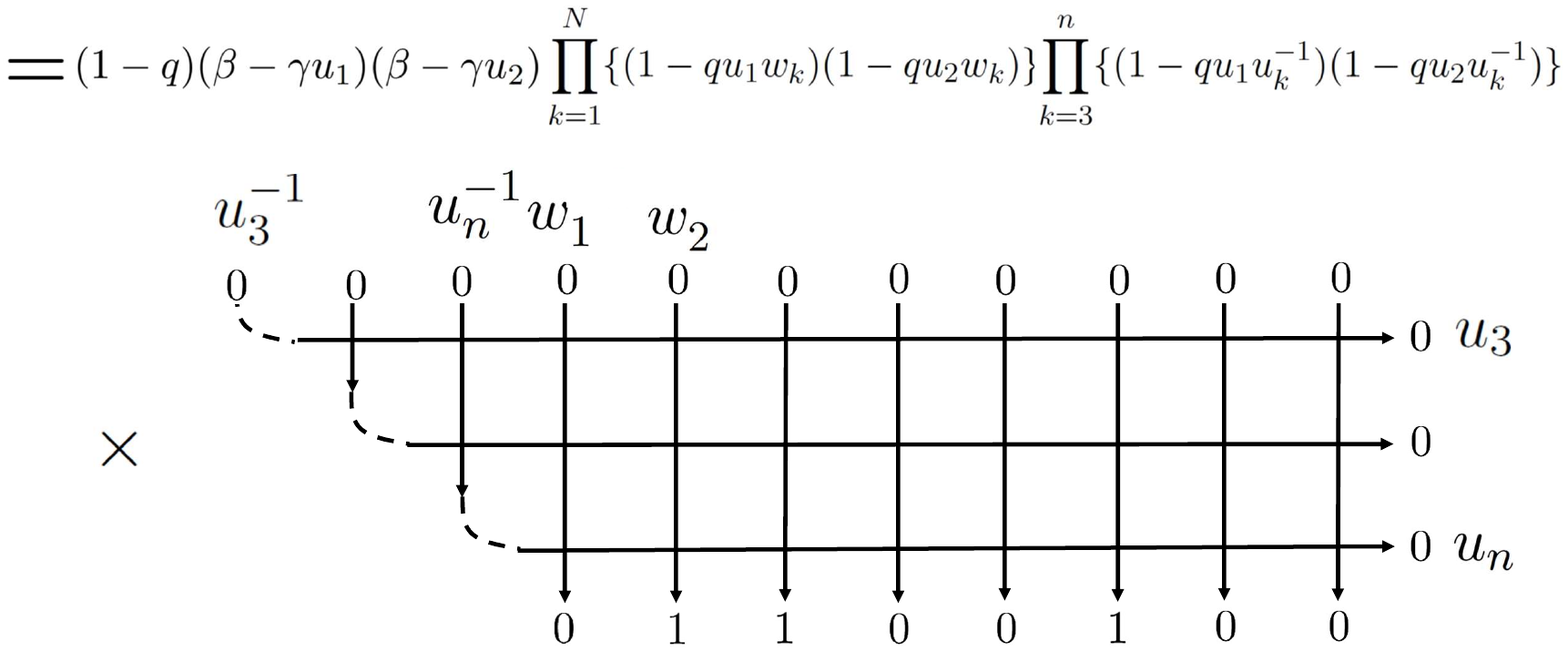}
  \caption{The reduction \eqref{reductionrelation}
  when $i=1$, $j=2$.}
  \label{fig:super}
\end{figure}

\begin{figure}[!htbp]
  \centering
   \includegraphics[width=0.7\linewidth]{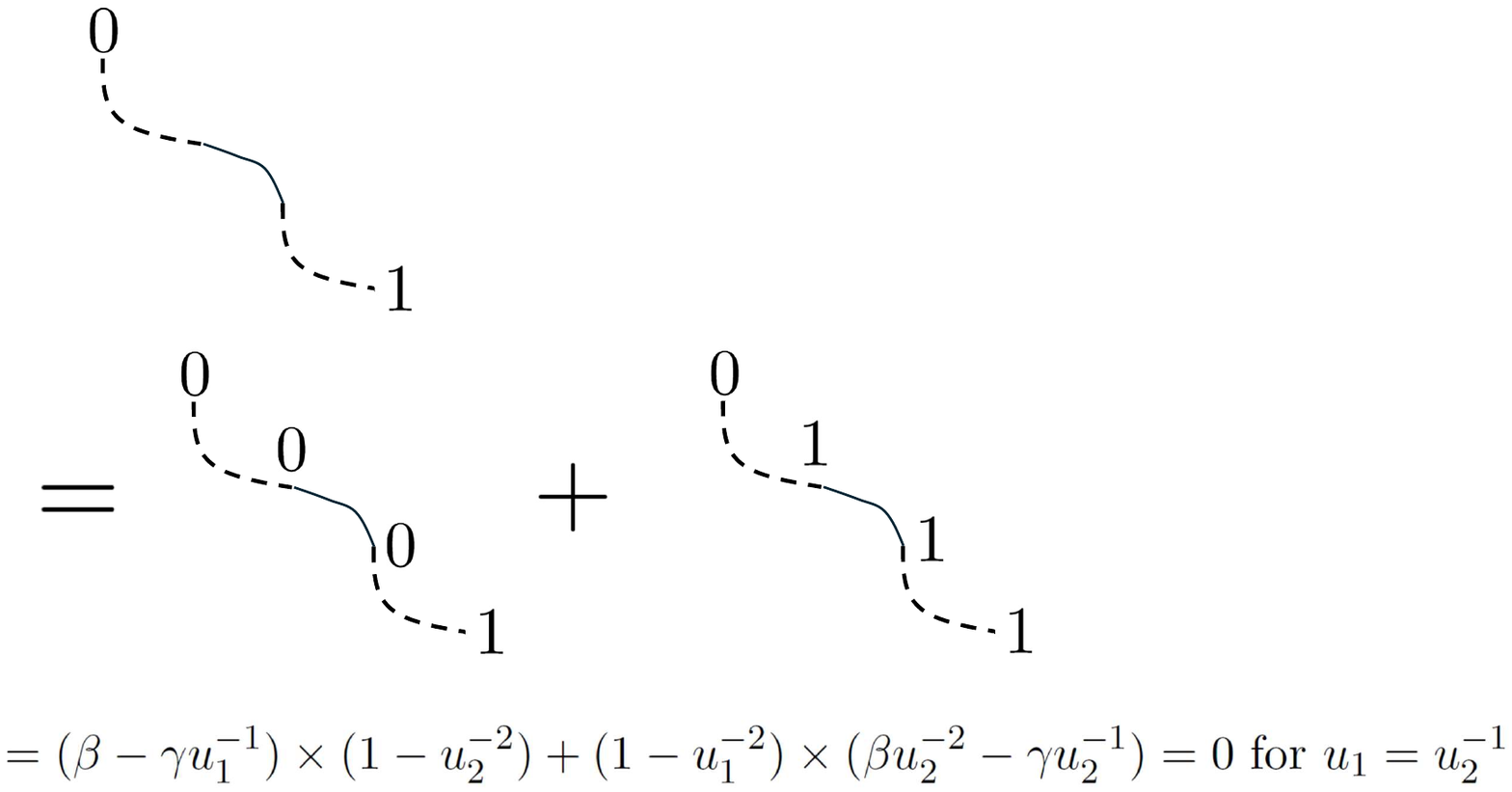}
  \caption{An auxiliary relation used in Figure~\ref{fig:super}.}
  \label{fig:auxsuper}
\end{figure}

\end{proof}

From
\eqref{reductionrelation},
 we have the following.

\begin{corollary} \label{susycor}
    
Assume $r\le n-2$. If $u_i u_j=1$, we have
\begin{align}
     WC_{N,n,r}(\mathbf{u}|\mathbf{w}|\mathbf{I})|_{u_i u_j=1} 
     =&(1-q) 
     \prod_{k=1}^N \{ (1-qu_i w_k)(1-qu_j w_k)  \} \nonumber \\
&   \times  \prod_{\substack{k=1 \\ k \neq i,j}}^n \{\
     (1-qu_i u_k^{-1})(1-q u_j u_k^{-1})
     \} \nonumber
     WC_{N,n-2,r}(\mathbf{u} \backslash \{ u_i, u_j \}|\mathbf{w}|\mathbf{I}), 
     \\
     WB_{N,n,r}(\mathbf{u}|\mathbf{w}|\mathbf{I})|_{u_i u_j=1} 
     =&(1-q) 
     \prod_{k=1}^N \{ (1-qu_i w_k)(1-qu_j w_k)  \} \nonumber  \\
& \times    \prod_{\substack{k=1 \\ k \neq i,j}}^n \{\
     (1-qu_i u_k^{-1})(1-q u_j u_k^{-1})
     \} \nonumber
     WB_{N,n-2,r}(\mathbf{u} \backslash \{ u_i, u_j \}|\mathbf{w}|\mathbf{I}).
\end{align}

\end{corollary}

\subsection{Relation with type $A$ weight functions}

For the relation with the type $A$ weight functions and its graphical proof below,
we assume
\[
|\mathbf{M}|=|\mathbf{I}|=r,
\qquad
\mathbf{M}=(m_1<\cdots<m_r).
\]
Thus the proposition below and the graphical reduction
to the rectangular type $A$ partition function concern only the equal-cardinality
specializations $\mathbf{u}=\mathbf{w}_{\mathbf{M}}$. They give another proof of the
$|\mathbf{M}|=|\mathbf{I}|$ part of Theorem~\ref{ffunctionproperties},
$\mathrm{(i)}$, and give a second form of the diagonal specialization
$F_{N,n,r}(\mathbf{w}_{\mathbf{I}}|\mathbf{w}|\mathbf{I})$ which will be used
in Section~\ref{sec:normalized-algebraic-properties}. After establishing this
relation, we give a separate graphical explanation of the vanishing when
$s<r$. This argument does not use the type $A$ relation.

\begin{prop}
We have
\begin{align}
F_{N,n,r}(\mathbf{w}_{\mathbf{M}}|\mathbf{w}|\mathbf{I}) 
=&
 (1-q)^{\frac{(n-r)(n-r-1)}{2}} (\beta-\gamma)^{n-r} 
 \prod_{1 \le j \le k \le r} (1-w_{m_j}^{-1} w_{m_k}^{-1})
\nonumber \\
&\times \prod_{j=1}^N (1-qw_j)^{n-r} \prod_{k=1}^r (1 - q w_{m_k}^{-1})^{n-r}  W_{N,r}(w_{m_1},\dots,w_{m_r}|\mathbf{w}|\mathbf{I}).
\label{fixedpointrelationFandWA}
\end{align}

\end{prop}

\begin{proof}

We investigate the partition function $F_{N,n,r}(\mathbf{w}_{\mathbf{M}}|\mathbf{w}|\mathbf{I})$ in its lattice-model realization. 
By leveraging its symmetry with respect to the $u$-variables, we can rearrange their assignment to the lines as shown in the top-left panel of Figure~\ref{fig:fromtri}. 
Upon setting $u_{r+1} = \dots = u_n = 1$, we observe that the configuration in the top $(n-r)$ rows is uniquely determined by the condition $[K(1)]_{0}^1 = 0$ from \eqref{eq:Ktri-components} and the ice rule \eqref{eq:ice-rule} of the $R$-matrix. 
Furthermore, the remaining part corresponds to a configuration where we can apply the same argument used to derive Lemma~\ref{factorizationlemma}. 
By collecting the corresponding $R$-matrix and $K$-matrix elements, we obtain the overall factor:
\begin{align}
& (1-q)^{\frac{(n-r)(n-r-1)}{2}} (\beta-\gamma)^{n-r} \prod_{1 \le j \le k \le r} (1-w_{m_j}^{-1} w_{m_k}^{-1})  \prod_{j=1}^N (1-qw_j)^{n-r} \prod_{k=1}^r (1 - q w_{m_k}^{-1})^{n-r}. \label{overallfactorfordiagonal}
\end{align}
The remaining part of the partition function then reduces to a rectangular shape, as depicted in the bottom panel of Figure~\ref{fig:fromtri}, which is nothing but the type $A$ partition function
$
W_{N,r}(w_{m_1},\dots,w_{m_r}|\mathbf{w}|\mathbf{I})$.
Hence we obtain
\eqref{fixedpointrelationFandWA}.

\end{proof}

\begin{figure}[!htbp]
  \centering
   \includegraphics[width=0.45\linewidth]{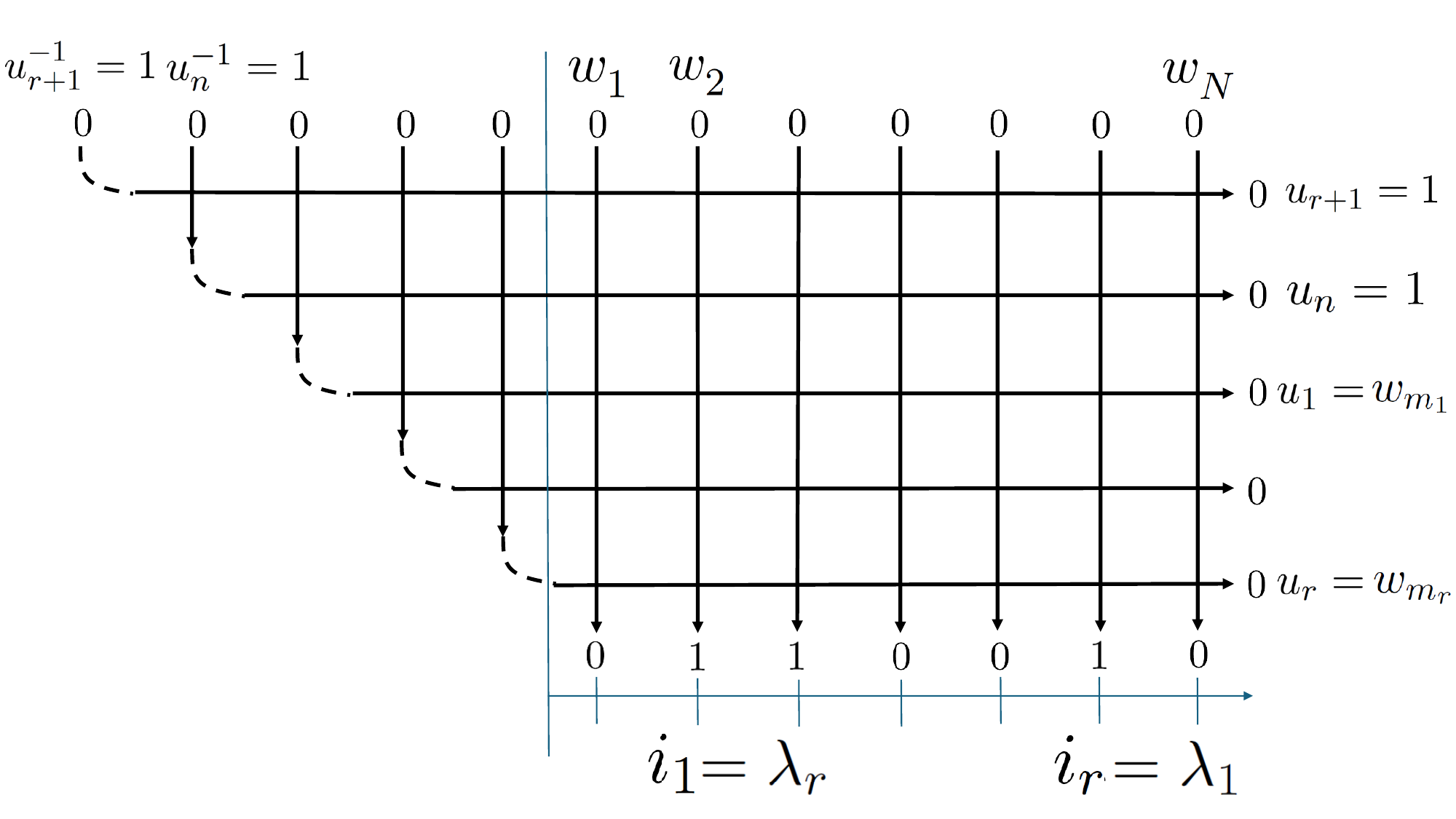}
 \includegraphics[width=0.45\linewidth]{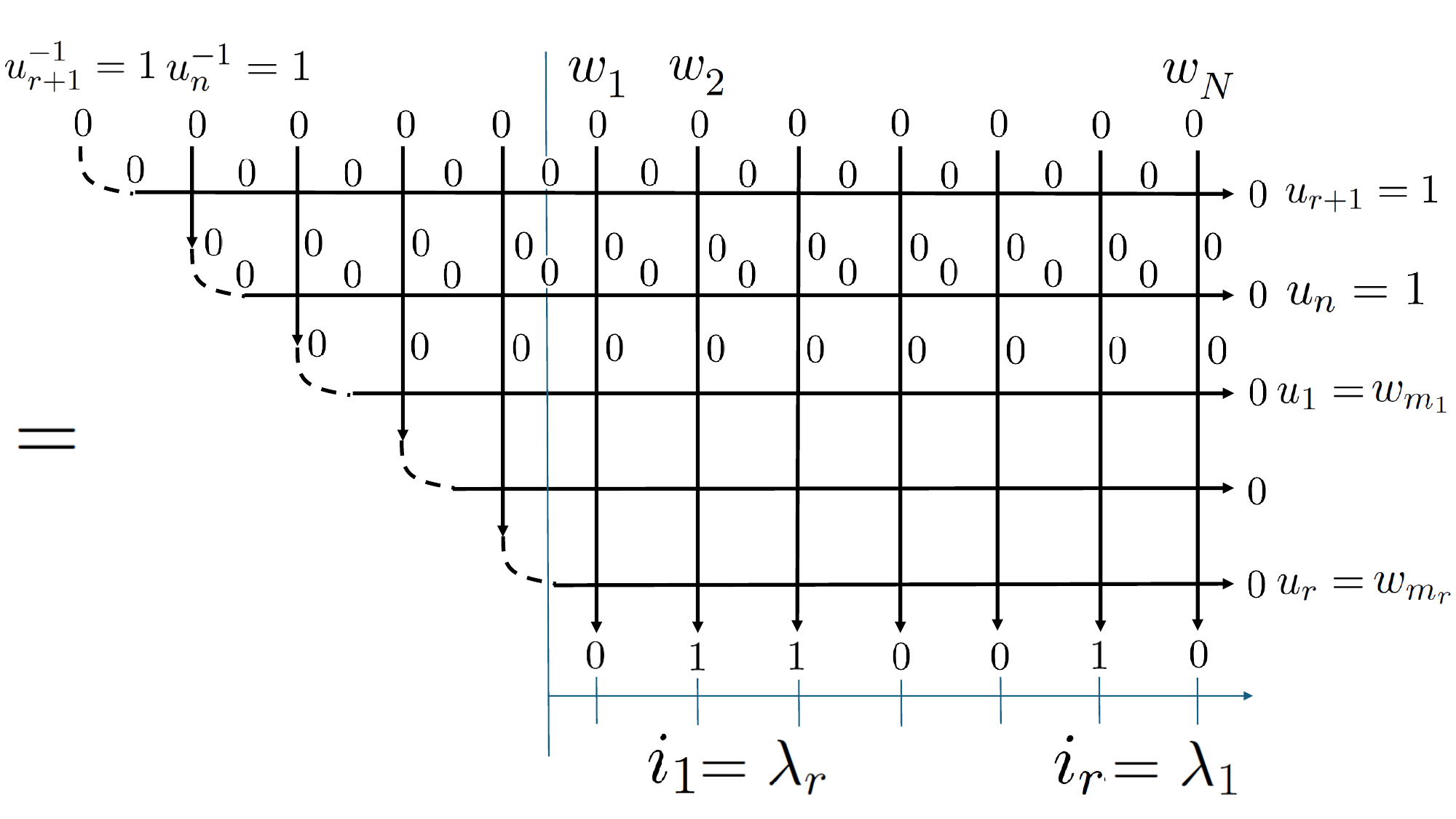}
\includegraphics[width=0.45\linewidth]{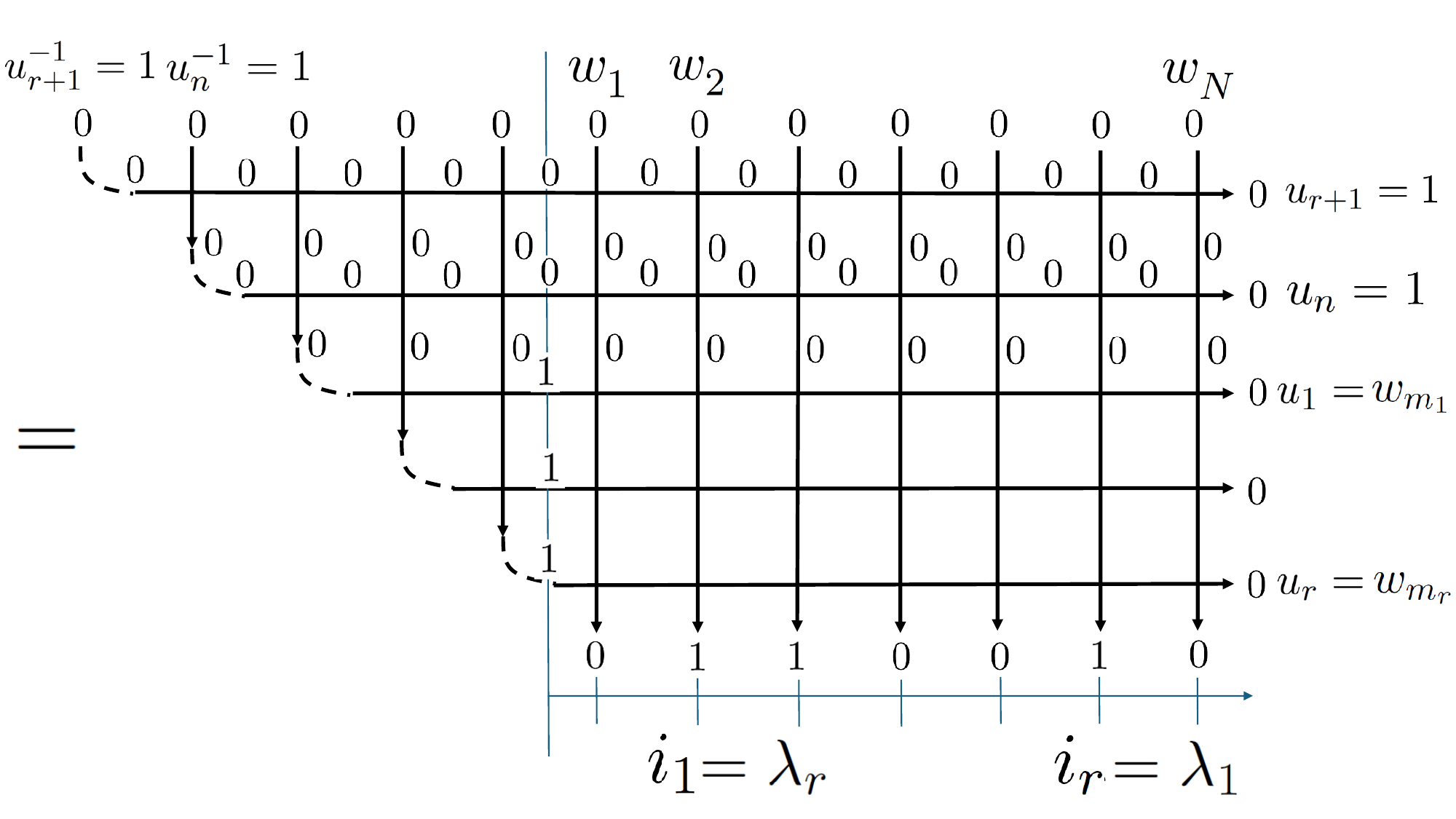}
\includegraphics[width=0.45\linewidth]{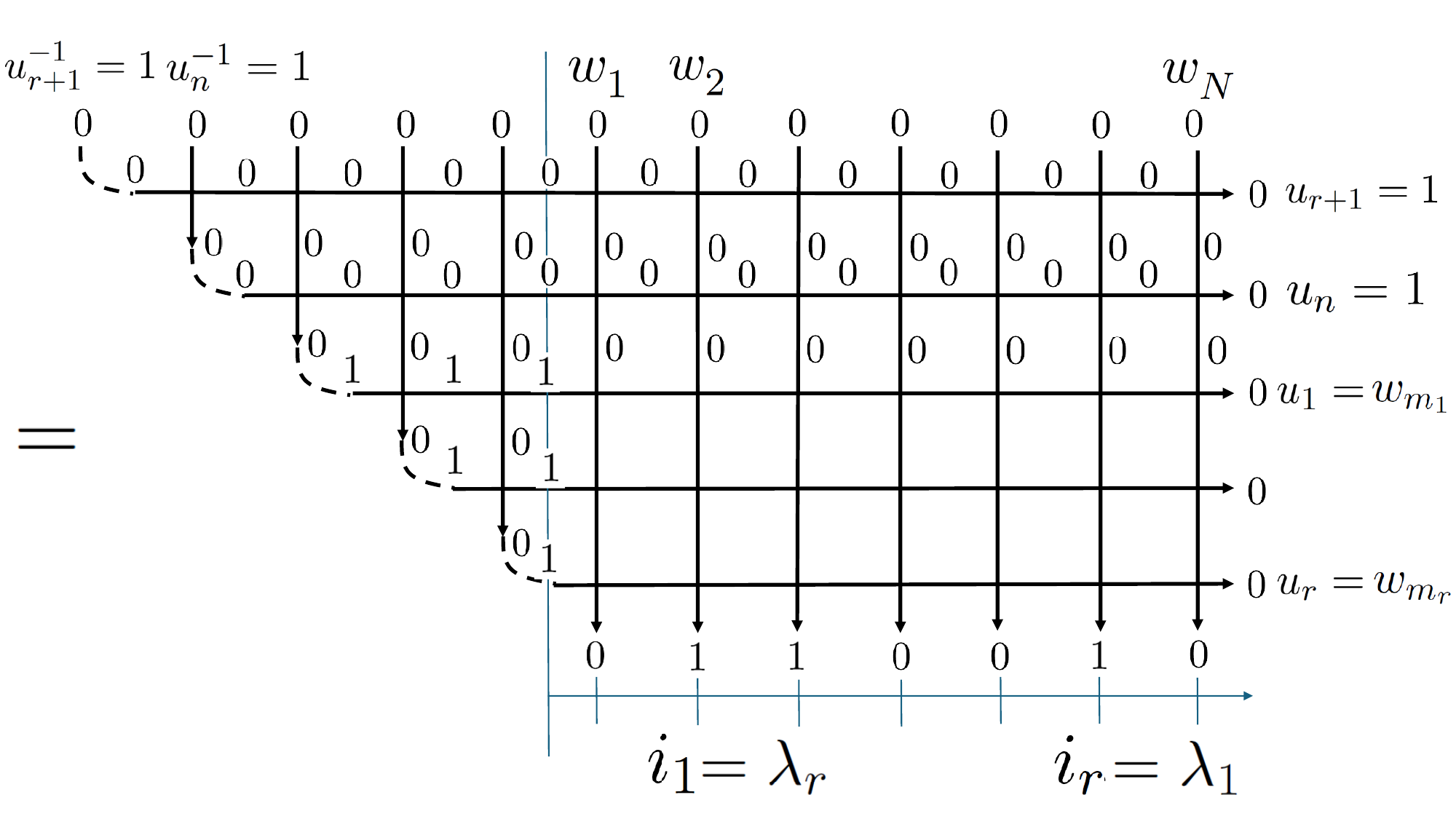}
 \includegraphics[width=0.45\linewidth]{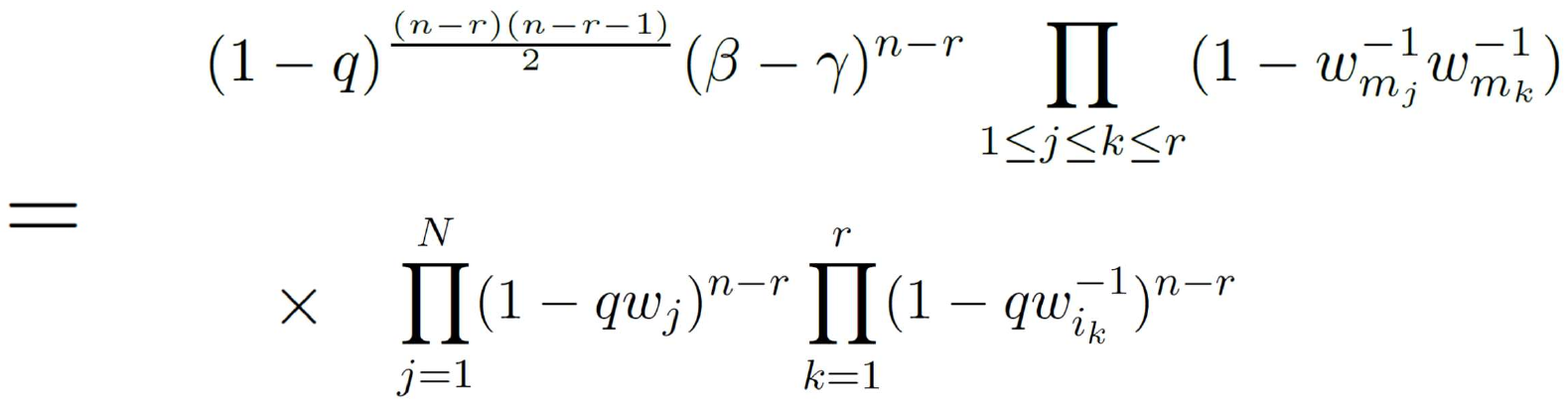}
\includegraphics[width=0.45\linewidth]{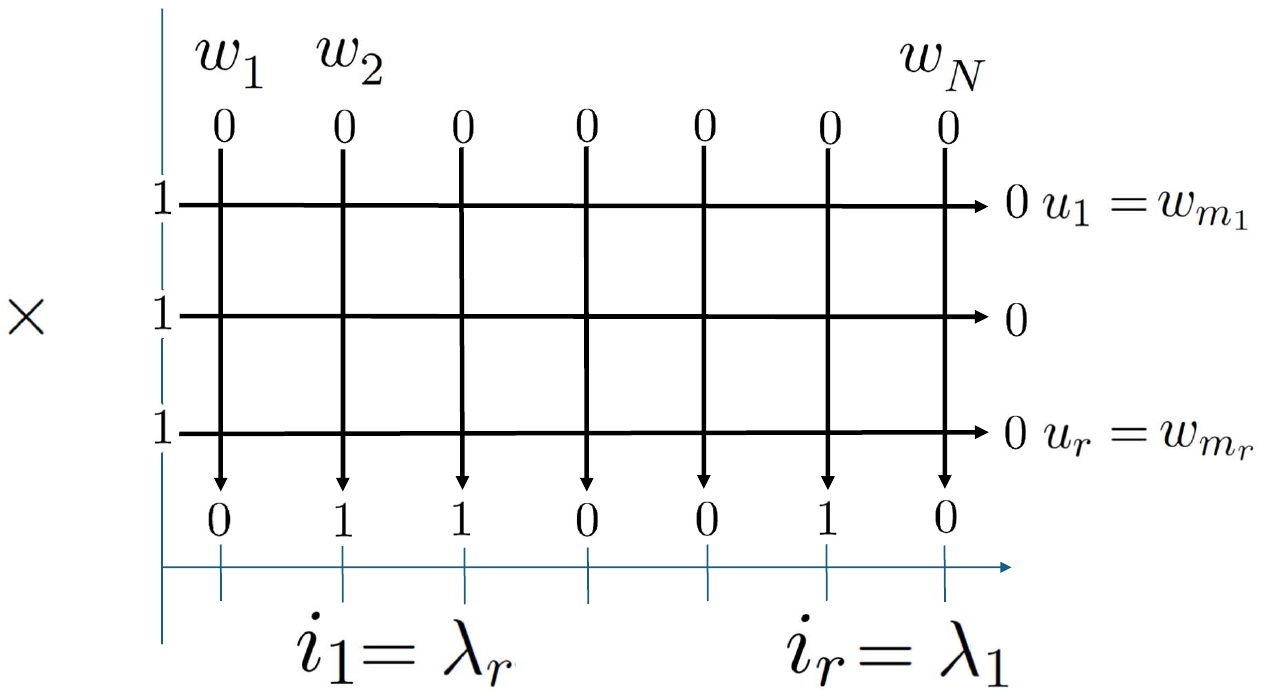}
  \caption{The partition function $F_{N,n,r}(\mathbf{w}_{\mathbf{M}}|\mathbf{w}|\mathbf{I})$
  reduces to the rectangular partition function
  multiplied by the overall factor
   \eqref{overallfactorfordiagonal}.}
  \label{fig:fromtri}
\end{figure}

\begin{figure}[htbp]
  \centering
  \includegraphics[width=0.7\linewidth]{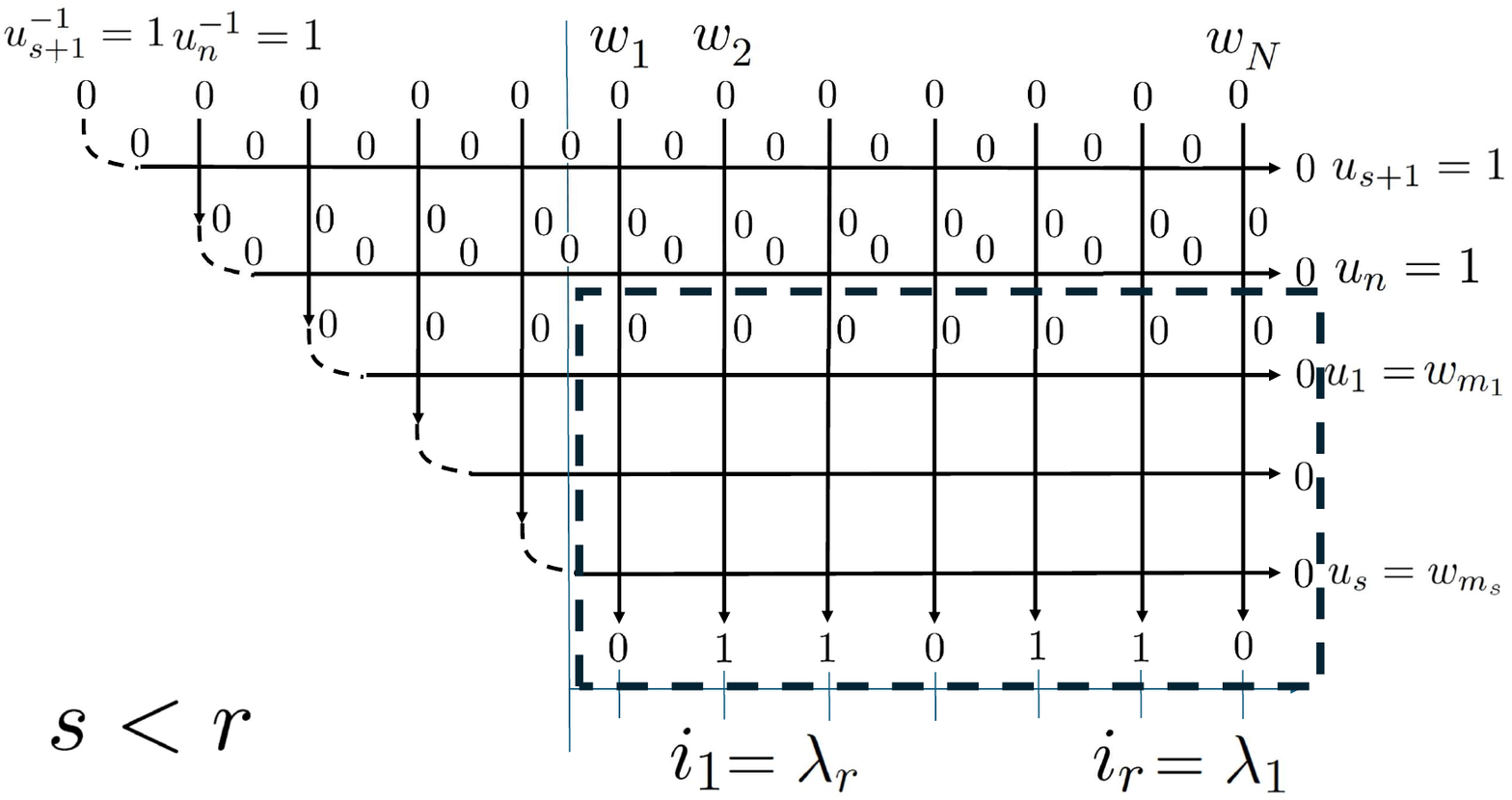} 
  \caption{Graphical explanation of the vanishing when $s<r$.
After freezing the upper $n-s$ rows, consider the remaining rectangular
part enclosed by the thick dashed lines. Its top and right boundary states are all $0$,
whereas its bottom boundary contains exactly $r$ edges in state $1$.
By the ice rule \eqref{eq:ice-rule}, the left boundary of this rectangle must therefore
contain exactly $r$ edges in state $1$. Since this boundary has only
$s<r$ edges, this is impossible, and the partition function vanishes.
  }
  \label{fig:lower-cardinality-vanishing}
\end{figure}

When $s<r$, the vanishing can also be seen graphically. First, by the
same argument as when $s=r$, the states in the upper $n-s$ rows
are uniquely determined, and we obtain the partial configuration shown
in Figure~\ref{fig:lower-cardinality-vanishing}. Notice that these
frozen rows are retained in the figure.

Now consider the rectangular part which is not frozen, enclosed by the
thick dashed lines in Figure~\ref{fig:lower-cardinality-vanishing}. The states
on its top and right boundaries are all $0$, while its bottom boundary
contains exactly $r$ edges in state $1$ and all the remaining edges in
state $0$. Therefore, by the ice rule \eqref{eq:ice-rule} for the $R$-matrix, the left
boundary of this rectangle must contain exactly $r$ edges in state $1$
for the partition function to be nonzero. On the other hand, this left
boundary consists of only $s$ edges. Since $s<r$, it is impossible for
exactly $r$ of them to be in state $1$. Hence there is no admissible
configuration, and the partition function vanishes.

\begin{proof}[Graphical proof of the equal-cardinality part of Theorem~\ref{ffunctionproperties}, $\mathrm{(i)}$]

Under the standing assumption $|\mathbf{M}|=|\mathbf{I}|=r$, we show that the rectangular partition functions
$
W_{N,r}(w_{m_1},\dots,w_{m_r}|\mathbf{w}|\mathbf{I})$
are given by
\begin{align}
\prod_{j=1}^r \prod_{k=1}^r (1-q w_{i_k}^{-1} w_{i_j}) 
\prod_{k=1}^{r+1} \left[ \prod_{i_{k-1} < j < i_k} \left\{ \prod_{l=k}^r (1-w_{i_l}^{-1} w_j) \prod_{l=1}^{k-1} (1-q w_{i_l}^{-1} w_j) \right\} \right], \label{rectdiagonalfactor}
\end{align}
where $i_0:=0$, $i_{r+1}:=N+1$
for the case $\mathbf{M}=\mathbf{I}$, and that they vanish unless the condition $\mathbf{M} \ge \mathbf{I}$ is satisfied.

When $\mathbf{M}=\mathbf{I}$, the configuration is uniquely determined starting from the leftmost column by the ice rule \eqref{eq:ice-rule}.
For each $i_j$-th column ($j=1, \dots, r$), the configuration is depicted in the top-left panel of Figure~\ref{fig:fromtritwo}. By \eqref{eq:R-components}, these columns give the weights $\prod_{k=1}^r (1-q w_{i_k}^{-1} w_{i_j})$, resulting in the factor
\begin{equation*}
\prod_{j=1}^r \prod_{k=1}^r (1-q w_{i_k}^{-1} w_{i_j}).
\end{equation*}
For a column index $j$ such that $i_{k-1} < j < i_k$, the configuration corresponds to the top-right panel of Figure~\ref{fig:fromtritwo}. The weight contribution from such a column is $\prod_{l=k}^r (1-w_{i_l}^{-1} w_j) \prod_{l=1}^{k-1} (1-q w_{i_l}^{-1} w_j)$, which leads to the factor
\begin{align}
\prod_{k=1}^{r+1} \left[ \prod_{i_{k-1} < j < i_k} \left\{ \prod_{l=k}^r (1-w_{i_l}^{-1} w_j) \prod_{l=1}^{k-1} (1-q w_{i_l}^{-1} w_j) \right\} \right]. \notag
\end{align}
The case $k=r+1$ in the preceding product accounts for the
columns $i_r+1,\ldots,N$: the product over $l=k,\ldots,r$
is empty, and their contribution is
$\prod_{j=i_r+1}^{N}\prod_{l=1}^{r}
(1-qw_{i_l}^{-1}w_j)$.
Together with the contributions from the columns
$i_1,\ldots,i_r$, this gives \eqref{rectdiagonalfactor}.

Conversely, if the condition $\mathbf{M} \ge \mathbf{I}$ is not satisfied, let $j$ be the minimal index such that $m_j < i_j$. By analyzing the configuration from the left, we find that the $m_j$-th column must take the form shown in the bottom panel of Figure~\ref{fig:fromtritwo}. In this case, there appears an inadmissible configuration (where, by the ice rule \eqref{eq:ice-rule}, the $R$-matrix element must be zero), implying that the rectangular partition functions vanish identically.

From \eqref{fixedpointrelationFandWA} and
 \eqref{rectdiagonalfactor},
 we have
\begin{align}
&F_{N,n,r}(\mathbf{w}_{\mathbf{I}}|\mathbf{w}|\mathbf{I})
=
 (1-q)^{\frac{(n-r)(n-r-1)}{2}} (\beta-\gamma)^{n-r} 
\nonumber \\
 &\times 
 \prod_{1 \le j \le k \le r} (1-w_{i_j}^{-1} w_{i_k}^{-1})
 \prod_{j=1}^N (1-qw_j)^{n-r} \prod_{k=1}^r (1 - q w_{i_k}^{-1})^{n-r} \nonumber \\
 &\times
\prod_{j=1}^r \prod_{k=1}^r (1-q w_{i_k}^{-1} w_{i_j}) \prod_{k=1}^{r+1} \left[ \prod_{i_{k-1} < j < i_k} \left\{ \prod_{l=k}^r (1-w_{i_l}^{-1} w_j) \prod_{l=1}^{k-1} (1-q w_{i_l}^{-1} w_j) \right\} \right].
\label{diagonalFfunctionanotherform}
\end{align}
A direct calculation shows that this expression is equivalent to \eqref{diagonalFfunction}.

\begin{figure}[!htbp]
  \centering
  \includegraphics[width=0.45\linewidth]{fromtrisix.pdf}
 \includegraphics[width=0.45\linewidth]{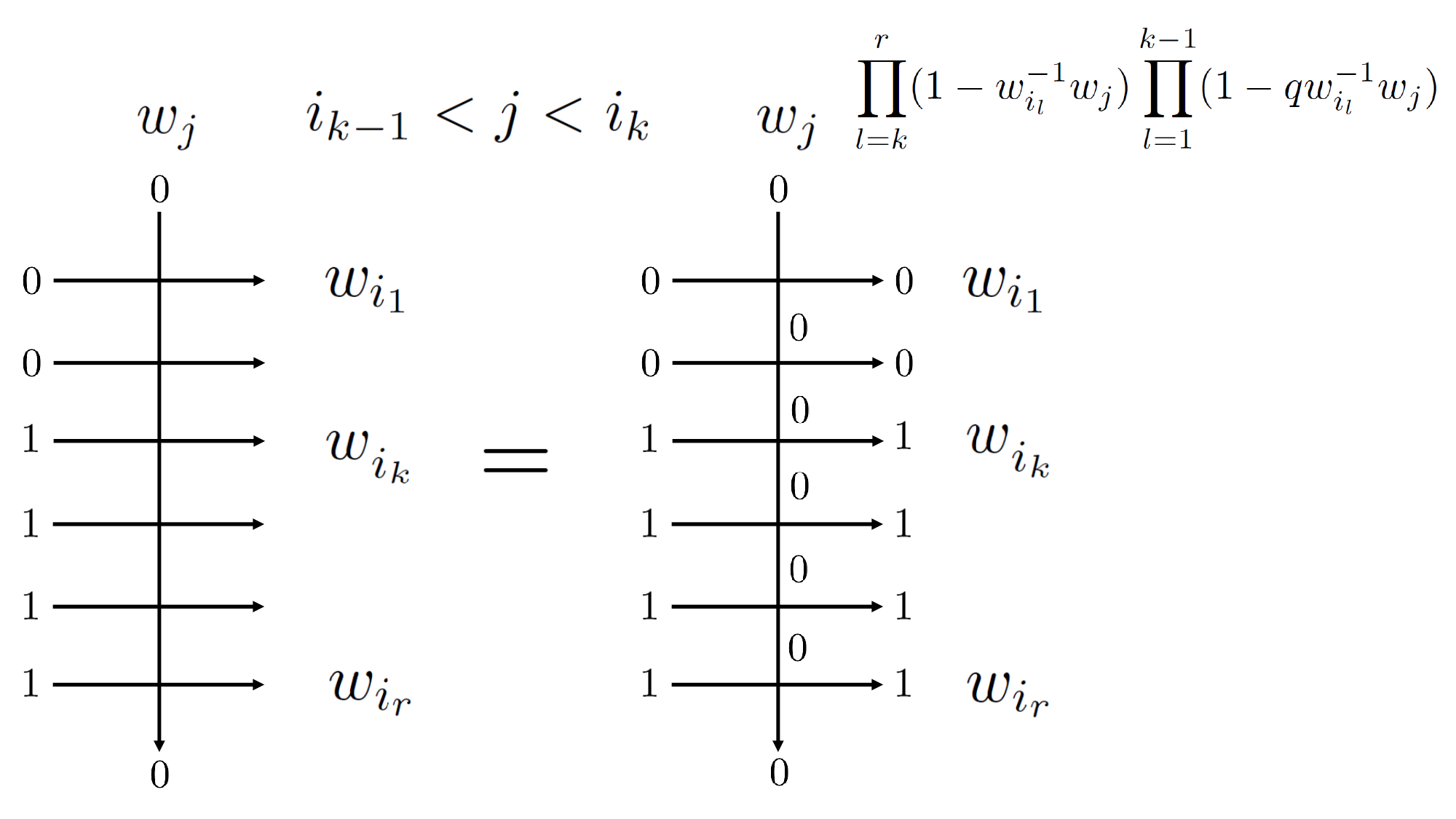}
\includegraphics[width=0.45\linewidth]{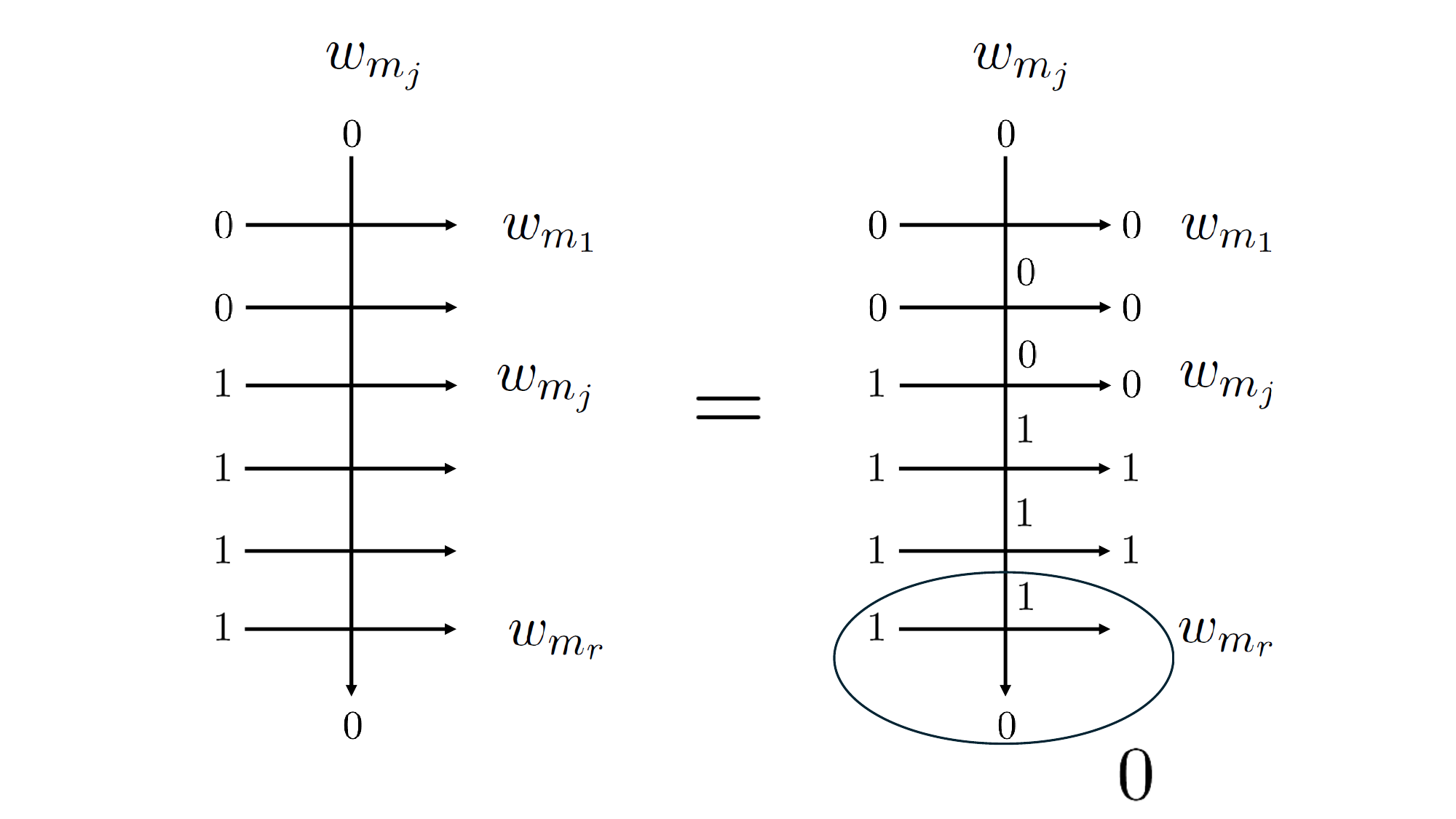}
\caption{Top: Column weights for the case $\mathbf{M}=\mathbf{I}$. Bottom: Illustration of the vanishing mechanism when $\mathbf{M} \ge \mathbf{I}$ is not satisfied.}
  \label{fig:fromtritwo}
\end{figure}

\end{proof}

At the end of this section,
let us rewrite $F_{N,n,r}(\mathbf{w}_{\mathbf{I}}|\mathbf{w}|\mathbf{I})$
into a form that will be used in
Section~\ref{sec:normalized-algebraic-properties}.
The final factor in equation \eqref{diagonalFfunctionanotherform}
\[
\prod_{k=1}^{r+1} \left[ \prod_{i_{k-1} < j < i_k} \left\{ \prod_{l=k}^r (1-w_{i_l}^{-1} w_j) \prod_{l=1}^{k-1} (1-q w_{i_l}^{-1} w_j) \right\} \right]
\]
can be rewritten in the following form using the complement set:
\begin{equation*}
\prod_{j \in \{1, 2, \dots, N\} \setminus \mathbf{I}} \left( \prod_{l: i_l > j} (1-w_{i_l}^{-1} w_j) \prod_{l: i_l < j} (1-q w_{i_l}^{-1} w_j) \right).
\end{equation*}

For this, we used
\begin{equation*}
\prod_{k=1}^{r+1} \prod_{i_{k-1} < j < i_k} = \prod_{j \in \{1, 2, \dots, N\} \setminus \mathbf{I}},
\end{equation*}
and the following:

    Since $l \ge k \iff i_l \ge i_k > j \iff i_l > j$,
    \begin{equation*}
    \prod_{l=k}^r (1-w_{i_l}^{-1} w_j) = \prod_{l: i_l > j} (1-w_{i_l}^{-1} w_j).
\end{equation*}
    
    Since $l \le k-1 \iff i_l \le i_{k-1} < j \iff i_l < j$,
    \begin{equation*}
    \prod_{l=1}^{k-1} (1-q w_{i_l}^{-1} w_j) = \prod_{l: i_l < j} (1-q w_{i_l}^{-1} w_j).
\end{equation*}

Therefore, $F_{N,n,r}(\mathbf{w}_{\mathbf{I}}|\mathbf{w}|\mathbf{I})$ can be rewritten in the following form:
\begin{align}
&F_{N,n,r}(\mathbf{w}_{\mathbf{I}}|\mathbf{w}|\mathbf{I}) \nonumber \\
=& (1-q)^{\frac{(n-r)(n-r-1)}{2}} (\beta-\gamma)^{n-r} \nonumber \\
&\times
 \prod_{1 \le j \le k \le r} (1-w_{i_j}^{-1} w_{i_k}^{-1}) 
  \prod_{j=1}^N (1-qw_j)^{n-r} \prod_{k=1}^r (1 - q w_{i_k}^{-1})^{n-r} 
\nonumber \\
&\times
\prod_{j=1}^r  \prod_{k=1}^r (1-q w_{i_k}^{-1} w_{i_j}) 
  \prod_{j \in \{1, 2, \dots, N\} \setminus \mathbf{I}} \left( \prod_{l: i_l > j} (1-w_{i_l}^{-1} w_j) \prod_{l: i_l < j} (1-q w_{i_l}^{-1} w_j) \right).
\label{diagonalFfunction_rewritten}
\end{align}

\section{Algebraic properties of the normalized weight functions}
\label{sec:normalized-algebraic-properties}

The normalized type $C$ and type $B$ weight functions and their main vanishing and evaluation
properties were introduced in Subsection~2.2. We prove the type $C$ vanishing and evaluation
properties using the lattice-model identities from Subsection~2.3 and the relation
with the type $A$ weight functions from Subsection~2.4. The corresponding type $B$
arguments are essentially the same and are omitted. We then state the normalized
recursive relations, supersymmetry properties, and
stability. The latter will be used systematically in Section~\ref{sec:normalized-specializations}.

\subsection{Type $C$}

\subsubsection{Proof of Theorem~\ref{typeCmainproperties}}
\begin{proof}
$
\widetilde{WC}_{N,n,r}(\mathbf{w}_{\mathbf{M}}|\mathbf{w}|\mathbf{I})=0$
unless $\mathbf{M} \ge \mathbf{I}$
follows from
Theorem~\ref{ffunctionproperties} $\mathrm{(i)}$.
Let us show
\eqref{normalizationWC}.
By definition \eqref{typeCnormalizeddefinition}, $\widetilde{WC}_{N,n,r}(\mathbf{w}_{\mathbf{I}}|\mathbf{w}|\mathbf{I}) = \frac{WC_{N,n,r}(\mathbf{w}_{\mathbf{I}}|\mathbf{w}|\mathbf{I})}{\mathcal{E}^C_{N,n,r}(\mathbf{w}_{\mathbf{I}}, \mathbf{w})}$. For brevity, we denote the complement of the index set by $\mathbf{I}^c := \{1, 2, \dots, N\} \setminus \mathbf{I}$. Executing the specialization substitution $\mathbf{u} = \mathbf{w}_{\mathbf{I}}$ in \eqref{typeCnormalizationfactor} and rearranging the equation, we obtain the following explicit form of the total normalization factor:
\begin{align}
\mathcal{E}_{N,n,r}^C(\mathbf{w}_{\mathbf{I}}, \mathbf{w}) =& (1-q)^{\frac{(n-r)(n-r-1)}{2}}
\prod_{j,k=1}^r (1 - q w_{i_j}^{-1} w_{i_k}) \prod_{j=1}^r (1 - q w_{i_j}^{-1})^{n-r}
\nonumber \\
&\quad \times \frac{  \prod_{k=1}^N (1 - q w_k)^{n-r} \prod_{j=1}^r \prod_{k=1}^N (1 - q w_{i_j} w_k) }{ \prod_{1 \le j < k \le r} (1 - q w_{i_j} w_{i_k}) \prod_{1 \le j \le k \le N} (1 - q w_j w_k) }.
\label{explicit_E_unified}
\end{align}

Rearranging using the expression \eqref{diagonalFfunction_rewritten} for $F_{N,n,r}(\mathbf{w}_{\mathbf{I}}|\mathbf{w}|\mathbf{I})$ and equation \eqref{explicit_E_unified} and setting $\gamma=0$, $\beta=1$, we obtain the form:
\begin{align}
\widetilde{WC}_{N,n,r}(\mathbf{w}_{\mathbf{I}}|\mathbf{w}|\mathbf{I})
&= \prod_{1 \le j \le k \le r} (1-w_{i_j}^{-1} w_{i_k}^{-1})  \prod_{j \in \mathbf{I}^c} \left( \prod_{l: i_l > j} (1-w_{i_l}^{-1} w_j) \prod_{l: i_l < j} (1-q w_{i_l}^{-1} w_j) \right) \nonumber \\
&\quad \times \frac{ \prod_{1 \le j < k \le r} (1 - q w_{i_j} w_{i_k}) \prod_{1 \le j \le k \le N} (1 - q w_j w_k) }{ \prod_{j=1}^r \prod_{k=1}^N (1 - q w_{i_j} w_k) }.
 \notag
\end{align}

Furthermore, simplifying using
\begin{align}
&\prod_{1 \le j \le k \le N} (1 - q w_j w_k) \nonumber \\
=& \prod_{1 \le j \le k \le r} (1 - q w_{i_j} w_{i_k})  \prod_{\substack{j \le k \\ j, 
k \in \mathbf{I}^c}} (1 - q w_j w_{k})  \prod_{j=1}^r \prod_{k \in \mathbf{I}^c} (1 - q w_{i_j} w_k),
 \notag
\end{align}
and
\begin{equation*}
\prod_{j=1}^r \prod_{k=1}^N (1 - q w_{i_j} w_k) = \prod_{j,k=1}^r (1 - q w_{i_j} w_{i_k})  \prod_{j=1}^r \prod_{k \in \mathbf{I}^c} (1 - q w_{i_j} w_k),
\end{equation*}
$\widetilde{WC}_{N,n,r}(\mathbf{w}_{\mathbf{I}}|\mathbf{w}|\mathbf{I})$ reduces to the following expression:
\begin{align}
\widetilde{WC}_{N,n,r}(\mathbf{w}_{\mathbf{I}}|\mathbf{w}|\mathbf{I})
&=  \prod_{1 \le j \le k \le r} (1-w_{i_j}^{-1} w_{i_k}^{-1})  \prod_{j \in \mathbf{I}^c} \prod_{l: i_l > j} (1-w_{i_l}^{-1} w_j)  \nonumber \\
&\quad \times  \prod_{\substack{j \le k \\ j, k \in \mathbf{I}^c}} (1 - q w_j w_{k})  \prod_{j \in \mathbf{I}^c} \prod_{l: i_l < j} (1-q w_{i_l}^{-1} w_j) .
\label{final_explicit_blocks}
\end{align}
The expressions \eqref{typecfirstblock} and \eqref{typecsecondblock} for $\prod_{\alpha \in \Inv(\lambda)} (1 - e^{-\alpha})$ and $\prod_{\alpha \in T(\lambda)} (1 - q e^{+\alpha})$ of type $C$ show that the right-hand side of \eqref{final_explicit_blocks} is $\prod_{\alpha \in \Inv(\lambda)} (1 - e^{-\alpha}) \prod_{\alpha \in T(\lambda)} (1 - q e^{+\alpha})$.

To prove $\mathrm{(ii)}$ and $\mathrm{(iii)}$,
we use the relation
\begin{align}
\widetilde{WC}_{N,n,r}
(\mathbf{w}_{\mathbf{M}}|\mathbf{w}|\mathbf{I})
={}&
\prod_{1\le j\le k\le r}
\left(1-w_{m_j}^{-1}w_{m_k}^{-1}\right)
\prod_{\substack{j\le k\\
j,k\in\{1,\dots,N\}\setminus\mathbf M}}
(1-qw_jw_k)
\nonumber\\
&\times
\frac{
W_{N,r}(w_{m_1},\dots,w_{m_r}|\mathbf w|\mathbf I)
}{
\displaystyle
\prod_{j,k=1}^r
(1-qw_{m_j}^{-1}w_{m_k})
}, \label{relationwithAnormalized}
\end{align}
which follows by combining \eqref{fixedpointrelationFandWA} with the normalized definition \eqref{typeCnormalizeddefinition} and the normalization factor \eqref{typeCnormalizationfactor}.
We also note the following identification with
type $A$ weight functions by Rim\'anyi--Tarasov--Varchenko:
\begin{align}
\frac{
W_{N,r}(w_{m_1},\dots,w_{m_r}|\mathbf w|\mathbf I)
}{
\displaystyle
\prod_{j,k=1}^r
(1-qw_{m_j}^{-1}w_{m_k})
}=\widetilde W_{\sigma_0,\mathcal I}^{\mathrm{RTV}}
(\mathbf{w}_{\mathbf{M}}, \mathbf{w},q).
\label{identificationtypeA}
\end{align}

To prove $\mathrm{(ii)}$, we use \eqref{relationwithAnormalized}.
First note that from Theorem~\ref{RTVtheorem} $\mathrm{(i)}$,
\eqref{identificationtypeA} is divisible by
\begin{align}
e_{\sigma_0,\mathcal{M},-}^{\mathrm{vert}}
&=\prod_{\substack{j \in \mathbf{M}^c, k \in \mathbf{M} \\ \sigma_0^{-1}(j)<\sigma_0^{-1}(k)}}
(1-q w_k^{-1} w_j)=\prod_{\substack{j \in \mathbf{M}^c, k \in \mathbf{M} \\ k <j }}
(1-q w_k^{-1} w_j) \nonumber \\
&=\prod_{j \in \{1, 2, \dots, N\} \setminus \mathbf{M}} \prod_{l: m_l < j} (1-q w_{m_l}^{-1} w_j). \notag
\end{align}

It is also explicit from \eqref{relationwithAnormalized} that
$\widetilde{WC}_{N,n,r}
(\mathbf{w}_{\mathbf{M}}|\mathbf{w}|\mathbf{I})$
contains the two factors
\[
\prod_{1\le j\le k\le r}
(1-w_{m_j}^{-1}w_{m_k}^{-1})
\quad\text{and}\quad
\prod_{\substack{j \le k \\ j, k \in \{1, 2, \dots, N\} \setminus \mathbf{M}}}
(1-q w_j w_{k}).
\]
Since the remaining normalized type $A$ factor is divisible by the vertical product displayed above, the full product in \eqref{typeCfullDivisibility} divides
$\widetilde{WC}_{N,n,r}(\mathbf{w}_{\mathbf{M}}|\mathbf{w}|\mathbf{I})$.

To prove $\mathrm{(iii)}$,
first note
that from Theorem~\ref{RTVtheorem} $\mathrm{(iii)}$
and the identification
\eqref{identificationtypeA},
we have
\begin{align}
&\varphi_{\mathbf{M}}
\Bigg(
\mathcal{N}
\Bigg(
\displaystyle
\prod_{j,k=1}^r
(1-qw_{m_j}^{-1}w_{m_k})^{-1}
W_{N,r}(w_{m_1},\dots,w_{m_r}|\mathbf w|\mathbf I)
\Bigg)
\Bigg) \nonumber \\
=&
\varphi_{\mathbf{M}} \Big(
\mathcal{N} \Big(
\widetilde W_{\sigma_0,\mathcal I}^{\mathrm{RTV}}
(\mathbf{w}_{\mathbf{M}},\mathbf{w},q)
\Big)
 \Big) \subset [0,r(N-r)-1], \label{newtfromA}
\end{align}
if $\mathbf{I} \neq \mathbf{M}$.

Using $\varphi_{\mathbf{M}} (\mathcal{N}( 1-w_{m_j}^{-1}w_{m_k}^{-1}))=[-2,0]$,
we have
\begin{align}
\varphi_{\mathbf{M}} \Bigg( \mathcal{N}
\Bigg(
\prod_{1\le j\le k\le r}
\left(1-w_{m_j}^{-1}w_{m_k}^{-1}\right)
\Bigg)
\Bigg)=[-r(r+1),0]. \label{newtfactorone}
\end{align}
A direct degree count gives
\begin{align}
\varphi_{\mathbf{M}} \Bigg( \mathcal{N}
\Bigg(
\prod_{\substack{j\le k\\
j,k\in\{1,\dots,N\}\setminus\mathbf M}}
(1-qw_jw_k)
\Bigg)
\Bigg)=[0,2(N-r)(N-r+1)]. \label{newtfactortwo}
\end{align}

From \eqref{relationwithAnormalized}, \eqref{newtfromA}, \eqref{newtfactorone},
\eqref{newtfactortwo} and the Minkowski-sum identity
$\mathcal N(fg)=\mathcal N(f)+\mathcal N(g)$, the claim follows.
\end{proof}

Note
\begin{align}
\varphi_{\mathbf{M}}(\mathcal{N}(
\widetilde{WC}_{N,n,r}
(\mathbf{w}_{\mathbf{M}}|\mathbf{w}|\mathbf{I})
)) = [-r(r+1), (N-r)(2N-r+2)], \notag
\end{align}
when $\mathbf{M}=\mathbf{I}$
which follows from $\mathrm{(i)}$,
or by using 
\begin{align}
&\varphi_{\mathbf{M}}
\Bigg(
\mathcal{N}
\Bigg(
\displaystyle
\prod_{j,k=1}^r
(1-qw_{m_j}^{-1}w_{m_k})^{-1}
W_{N,r}(w_{m_1},\dots,w_{m_r}|\mathbf w|\mathbf I)
\Bigg)
\Bigg) \nonumber \\
=&
\varphi_{\mathbf{M}} \Big(
\mathcal{N} \Big(
\widetilde W_{\sigma_0,\mathcal I}^{\mathrm{RTV}}
(\mathbf{w}_{\mathbf{M}},\mathbf{w},q)
\Big)
 \Big) =[0,r(N-r)], \notag 
\end{align}
when $\mathbf{M} = \mathbf{I}$.

\subsubsection{Recursive relations and stability}

\begin{prop}
Let $i \in \{1, \dots, N-1\}$.
For the normalized weight function $\widetilde{WC}_{N,n,r}(\mathbf{u}|\mathbf{w}|\mathbf{I})$, the following relations hold.

(i) If $i+1 = i_k \in \mathbf{I}$ and $i \notin \mathbf{I}$, we have
\begin{align}
&\widetilde{WC}_{N,n,r}(\mathbf{u}|\mathbf{w}|\mathbf{I}^{(i)}) \nonumber \\
=& \frac{1}{w_i-w_{i+1}} \Big( (w_i-q w_{i+1}) \widetilde{WC}_{N,n,r}(\mathbf{u}|s_i \mathbf{w}|\mathbf{I}) - (1-q)w_{i+1} \widetilde{WC}_{N,n,r}(\mathbf{u}|\mathbf{w}|\mathbf{I}) \Big), \label{recrelone_tilde} \\
&\widetilde{WC}_{N,n,r}(\mathbf{u}|\mathbf{w}|\mathbf{I}) \nonumber \\
=& \frac{1}{q(w_i-w_{i+1})} \Big( (w_i-q w_{i+1}) \widetilde{WC}_{N,n,r}(\mathbf{u}|s_i \mathbf{w}|\mathbf{I}^{(i)}) 
- (1-q)w_{i} \widetilde{WC}_{N,n,r}(\mathbf{u}|\mathbf{w}|\mathbf{I}^{(i)}) \Big). \label{recreltwo_tilde}
\end{align}
Here $\mathbf{I}^{(i)} = \{ i_1, \dots, i_{k-1}, i, i_{k+1}, \dots, i_r \}$.

(ii) If $i,i+1 \in \mathbf{I}$ or $i,i+1 \not\in \mathbf{I}$, we have
\begin{align}
\widetilde{WC}_{N,n,r}(\mathbf{u}|\mathbf{w}|\mathbf{I}) = \widetilde{WC}_{N,n,r}(\mathbf{u}|s_i \mathbf{w}|\mathbf{I}).
\label{recrelthree_tilde}
\end{align}
\end{prop}

\begin{proof}
By \eqref{typeCnormalizeddefinition}, $WC_{N,n,r}(\mathbf{u}|\mathbf{w}|\mathbf{I}) = \widetilde{WC}_{N,n,r}(\mathbf{u}|\mathbf{w}|\mathbf{I}) \cdot \mathcal{E}_{N,n,r}^C(\mathbf{u}, \mathbf{w})$.
We check the action of the simple transposition $s_i$ (swapping $w_i$ and $w_{i+1}$) on the normalization factor $\mathcal{E}^C_{N,n,r}(\mathbf{u}, \mathbf{w})$.
By \eqref{typeCnormalizationfactor}, the variable $\mathbf{w}$ in $\mathcal{E}^C_{N,n,r}(\mathbf{u}, \mathbf{w})$ appears only in the products $\prod_{j=1}^n \prod_{k=1}^N (1 - q u_j w_k)$ and $\prod_{1 \le j \le k \le N} (1 - q w_j w_k)$.
Since these are both symmetric products over all $k \in \{1, \dots, N\}$, they are invariant under the action of any simple transposition $s_i$. Thus,
$
\mathcal{E}^C_{N,n,r}(\mathbf{u}, s_i \mathbf{w}) = \mathcal{E}^C_{N,n,r}(\mathbf{u}, \mathbf{w}).
$
Therefore, substituting $WC_{N,n,r} = \widetilde{WC}_{N,n,r} \cdot \mathcal{E}^C_{N,n,r}$ into the recursive relations of the
$\gamma=0$, $\beta=1$ case of
Proposition~\ref{diddiffbeforenorm}, the invariant factor $\mathcal{E}^C_{N,n,r}(\mathbf{u}, \mathbf{w})$ completely cancels from both sides of all equations, so $\widetilde{WC}_{N,n,r}$ satisfies the same recursive relations as the original $WC_{N,n,r}$.
\end{proof}

\begin{prop}
The normalized weight function $\widetilde{WC}_{N,n,r}$ satisfies the following boundary recursive relations.

(i) If $1 \in \mathbf{I}$, we have
\begin{align}
\widetilde{WC}_{N,n,r-1}(\mathbf{u}|\mathbf{w}|\mathbf{I}^{(0)})
=& \frac{1}{1-w_1^2} \Bigg( (1 - q w_1^2) \widetilde{WC}_{N,n+1,r}(\mathbf{u}, w_1^{-1}|w_1^{-1}, w_2, \dotsm, w_N|\mathbf{I})
\nonumber \\
 &- (1-q) w_1^2 \widetilde{WC}_{N,n,r}(\mathbf{u}|\mathbf{w}|\mathbf{I}) \Bigg), \label{pizerorecursionone_tilde}
\end{align}
Here $\mathbf{I}^{(0)}=\{ i_2,\dotsm,i_r \}$.

(ii) If $1 \not\in \mathbf{I}$, we have
\begin{align}
\widetilde{WC}_{N,n+1,r}(\mathbf{u}, w_1^{-1}|w_1^{-1}, w_2, \dotsm, w_N|\mathbf{I})
= \frac{1-q}{1 - q w_1^2} \widetilde{WC}_{N,n,r}(\mathbf{u}|\mathbf{w}|\mathbf{I}). \label{pizerorecursiontwo_tilde}
\end{align}
\end{prop}

\begin{proof}
For brevity, we write $\mathbf{u}' = (\mathbf{u}, w_1^{-1})$ and \\ $\mathbf{w}' = (w_1^{-1}, w_2, \dotsm, w_N)$ in the proof.
From \eqref{typeCnormalizationfactor}, the ratio under the parameter shift $(\mathbf{u}', \mathbf{w}')$ is
\begin{align}
\frac{\mathcal{E}_{N,n+1,r}^C(\mathbf{u}', \mathbf{w}')}{\mathcal{E}_{N,n,r}^C(\mathbf{u}, \mathbf{w})}
&= (1-q w_1^2)  \prod_{j=1}^n (1 - q w_1^{-1} u_j^{-1}) \prod_{k=2}^N (1 - q w_1 w_k).
\label{E_ratio_shift}
\end{align}
Also, regarding the case where the index $r$ changes to $r-1$, since only the exponent of the coefficient part changes, we immediately obtain the following:
\begin{align}
\frac{\mathcal{E}_{N,n,r}^C(\mathbf{u}, \mathbf{w})}{\mathcal{E}^C_{N,n,r-1}(\mathbf{u}, \mathbf{w})}  = 1-q.
\label{E_ratio_r}
\end{align}

To prove (i), we substitute $WC_{N,n,r} = \widetilde{WC}_{N,n,r} \cdot \mathcal{E}_{N,n,r}^C$ into the original relation $\gamma=0$, $\beta=1$ of \eqref{pizerorecursionone} and divide both sides by $\mathcal{E}_{N,n,r-1}^C(\mathbf{u}, \mathbf{w})$:
\begin{align*}
&\widetilde{WC}_{N,n,r-1}(\mathbf{u}|\mathbf{w}|\mathbf{I}^{(0)}) \nonumber \\
&= \frac{1}{1-w_1^2} \Bigg( (1-q)^{-1} \prod_{j=1}^n (1-q w_1^{-1} u_j^{-1})^{-1} \prod_{k=2}^{N} (1-q w_1 w_k)^{-1} \frac{\mathcal{E}_{N,n+1,r}^C(\mathbf{u}', \mathbf{w}')}{\mathcal{E}_{N,n,r-1}^C(\mathbf{u}, \mathbf{w})} 
\nonumber \\
&\times \widetilde{WC}_{N,n+1,r}(\mathbf{u}'|\mathbf{w}'|\mathbf{I})  -  w_1^2 \frac{\mathcal{E}_{N,n,r}^C(\mathbf{u}, \mathbf{w})}{\mathcal{E}_{N,n,r-1}^C(\mathbf{u}, \mathbf{w})} \widetilde{WC}_{N,n,r}(\mathbf{u}|\mathbf{w}|\mathbf{I}) \Bigg).
\end{align*}
Decomposing the ratio of the normalization factors 
in the first term as \\
$\frac{\mathcal{E}^C_{N,n+1,r}(\mathbf{u}', \mathbf{w}')}{\mathcal{E}^C_{N,n,r-1}(\mathbf{u}, \mathbf{w})} = \frac{\mathcal{E}_{N,n+1,r}^C(\mathbf{u}', \mathbf{w}')}{\mathcal{E}^C_{N,n,r}(\mathbf{u}, \mathbf{w})} \cdot \frac{\mathcal{E}^C_{N,n,r}(\mathbf{u}, \mathbf{w})}{\mathcal{E}^C_{N,n,r-1}(\mathbf{u}, \mathbf{w})}$ and applying equations \eqref{E_ratio_shift} and \eqref{E_ratio_r}, equation \eqref{pizerorecursionone_tilde} is derived.

To prove (ii), we substitute $WC_{N,n,r} = \widetilde{WC}_{N,n,r} \cdot \mathcal{E}_{N,n,r}^C$ into the original relation $\gamma=0$, $\beta=1$ of
\eqref{pizerorecursiontwo} and divide both sides by $\mathcal{E}_{N,n+1,r}^C(\mathbf{u}', \mathbf{w}')$. This yields the following:
\begin{align*}
&\widetilde{WC}_{N,n+1,r}(\mathbf{u}'|\mathbf{w}'|\mathbf{I}) \nonumber \\
=& (1-q)  \prod_{j=1}^n (1-q w_1^{-1} u_j^{-1}) \prod_{k=2}^{N} (1-q w_1 w_k) \cdot \frac{\mathcal{E}^C_{N,n,r}(\mathbf{u}, \mathbf{w})}{\mathcal{E}_{N,n+1,r}^C(\mathbf{u}', \mathbf{w}')} \widetilde{WC}_{N,n,r}(\mathbf{u}|\mathbf{w}|\mathbf{I}).
\end{align*}
Substituting the inverse of the ratio \eqref{E_ratio_shift} into this equation derives equation \eqref{pizerorecursiontwo_tilde}.

\end{proof}

\begin{prop}
Assume $r\le n-2$. We have
\begin{align}
 \widetilde{WC}_{N,n,r}(\mathbf{u}|\mathbf{w}|\mathbf{I})\Big|_{u_i u_j = 1}& = \frac{(1-q)^2}{(1-qu_i^2)(1-qu_j^2)} \widetilde{WC}_{N,n-2,r}(\mathbf{u} \setminus \{u_i, u_j\}|\mathbf{w}|\mathbf{I}). \label{normalizedsusytwo} 
\end{align}
\end{prop}

\begin{proof}
Direct substitution in \eqref{typeCnormalizationfactor} gives
\begin{align}
    &\frac{\mathcal{E}_{N,n,r}^C(\mathbf{u},\mathbf{w})|_{u_i u_j = 1}}{\mathcal{E}_{N,n-2,r}^C(\mathbf{u} \setminus \{u_i, u_j\}, \mathbf{w})} 
    = \frac{1}{1-q}(1-qu_i^2)(1-qu_j^2) 
   \nonumber \\ 
&\times    \prod_{k=1}^N (1-qu_iw_k)(1-qu_jw_k)  \prod_{\substack{k=1 \\ k \neq i,j}}^n (1-qu_k^{-1}u_i)(1-qu_k^{-1}u_j). \notag
\end{align}
Combining with $\gamma=0$, $\beta=1$ of
\eqref{reductionrelation},
we obtain \eqref{normalizedsusytwo}.
\end{proof}

\begin{prop}
Assume $r\le n-1$. We have
\begin{align}
\widetilde{WC}_{N,n,r}(\mathbf{u}|\mathbf{w}|\mathbf{I})\Big|_{u_n = 1}& =\widetilde{WC}_{N,n-1,r}(\mathbf{u} \setminus \{ u_n \}|\mathbf{w}|\mathbf{I}).
\label{stabilityWC}
\end{align}
\end{prop}

\begin{proof}

Lemma~\ref{onevariablesp} gives \eqref{Fonevariablesp}.
Direct substitution in \eqref{typeCnormalizationfactor} gives
\begin{align}
\left.
\mathcal{E}_{N,n,r}^{C}(\mathbf{u},\mathbf{w})
\right|_{u_n=1}
&=
\prod_{k=1}^{N}(1-qw_k)
\prod_{k=1}^{n-1}(1-qu_k^{-1})
\,\mathcal{E}^{C}_{N,n-1,r}(\mathbf{u}\setminus\{u_n\},\mathbf{w}). \notag
\end{align}
Combining with $\gamma=0, \beta=1$ of
\eqref{Fonevariablesp}, we obtain \eqref{stabilityWC}.

\end{proof}

\subsection{Type $B$}

\subsubsection{Equal-cardinality relation}

For later use, we record the relation
\begin{align}
&\widetilde{WB}_{N,n,r}
(\mathbf{w}_{\mathbf{M}}|\mathbf{w}|\mathbf{I})
={}
\prod_{j=1}^r (1-w_{m_j}^{-1})
\prod_{1\le j< k\le r}
\left(1-w_{m_j}^{-1}w_{m_k}^{-1}\right)
\nonumber \\
&\times
\prod_{j\in\{1,\dots,N\}\setminus\mathbf M}
(1-q w_j)
\prod_{\substack{j< k\\
j,k\in\{1,\dots,N\}\setminus\mathbf M}}
(1-qw_jw_k)
\frac{
W_{N,r}(w_{m_1},\dots,w_{m_r}|\mathbf w|\mathbf I)
}{
\displaystyle
\prod_{j,k=1}^r
(1-qw_{m_j}^{-1}w_{m_k})
}. \label{relationwithAnormalizedtypeB}
\end{align}
which follows directly from \eqref{fixedpointrelationFandWA}, \eqref{WBfunction},
\eqref{typeBnormalizeddefinition}, and \eqref{typeBnormalizationfactor}.

Note
\begin{align}
\varphi_{\mathbf{M}}(\mathcal{N}(
\widetilde{WB}_{N,n,r}
(\mathbf{w}_{\mathbf{M}}|\mathbf{w}|\mathbf{I})
)) = [-r^2, (N-r)(2N-r)], \notag
\end{align}
when $\mathbf{M}=\mathbf{I}$ which follows from $(\mathrm{i})$.

\subsubsection{Recursive relations and stability}

\begin{prop} \label{prop:B_si}
Let $i \in \{1, \dotsm, N-1\}$. For the normalized weight function $\widetilde{WB}_{N,n,r}(\mathbf{u}|\mathbf{w}|\mathbf{I})$, the following relations hold.

(i) If $i+1 = i_k \in \mathbf{I}$ and $i \notin \mathbf{I}$, we have
\begin{align}
&\widetilde{WB}_{N,n,r}(\mathbf{u}|\mathbf{w}|\mathbf{I}^{(i)}) \nonumber \\
=& \frac{1}{w_i-w_{i+1}} \Big( (w_i-q w_{i+1}) \widetilde{WB}_{N,n,r}(\mathbf{u}|s_i \mathbf{w}|\mathbf{I}) - (1-q)w_{i+1} \widetilde{WB}_{N,n,r}(\mathbf{u}|\mathbf{w}|\mathbf{I}) \Big), \\
&\widetilde{WB}_{N,n,r}(\mathbf{u}|\mathbf{w}|\mathbf{I}) \nonumber \\
=& \frac{1}{q(w_i-w_{i+1})} \Big( (w_i-q w_{i+1}) \widetilde{WB}_{N,n,r}(\mathbf{u}|s_i \mathbf{w}|\mathbf{I}^{(i)}) - (1-q)w_{i} \widetilde{WB}_{N,n,r}(\mathbf{u}|\mathbf{w}|\mathbf{I}^{(i)}) \Big),
\end{align}
here $\mathbf{I}^{(i)} = \{ i_1, \dotsm, i_{k-1}, i, i_{k+1}, \dotsm, i_r \}$.

(ii) If $i,i+1 \in \mathbf{I}$ or $i,i+1 \not\in \mathbf{I}$, we have
\begin{align}
\widetilde{WB}_{N,n,r}(\mathbf{u}|\mathbf{w}|\mathbf{I}) = \widetilde{WB}_{N,n,r}(\mathbf{u}|s_i \mathbf{w}|\mathbf{I}).
\end{align}
\end{prop}

\begin{prop} \label{prop:B_shift}
The normalized weight function $\widetilde{WB}_{N,n,r}$ satisfies the following boundary recursive relations.

(i) If $1 \in \mathbf{I}$, we have
\begin{align}
\widetilde{WB}_{N,n,r-1}(\mathbf{u}|\mathbf{w}|\mathbf{I}^{(0)}) =& \frac{1}{1-w_1} \Bigg( (1 - q w_1) \widetilde{WB}_{N,n+1,r}(\mathbf{u}, w_1^{-1}|w_1^{-1}, w_2, \dotsm, w_N|\mathbf{I}) \nonumber \\
&- (1-q) w_1 \widetilde{WB}_{N,n,r}(\mathbf{u}|\mathbf{w}|\mathbf{I}) \Bigg),
\end{align}
here $\mathbf{I}^{(0)}=\{ i_2,\dotsm,i_r \}$.

(ii) If $1 \not\in \mathbf{I}$, we have
\begin{align}
\widetilde{WB}_{N,n+1,r}(\mathbf{u}, w_1^{-1}|w_1^{-1}, w_2, \dotsm, w_N|\mathbf{I}) = \frac{1-q}{1-qw_1} \widetilde{WB}_{N,n,r}(\mathbf{u}|\mathbf{w}|\mathbf{I}).
\end{align}
\end{prop}

\begin{prop}
Assume $r\le n-2$. We have
\begin{align}
 \widetilde{WB}_{N,n,r}(\mathbf{u}|\mathbf{w}|\mathbf{I})\Big|_{u_i u_j = 1}& = \frac{(1-q)^2}{(1-qu_i)(1-qu_j)} \widetilde{WB}_{N,n-2,r}(\mathbf{u} \setminus \{u_i, u_j\}|\mathbf{w}|\mathbf{I}). \label{normalizedsusythree}
\end{align}
\end{prop}

\begin{prop}
Assume $r\le n-1$. We have
\begin{align}
\widetilde{WB}_{N,n,r}(\mathbf{u}|\mathbf{w}|\mathbf{I})\Big|_{u_n = 1}& =\widetilde{WB}_{N,n-1,r}(\mathbf{u} \setminus \{ u_n \}|\mathbf{w}|\mathbf{I}).
\label{stabilityWB}
\end{align}
\end{prop}

\section{Further properties of the normalized weight functions}
\label{sec:normalized-specializations}

We further derive properties of the normalized weight functions,
including branching formulas and orthogonality relations.
Recall that $N$ and $n$ are the numbers of $w$-variables and
$u$-variables, respectively, and that $r=|\mathbf I|$ is the
cardinality of the label $\mathbf I$.
For $\mathbf M=\{m_1<\cdots<m_s\}\subset\{1,\ldots,N\}$,
we write $s=|\mathbf M|$ and specialize
\[
\mathbf u=(w_{m_1},\ldots,w_{m_s},
\underbrace{1,\ldots,1}_{n-s}),
\qquad 0\le s\le\min(n,N).
\]
Thus $r$ specifies the label cardinality, whereas $s$ specifies
the number of selected $w$-variables in the specialization.
We consider both $s=r$ and $s\ne r$.

\subsection{Type $C$}

\subsubsection{General setup}
\label{sec:typeC-specialization}

Let
\[
\mathbf{M}=\{m_1<\cdots<m_s\}\subset\{1,\dots,N\},
\qquad 0\le s\le n,
\qquad
\mathbf{M}^{\mathrm c}=\{1,\dots,N\}\setminus\mathbf{M},
\]
and set
\[
\mathbf{w}_{\mathbf{M}}:=
(w_{m_1},\ldots,w_{m_s},\underbrace{1,\ldots,1}_{n-s}).
\]
Repeated application of \eqref{stabilityWC} gives, if $s\ge r$,
\begin{equation}
\widetilde{WC}_{N,n,r}(\mathbf{w}_{\mathbf{M}}|\mathbf{w}|\mathbf{I})
=
\widetilde{WC}_{N,s,r}(w_{m_1},\ldots,w_{m_s}|\mathbf{w}|\mathbf{I}),
\label{eq:typeC-stability-s-ge-r}
\end{equation}
whereas if $s<r$,
\begin{equation}
\widetilde{WC}_{N,n,r}(\mathbf{w}_{\mathbf{M}}|\mathbf{w}|\mathbf{I})
=
\widetilde{WC}_{N,r,r}\Bigl(
w_{m_1},\ldots,w_{m_s},\underbrace{1,\ldots,1}_{r-s}
|\mathbf{w}|\mathbf{I}\Bigr).
\label{eq:typeC-stability-s-lt-r}
\end{equation}
We organize the consequences in the order $s<r$, $s=r$, and $s>r$.

\subsubsection{The case $s<r$}

\begin{prop}
If $|\mathbf{M}|<|\mathbf{I}|=r$, then
\begin{equation}
\widetilde{WC}_{N,n,r}(\mathbf{w}_{\mathbf{M}}|\mathbf{w}|\mathbf{I})=0.
\label{eq:typeC-lower-cardinality-vanishing}
\end{equation}
\end{prop}

\begin{proof}
By \eqref{eq:typeC-stability-s-lt-r}, it is enough to consider
$\widetilde{WC}_{N,r,r}$ at a specialization containing at least one variable
$u_j=1$.  The factorization \eqref{typeCnrFactorization} shows that the
corresponding unnormalized function $WC_{N,r,r}$ vanishes because it contains
the factor $1-u_j^{-2}$.  By the normalized definition
\eqref{typeCnormalizeddefinition} and the normalization factor
\eqref{typeCnormalizationfactor}, $\mathcal E^C_{N,r,r}$ is finite and
nonzero at this specialization for generic parameters.  Hence the normalized
value also vanishes.
\end{proof}

\subsubsection{Equal cardinality: $s=r$}
For $|\mathbf{M}|=r$, the relation with the normalized type $A$ weight function is exactly \eqref{relationwithAnormalized}.  This is precisely the cardinality considered in Theorem~\ref{typeCmainproperties}.  Together with the support condition and Newton degree results below, the equal-cardinality specialization recovers $\mathrm{(i)}$--$\mathrm{(iii)}$ of that theorem; we summarize this at the end of the type $C$ subsection.

\subsubsection{The case $s>r$: branching formula}

Let
\[
\mathbf{M}=\{m_1<\cdots<m_s\}\subset\{1,\dots,N\},
\qquad
\mathbf{M}^{\mathrm c}=\{1,\dots,N\}\setminus\mathbf{M},
\qquad r\le s\le n.
\]
For each
$
\mathbf{K}=\{k_1<\cdots<k_r\}\subset\mathbf{M},
$
write
$
\mathbf{L}=\mathbf{M}\setminus\mathbf{K}
=\{l_1<\cdots<l_{s-r}\}.
$

\begin{thm}
For every $s=|\mathbf{M}|$ with $r=|\mathbf{I}|\le s\le n$,
\begin{align}
\widetilde{WC}_{N,n,r}(\mathbf{w}_{\mathbf{M}}|\mathbf{w}|\mathbf{I})
&=(1-q)^{s-r}
\prod_{\substack{j\le k\\j,k\in\mathbf{M}^{\mathrm c}}}
(1-qw_jw_k)
\sum_{\substack{\mathbf{K}\subset\mathbf{M}\\|\mathbf{K}|=r}}
\frac{
W_{N,r}(w_{k_1},\ldots,w_{k_r}|\mathbf{w}|\mathbf{I})
}{
\displaystyle
\prod_{a,b=1}^{r}(1-qw_{k_a}^{-1}w_{k_b})
}
\notag\\
&\quad\times
\frac{
\displaystyle
\prod_{1\le a\le b\le r}
(1-w_{k_a}^{-1}w_{k_b}^{-1})
\prod_{a=1}^{r}\prod_{b=1}^{s-r}
(1-w_{k_a}^{-1}w_{l_b}^{-1})
}{
\displaystyle
\prod_{a=1}^{r}\prod_{b=1}^{s-r}
(1-w_{k_a}^{-1}w_{l_b})
}
\notag\\
&\quad\times
\prod_{1\le a<b\le s-r}
(1-qw_{l_a}^{-1}w_{l_b}^{-1})
\prod_{a=1}^{s-r}
\prod_{j\in\mathbf{M}^{\mathrm c}}
(1-qw_{l_a}^{-1}w_j).
\label{eq:typeC-cross-sector-branching}
\end{align}
For $r=0$, set $W_{N,0}=1$, and interpret every empty product as $1$.
\end{thm}

\medskip
\noindent\textbf{Remark.}

If $s=r$, then $\mathbf{K}=\mathbf{M}$ is the only term.
All products involving $\mathbf{L}$ are empty and $(1-q)^{s-r}=1$.
Thus~\eqref{eq:typeC-cross-sector-branching} reduces exactly to
\eqref{relationwithAnormalized}.

\medskip

\begin{proof}
By \eqref{eq:typeC-stability-s-ge-r},
$
\widetilde{WC}_{N,n,r}
(\mathbf{w}_{\mathbf{M}}|\mathbf{w}|\mathbf{I})
=
\widetilde{WC}_{N,s,r}
(w_{m_1},\ldots,w_{m_s}|\mathbf{w}|\mathbf{I}).
$
Therefore it is enough to decompose
$WC_{N,s,r}(w_{m_1},\ldots,w_{m_s}|\mathbf{w}|\mathbf{I})$.
Fix
$
\mathbf{K}=\{k_1<\cdots<k_r\}\subset\mathbf{M}$
and $
\mathbf{L}=\mathbf{M}\setminus\mathbf{K}
=\{l_1<\cdots<l_{s-r}\}.
$
By the definition \eqref{WCfunction},
$WC_{N,s,r}$ is obtained from the explicit formula \eqref{Motherfunction}
for $F_{N,s,r}$ by setting $\gamma=0$ and $\beta=1$.
We apply this specialized form of \eqref{Motherfunction} and collect the
permutations for which the variables in the first $r$ positions are precisely
$w_{k_1},\ldots,w_{k_r}$ and those in the remaining $s-r$ positions are
$w_{l_1},\ldots,w_{l_{s-r}}$.  The factors coupling these two sets are
invariant under permutations within either set.  Consequently, before any
further simplification, the specialized form of \eqref{Motherfunction} can be
written as the following sum over subsets:
\begin{align}
&WC_{N,s,r}(w_{m_1},\ldots,w_{m_s}|\mathbf{w}|\mathbf{I})
\notag\\
&=\frac{1}{(s-r)!}
\sum_{\substack{\mathbf{K}\subset\mathbf{M}\\|\mathbf{K}|=r}}
\prod_{a=1}^{r}w_{k_a}^{-(N+s)}
\prod_{b=1}^{s-r}w_{l_b}^{-(N+s)}
\notag\\
&\quad\times
\sum_{\sigma\in S_r}\sigma\Bigg[
\prod_{j=1}^{r}
\left\{
(1-q)(u_j^2-1)
\prod_{c=i_j+1}^{N}(u_j-qw_c)
\prod_{c=1}^{i_j-1}(u_j-w_c)
\right\}
\notag\\
&\hspace{40mm}\times
\prod_{1\le j<k\le r}
\frac{qu_j-u_k}{u_j-u_k}(u_ju_k-1)
\Bigg]_{(u_1,\ldots,u_r)=(w_{k_1},\ldots,w_{k_r})}
\notag\\
&\quad\times
\sum_{\rho\in S_{s-r}}\rho\Bigg[
\prod_{b=1}^{s-r}
\left\{
 u_{r+b}\prod_{c=1}^{N}(u_{r+b}-qw_c)
\right\}
\notag\\
&\quad\times
\prod_{a=1}^{r}\prod_{b=1}^{s-r}
\frac{qu_{r+b}-w_{k_a}}{u_{r+b}-w_{k_a}}
(w_{k_a}u_{r+b}-1)
\prod_{1\le a<b\le s-r}(u_{r+a}u_{r+b}-q)
\Bigg]_{\substack{
(u_{r+1},\ldots,u_s)\\
=(w_{l_1},\ldots,w_{l_{s-r}})}}.
\label{eq:typeC-subset-sum-before-grouping}
\end{align}
We first evaluate the sum over $\rho$.  The action of $\rho$ permutes only
$u_{r+1},\ldots,u_s$.  Since each of the three products below is symmetric in
these variables, all $(s-r)!$ terms are equal.  Therefore
\begin{align}
&\frac{1}{(s-r)!}
\sum_{\rho\in S_{s-r}}\rho\Bigg[
\prod_{b=1}^{s-r}
\left\{
 u_{r+b}\prod_{c=1}^{N}(u_{r+b}-qw_c)
\right\}
\notag\\
&\times
\prod_{a=1}^{r}\prod_{b=1}^{s-r}
\frac{qu_{r+b}-w_{k_a}}{u_{r+b}-w_{k_a}}
(w_{k_a}u_{r+b}-1)
\prod_{1\le a<b\le s-r}(u_{r+a}u_{r+b}-q)
\Bigg]_{\substack{
(u_{r+1},\ldots,u_s)\\
=(w_{l_1},\ldots,w_{l_{s-r}})}}
\notag\\
&=
\prod_{b=1}^{s-r}
\left\{
 w_{l_b}\prod_{c=1}^{N}(w_{l_b}-qw_c)
\right\}
\prod_{a=1}^{r}\prod_{b=1}^{s-r}
\frac{qw_{l_b}-w_{k_a}}{w_{l_b}-w_{k_a}}
(w_{k_a}w_{l_b}-1)
\prod_{1\le a<b\le s-r}(w_{l_a}w_{l_b}-q).
 \notag
\end{align}
Thus the subset decomposition
\eqref{eq:typeC-subset-sum-before-grouping} becomes the following completely
explicit formula:
\begin{align}
&WC_{N,s,r}(w_{m_1},\ldots,w_{m_s}|\mathbf{w}|\mathbf{I})
=
\sum_{\substack{\mathbf{K}\subset\mathbf{M}\\|\mathbf{K}|=r}}
\prod_{a=1}^{r}w_{k_a}^{-(N+s)}
\prod_{b=1}^{s-r}w_{l_b}^{-(N+s)}
\notag\\
&\quad\times
\sum_{\sigma\in S_r}\sigma\Bigg[
\prod_{j=1}^{r}
\left\{
(1-q)(u_j^2-1)
\prod_{c=i_j+1}^{N}(u_j-qw_c)
\prod_{c=1}^{i_j-1}(u_j-w_c)
\right\}
\notag\\
&\quad\times
\prod_{1\le j<k\le r}
\frac{qu_j-u_k}{u_j-u_k}(u_ju_k-1)
\Bigg]_{(u_1,\ldots,u_r)=(w_{k_1},\ldots,w_{k_r})}
\prod_{b=1}^{s-r}
\left\{
w_{l_b}\prod_{c=1}^{N}(w_{l_b}-qw_c)
\right\}
\notag\\
&\quad\times
\prod_{a=1}^{r}\prod_{b=1}^{s-r}
\frac{qw_{l_b}-w_{k_a}}{w_{l_b}-w_{k_a}}
(w_{k_a}w_{l_b}-1)
\prod_{1\le a<b\le s-r}(w_{l_a}w_{l_b}-q).
\label{eq:typeC-subset-sum-after-rho}
\end{align}
We next identify the $\sigma$-sum in each summand of
\eqref{eq:typeC-subset-sum-after-rho}.  Applying \eqref{Motherfunction}
with $n=r$, $\gamma=0$, and $\beta=1$, and then using the definition
\eqref{WCfunction},
\begin{align}
&\prod_{a=1}^{r}w_{k_a}^{-(N+s)}
\sum_{\sigma\in S_r}\sigma\Bigg[
\prod_{j=1}^{r}
\left\{
(1-q)(u_j^2-1)
\prod_{c=i_j+1}^{N}(u_j-qw_c)
\prod_{c=1}^{i_j-1}(u_j-w_c)
\right\}
\notag\\
&\hspace{34mm}\times
\prod_{1\le j<k\le r}
\frac{qu_j-u_k}{u_j-u_k}(u_ju_k-1)
\Bigg]_{(u_1,\ldots,u_r)=(w_{k_1},\ldots,w_{k_r})}
\notag\\
&\qquad=
WC_{N,r,r}
(w_{k_1},\ldots,w_{k_r}|\mathbf{w}|\mathbf{I})
\prod_{a=1}^{r}w_{k_a}^{-(s-r)}.
\label{eq:typeC-first-r-contribution}
\end{align}
Indeed, the monomial prefactor in $WC_{N,r,r}$ is
$\prod_{a=1}^{r}w_{k_a}^{-(N+r)}$, whereas the monomial prefactor inherited
from $WC_{N,s,r}$ is $\prod_{a=1}^{r}w_{k_a}^{-(N+s)}$.
Substituting~\eqref{eq:typeC-first-r-contribution} into
\eqref{eq:typeC-subset-sum-after-rho}, we obtain
\begin{align}
&WC_{N,s,r}(w_{m_1},\ldots,w_{m_s}|\mathbf{w}|\mathbf{I})
\notag\\
&=
\sum_{\substack{\mathbf{K}\subset\mathbf{M}\\|\mathbf{K}|=r}}
WC_{N,r,r}(w_{k_1},\ldots,w_{k_r}|\mathbf{w}|\mathbf{I})
\prod_{a=1}^{r}w_{k_a}^{-(s-r)}
\prod_{b=1}^{s-r}w_{l_b}^{-(N+s)}
\notag\\
&\quad\times
\prod_{b=1}^{s-r}
\left\{
w_{l_b}\prod_{c=1}^{N}(w_{l_b}-qw_c)
\right\}
\prod_{a=1}^{r}\prod_{b=1}^{s-r}
\frac{qw_{l_b}-w_{k_a}}{w_{l_b}-w_{k_a}}
(w_{k_a}w_{l_b}-1)
\prod_{1\le a<b\le s-r}(w_{l_a}w_{l_b}-q).
\label{eq:typeC-sum-through-WCrr}
\end{align}

We use
\begin{align}
&w_{l_b}^{-(N+s)}w_{l_b}
\prod_{c=1}^{N}(w_{l_b}-qw_c)
=
w_{l_b}^{1-s}
\prod_{k=1}^{N}(1-qw_{l_b}^{-1}w_k),
 \notag\\
&\frac{qw_{l_b}-w_{k_a}}{w_{l_b}-w_{k_a}}
(w_{k_a}w_{l_b}-1)
=
w_{k_a}w_{l_b}
\frac{
(1-qw_{k_a}^{-1}w_{l_b})
(1-w_{k_a}^{-1}w_{l_b}^{-1})
}{
1-w_{k_a}^{-1}w_{l_b}
},
 \notag\\
&w_{l_a}w_{l_b}-q
=
w_{l_a}w_{l_b}
(1-qw_{l_a}^{-1}w_{l_b}^{-1}).
 \notag
\end{align}

Applying~\eqref{relationwithAnormalized} with $n=r$ and $\mathbf{M}=\mathbf{K}$, and then using the normalized definition \eqref{typeCnormalizeddefinition}, gives
\begin{align}
&WC_{N,r,r}(w_{k_1},\ldots,w_{k_r}|\mathbf{w}|\mathbf{I})
\notag\\
&=
\mathcal E^C_{N,r,r}(w_{k_1},\ldots,w_{k_r},\mathbf{w})
\prod_{1\le a\le b\le r}
(1-w_{k_a}^{-1}w_{k_b}^{-1})
\prod_{\substack{j\le k\\j,k\in\mathbf{K}^{\mathrm c}}}
(1-qw_jw_k)
\frac{
W_{N,r}(w_{k_1},\ldots,w_{k_r}|\mathbf{w}|\mathbf{I})
}{
\displaystyle
\prod_{a,b=1}^{r}(1-qw_{k_a}^{-1}w_{k_b})
}.
 \notag
\end{align}
Here $\mathbf{K}^{\mathrm c}=\{1,\dots,N\}\setminus\mathbf{K}$.

Moreover,
\begin{align}
\frac{
\mathcal E^C_{N,r,r}(w_{k_1},\ldots,w_{k_r},\mathbf{w})
}{
\mathcal E^C_{N,s,r}(w_{m_1},\ldots,w_{m_s},\mathbf{w})
}
&=(1-q)^{s-r}
\frac{
\displaystyle
\prod_{a=1}^{r}\prod_{b=1}^{s-r}
(1-qw_{k_a}w_{l_b})
\prod_{1\le a<b\le s-r}
(1-qw_{l_a}w_{l_b})
}{
\displaystyle
\prod_{a=1}^{r}\prod_{b=1}^{s-r}
(1-qw_{k_a}^{-1}w_{l_b})
\prod_{a=1}^{s-r}\prod_{b=1}^{r}
(1-qw_{l_a}^{-1}w_{k_b})
}
\notag\\
&\qquad\times
\frac{1}{
\displaystyle
\prod_{a,b=1}^{s-r}(1-qw_{l_a}^{-1}w_{l_b})
\prod_{a=1}^{s-r}\prod_{j=1}^{N}(1-qw_{l_a}w_j)
}.
 \notag
\end{align}

Substituting these identities into the summand of
\eqref{eq:typeC-sum-through-WCrr} indexed by $\mathbf{K}$, and then dividing
by the normalization factor
$
\mathcal E^C_{N,s,r}(w_{m_1},\ldots,w_{m_s},\mathbf{w}),
$
we obtain
\begin{align}
&(1-q)^{s-r}
\prod_{\substack{j\le k\\j,k\in\mathbf{M}^{\mathrm c}}}
(1-qw_jw_k)
\frac{
W_{N,r}(w_{k_1},\ldots,w_{k_r}|\mathbf{w}|\mathbf{I})
}{
\displaystyle
\prod_{a,b=1}^{r}(1-qw_{k_a}^{-1}w_{k_b})
}
\notag\\
&\quad\times
\frac{
\displaystyle
\prod_{1\le a\le b\le r}
(1-w_{k_a}^{-1}w_{k_b}^{-1})
\prod_{a=1}^{r}\prod_{b=1}^{s-r}
(1-w_{k_a}^{-1}w_{l_b}^{-1})
}{
\displaystyle
\prod_{a=1}^{r}\prod_{b=1}^{s-r}
(1-w_{k_a}^{-1}w_{l_b})
}
\notag\\
&\quad\times
\prod_{1\le a<b\le s-r}
(1-qw_{l_a}^{-1}w_{l_b}^{-1})
\prod_{a=1}^{s-r}
\prod_{j\in\mathbf{M}^{\mathrm c}}
(1-qw_{l_a}^{-1}w_j).
\label{eq:typeC-one-branch-explicit}
\end{align}
Summing~\eqref{eq:typeC-one-branch-explicit} over all
$\mathbf{K}\subset\mathbf{M}$ with $|\mathbf{K}|=r$ proves the formula.
\end{proof}

\subsubsection{Support condition}

Define the strict partitions
\[
\lambda(\mathbf{I})=(i_r>\cdots>i_1),
\qquad
\lambda(\mathbf{M})=(m_s>\cdots>m_1).
\]

\begin{prop}
Let $s=|\mathbf{M}|$ and assume $r\le s\le n$. Then
\begin{equation*}
\widetilde{WC}_{N,n,r}(\mathbf{w}_{\mathbf{M}}|\mathbf{w}|\mathbf{I})\ne0
\end{equation*}
only if
\begin{equation}
i_j\le m_{s-r+j}
\qquad(j=1,\ldots,r).
\label{eq:typeC-shifted-dominance}
\end{equation}
This is equivalent to $\lambda(\mathbf{I})\subseteq\lambda(\mathbf{M})$.
\end{prop}

\begin{proof}
The $\mathbf{K}$-term of~\eqref{eq:typeC-cross-sector-branching} contains
\[
\frac{W_{N,r}(w_{k_1},\ldots,w_{k_r}|\mathbf{w}|\mathbf{I})}
{\displaystyle\prod_{a,b=1}^{r}(1-qw_{k_a}^{-1}w_{k_b})}.
\]
By type $A$ triangularity (see the argument following \eqref{rectdiagonalfactor}), this factor vanishes unless $\mathbf{K}\ge \mathbf{I}$.
The necessary and sufficient condition for $\mathbf{M}$ to contain an $r$-element subset $\mathbf{K}\ge \mathbf{I}$ is
$
\{m_{s-r+1},\ldots,m_s\}\ge \mathbf{I}.
$
This is precisely~\eqref{eq:typeC-shifted-dominance}.
\end{proof}

\subsubsection{Laurent polynomiality}

We give a short geometric localization argument for Laurent polynomiality.
We first specify the $\lambda_y$-class used below. For an equivariant vector bundle $E$ and a formal parameter $y$, define
\begin{equation*}
\lambda_y(E)
:=
\sum_{j=0}^{\operatorname{rank}E}
y^j\left[\bigwedge^j E\right].
\end{equation*}
If the weights of $E$ at a fixed point are
$\chi_1,\ldots,\chi_{\operatorname{rank}E}$, then
\begin{equation}
\lambda_y(E)=
\prod_{j=1}^{\operatorname{rank}E}(1+y\chi_j).
\label{eq:typeC-lambda-y-weights}
\end{equation}
Below we use in particular $\lambda_{-1}$ and $\lambda_{-q}$.
Throughout this geometric discussion, for a vector bundle $E$, the notation
$E^*$ denotes its dual bundle. Thus, if the equivariant weights of $E$ at a fixed point are
$\chi_1,\ldots,\chi_{\operatorname{rank}E}$, then the weights of $E^*$ are
$\chi_1^{-1},\ldots,\chi_{\operatorname{rank}E}^{-1}$. In particular,
$\mathscr S^*$ and $\mathscr Q^*$ denote the dual tautological and dual quotient bundles,
respectively, and $T_{\mathbf K}^*X_{\mathbf M}$ denotes the cotangent space of
$X_{\mathbf M}$ at the fixed point $\mathbf K$.

\begin{prop}
For every $\mathbf{M}=\{m_1<\cdots<m_s\}\subset\{1,\dots,N\}$ with $s\le n$,
\begin{equation}
\widetilde{WC}_{N,n,r}(\mathbf{w}_{\mathbf{M}}|\mathbf{w}|\mathbf{I})
\in
\mathbb{C}(q)[w_1^{\pm1},\ldots,w_N^{\pm1}].
\label{eq:typeC-laurent-polynomiality}
\end{equation}
Moreover, this Laurent polynomial is divisible by
\begin{equation}
\prod_{\substack{j\le k\\j,k\in\mathbf{M}^{\mathrm c}}}
(1-qw_jw_k).
\label{eq:typeC-complement-factor-general-s}
\end{equation}
\end{prop}

\begin{proof}
If $|\mathbf{M}|<r$, the value is zero by~\eqref{eq:typeC-lower-cardinality-vanishing}.
Assume henceforth that $s=|\mathbf{M}|\ge r$, and set
\[
X_{\mathbf{M}}:=\operatorname{Gr}
\left(r,\bigoplus_{j=1}^{s}\mathbb{C}_{w_{m_j}}\right).
\]
Here $\mathbb{C}_{w_{m_j}}$ is the one-dimensional $T$-representation of weight
$w_{m_j}$. Denote the tautological and quotient bundles by
$\mathscr S$ and $\mathscr Q$, respectively.
At the fixed point
$\mathbf{K}=\{k_1<\cdots<k_r\}$, write
$
\mathbf{L}=\mathbf{M}\setminus\mathbf{K}
=\{l_1<\cdots<l_{s-r}\}.
$
The weights of $\mathscr S$ are $w_{k_1},\ldots,w_{k_r}$, and the weights of
$\mathscr Q$ are $w_{l_1},\ldots,w_{l_{s-r}}$.

The factor
\[
\frac{
W_{N,r}(w_{k_1},\ldots,w_{k_r}|\mathbf{w}|\mathbf{I})
}{
\displaystyle
\prod_{a,b=1}^{r}(1-qw_{k_a}^{-1}w_{k_b})
}
\]
is the restriction to $X_{\mathbf{M}}$ of the normalized type $A$ weight function class~\cite{RTV}, using
$K_T(T^*\operatorname{Gr}(r,N))\simeq K_T(\operatorname{Gr}(r,N))$.
By~\eqref{eq:typeC-lambda-y-weights}, the remaining numerator products are
\begin{align}
&\prod_{1\le a\le b\le r}
(1-w_{k_a}^{-1}w_{k_b}^{-1})
=
\lambda_{-1}(\operatorname{Sym}^2\mathscr S^*)|_{\mathbf{K}},
\notag\\
&\prod_{a=1}^{r}\prod_{b=1}^{s-r}
(1-w_{k_a}^{-1}w_{l_b}^{-1})
=
\lambda_{-1}(\mathscr S^*\otimes\mathscr Q^*)|_{\mathbf{K}},
\notag\\
&\prod_{1\le a<b\le s-r}
(1-qw_{l_a}^{-1}w_{l_b}^{-1})
=
\lambda_{-q}(\bigwedge^2\mathscr Q^*)|_{\mathbf{K}},
\notag\\
&\prod_{a=1}^{s-r}\prod_{j\in\mathbf{M}^{\mathrm c}}
(1-qw_{l_a}^{-1}w_j)
=
\lambda_{-q}
\left(
\mathscr Q^*\otimes
\bigoplus_{k\in\mathbf{M}^{\mathrm c}}\mathbb{C}_{w_k}
\right) \Bigg|_{\mathbf{K}}.
 \notag
\end{align}
Thus the numerator is the fixed-point restriction of a single global equivariant
$K$-class.

For the denominator,
\begin{align}
&\prod_{a=1}^{r}\prod_{b=1}^{s-r}
(1-w_{k_a}^{-1}w_{l_b})
=(-1)^{r(s-r)}
\frac{(\det\mathscr Q|_{\mathbf{K}})^r}
{(\det\mathscr S|_{\mathbf{K}})^{s-r}}
\lambda_{-1}(T_{\mathbf{K}}^*X_{\mathbf{M}}).
\label{eq:typeC-euler-denominator}
\end{align}
Therefore, the fixed-point sum in~\eqref{eq:typeC-cross-sector-branching} is the
pushforward of a single equivariant $K$-class with a line-bundle twist.
Its value lies in the representation ring and is therefore a Laurent polynomial.
The divisibility follows from the factor outside the sum in
\eqref{eq:typeC-cross-sector-branching}.
\end{proof}

\subsubsection{Newton degrees}

Throughout this subsection, all degree intervals and Newton degrees are taken with respect to $\varphi_{\mathbf M}$. For the degree counts below, we use the grading on $\mathbb C(q)[w_1^{\pm1},\ldots,w_N^{\pm1}]$ determined by
\[
\deg_{\mathbf M}(w_k)=
\begin{cases}
1,&k\in\mathbf M,\\
2,&k\notin\mathbf M,
\end{cases}
\]
with $\mathbb C(q)$ in degree $0$. Thus, for
$\nu=(\nu_1,\ldots,\nu_N)\in\mathbb Z^N$,
\[
\deg_{\mathbf M}(w_1^{\nu_1}\cdots w_N^{\nu_N})
=\sum_{j\in\mathbf M}\nu_j+2\sum_{j\notin\mathbf M}\nu_j
=\varphi_{\mathbf M}(\nu).
\]
For a nonzero Laurent polynomial $f$, its degree interval is $\varphi_{\mathbf M}(\mathcal N(f))$.

Assume below that $r\le s=|\mathbf{M}|\le n$.
After the substitution $u_j=w_{m_j}$, every $u_j$ has degree $1$.
Consequently, to prove the following proposition, we use the explicit formula
\eqref{Motherfunction} for $F_{N,s,r}$ and its type $C$ specialization
$\gamma=0$, $\beta=1$ from \eqref{WCfunction}. It is enough to find
uniform lower and upper bounds for the degrees of the homogeneous components in each summand of
this specialized form of \eqref{Motherfunction}. We compute these bounds
factor by factor.

\begin{prop}
For every $s=|\mathbf{M}|$ with $r\le s\le n$,
\begin{equation}
\varphi_{\mathbf{M}}\left(
\mathcal N\left(
WC_{N,s,r}(w_{m_1},\ldots,w_{m_s}|\mathbf{w}|\mathbf{I})
\right)
\right)
\subset
[-s^2+s-2r,\ s(N-s)-r+|\mathbf{I}\cap \mathbf{M}|].
\label{eq:typeC-raw-newton-bound}
\end{equation}
\end{prop}

\begin{proof}
By \eqref{WCfunction}, we use \eqref{Motherfunction} with $n=s$ and
$\gamma=0$, $\beta=1$. After specializing
$(u_1,\ldots,u_s)=(w_{m_1},\ldots,w_{m_s})$, each $u_j$ occurring in every
summand of this specialized formula is one of
$w_{m_1},\ldots,w_{m_s}$. Hence $\deg_{\mathbf M}u_j=1$ for $j=1,\ldots,s$.

It follows that $u_j^2-1$ has degree interval $[0,2]$. Each factor of the form $u_j-qw_k$ or $u_j-w_k$ is homogeneous of degree $1$ when $k\in\mathbf{M}$, while its two terms have degrees $1$ and $2$ when $k\notin\mathbf{M}$. Moreover, each quotient $(qu_j-u_k)/(u_j-u_k)$ is homogeneous of degree $0$, since its numerator and denominator are homogeneous of the same degree. Finally, each factor $u_ju_k-1$ or $u_ju_k-q$ has degree interval $[0,2]$.

We now apply these observations to the four parts of a fixed summand of the
specialized form of \eqref{Motherfunction}. First, the monomial prefactor
$\prod_{j=1}^{s}u_j^{-(N+s)}$ has the single degree $-s(N+s)$, and hence
degree interval $[-s(N+s),-s(N+s)]$.

Next, fix $j\in\{1,\ldots,r\}$ and consider
\[
(u_j^2-1)
\prod_{k=i_j+1}^{N}(u_j-qw_k)
\prod_{k=1}^{i_j-1}(u_j-w_k).
\]
There are $N-1$ factors of the form $u_j-qw_k$ or $u_j-w_k$, since the factor with $k=i_j$ is omitted. Every one of these factors contributes at least degree $1$, so the lower endpoint is $N-1$.

Among the indices $k\ne i_j$, the number of indices outside $\mathbf{M}$ is
\[
\begin{cases}
N-s, & i_j\in\mathbf{M},\\
N-s-1, & i_j\notin\mathbf{M}.
\end{cases}
\]
Indeed, if $i_j\in\mathbf{M}$, the omitted index $i_j$ is not one of the
$N-s$ indices outside $\mathbf{M}$, whereas if $i_j\notin\mathbf{M}$, one of
those $N-s$ indices is omitted. Each remaining index outside $\mathbf{M}$ can
increase the upper degree by one, and the factor $u_j^2-1$ contributes at most
degree $2$. Therefore the upper endpoint is
\[
\begin{cases}
2N-s+1, & i_j\in\mathbf{M},\\
2N-s, & i_j\notin\mathbf{M}.
\end{cases}
\]
When these contributions are multiplied over $j=1,\ldots,r$, the first case
occurs exactly $|\mathbf{I}\cap\mathbf{M}|$ times. Hence the upper endpoint
of the product is
\[
r(2N-s)+|\mathbf{I}\cap\mathbf{M}|.
\]
Thus we obtain
\begin{align}
&\prod_{j=1}^{r}
\left\{
(u_j^2-1)
\prod_{k=i_j+1}^{N}(u_j-qw_k)
\prod_{k=1}^{i_j-1}(u_j-w_k)
\right\} \notag
\end{align}
has degree interval contained in
$[r(N-1),\,r(2N-s)+|\mathbf{I}\cap\mathbf{M}|]$.

For each $j=r+1,\ldots,s$, consider $u_j\prod_{k=1}^{N}(u_j-qw_k)$. The factor $u_j$ has degree $1$. Among the other $N$ factors, the $s$ factors indexed by $k\in\mathbf{M}$ are homogeneous of degree $1$, whereas the $N-s$ factors indexed by $k\notin\mathbf{M}$ have degree interval $[1,2]$. Hence the degree interval for one such $j$ is contained in $[1+N,1+s+2(N-s)]=[N+1,2N-s+1]$. Multiplying over $j=r+1,\ldots,s$ gives the degree interval $[(s-r)(N+1),(s-r)(2N-s+1)]$.

It remains to consider the factors involving two $u$-variables:
\begin{align*}
&\prod_{1\le j<k\le r}
\frac{qu_j-u_k}{u_j-u_k}(u_ju_k-1)
\prod_{j=r+1}^{s}\prod_{k=1}^{r}
\frac{qu_j-u_k}{u_j-u_k}(u_ju_k-1)
\prod_{r+1\le j<k\le s}(u_ju_k-q).
\end{align*}
Every quotient has degree $0$, while every remaining factor has degree interval $[0,2]$. The total number of these remaining factors is $\binom r2+r(s-r)+\binom{s-r}{2}=\binom s2$. Consequently, their product has degree interval $\Big[0,2\binom s2 \Big]=[0,s(s-1)]$.

Adding the lower endpoints gives $-s(N+s)+r(N-1)+(s-r)(N+1)=-s^2+s-2r$. Adding the upper endpoints gives
\begin{align*}
&-s(N+s)+r(2N-s)+|\mathbf{I}\cap\mathbf{M}|
 +(s-r)(2N-s+1)+s(s-1) \notag \\
=&s(N-s)-r+|\mathbf{I}\cap\mathbf{M}|.
\end{align*}
Thus each summand is a finite sum of homogeneous rational functions of degrees between $-s^2+s-2r$ and $s(N-s)-r+|\mathbf{I}\cap\mathbf{M}|$. Since the sum over the permutations in the specialized form of \eqref{Motherfunction} is a Laurent polynomial, its homogeneous components lie in the same range. Therefore its degree interval is contained in $[-s^2+s-2r,\,s(N-s)-r+|\mathbf{I}\cap\mathbf{M}|]$, and~\eqref{eq:typeC-raw-newton-bound} follows.

\end{proof}

\begin{thm}
For every $s=|\mathbf{M}|$ with $r=|\mathbf{I}|\le s\le n$,
\begin{align}
&\varphi_{\mathbf{M}}\left(
\mathcal N\left(
\widetilde{WC}_{N,n,r}(\mathbf{w}_{\mathbf{M}}|\mathbf{w}|\mathbf{I})
\right)
\right)
\subset
\left[
-s^2+s-2r,
(N-s)(2N-s+2)-r+|\mathbf{I}\cap \mathbf{M}|
\right].
\label{eq:typeC-strong-newton-bound}
\end{align}
\end{thm}

\begin{proof}
Using the stability relation \eqref{eq:typeC-stability-s-ge-r}, the normalized definition \eqref{typeCnormalizeddefinition}, and then substituting $u_j=w_{m_j}$ into the normalization factor \eqref{typeCnormalizationfactor}, we obtain
\begin{align}
&\widetilde{WC}_{N,n,r}(\mathbf{w}_{\mathbf{M}}|\mathbf{w}|\mathbf{I})
=(1-q)^{s-r}
\prod_{\substack{j\le k\\j,k\in\mathbf{M}^{\mathrm c}}}
(1-qw_jw_k)
\frac{
WC_{N,s,r}(w_{m_1},\ldots,w_{m_s}|\mathbf{w}|\mathbf{I})
}{
\displaystyle
\prod_{j,k=1}^{s}(1-qw_{m_j}^{-1}w_{m_k})
}.
 \notag
\end{align}
The degree intervals of the products appearing here are
\begin{align*}
(1-q)^{s-r}&:\ [0,0],\\
\prod_{j,k=1}^{s}(1-qw_{m_j}^{-1}w_{m_k})&:\ [0,0],\\
\prod_{\substack{j\le k\\j,k\in\{1,\dots,N\}\setminus \mathbf{M}}}(1-qw_jw_k)
&:\ \left[0,4\tbinom{N-s+1}{2} \right]
=\left[0,2(N-s)(N-s+1) \right].
\end{align*}
Thus the lower endpoint is unchanged, while the upper endpoint becomes
\[
s(N-s)-r+|\mathbf{I}\cap \mathbf{M}|+2(N-s)(N-s+1)
=(N-s)(2N-s+2)-r+|\mathbf{I}\cap \mathbf{M}|.
\]
\end{proof}

\begin{corollary}
Let $s=|\mathbf{M}|$ satisfy $r\le s\le n$. If $\mathbf{I}\not\subseteq \mathbf{M}$ as sets, then $|\mathbf{I}\cap \mathbf{M}|\le r-1$, and therefore
\begin{align}
&\varphi_{\mathbf{M}}\left(
\mathcal N\left(
\widetilde{WC}_{N,n,r}(\mathbf{w}_{\mathbf{M}}|\mathbf{w}|\mathbf{I})
\right)
\right)
\subset[-s(s+1),(N-s)(2N-s+2)-1].
\label{eq:typeC-strict-smallness-sufficient-condition}
\end{align}
\end{corollary}

\begin{proof}
Use $|\mathbf{I}\cap\mathbf{M}|\le r-1$ in
\eqref{eq:typeC-strong-newton-bound}. Since $s\ge r$, we also have
$-s^2+s-2r\ge -s(s+1)$, which gives the displayed weaker lower bound.
\end{proof}

\medskip
\noindent\textbf{Recovery of Theorem~\ref{typeCmainproperties} when $s=r$.}
We record explicitly how the preceding specialization results reduce at $s=r$ to the three equal-cardinality properties of Theorem~\ref{typeCmainproperties}.  First, setting $s=r$ in the support condition \eqref{eq:typeC-shifted-dominance} gives $m_j\ge i_j$ for $1\le j\le r$, and hence recovers the vanishing statement in $\mathrm{(i)}$.  For $\mathbf{M}=\mathbf{I}$, the $s=r$ specialization of \eqref{eq:typeC-cross-sector-branching} is \eqref{relationwithAnormalized}; together with \eqref{identificationtypeA}, Theorem~\ref{RTVtheorem} $\mathrm{(ii)}$, and the explicit type $C$ root products \eqref{typecfirstblock} and \eqref{typecsecondblock}, this gives the diagonal value \eqref{normalizationWC}.  Second, in the same equal-cardinality relation \eqref{relationwithAnormalized}, the first two factors are explicit, while Theorem~\ref{RTVtheorem} $\mathrm{(i)}$ supplies the vertical factor; this recovers the full divisibility in $\mathrm{(ii)}$.  Finally, if $s=r$ and $\mathbf{M}\ne\mathbf{I}$, then $\mathbf{I}\not\subseteq\mathbf{M}$, so \eqref{eq:typeC-strict-smallness-sufficient-condition} becomes
\[
\varphi_{\mathbf{M}}\bigl(\mathcal N(\widetilde{WC}_{N,n,r}(\mathbf{w}_{\mathbf{M}}|\mathbf{w}|\mathbf{I}))\bigr)
\subset[-r(r+1),(N-r)(2N-r+2)-1],
\]
which is exactly the type $C$ $\mathbf{M}$-smallness assertion in $\mathrm{(iii)}$.

\subsection{Type $B$}

\subsubsection{General setup}
\label{sec:typeB-specialization}

The proofs of the type $B$ statements below are parallel to the corresponding type $C$ arguments, with the type $B$ specialization and normalization factors in place of the type $C$ ones, and are therefore omitted.

Repeated application of \eqref{stabilityWB} gives, if $s\ge r$,
\begin{equation*}
\widetilde{WB}_{N,n,r}(\mathbf{w}_{\mathbf{M}}|\mathbf{w}|\mathbf{I})
=
\widetilde{WB}_{N,s,r}(w_{m_1},\ldots,w_{m_s}|\mathbf{w}|\mathbf{I}),
\end{equation*}
and, if $s<r$,
\begin{equation*}
\widetilde{WB}_{N,n,r}(\mathbf{w}_{\mathbf{M}}|\mathbf{w}|\mathbf{I})
=
\widetilde{WB}_{N,r,r}\Bigl(
w_{m_1},\ldots,w_{m_s},\underbrace{1,\ldots,1}_{r-s}
|\mathbf{w}|\mathbf{I}\Bigr).
\end{equation*}
We again proceed in the order $s<r$, $s=r$, and $s>r$.

\subsubsection{The case $s<r$}

\begin{prop}
If $|\mathbf{M}|<|\mathbf{I}|=r$, then
\begin{equation}
\widetilde{WB}_{N,n,r}(\mathbf{w}_{\mathbf{M}}|\mathbf{w}|\mathbf{I})=0.
\label{eq:typeB-lower-cardinality-vanishing}
\end{equation}
\end{prop}

\subsubsection{Equal cardinality: $s=r$}
For $|\mathbf{M}|=r$, the relation with the normalized type $A$ weight function is exactly \eqref{relationwithAnormalizedtypeB}.  This is precisely the cardinality considered in Theorem~\ref{typeBmainproperties}.  As in type $C$, the results below recover $\mathrm{(i)}$--$\mathrm{(iii)}$ when $s=r$.

\subsubsection{The case $s>r$: branching formula}

Let
\[
\mathbf{M}=\{m_1<\cdots<m_s\}\subset\{1,\dots,N\},
\qquad
\mathbf{M}^{\mathrm c}=\{1,\dots,N\}\setminus\mathbf{M},
\qquad r\le s\le n.
\]
For each
$
\mathbf{K}=\{k_1<\cdots<k_r\}\subset\mathbf{M},
$
write
$
\mathbf{L}=\mathbf{M}\setminus\mathbf{K}
=\{l_1<\cdots<l_{s-r}\}.
$

\begin{thm}
For every $s=|\mathbf{M}|$ with $r=|\mathbf{I}|\le s\le n$,
\begin{align}
&\widetilde{WB}_{N,n,r}(\mathbf{w}_{\mathbf{M}}|\mathbf{w}|\mathbf{I})
\notag\\
&=(1-q)^{s-r}
\prod_{j\in\mathbf{M}^{\mathrm c}}(1-qw_j)
\prod_{\substack{j<k\\j,k\in\mathbf{M}^{\mathrm c}}}(1-qw_jw_k)
\sum_{\substack{\mathbf{K}\subset\mathbf{M}\\|\mathbf{K}|=r}}
\frac{
W_{N,r}(w_{k_1},\ldots,w_{k_r}|\mathbf{w}|\mathbf{I})
}{
\displaystyle
\prod_{a,b=1}^{r}(1-qw_{k_a}^{-1}w_{k_b})
}
\notag\\
&\quad\times
\frac{
\displaystyle
\prod_{a=1}^{r}(1-w_{k_a}^{-1})
\prod_{1\le a<b\le r}(1-w_{k_a}^{-1}w_{k_b}^{-1})
\prod_{a=1}^{r}\prod_{b=1}^{s-r}(1-w_{k_a}^{-1}w_{l_b}^{-1})
}{
\displaystyle
\prod_{a=1}^{r}\prod_{b=1}^{s-r}(1-w_{k_a}^{-1}w_{l_b})
}
\notag\\
&\quad\times
\prod_{1\le a<b\le s-r}(1-qw_{l_a}^{-1}w_{l_b}^{-1})
\prod_{a=1}^{s-r}\prod_{j\in\mathbf{M}^{\mathrm c}}
(1-qw_{l_a}^{-1}w_j).
\label{eq:typeB-cross-sector-branching}
\end{align}
For $r=0$, set $W_{N,0}=1$, and interpret every empty product as $1$.
If $s=r$, only the term $\mathbf{K}=\mathbf{M}$ remains, and the formula agrees with
\eqref{relationwithAnormalizedtypeB}.
\end{thm}

\subsubsection{Support condition}

Define the strict partitions corresponding to
$\mathbf{I}=\{i_1<\cdots<i_r\}$ and $\mathbf{M}=\{m_1<\cdots<m_s\}$ by
\[
\lambda(\mathbf{I})=(i_r,\ldots,i_1),
\qquad
\lambda(\mathbf{M})=(m_s,\ldots,m_1).
\]

\begin{prop}
Let $s=|\mathbf{M}|$ and assume $r\le s\le n$. Then
\[
\widetilde{WB}_{N,n,r}(\mathbf{w}_{\mathbf{M}}|\mathbf{w}|\mathbf{I})\ne0
\]
only if
\begin{equation}
i_j\le m_{s-r+j}
\qquad(1\le j\le r).
\label{eq:typeB-generalized-support-componentwise}
\end{equation}
This condition is equivalent to $\lambda(\mathbf{I})\subseteq\lambda(\mathbf{M})$.
\end{prop}

\subsubsection{Laurent polynomiality}

\begin{prop}
For every $\mathbf{M}=\{m_1<\cdots<m_s\}\subset\{1,\dots,N\}$ with $s\le n$,
\[
\widetilde{WB}_{N,n,r}(\mathbf{w}_{\mathbf{M}}|\mathbf{w}|\mathbf{I})
\in
\mathbb{C}(q)[w_1^{\pm1},\ldots,w_N^{\pm1}].
\]
\end{prop}

\begin{prop}
For every $s=|\mathbf{M}|$ with $r\le s\le n$, the Laurent polynomial
\[
\widetilde{WB}_{N,n,r}(\mathbf{w}_{\mathbf{M}}|\mathbf{w}|\mathbf{I})
\]
is divisible by
\begin{equation}
\prod_{j\in\mathbf{M}^{\mathrm c}}(1-qw_j)
\prod_{\substack{j<k\\j,k\in\mathbf{M}^{\mathrm c}}}(1-qw_jw_k).
\label{eq:typeB-complement-factor}
\end{equation}
\end{prop}

\subsubsection{Newton degrees}

\begin{prop}
For every $s=|\mathbf{M}|$ with $r\le s\le n$,
\begin{equation}
\varphi_{\mathbf{M}}\left(
\mathcal N\left(
WB_{N,s,r}(w_{m_1},\ldots,w_{m_s}|\mathbf{w}|\mathbf{I})
\right)
\right)
\subset
[-s^2+s-r,\ s(N-s)-r+|\mathbf{I}\cap\mathbf{M}|].
\label{eq:typeB-raw-newton-bound}
\end{equation}
\end{prop}

\begin{thm}
For every $s=|\mathbf{M}|$ with $r=|\mathbf{I}|\le s\le n$,
\begin{align}
&\varphi_{\mathbf{M}}\left(
\mathcal N\left(
\widetilde{WB}_{N,n,r}(\mathbf{w}_{\mathbf{M}}|\mathbf{w}|\mathbf{I})
\right)
\right)
\subset
\left[
-s^2+s-r,
(N-s)(2N-s)-r+|\mathbf{I}\cap\mathbf{M}|
\right].
\label{eq:typeB-strong-newton-bound}
\end{align}
\end{thm}

\begin{corollary}
Let $s=|\mathbf{M}|$ satisfy $r\le s\le n$. If $\mathbf{I}\not\subseteq\mathbf{M}$ as sets, then $|\mathbf{I}\cap\mathbf{M}|\le r-1$, and therefore
\begin{align}
&\varphi_{\mathbf{M}}\left(
\mathcal N\left(
\widetilde{WB}_{N,n,r}(\mathbf{w}_{\mathbf{M}}|\mathbf{w}|\mathbf{I})
\right)
\right)
\subset[-s^2,(N-s)(2N-s)-1].
\label{eq:typeB-strict-smallness-sufficient-condition}
\end{align}
\end{corollary}

\medskip
\noindent\textbf{Recovery of Theorem~\ref{typeBmainproperties} when $s=r$.}
The same comparison as in type $C$ applies.  The support condition with $s=r$ gives the vanishing statement in $\mathrm{(i)}$, while the specialization $\mathbf{M}=\mathbf{I}$ of the equal-cardinality relation \eqref{relationwithAnormalizedtypeB}, together with Theorem~\ref{RTVtheorem} $\mathrm{(ii)}$ and the explicit type $B$ root products above, gives the diagonal value \eqref{normalizationWB}.  In \eqref{relationwithAnormalizedtypeB}, the type $B$ factors are explicit, and Theorem~\ref{RTVtheorem} $\mathrm{(i)}$ supplies the remaining vertical factor, recovering $\mathrm{(ii)}$.  Finally, if $s=r$ and $\mathbf{M}\ne\mathbf{I}$, then \eqref{eq:typeB-strict-smallness-sufficient-condition} becomes
\[
\varphi_{\mathbf{M}}\bigl(\mathcal N(\widetilde{WB}_{N,n,r}(\mathbf{w}_{\mathbf{M}}|\mathbf{w}|\mathbf{I}))\bigr)
\subset[-r^2,(N-r)(2N-r)-1],
\]
which is exactly the type $B$ $\mathbf{M}$-smallness assertion in $\mathrm{(iii)}$.

\subsection{Orthogonality relations}
\label{subsec:BC-orthogonality-branching-formula-inversion}

We prove an orthogonality relation in each fixed-cardinality sector and an
all-cardinality biorthogonality relation, and then derive an inversion formula
for the branching coefficients.

Throughout this subsection, set
\begin{equation*}
[N]:=\{1,\ldots,N\},
\qquad
\mathbf A^{\mathrm c}:=[N]\setminus\mathbf A
\quad(\mathbf A\subset[N]).
\end{equation*}
If $\mathbf A=\{a_1<\cdots<a_s\}\subset[N]$ with $s\le n$, set
\begin{equation}
\mathbf w_{\mathbf A}
:=(w_{a_1},\ldots,w_{a_s},\underbrace{1,\ldots,1}_{n-s}).
\label{eq:BC-orth-common-wA}
\end{equation}
In expressions involving $\widetilde{WX}_{N,n,r}$,
the tuple $\mathbf w_{\mathbf A}$ has $n$ entries,
with $n-|\mathbf A|$ trailing entries equal to $1$.
\subsubsection{Orthogonality in the equal-cardinality case}

Let $X\in\{B,C\}$, and write $\widetilde{WX}$ for $\widetilde{WC}$ when
$X=C$ and for $\widetilde{WB}$ when $X=B$.
To prepare for a matrix formulation of the specialized
weight functions, we introduce notation for the prefactors
in the equal-cardinality factorizations
\eqref{relationwithAnormalized} and
\eqref{relationwithAnormalizedtypeB}.
These factors depend on the specialization subset $\mathbf M$
but not on the label $\mathbf I$.
For an $r$-element subset $\mathbf M\subset[N]$, set
\begin{equation}
D^C_r(\mathbf M;\mathbf w,q)
:=
\prod_{\substack{a\le b\\a,b\in\mathbf M}}
(1-w_a^{-1}w_b^{-1})
\prod_{\substack{a\le b\\a,b\in\mathbf M^{\mathrm c}}}
(1-qw_aw_b),
\label{eq:local-DC-orth}
\end{equation}
and
\begin{align}
D^B_r(\mathbf M;\mathbf w,q)
:={}&
\prod_{a\in\mathbf M}(1-w_a^{-1})
\prod_{\substack{a<b\\a,b\in\mathbf M}}
(1-w_a^{-1}w_b^{-1})
\prod_{a\in\mathbf M^{\mathrm c}}(1-qw_a)
\prod_{\substack{a<b\\a,b\in\mathbf M^{\mathrm c}}}
(1-qw_aw_b).
\label{eq:local-DB-orth}
\end{align}
We arrange the specializations of the normalized type $A$
weight functions into a matrix, with rows indexed by
$r$-element subsets $\mathbf M\subset[N]$ and columns by
$r$-element subsets $\mathbf I\subset[N]$.
Writing $\mathbf M=\{m_1<\cdots<m_r\}$, define its entries by
\begin{equation}
\Phi_{\mathbf M,\mathbf I}(\mathbf w,q)
:=
\widetilde W^{\mathrm{RTV}}_{\sigma_0,\mathcal I}
(w_{m_1},\ldots,w_{m_r},\mathbf w,q),
\qquad
\mathcal I=(\mathbf I,\mathbf I^{\mathrm c}).
\label{eq:PhiX-def}
\end{equation}
Denote this matrix by
\begin{equation}
\Phi^{(r)}(\mathbf w,q)
:=
\bigl(\Phi_{\mathbf M,\mathbf I}(\mathbf w,q)\bigr)_{|\mathbf M|=|\mathbf I|=r}.
\label{eq:Phi-fixed-cardinality-matrix}
\end{equation}
By \eqref{relationwithAnormalized},
\eqref{identificationtypeA},
and
\eqref{relationwithAnormalizedtypeB}, for $X\in\{B,C\}$,
\begin{equation}
\widetilde{WX}_{N,n,r}
(\mathbf w_{\mathbf M}\,|\mathbf w\,|\mathbf I)
=
D^X_r(\mathbf M;\mathbf w,q)\,
\Phi_{\mathbf M,\mathbf I}(\mathbf w,q).
\label{eq:PhiX-is-RTV}
\end{equation}
Thus the specialized type $B/C$ weight functions form a matrix
obtained from $\Phi^{(r)}$ by multiplying each row
by the corresponding factor $D_r^X$.

Let
\[
\rho(a)=N+1-a,
\qquad
\rho(\mathbf I)=\{N+1-i:i\in\mathbf I\},
\]
and set
\[
\mathbf w^{\rho}=(w_N,\ldots,w_1),
\qquad
\mathbf w^{-1}=(w_1^{-1},\ldots,w_N^{-1}).
\]
For a subset $\mathbf A\subset[N]$, the notation
$(\mathbf w^{\rho})_{\rho(\mathbf A)}$ and
$(\mathbf w^{-1})_{\mathbf A}$ means that the specialization convention
\eqref{eq:BC-orth-common-wA} is applied to
$\mathbf w^{\rho}$ and $\mathbf w^{-1}$, respectively.
Likewise, $D^X_r(\mathbf M;\mathbf w^{-1},q^{-1})$ denotes simultaneous
substitution $(\mathbf w,q)\mapsto(\mathbf w^{-1},q^{-1})$ in the defining
product for $D^X_r$.
Moreover,
\begin{equation}
D^X_r(\rho(\mathbf M);\mathbf w^{\rho},q)
=D^X_r(\mathbf M;\mathbf w,q).
\label{eq:local-D-reversal}
\end{equation}
We use the notation
$P_{\mathbf M}(\mathbf w)$, $R_{\mathbf M}(\mathbf w)$,
$Q_{\mathbf M}(\mathbf w,q)$, and $p(\mathbf M)$ from
\eqref{eq:RTV-Gr-P}--\eqref{eq:RTV-Gr-p}.

\begin{thm}
\label{thm:BC-equal-cardinality-orthogonality}
For $X\in\{B,C\}$ and any $r$-element subsets
$\mathbf J,\mathbf K\subset[N]$,
\begin{align}
&\sum_{\substack{\mathbf M\subset[N]\\|\mathbf M|=r}}
\frac{
q^{p(\mathbf K)}P_{\mathbf M}(\mathbf w)
}{
R_{\mathbf M}(\mathbf w)
Q_{\mathbf M}(\mathbf w,q)
D^X_r(\mathbf M;\mathbf w,q)
D^X_r(\mathbf M;\mathbf w^{-1},q^{-1})
}
\notag\\
&\qquad\times
\widetilde{WX}_{N,n,r}
\bigl((\mathbf w^{\rho})_{\rho(\mathbf M)}
\,|\mathbf w^{\rho}\,|\rho(\mathbf J)\bigr)
\left.
\widetilde{WX}_{N,n,r}
\bigl((\mathbf w^{-1})_{\mathbf M}
\,|\mathbf w^{-1}\,|\mathbf K\bigr)
\right|_{q\mapsto q^{-1}}
=\delta_{\mathbf J,\mathbf K}.
\label{eq:BC-equal-cardinality-orthogonality}
\end{align}
\end{thm}

\begin{proof}
This is a direct rewriting of the type $A$ orthogonality relation
\eqref{eq:RTV-Gr-orthogonality}.  Setting $\mathbf z=\mathbf w$ and $h=q$
in that relation gives
\begin{align*}
\delta_{\mathbf J,\mathbf K}
={}&
\sum_{\substack{\mathbf M\subset[N]\\ |\mathbf M|=r}}
q^{p(\mathbf K)}
\frac{
P_{\mathbf M}(\mathbf w)
}{
R_{\mathbf M}(\mathbf w)Q_{\mathbf M}(\mathbf w,q)
}
\\
&\qquad\times
\widetilde W^{\mathrm{RTV}}_{\mathrm{id},\mathcal J}
(w_{m_1},\ldots,w_{m_r},\mathbf w,q)
\,
\widetilde W^{\mathrm{RTV}}_{\sigma_0,\mathcal K}
(w_{m_1}^{-1},\ldots,w_{m_r}^{-1},
 \mathbf w^{-1},q^{-1}),
\end{align*}
where
$
\mathcal J=(\mathbf J,\mathbf J^{\mathrm c}), \
\mathcal K=(\mathbf K,\mathbf K^{\mathrm c}).$
By \eqref{eq:PhiX-def} and
\eqref{eq:RTV-id-via-reversal},
\begin{equation}
\widetilde W^{\mathrm{RTV}}_{\mathrm{id},\mathcal J}
(w_{m_1},\ldots,w_{m_r},\mathbf w,q)
=
\Phi_{\rho(\mathbf M),\rho(\mathbf J)}
(\mathbf w^{\rho},q).
\label{eq:PhiX-op-is-id}
\end{equation}
Using \eqref{eq:PhiX-is-RTV} with
$(\mathbf M,\mathbf I,\mathbf w)$ replaced by
$(\rho(\mathbf M),\rho(\mathbf J),\mathbf w^{\rho})$, and then
\eqref{eq:local-D-reversal}, we obtain
\[
\widetilde W^{\mathrm{RTV}}_{\mathrm{id},\mathcal J}
(w_{m_1},\ldots,w_{m_r},\mathbf w,q)
=
\frac{
\widetilde{WX}_{N,n,r}
\bigl(
(\mathbf w^{\rho})_{\rho(\mathbf M)}
\,|\mathbf w^{\rho}\,|\rho(\mathbf J)
\bigr)
}{
D^X_r(\mathbf M;\mathbf w,q)
}.
\]
Similarly, \eqref{eq:PhiX-def} and \eqref{eq:PhiX-is-RTV}, after the
substitution $(\mathbf w,q)\mapsto(\mathbf w^{-1},q^{-1})$, give
\[
\widetilde W^{\mathrm{RTV}}_{\sigma_0,\mathcal K}
(w_{m_1}^{-1},\ldots,w_{m_r}^{-1},
 \mathbf w^{-1},q^{-1})
=
\frac{
\left.
\widetilde{WX}_{N,n,r}
\bigl(
(\mathbf w^{-1})_{\mathbf M}
\,|\mathbf w^{-1}\,|\mathbf K
\bigr)
\right|_{q\mapsto q^{-1}}
}{
D^X_r(\mathbf M;\mathbf w^{-1},q^{-1})
}.
\]
Substituting these two identities into
\eqref{eq:RTV-Gr-orthogonality} gives
\eqref{eq:BC-equal-cardinality-orthogonality}.
\end{proof}

\begin{corollary}
\label{cor:Phi-inverse}
The matrix $\Phi^{(r)}(\mathbf w,q)$ is invertible.  Its inverse is given by
\begin{align}
\bigl(\Phi^{(r)}(\mathbf w,q)^{-1}\bigr)_{\mathbf I,\mathbf K}
={}&
q^{-p(\mathbf I)}
\frac{
P_{\mathbf K}(\mathbf w^{-1})
}{
R_{\mathbf K}(\mathbf w^{-1})
Q_{\mathbf K}(\mathbf w^{-1},q^{-1})
}
\Phi_{\rho(\mathbf K),\rho(\mathbf I)}
\bigl((\mathbf w^{-1})^{\rho},q^{-1}\bigr).
\label{eq:Phi-inverse}
\end{align}
\end{corollary}

\begin{proof}
Apply Theorem~\ref{thm:BC-equal-cardinality-orthogonality} after the
substitution $(\mathbf w,q)\mapsto(\mathbf w^{-1},q^{-1})$, and use
\eqref{eq:PhiX-is-RTV} to replace the two normalized type $B/C$ weight
functions by the corresponding $\Phi$-factors.  After cancelling the
$D^X_r$-factors, the resulting orthogonality identity is precisely the
matrix-inverse relation for $\Phi^{(r)}(\mathbf w,q)$, with the inverse
entries given by the right-hand side of \eqref{eq:Phi-inverse}.
\end{proof}

\medskip
\noindent\textbf{Example.}
Take $N=2$, $n=r=1$, and $X=C$, and let
$
\mathbf J=\{2\}, \ \mathbf K=\{1\}.
$
Then
$
\rho(\mathbf J)=\{1\},\ p(\mathbf K)=1.
$
The sum in Theorem~\ref{thm:BC-equal-cardinality-orthogonality}
has two terms, corresponding to
$\mathbf M=\{1\}$ and $\mathbf M=\{2\}$.
From \eqref{eq:RTV-Gr-P}--\eqref{eq:RTV-Gr-Q},
\begin{align*}
P_{\{1\}}(\mathbf w)&=-\frac{w_2}{w_1},&
R_{\{1\}}(\mathbf w)&=1-\frac{w_2}{w_1},&
Q_{\{1\}}(\mathbf w,q)&=1-q\frac{w_2}{w_1},\\
P_{\{2\}}(\mathbf w)&=-\frac{w_1}{w_2},&
R_{\{2\}}(\mathbf w)&=1-\frac{w_1}{w_2},&
Q_{\{2\}}(\mathbf w,q)&=1-q\frac{w_1}{w_2}.
\end{align*}
Moreover,
\begin{align*}
D^C_1(\{1\};\mathbf w,q)&=(1-w_1^{-2})(1-qw_2^2),\\
D^C_1(\{2\};\mathbf w,q)&=(1-w_2^{-2})(1-qw_1^2),\\
D^C_1(\{1\};\mathbf w^{-1},q^{-1})
&=(1-w_1^2)(1-q^{-1}w_2^{-2}),\\
D^C_1(\{2\};\mathbf w^{-1},q^{-1})
&=(1-w_2^2)(1-q^{-1}w_1^{-2}).
\end{align*}
Using \eqref{eq:PhiX-is-RTV}, the four values appearing in the two summands are
\begin{align}
\widetilde{WC}_{2,1,1}
\bigl(w_1\,|w_2,w_1\,|\{1\}\bigr)
&=D^C_1(\{1\};\mathbf w,q)(1-q),
 \notag\\
\widetilde{WC}_{2,1,1}
\bigl(w_2\,|w_2,w_1\,|\{1\}\bigr)
&=D^C_1(\{2\};\mathbf w,q)
\left(1-q\frac{w_1}{w_2}\right),
 \notag\\
\left.\widetilde{WC}_{2,1,1}
\bigl(w_1^{-1}\,|w_1^{-1},w_2^{-1}\,|\{1\}\bigr)
\right|_{q\mapsto q^{-1}}
&=D^C_1(\{1\};\mathbf w^{-1},q^{-1})
\left(1-q^{-1}\frac{w_1}{w_2}\right),
 \notag\\
\left.\widetilde{WC}_{2,1,1}
\bigl(w_2^{-1}\,|w_1^{-1},w_2^{-1}\,|\{1\}\bigr)
\right|_{q\mapsto q^{-1}}
&=D^C_1(\{2\};\mathbf w^{-1},q^{-1})(1-q^{-1}).
 \notag
\end{align}
Substituting these expressions into
Theorem~\ref{thm:BC-equal-cardinality-orthogonality},
the factors $D^C_1$ cancel.  The contribution from
$\mathbf M=\{1\}$ is
\begin{align}
T_{\{1\}}
&=q\frac{-w_2/w_1}
{(1-w_2/w_1)(1-qw_2/w_1)}
(1-q)\left(1-q^{-1}\frac{w_1}{w_2}\right)=\frac{1-q}{1-w_2/w_1},
 \notag
\end{align}
whereas the contribution from $\mathbf M=\{2\}$ is
\begin{align}
T_{\{2\}}
&=q\frac{-w_1/w_2}
{(1-w_1/w_2)(1-qw_1/w_2)}
\left(1-q\frac{w_1}{w_2}\right)(1-q^{-1})=-\frac{1-q}{1-w_2/w_1}.
 \notag
\end{align}
Therefore
$
T_{\{1\}}+T_{\{2\}}=0
=\delta_{\mathbf J,\mathbf K},
$
as required.

\subsubsection{All-cardinality biorthogonality}
\label{subsubsec:BC-all-cardinality-biorthogonality}

We first introduce the notation and matrices needed to state
the all-cardinality biorthogonality relation.
The branching formulas will express the matrix of specialized
type $B/C$ weight functions as a product of a branching matrix
and a block diagonal matrix formed from the type $A$ matrices.

For $\mathbf A\subset[N]$, we use
$|\lambda(\mathbf A)|$ for the size of the strict partition
$\lambda(\mathbf A)$ introduced in Subsection~2.2.2; thus
\[
|\lambda(\mathbf A)|=\sum_{a\in\mathbf A}a.
\]
Here $|\lambda(\mathbf A)|$ denotes the sum of the parts,
whereas $|\mathbf A|$ denotes the number of elements of
$\mathbf A$, equivalently the number of parts of
$\lambda(\mathbf A)$.

Let
\[
\mathbf M=\{m_1<\cdots<m_s\},\qquad
\mathbf K=\{k_1<\cdots<k_r\}\subset\mathbf M,
\qquad
\Omega:=\mathbf M\setminus\mathbf K
=\{l_1<\cdots<l_{s-r}\}.
\]
Recall that $\Phi^{(r)}(\mathbf w,q)$ is the matrix of
specializations of the normalized type $A$ weight functions,
with entries $\Phi_{\mathbf K,\mathbf I}(\mathbf w,q)$
for $|\mathbf K|=|\mathbf I|=r$.
By \eqref{identificationtypeA}, these entries can be written as
\begin{equation*}
\frac{
W_{N,r}(w_{k_1},\ldots,w_{k_r}\,|\mathbf w\,|\mathbf I)
}{
\displaystyle
\prod_{a,b=1}^{r}(1-qw_{k_a}^{-1}w_{k_b})
}
=
\Phi_{\mathbf K,\mathbf I}(\mathbf w,q).
\end{equation*}
Combining the prefactor preceding the sum with the remaining factors in
the summand 
in the right hand side of \eqref{eq:typeC-cross-sector-branching} and
\eqref{eq:typeB-cross-sector-branching}
indexed by $\mathbf K$, and rewriting the result in terms of
the subsets $\mathbf K$ and $\Omega$, define
\begin{align}
&\operatorname{Br}^C_{\mathbf M,\mathbf K}(\mathbf w,q)
:={}
(1-q)^{|\Omega|}
\prod_{\substack{a\le b\\a,b\in\mathbf M^{\mathrm c}}}
(1-qw_aw_b)
\notag\\
&\times
\prod_{\substack{a\le b\\a,b\in\mathbf K}}
(1-w_a^{-1}w_b^{-1})
\prod_{\substack{k\in\mathbf K\\t\in\Omega}}
\frac{1-w_k^{-1}w_t^{-1}}{1-w_k^{-1}w_t}
\prod_{\substack{a<b\\a,b\in\Omega}}
(1-qw_a^{-1}w_b^{-1})
\prod_{\substack{t\in\Omega\\j\in\mathbf M^{\mathrm c}}}
(1-qw_t^{-1}w_j),
\label{eq:allcard-Br-C}
\end{align}
and
\begin{align}
&\operatorname{Br}^B_{\mathbf M,\mathbf K}(\mathbf w,q)
:={}
(1-q)^{|\Omega|}
\prod_{j\in\mathbf M^{\mathrm c}}(1-qw_j)
\prod_{\substack{a<b\\a,b\in\mathbf M^{\mathrm c}}}
(1-qw_aw_b)
\notag\\
&\times
\prod_{k\in\mathbf K}(1-w_k^{-1})
\prod_{\substack{a<b\\a,b\in\mathbf K}}
(1-w_a^{-1}w_b^{-1})
\prod_{\substack{k\in\mathbf K\\t\in\Omega}}
\frac{1-w_k^{-1}w_t^{-1}}{1-w_k^{-1}w_t}
\prod_{\substack{a<b\\a,b\in\Omega}}
(1-qw_a^{-1}w_b^{-1})
\prod_{\substack{t\in\Omega\\j\in\mathbf M^{\mathrm c}}}
(1-qw_t^{-1}w_j).
\label{eq:allcard-Br-B}
\end{align}
With this notation, the branching formulas
\eqref{eq:typeC-cross-sector-branching} and
\eqref{eq:typeB-cross-sector-branching} become
\begin{equation}
\widetilde{WX}_{N,n,r}
(\mathbf w_{\mathbf M}\,|\mathbf w\,|\mathbf I)
=
\sum_{\substack{\mathbf K\subset\mathbf M\\|\mathbf K|=r}}
\operatorname{Br}^X_{\mathbf M,\mathbf K}(\mathbf w,q)
\Phi_{\mathbf K,\mathbf I}(\mathbf w,q),
\qquad X\in\{B,C\},
\label{eq:allcard-branching-factorization}
\end{equation}
for $r=|\mathbf I|\le |\mathbf M|\le n$.

To treat all subsets of $[N]$ simultaneously, we now choose
$n=N$. Accordingly, $\mathbf w_{\mathbf A}$ is padded with
$N-|\mathbf A|$ entries equal to $1$.
By stability, any common choice $n\ge N$ gives the same
specialized values.
For $\mathbf A\subset[N]$, define
\begin{equation}
\Delta^X_{\mathbf A}(\mathbf w,q)
:=
\widetilde{WX}_{N,N,|\mathbf A|}
(\mathbf w_{\mathbf A}\,|\mathbf w\,|\mathbf A)
=
D^X_{|\mathbf A|}(\mathbf A;\mathbf w,q)\,
\Phi_{\mathbf A,\mathbf A}(\mathbf w,q),
\label{eq:allcard-Delta-factor}
\end{equation}
where the second equality follows from \eqref{eq:PhiX-is-RTV}.
In particular,
\begin{equation*}
\operatorname{Br}^X_{\mathbf M,\mathbf M}(\mathbf w,q)
=D^X_{|\mathbf M|}(\mathbf M;\mathbf w,q).
\end{equation*}

We now introduce a further normalization at the matrix level
so that all diagonal entries are $1$.
We divide each row of the type $A$ specialization matrices
by its diagonal entry and adjust the branching coefficients
accordingly, preserving the branching factorization.
All matrices below are indexed by subsets of $[N]$.
Let
$\widehat\Phi(\mathbf w,q)
=(\widehat\Phi_{\mathbf K,\mathbf I}(\mathbf w,q))$ be given by
\begin{equation}
\widehat\Phi_{\mathbf K,\mathbf I}(\mathbf w,q)
:=
\begin{cases}
\displaystyle
\frac{\Phi_{\mathbf K,\mathbf I}(\mathbf w,q)}
{\Phi_{\mathbf K,\mathbf K}(\mathbf w,q)},
&|\mathbf K|=|\mathbf I|,\\[3mm]
0,&|\mathbf K|\ne|\mathbf I|,
\end{cases}
\label{eq:allcard-Phihat-def}
\end{equation}
and let
$\widehat{\mathrm{Br}}{}^X(\mathbf w,q)
=(\widehat{\mathrm{Br}}{}^X_{\mathbf M,\mathbf K}(\mathbf w,q))$
be given by
\begin{equation}
\widehat{\mathrm{Br}}{}^X_{\mathbf M,\mathbf K}(\mathbf w,q)
:=
\begin{cases}
\displaystyle
\frac{
\operatorname{Br}^X_{\mathbf M,\mathbf K}(\mathbf w,q)
\Phi_{\mathbf K,\mathbf K}(\mathbf w,q)
}{
\Delta^X_{\mathbf M}(\mathbf w,q)
},
&\mathbf K\subset\mathbf M,\\[3mm]
0,&\mathbf K\not\subset\mathbf M.
\end{cases}
\label{eq:allcard-Bhat-def}
\end{equation}
The matrix $\widehat\Phi(\mathbf w,q)$ is block diagonal with respect to cardinality,
while $\widehat{\mathrm{Br}}{}^X(\mathbf w,q)$ is unitriangular with respect to inclusion;
both have diagonal entries $1$.

Finally, assemble the specialized type $B/C$ weight functions
into a matrix and normalize each row by its diagonal value.
Let $\mathsf A^X(\mathbf w,q)
=(\mathsf A^X_{\mathbf M,\mathbf I}(\mathbf w,q))$ be the resulting matrix, with entries
\begin{equation}
\mathsf A^X_{\mathbf M,\mathbf I}(\mathbf w,q)
:=
\frac{
\widetilde{WX}_{N,N,|\mathbf I|}
(\mathbf w_{\mathbf M}\,|\mathbf w\,|\mathbf I)
}{
\Delta^X_{\mathbf M}(\mathbf w,q)
}.
\label{eq:allcard-A-def}
\end{equation}
If $|\mathbf I|>|\mathbf M|$, then
$\mathsf A^X_{\mathbf M,\mathbf I}(\mathbf w,q)=0$ by
\eqref{eq:typeC-lower-cardinality-vanishing} and
\eqref{eq:typeB-lower-cardinality-vanishing}.  The corresponding entry of
$\widehat{\mathrm{Br}}{}^X\widehat\Phi$ is also zero, since a nonzero
summand would require both $\mathbf K\subset\mathbf M$ and
$|\mathbf K|=|\mathbf I|$.  If $|\mathbf I|\le|\mathbf M|$, the definitions
above and \eqref{eq:allcard-branching-factorization} give
\begin{align*}
\bigl(\widehat{\mathrm{Br}}{}^X(\mathbf w,q)\widehat\Phi(\mathbf w,q)\bigr)_{\mathbf M,\mathbf I}
&=
\sum_{\mathbf K\subset[N]}
\widehat{\mathrm{Br}}{}^X_{\mathbf M,\mathbf K}(\mathbf w,q)
\widehat\Phi_{\mathbf K,\mathbf I}(\mathbf w,q) \\
&=
\frac{
\widetilde{WX}_{N,N,|\mathbf I|}
(\mathbf w_{\mathbf M}\,|\mathbf w\,|\mathbf I)
}{
\Delta^X_{\mathbf M}(\mathbf w,q)
} 
=
\mathsf A^X_{\mathbf M,\mathbf I}(\mathbf w,q).
\end{align*}
Hence, for all $\mathbf M,\mathbf I\subset[N]$,
\begin{equation}
\mathsf A^X(\mathbf w,q)
=\widehat{\mathrm{Br}}{}^X(\mathbf w,q)\widehat\Phi(\mathbf w,q).
\label{eq:allcard-P-factorization}
\end{equation}

\begin{thm}
\label{thm:BC-all-cardinality-biorthogonality}
For $X\in\{B,C\}$ and arbitrary subsets $\mathbf M,\mathbf L\subset[N]$,
\begin{equation}
\sum_{\mathbf I\subset[N]}
q^{|\lambda(\mathbf I)|-|\lambda(\mathbf L)|}
\mathsf A^X_{\mathbf M,\mathbf I}(\mathbf w,q)
\mathsf A^X_{\mathbf L^{\mathrm c},\mathbf I^{\mathrm c}}
(\mathbf w,q^{-1})
=\delta_{\mathbf M,\mathbf L}.
\label{eq:BC-all-cardinality-biorthogonality-matrix}
\end{equation}
Equivalently,
\begin{align}
&\sum_{\mathbf I\subset[N]}
q^{|\lambda(\mathbf I)|-|\lambda(\mathbf L)|}
\frac{
\widetilde{WX}_{N,N,|\mathbf I|}
(\mathbf w_{\mathbf M}\,|\mathbf w\,|\mathbf I)
}{\Delta^X_{\mathbf M}(\mathbf w,q)}
\frac{
\left.
\widetilde{WX}_{N,N,N-|\mathbf I|}
(\mathbf w_{\mathbf L^{\mathrm c}}\,|\mathbf w\,|\mathbf I^{\mathrm c})
\right|_{q\mapsto q^{-1}}
}{
\Delta^X_{\mathbf L^{\mathrm c}}(\mathbf w,q^{-1})
}
=\delta_{\mathbf M,\mathbf L}.
\label{eq:BC-all-cardinality-biorthogonality}
\end{align}
\end{thm}

The theorem is equivalent to the following formula for the inverse of
$\mathsf A^X(\mathbf w,q)$:
\begin{equation}
\bigl(\mathsf A^X(\mathbf w,q)^{-1}\bigr)_{\mathbf I,\mathbf L}
=
q^{|\lambda(\mathbf I)|-|\lambda(\mathbf L)|}
\mathsf A^X_{\mathbf L^{\mathrm c},\mathbf I^{\mathrm c}}
(\mathbf w,q^{-1}).
\label{eq:allcard-A-inverse-goal}
\end{equation}
Indeed, substituting \eqref{eq:allcard-A-inverse-goal} into the
$(\mathbf M,\mathbf L)$ entry of
$\mathsf A^X(\mathbf w,q)\mathsf A^X(\mathbf w,q)^{-1}=\mathrm{Id}$
gives exactly
\eqref{eq:BC-all-cardinality-biorthogonality-matrix}.

We therefore compute $\mathsf A^X(\mathbf w,q)^{-1}$ using
\eqref{eq:allcard-P-factorization}.  Since
\[
\mathsf A^X(\mathbf w,q)
=
\widehat{\mathrm{Br}}{}^X(\mathbf w,q)\widehat\Phi(\mathbf w,q),
\]
it is enough to determine the inverses of
$\widehat\Phi(\mathbf w,q)$ and
$\widehat{\mathrm{Br}}{}^X(\mathbf w,q)$.

\begin{lem}
\label{lem:allcard-Phihat-inverse}
For arbitrary subsets $\mathbf I,\mathbf K\subset[N]$,
\begin{equation}
\bigl(\widehat\Phi(\mathbf w,q)^{-1}\bigr)_{\mathbf I,\mathbf K}
=
q^{|\lambda(\mathbf I)|-|\lambda(\mathbf K)|}
\widehat\Phi_{\mathbf K^{\mathrm c},\mathbf I^{\mathrm c}}
(\mathbf w,q^{-1}).
\label{eq:allcard-Phihat-inverse-components}
\end{equation}
Both sides are zero unless $|\mathbf I|=|\mathbf K|$.
\end{lem}

\begin{proof}
By definition, $\widehat\Phi(\mathbf w,q)$ is block diagonal with respect to
cardinality.  Hence its inverse is block diagonal as well, so the left-hand
side of \eqref{eq:allcard-Phihat-inverse-components} vanishes when
$|\mathbf I|\ne|\mathbf K|$.  On the right-hand side,
$\widehat\Phi_{\mathbf K^{\mathrm c},\mathbf I^{\mathrm c}}$ vanishes unless
$|\mathbf K^{\mathrm c}|=|\mathbf I^{\mathrm c}|$, which is equivalent to
$|\mathbf K|=|\mathbf I|$.  It remains to consider
$|\mathbf I|=|\mathbf K|=r$.

By \eqref{eq:allcard-Phihat-def},
\[
\widehat\Phi^{(r)}(\mathbf w,q)
=
\operatorname{diag}
\bigl(\Phi_{\mathbf A,\mathbf A}(\mathbf w,q)^{-1}\bigr)_{|\mathbf A|=r}
\Phi^{(r)}(\mathbf w,q),
\]
and therefore
\[
\bigl(\widehat\Phi(\mathbf w,q)^{-1}\bigr)_{\mathbf I,\mathbf K}
=
\bigl(\Phi^{(r)}(\mathbf w,q)^{-1}\bigr)_{\mathbf I,\mathbf K}
\Phi_{\mathbf K,\mathbf K}(\mathbf w,q).
\]
Together with Corollary~\ref{cor:Phi-inverse}, this gives
\begin{align}
\bigl(\widehat\Phi(\mathbf w,q)^{-1}\bigr)_{\mathbf I,\mathbf K}
={}&
q^{-p(\mathbf I)}
\frac{
P_{\mathbf K}(\mathbf w^{-1})
\Phi_{\mathbf K,\mathbf K}(\mathbf w,q)
}{
R_{\mathbf K}(\mathbf w^{-1})
Q_{\mathbf K}(\mathbf w^{-1},q^{-1})
}
\Phi_{\rho(\mathbf K),\rho(\mathbf I)}
\bigl((\mathbf w^{-1})^{\rho},q^{-1}\bigr).
\label{eq:allcard-Phihat-inverse-before-duality}
\end{align}

For a subspace $\mathcal U\subset\mathbb C^N$, put
$
\mathcal U^{\perp}
:=\{f\in(\mathbb C^N)^*:f|_{\mathcal U}=0\}.
$
Then $\mathcal U\mapsto\mathcal U^{\perp}$ identifies
$\operatorname{Gr}(r,\mathbb C^N)$ with
$\operatorname{Gr}(N-r,(\mathbb C^N)^*)$.
The coordinate fixed point indexed by $\mathbf M$ is sent to the one indexed
by $\mathbf M^{\mathrm c}$, and the torus weights on $(\mathbb C^N)^*$ are
$z_1^{-1},\ldots,z_N^{-1}$.  This identification lifts to the cotangent
bundles and, after inversion of the torus characters, induces the corresponding
equivariant $K$-theory isomorphism.  Hence the transformed fixed-point
restrictions below come from an equivariant $K$-class on
$T^*\operatorname{Gr}(N-r,(\mathbb C^N)^*)$.
Let $h$ be the parameter in Theorem~\ref{RTVtheorem}, and put
$
\mathcal M=(\mathbf M,\mathbf M^{\mathrm c}),
\
\mathcal M^{\vee}=(\mathbf M^{\mathrm c},\mathbf M).
$
We check conditions $\mathrm{(i)}$--$\mathrm{(iii)}$ of
Theorem~\ref{RTVtheorem} after the substitution
$\mathbf z\mapsto\mathbf z^{-1}$ and the replacement
$\mathcal M\mapsto\mathcal M^{\vee}$.
For condition $\mathrm{(i)}$,
\begin{align}
e_{\mathrm{id},\mathcal M,-}^{\mathrm{vert}}
(\mathbf z^{-1},h)
&=
\prod_{\substack{a\in\mathbf M,\ b\in\mathbf M^{\mathrm c}\\ b<a}}
\left(1-h\frac{z_a}{z_b}\right)
=
 e_{\sigma_0,\mathcal M^{\vee},-}^{\mathrm{vert}}
(\mathbf z,h).
 \notag
\end{align}
For condition $\mathrm{(ii)}$, put
$\mathcal I^{\vee}=(\mathbf I^{\mathrm c},\mathbf I)$.  At the diagonal fixed
point, \eqref{general-diagonal-factor} gives
\begin{align}
&P_{\mathrm{id},\mathcal I}(\mathbf z^{-1})
 e_{\mathrm{id},\mathcal I}(\mathbf z^{-1},h)
=
\prod_{\substack{a\in\mathbf I,\ b\in\mathbf I^{\mathrm c}\\ a<b}}
\left(1-\frac{z_a}{z_b}\right)
\prod_{\substack{a\in\mathbf I,\ b\in\mathbf I^{\mathrm c}\\ b<a}}
\left(1-h\frac{z_a}{z_b}\right)
\notag\\
&=
P_{\sigma_0,\mathcal I^{\vee}}(\mathbf z)
 e_{\sigma_0,\mathcal I^{\vee}}(\mathbf z,h).
 \notag
\end{align}

It remains to check the smallness condition $\mathrm{(iii)}$.  Suppose
$\mathcal M\ne\mathcal I$ and set
$
f_{\mathcal M}(\mathbf z)
:=
\widetilde W_{\mathrm{id},\mathcal I}^{\mathrm{RTV}}
(\mathbf z_{\mathbf M},\mathbf z,h).
$
By \eqref{eq:general-sigma-W} and \eqref{eq:general-Wtilde-def}, after the
fixed-point specialization $\mathbf t=\mathbf z_{\mathbf M}$ all the
$z$-variables occur through ratios $z_a/z_b$.  Hence every Laurent monomial
$
\mathbf z^{\nu}=z_1^{\nu_1}\cdots z_N^{\nu_N}
$
occurring in $f_{\mathcal M}(\mathbf z)$ satisfies
\begin{align}
\sum_{j=1}^N\nu_j=0. \label{sumtozero}
\end{align}
Under $\mathbf z\mapsto\mathbf z^{-1}$ the monomial $\mathbf z^\nu$ becomes
$\mathbf z^{-\nu}$, and
\begin{align}
\varphi_{\mathcal M^{\vee}}(-\nu)
&=-\sum_{j\in\mathbf M^{\mathrm c}}\nu_j
  -2\sum_{j\in\mathbf M}\nu_j
=\sum_{j\in\mathbf M}\nu_j
  +2\sum_{j\in\mathbf M^{\mathrm c}}\nu_j
=\varphi_{\mathcal M}(\nu),
 \notag
\end{align}
where the second equality uses \eqref{sumtozero}.  Moreover,
\[
\mathcal N\bigl(f_{\mathcal M}(\mathbf z^{-1})\bigr)
=-\mathcal N\bigl(f_{\mathcal M}(\mathbf z)\bigr).
\]
Therefore
\[
\varphi_{\mathcal M^{\vee}}
\bigl(\mathcal N(f_{\mathcal M}(\mathbf z^{-1}))\bigr)
=
\varphi_{\mathcal M}
\bigl(\mathcal N(f_{\mathcal M}(\mathbf z))\bigr)
\subset[0,r(N-r)-1].
\]
Since $\mathcal M\ne\mathcal I$ implies
$\mathcal M^{\vee}\ne\mathcal I^{\vee}$, and the smallness interval for
$\operatorname{Gr}(N-r,N)$ is again $[0,r(N-r)-1]$, condition
$\mathrm{(iii)}$ is preserved.

The uniqueness statement in Theorem~\ref{RTVtheorem} therefore gives, for
all $\mathbf M\subset[N]$ with $|\mathbf M|=r$,
\begin{equation}
\widetilde W_{\mathrm{id},\mathcal I}^{\mathrm{RTV}}
(\mathbf z_{\mathbf M}^{-1},\mathbf z^{-1},h)
=
\widetilde W_{\sigma_0,\mathcal I^{\vee}}^{\mathrm{RTV}}
(\mathbf z_{\mathbf M^{\mathrm c}},\mathbf z,h).
\label{eq:allcard-RTV-complement-duality}
\end{equation}
Taking $\mathbf M=\mathbf K$, $\mathbf z=\mathbf w$, and $h=q^{-1}$, and
using \eqref{eq:PhiX-def} and \eqref{eq:PhiX-op-is-id}, we obtain
\begin{equation}
\Phi_{\rho(\mathbf K),\rho(\mathbf I)}
\bigl((\mathbf w^{-1})^{\rho},q^{-1}\bigr)
=
\Phi_{\mathbf K^{\mathrm c},\mathbf I^{\mathrm c}}
(\mathbf w,q^{-1}).
\label{eq:allcard-Phi-complement-duality}
\end{equation}

By \eqref{eq:RTV-Gr-P}--\eqref{eq:RTV-Gr-Q},
Theorem~\ref{RTVtheorem} $\mathrm{(ii)}$, and
\eqref{general-diagonal-factor}, after simplification we obtain
\begin{equation}
\frac{
P_{\mathbf K}(\mathbf w^{-1})
\Phi_{\mathbf K,\mathbf K}(\mathbf w,q)
\Phi_{\mathbf K^{\mathrm c},\mathbf K^{\mathrm c}}(\mathbf w,q^{-1})
}{
R_{\mathbf K}(\mathbf w^{-1})
Q_{\mathbf K}(\mathbf w^{-1},q^{-1})
}
=q^{p(\mathbf K)}.
\label{eq:allcard-scalar-complement}
\end{equation}
Substituting \eqref{eq:allcard-Phi-complement-duality} and
\eqref{eq:allcard-scalar-complement} into
\eqref{eq:allcard-Phihat-inverse-before-duality} yields
\begin{align}
\bigl(\widehat\Phi(\mathbf w,q)^{-1}\bigr)_{\mathbf I,\mathbf K}
&=
q^{-p(\mathbf I)+p(\mathbf K)}
\frac{
\Phi_{\mathbf K^{\mathrm c},\mathbf I^{\mathrm c}}(\mathbf w,q^{-1})
}{
\Phi_{\mathbf K^{\mathrm c},\mathbf K^{\mathrm c}}(\mathbf w,q^{-1})
}
=
q^{-p(\mathbf I)+p(\mathbf K)}
\widehat\Phi_{\mathbf K^{\mathrm c},\mathbf I^{\mathrm c}}
(\mathbf w,q^{-1}).
\label{eq:allcard-Psihat-before-lambda}
\end{align}
If $\mathbf A=\{a_1<\cdots<a_r\}$, then by \eqref{eq:RTV-Gr-p} and
$|\lambda(\mathbf A)|=\sum_{j=1}^r a_j$,
\begin{align}
p(\mathbf A)+|\lambda(\mathbf A)|
&=
\sum_{j=1}^r(N-r+j-a_j)+\sum_{j=1}^r a_j
=
r(N-r)+\frac{r(r+1)}2
=
rN-\binom r2.
 \notag
\end{align}
Since the right-hand side depends only on $r$,
\[
p(\mathbf K)-p(\mathbf I)
=
|\lambda(\mathbf I)|-|\lambda(\mathbf K)|.
\]
Substituting this into \eqref{eq:allcard-Psihat-before-lambda} proves
\eqref{eq:allcard-Phihat-inverse-components}.
\end{proof}

\begin{lem}
\label{lem:allcard-Bhat-product}
Let $\mathbf K\subset\mathbf M$ and put
$\Omega=\mathbf M\setminus\mathbf K$.  Then
\begin{align}
\widehat{\mathrm{Br}}{}^C_{\mathbf M,\mathbf K}(\mathbf w,q)
={}&
(1-q)^{|\Omega|}
\prod_{t\in\Omega}\frac{1}{1-w_t^{-2}}
\prod_{\substack{a<b\\a,b\in\Omega}}
\frac{1-qw_a^{-1}w_b^{-1}}{1-w_a^{-1}w_b^{-1}}
\notag\\
&\times
\prod_{\substack{k\in\mathbf K,\ t\in\Omega\\k<t}}
\frac{1-qw_k^{-1}w_t}{1-w_k^{-1}w_t}
\prod_{\substack{t\in\Omega,\ j\in\mathbf M^{\mathrm c}\\j<t}}
\frac{1-qw_t^{-1}w_j}{1-w_t^{-1}w_j},
\label{eq:allcard-Bhat-product-C}
\end{align}
and
\begin{align}
\widehat{\mathrm{Br}}{}^B_{\mathbf M,\mathbf K}(\mathbf w,q)
={}&
(1-q)^{|\Omega|}
\prod_{t\in\Omega}\frac{1}{1-w_t^{-1}}
\prod_{\substack{a<b\\a,b\in\Omega}}
\frac{1-qw_a^{-1}w_b^{-1}}{1-w_a^{-1}w_b^{-1}}
\notag\\
&\times
\prod_{\substack{k\in\mathbf K,\ t\in\Omega\\k<t}}
\frac{1-qw_k^{-1}w_t}{1-w_k^{-1}w_t}
\prod_{\substack{t\in\Omega,\ j\in\mathbf M^{\mathrm c}\\j<t}}
\frac{1-qw_t^{-1}w_j}{1-w_t^{-1}w_j}.
\label{eq:allcard-Bhat-product-B}
\end{align}
\end{lem}

\begin{proof}
We give the calculation for type $C$.  By
\eqref{eq:allcard-Bhat-def} and \eqref{eq:allcard-Delta-factor},
\[
\widehat{\mathrm{Br}}{}^C_{\mathbf M,\mathbf K}(\mathbf w,q)
=
\frac{
\operatorname{Br}^C_{\mathbf M,\mathbf K}(\mathbf w,q)
\Phi_{\mathbf K,\mathbf K}(\mathbf w,q)
}{D^C_{|\mathbf M|}(\mathbf M;\mathbf w,q)
\Phi_{\mathbf M,\mathbf M}(\mathbf w,q)
}
.
\]
Put $\Omega=\mathbf M\setminus\mathbf K$.  The four factors on the
right-hand side are as follows.  First,
\begin{align*}
\operatorname{Br}^C_{\mathbf M,\mathbf K}(\mathbf w,q)
={}&
(1-q)^{|\Omega|}
\prod_{\substack{a\le b\\a,b\in\mathbf M^{\mathrm c}}}
(1-qw_aw_b)
\prod_{\substack{a\le b\\a,b\in\mathbf K}}
(1-w_a^{-1}w_b^{-1})
\\
&\times
\prod_{\substack{k\in\mathbf K\\t\in\Omega}}
\frac{1-w_k^{-1}w_t^{-1}}{1-w_k^{-1}w_t}
\prod_{\substack{a<b\\a,b\in\Omega}}
(1-qw_a^{-1}w_b^{-1})
\prod_{\substack{t\in\Omega\\j\in\mathbf M^{\mathrm c}}}
(1-qw_t^{-1}w_j),
\end{align*}
and
\[
D^C_{|\mathbf M|}(\mathbf M;\mathbf w,q)
=
\prod_{\substack{a\le b\\a,b\in\mathbf M}}
(1-w_a^{-1}w_b^{-1})
\prod_{\substack{a\le b\\a,b\in\mathbf M^{\mathrm c}}}
(1-qw_aw_b).
\]
Next, if
$\mathbf K=\{k_1<\cdots<k_r\}$ and
$\mathbf M=\{m_1<\cdots<m_s\}$, then by \eqref{eq:PhiX-def},
Theorem~\ref{RTVtheorem} $\mathrm{(ii)}$, and
\eqref{general-diagonal-factor},
\begin{align*}
\Phi_{\mathbf K,\mathbf K}(\mathbf w,q)
&=
\widetilde W^{\mathrm{RTV}}_{\sigma_0,\mathcal K}
(w_{k_1},\ldots,w_{k_r},\mathbf w,q)\\
&=
\prod_{\substack{k\in\mathbf K,\ a\in\mathbf K^{\mathrm c}\\a<k}}
(1-w_k^{-1}w_a)
\prod_{\substack{k\in\mathbf K,\ a\in\mathbf K^{\mathrm c}\\k<a}}
(1-qw_k^{-1}w_a),
\\[1mm]
\Phi_{\mathbf M,\mathbf M}(\mathbf w,q)
&=
\widetilde W^{\mathrm{RTV}}_{\sigma_0,\mathcal M}
(w_{m_1},\ldots,w_{m_s},\mathbf w,q)\\
&=
\prod_{\substack{m\in\mathbf M,\ j\in\mathbf M^{\mathrm c}\\j<m}}
(1-w_m^{-1}w_j)
\prod_{\substack{m\in\mathbf M,\ j\in\mathbf M^{\mathrm c}\\m<j}}
(1-qw_m^{-1}w_j),
\end{align*}
where
$\mathcal K=(\mathbf K,\mathbf K^{\mathrm c})$ and
$\mathcal M=(\mathbf M,\mathbf M^{\mathrm c})$.
Substituting these four expressions and simplifying, using
\[
\mathbf M=\mathbf K\sqcup\Omega,
\qquad
\mathbf K^{\mathrm c}=\Omega\sqcup\mathbf M^{\mathrm c},
\]
gives
\eqref{eq:allcard-Bhat-product-C}.  The type $B$ formula is obtained
by the same simplification from \eqref{eq:allcard-Br-B} and
\eqref{eq:local-DB-orth}.
\end{proof}

For two variables $x,y$, define
\begin{align}
H_{00}(x,y)&:=xy-q,\ \ \
H_{01}(x,y):=\frac{(qx-y)(xy-1)}{x-y},\notag\\
H_{10}(x,y)&:=\frac{(x-qy)(xy-1)}{x-y},\ \ \
H_{11}(x,y):=qxy-1.
\label{eq:allcard-hweights}
\end{align}
They satisfy $H_{ab}(x,y)=H_{ba}(y,x)$.

\begin{lem}
\label{lem:allcard-universal-cancellation}
For $m\ge1$ and generic variables $x_1,\ldots,x_m$,
\begin{equation}
\sum_{\epsilon\in\{0,1\}^m}
(-1)^{m-|\epsilon|}
\prod_{1\le i<j\le m}
H_{\epsilon_i,\epsilon_j}(x_i,x_j)=0,
\label{eq:allcard-universal-cancellation}
\end{equation}
where $|\epsilon|=\epsilon_1+\cdots+\epsilon_m$.
\end{lem}

\begin{proof}
Denote the left-hand side of
\eqref{eq:allcard-universal-cancellation} by
$F_m(x_1,\ldots,x_m)$ and regard it as a rational function of $x_m$.
The only possible finite poles are $x_m=x_i$ with $i<m$.  The residues of
the two off-diagonal factors are
\begin{align}
\operatorname*{Res}_{y=x}H_{01}(x,y)
&=x(1-q)(x^2-1),\ \ \
\operatorname*{Res}_{y=x}H_{10}(x,y)
=-x(1-q)(x^2-1).
 \notag
\end{align}
Fix $\epsilon_j$ for $j\ne i,m$.  The two terms with
$(\epsilon_i,\epsilon_m)=(0,1)$ and $(1,0)$ have the same alternating sign,
and at $x_m=x_i$ all other pair factors agree after interchanging
$\epsilon_i$ and $\epsilon_m$, using
$H_{a,b}(x,y)=H_{b,a}(y,x)$.  Their residues therefore cancel.  Hence
$F_m$ is a polynomial in $x_m$.

To determine the degree in $x_m$, keep $x_1,\ldots,x_{m-1}$ fixed and
set $y=x_m$.  From the definitions of the four functions $H_{a,b}$,
\begin{equation*}
H_{a,b}(x,y)=q^a x y+O(1)
\qquad (y\to\infty),
\qquad a,b\in\{0,1\}.
\end{equation*}
In particular, the coefficient of the leading term is independent of the
second index $b$.  Fix $\epsilon_1,\ldots,\epsilon_{m-1}$.  The factors in
the corresponding summand of $F_m$ that contain $x_m$ satisfy
\begin{align}
\prod_{i=1}^{m-1}
H_{\epsilon_i,\epsilon_m}(x_i,x_m)
={}&
q^{\epsilon_1+\cdots+\epsilon_{m-1}}
\left(\prod_{i=1}^{m-1}x_i\right)x_m^{m-1}
+O(x_m^{m-2}).
 \notag
\end{align}
The leading term on the right-hand side is the same for
$\epsilon_m=0$ and $\epsilon_m=1$, whereas the alternating signs of these
two terms in $F_m$ are opposite.  Hence their coefficients of
$x_m^{m-1}$ cancel.  Therefore
\begin{equation}
\deg_{x_m}F_m\le m-2.
\label{eq:allcard-degree}
\end{equation}

We now prove that $F_m$ has $m-1$ roots.  The case $m=1$ is immediate.
For $m=2$, \eqref{eq:allcard-degree} shows that $F_2$ is constant, while
\[
F_2(x_1,x_1^{-1})
=H_{00}(x_1,x_1^{-1})+H_{11}(x_1,x_1^{-1})=0.
\]
Hence $F_2=0$.  Assume $m\ge3$ and that the identity is known for $m-2$.
Fix $k<m$ and
specialize $x_m=x_k^{-1}$.  Since
$
H_{01}(x_k,x_k^{-1})=H_{10}(x_k,x_k^{-1})=0,
$
only the terms with
$
\epsilon_k=\epsilon_m=:s,
\ s\in\{0,1\},
$
survive.  Let $j\ne k,m$, put $u=x_j$, and write
$\eta=\epsilon_j$.  Using
$H_{ab}(x,y)=H_{ba}(y,x)$ when necessary, the product of the two factors
involving $x_j$ and the variables $x_k,x_m$ can be written as
$H_{\eta,s}(u,x_k)H_{\eta,s}(u,x_k^{-1})$.  For
$\eta\in\{0,1\}$ one has
\begin{align}
H_{\eta,0}(u,x)H_{\eta,0}(u,x^{-1})
&=\frac{(ux-q)(u-qx)}{x}, \
H_{\eta,1}(u,x)H_{\eta,1}(u,x^{-1})
=\frac{(qu-x)(qux-1)}{x}.
\label{eq:allcard-fusion}
\end{align}
For fixed $s$, the right-hand side is independent of the remaining bit
$\eta=\epsilon_j$.  Hence all factors involving $k$ or $m$ are independent
of $\epsilon_j$ for $j\ne k,m$ and may be taken outside the alternating sum
over the remaining bits.  The remaining product and alternating signs are
exactly those defining $F_{m-2}$ in the variables
$x_1,\ldots,\widehat{x_k},\ldots,x_{m-1}$.  By the induction hypothesis,
this sum vanishes for each $s$.  Therefore
\begin{equation*}
F_m\big|_{x_m=x_k^{-1}}=0
\qquad(k=1,\ldots,m-1).
\end{equation*}
The polynomial $F_m$ therefore has degree at most $m-2$ and the $m-1$
distinct roots $x_1^{-1},\ldots,x_{m-1}^{-1}$.  Hence $F_m=0$
identically.
\end{proof}

\begin{prop}
\label{prop:allcard-Bhat-inverse}
For $X\in\{B,C\}$,
\begin{equation}
\bigl(\widehat{\mathrm{Br}}{}^X(\mathbf w,q)^{-1}\bigr)_{\mathbf I,\mathbf M}
=
q^{|\lambda(\mathbf I)|-|\lambda(\mathbf M)|}
\widehat{\mathrm{Br}}{}^X_{\mathbf M^{\mathrm c},\mathbf I^{\mathrm c}}
(\mathbf w,q^{-1}).
\label{eq:allcard-Bhat-inverse}
\end{equation}
\end{prop}

\begin{proof}
We prove the type $C$ case. The type $B$ case can be proved in the same way.
It is enough to multiply
$\widehat{\mathrm{Br}}{}^C(\mathbf w,q)$ by the matrix on the right-hand
side of \eqref{eq:allcard-Bhat-inverse} for $X=C$ and show that the product
is the identity.  Its $(\mathbf M,\mathbf L)$ entry is
\begin{equation}
\sum_{\mathbf K\subset[N]}
\widehat{\mathrm{Br}}{}^C_{\mathbf M,\mathbf K}(\mathbf w,q)
q^{|\lambda(\mathbf K)|-|\lambda(\mathbf L)|}
\widehat{\mathrm{Br}}{}^C_{\mathbf L^{\mathrm c},\mathbf K^{\mathrm c}}
(\mathbf w,q^{-1}).
\label{eq:allcard-boundary-product}
\end{equation}
By triangularity, a summand can be nonzero only when
$\mathbf L\subset\mathbf K\subset\mathbf M$.  If $\mathbf M=\mathbf L$,
only $\mathbf K=\mathbf M$ occurs, and the value is $1$ because the
diagonal entries of $\widehat{\mathrm{Br}}{}^C$ are $1$.

Assume now that $\mathbf L\subsetneq\mathbf M$ and write
\[
\mathcal D:=\mathbf M\setminus\mathbf L
=\{d_1<\cdots<d_m\}.
\]
Every intermediate subset $\mathbf K$ with
$\mathbf L\subset\mathbf K\subset\mathbf M$ is uniquely of the form
\[
\mathbf K=\mathbf K(\mathbf U):=\mathbf L\cup\mathbf U,
\qquad
\mathbf U\subset\mathcal D.
\]
For $\mathbf U\subset\mathcal D$, let
$\epsilon=(\epsilon_1,\ldots,\epsilon_m)\in\{0,1\}^m$ be defined by
\[
\epsilon_i=
\begin{cases}
1,& d_i\in\mathbf U,\\
0,& d_i\notin\mathbf U.
\end{cases}
\]
Since $|\lambda(\mathbf A)|=\sum_{a\in\mathbf A}a$, we have
\begin{equation*}
|\lambda(\mathbf K(\mathbf U))|-|\lambda(\mathbf L)|
=
\sum_{d\in\mathbf U}d
=
\sum_{i=1}^m d_i\epsilon_i.
\end{equation*}
Thus, if we set
\begin{equation}
S(\mathbf U)
:=
\widehat{\mathrm{Br}}{}^C_{\mathbf M,\mathbf K(\mathbf U)}(\mathbf w,q)
q^{\sum_{d\in\mathbf U}d}
\widehat{\mathrm{Br}}{}^C_{\mathbf L^{\mathrm c},
\mathbf K(\mathbf U)^{\mathrm c}}(\mathbf w,q^{-1}),
\label{eq:allcard-SU-def}
\end{equation}
then \eqref{eq:allcard-boundary-product} is simply
$\displaystyle
\sum_{\mathbf U\subset\mathcal D}S(\mathbf U).
$

Before treating the general case, it is useful to see explicitly what
happens when $|\mathcal D|=1$.  Write
$
\mathcal D=\{d\},
\
\mathbf M=\mathbf L\sqcup\{d\}.
$
There are only two subsets of $\mathcal D$, namely $\emptyset$ and
$\{d\}$.  If $\mathbf U=\emptyset$, then
$\mathbf K(\mathbf U)=\mathbf L$, and hence
\begin{align*}
S(\emptyset)
&=
\widehat{\mathrm{Br}}{}^C_{\mathbf M,\mathbf L}(\mathbf w,q)
\widehat{\mathrm{Br}}{}^C_{\mathbf L^{\mathrm c},\mathbf L^{\mathrm c}}
(\mathbf w,q^{-1})=
\widehat{\mathrm{Br}}{}^C_{\mathbf M,\mathbf L}(\mathbf w,q),
\end{align*}
since the diagonal entries of $\widehat{\mathrm{Br}}{}^C$ are $1$.
If $\mathbf U=\{d\}$, then $\mathbf K(\mathbf U)=\mathbf M$, so
\begin{align*}
S(\{d\})
&=
\widehat{\mathrm{Br}}{}^C_{\mathbf M,\mathbf M}(\mathbf w,q)
q^d
\widehat{\mathrm{Br}}{}^C_{\mathbf L^{\mathrm c},\mathbf M^{\mathrm c}}
(\mathbf w,q^{-1})=
q^d
\widehat{\mathrm{Br}}{}^C_{\mathbf L^{\mathrm c},\mathbf M^{\mathrm c}}
(\mathbf w,q^{-1}).
\end{align*}
Using \eqref{eq:allcard-Bhat-product-C}, these two nontrivial matrix entries
are
\begin{align*}
\widehat{\mathrm{Br}}{}^C_{\mathbf M,\mathbf L}(\mathbf w,q)
={}&
\frac{1-q}{1-w_d^{-2}}
\prod_{\substack{l\in\mathbf L\\ l<d}}
\frac{1-qw_l^{-1}w_d}{1-w_l^{-1}w_d}
\prod_{\substack{a\in\mathbf M^{\mathrm c}\\ a<d}}
\frac{1-qw_d^{-1}w_a}{1-w_d^{-1}w_a},\\
\widehat{\mathrm{Br}}{}^C_{\mathbf L^{\mathrm c},\mathbf M^{\mathrm c}}
(\mathbf w,q^{-1})
={}&
\frac{1-q^{-1}}{1-w_d^{-2}}
\prod_{\substack{a\in\mathbf M^{\mathrm c}\\ a<d}}
\frac{1-q^{-1}w_a^{-1}w_d}{1-w_a^{-1}w_d}
\prod_{\substack{l\in\mathbf L\\ l<d}}
\frac{1-q^{-1}w_d^{-1}w_l}{1-w_d^{-1}w_l}.
\end{align*}
For every index smaller than $d$,
\begin{align*}
\frac{1-q^{-1}w_a^{-1}w_d}{1-w_a^{-1}w_d}
&=
q^{-1}
\frac{1-qw_d^{-1}w_a}{1-w_d^{-1}w_a},\ \ \
\frac{1-q^{-1}w_d^{-1}w_l}{1-w_d^{-1}w_l}
=
q^{-1}
\frac{1-qw_l^{-1}w_d}{1-w_l^{-1}w_d}.
\end{align*}
Moreover,
$\displaystyle
\frac{1-q^{-1}}{1-q}=-q^{-1}.
$
Since $\mathbf M=\mathbf L\sqcup\{d\}$, every index smaller
than $d$ belongs to exactly one of $\mathbf L$ and
$\mathbf M^{\mathrm c}$.
Thus the two products contribute $d-1$ factors of $q^{-1}$
in total. Therefore
\begin{equation*}
\widehat{\mathrm{Br}}{}^C_{\mathbf L^{\mathrm c},\mathbf M^{\mathrm c}}
(\mathbf w,q^{-1})
=
-q^{-d}
\widehat{\mathrm{Br}}{}^C_{\mathbf M,\mathbf L}(\mathbf w,q).
\end{equation*}
Consequently,
\[
S(\emptyset)+S(\{d\})
=
\widehat{\mathrm{Br}}{}^C_{\mathbf M,\mathbf L}(\mathbf w,q)
+q^d
\widehat{\mathrm{Br}}{}^C_{\mathbf L^{\mathrm c},\mathbf M^{\mathrm c}}
(\mathbf w,q^{-1})
=0.
\]

We now prove the general case. Fix $\mathbf U\subseteq\mathcal D$, where
\[
\mathcal D=\mathbf M\setminus\mathbf L=\{d_1<\cdots<d_m\},
\qquad \mathbf K(\mathbf U)=\mathbf L\sqcup\mathbf U.
\]
Put $x_j=w_{d_j}$, and let $\epsilon_j=1$ if $d_j\in\mathbf U$, and
$\epsilon_j=0$ otherwise. Recall that
\[
S(\mathbf U)=
\widehat{\mathrm{Br}}{}^C_{\mathbf M,\mathbf K(\mathbf U)}(\mathbf w,q)
q^{\sum_jd_j\epsilon_j}
\widehat{\mathrm{Br}}{}^C_{\mathbf L^{\mathrm c},\mathbf K(\mathbf U)^{\mathrm c}}
(\mathbf w,q^{-1}).
\]
We first expand both branching coefficients, organize their factors, and
then compare the resulting expression with $S(\emptyset)$.
All quotients below are taken in $\mathbb C(q,w_1,\ldots,w_N)$.

Since $\mathbf M\setminus\mathbf K(\mathbf U)=\mathcal D\setminus\mathbf U$,
applying \eqref{eq:allcard-Bhat-product-C} to the first coefficient gives
\begin{align}
&\widehat{\mathrm{Br}}{}^C_{\mathbf M,\mathbf K(\mathbf U)}(\mathbf w,q)
\notag\\
&=(1-q)^{m-|\mathbf U|}
\prod_{t\in\mathcal D\setminus\mathbf U}\frac{1}{1-w_t^{-2}}
\prod_{\substack{a<b\\a,b\in\mathcal D\setminus\mathbf U}}
\frac{1-qw_a^{-1}w_b^{-1}}{1-w_a^{-1}w_b^{-1}}
\notag\\
&\quad\times
\prod_{\substack{k\in\mathbf L\sqcup\mathbf U,\ t\in\mathcal D\setminus\mathbf U\\k<t}}
\frac{1-qw_k^{-1}w_t}{1-w_k^{-1}w_t}
\prod_{\substack{t\in\mathcal D\setminus\mathbf U,\ b\in\mathbf M^{\mathrm c}\\b<t}}
\frac{1-qw_t^{-1}w_b}{1-w_t^{-1}w_b}.
\label{eq:allcard-SU-first-expanded}
\end{align}
For the second coefficient, note that
\[
[N]=\mathbf M^{\mathrm c}\sqcup\mathbf L\sqcup\mathcal D,
\qquad
\mathbf K(\mathbf U)=\mathbf L\sqcup\mathbf U,
\]
where $\mathbf U\subset\mathcal D$. Taking complements in $[N]$ and using
$A^{\mathrm c}\setminus B^{\mathrm c}=B\setminus A$ gives
\begin{gather*}
\mathbf L^{\mathrm c}\setminus\mathbf K(\mathbf U)^{\mathrm c}
=\mathbf K(\mathbf U)\setminus\mathbf L=\mathbf U,
\qquad
\mathbf K(\mathbf U)^{\mathrm c}
=\mathbf M^{\mathrm c}\sqcup(\mathcal D\setminus\mathbf U),
\qquad
(\mathbf L^{\mathrm c})^{\mathrm c}=\mathbf L.
\end{gather*}
Thus \eqref{eq:allcard-Bhat-product-C}, with $q$ replaced by $q^{-1}$, gives
\begin{align}
&\widehat{\mathrm{Br}}{}^C_{\mathbf L^{\mathrm c},\mathbf K(\mathbf U)^{\mathrm c}}
(\mathbf w,q^{-1})\notag\\
&=(1-q^{-1})^{|\mathbf U|}
\prod_{t\in\mathbf U}\frac{1}{1-w_t^{-2}}
\prod_{\substack{a<b\\a,b\in\mathbf U}}
\frac{1-q^{-1}w_a^{-1}w_b^{-1}}{1-w_a^{-1}w_b^{-1}}
\notag\\
&\quad\times
\prod_{\substack{k\in\mathbf M^{\mathrm c}\sqcup(\mathcal D\setminus\mathbf U),\ t\in\mathbf U\\k<t}}
\frac{1-q^{-1}w_k^{-1}w_t}{1-w_k^{-1}w_t}
\prod_{\substack{t\in\mathbf U,\ b\in\mathbf L\\b<t}}
\frac{1-q^{-1}w_t^{-1}w_b}{1-w_t^{-1}w_b}.
\label{eq:allcard-SU-second-expanded}
\end{align}
In these formulas, the index in the last product is denoted by $b$ to
avoid confusion with the index $j$ in $d_j$.

The prefactors and the single-index products in
\eqref{eq:allcard-SU-first-expanded} and \eqref{eq:allcard-SU-second-expanded},
together with the explicit power of $q$ in $S(\mathbf U)$, multiply to
\[
q^{\sum_jd_j\epsilon_j}
(1-q)^{m-|\mathbf U|}(1-q^{-1})^{|\mathbf U|}
\prod_{j=1}^{m}\frac{1}{1-x_j^{-2}}.
\]
Here the two single-index products combine into the product over
$\mathcal D$, because $\mathcal D=(\mathcal D\setminus\mathbf U)\sqcup\mathbf U$.

We now organize the remaining two-index factors in
\eqref{eq:allcard-SU-first-expanded} and \eqref{eq:allcard-SU-second-expanded}.
First fix $d_j\in\mathcal D$ and collect the factors involving $d_j$ and
an index in $\mathbf L$ or $\mathbf M^{\mathrm c}$.
These are precisely the indices not in $\mathcal D$, since
$[N]=\mathbf L\sqcup\mathcal D\sqcup\mathbf M^{\mathrm c}$.

Suppose that $\epsilon_j=0$, so that $d_j\in\mathcal D\setminus\mathbf U$.
In the first mixed product in \eqref{eq:allcard-SU-first-expanded}, set
$t=d_j$ and restrict $k$ to $\mathbf L$.
In the last product in \eqref{eq:allcard-SU-first-expanded}, set $t=d_j$
and retain all $b\in\mathbf M^{\mathrm c}$ satisfying $b<d_j$.
These factors give
\begin{equation}
E_j^{(0)}:=
\prod_{\substack{k\in\mathbf L\\k<d_j}}
\frac{1-qw_k^{-1}x_j}{1-w_k^{-1}x_j}
\prod_{\substack{b\in\mathbf M^{\mathrm c}\\b<d_j}}
\frac{1-qx_j^{-1}w_b}{1-x_j^{-1}w_b}.
\label{eq:allcard-E0-def}
\end{equation}
The remaining choices $k\in\mathbf U$ in the first mixed product
(recall that $\mathbf K(\mathbf U)=\mathbf L\sqcup\mathbf U$) involve
two indices in $\mathcal D$ and will be considered below.
There is no factor involving $d_j$ and an index in $\mathbf L$ or
$\mathbf M^{\mathrm c}$ in \eqref{eq:allcard-SU-second-expanded}:
every two-index factor in \eqref{eq:allcard-SU-second-expanded} contains
an index in $\mathbf U$, whereas neither member of such a pair belongs
to $\mathbf U$.

Suppose instead that $\epsilon_j=1$, so that $d_j\in\mathbf U$.
In the first mixed product in \eqref{eq:allcard-SU-second-expanded}, set
$t=d_j$ and restrict $k$ to $\mathbf M^{\mathrm c}$.
In the last product in \eqref{eq:allcard-SU-second-expanded}, set $t=d_j$
and retain all $b\in\mathbf L$ satisfying $b<d_j$.
These factors give
\begin{equation}
E_j^{(1)}:=
\prod_{\substack{k\in\mathbf M^{\mathrm c}\\k<d_j}}
\frac{1-q^{-1}w_k^{-1}x_j}{1-w_k^{-1}x_j}
\prod_{\substack{b\in\mathbf L\\b<d_j}}
\frac{1-q^{-1}x_j^{-1}w_b}{1-x_j^{-1}w_b}.
\label{eq:allcard-E1-def}
\end{equation}
The remaining choices $k\in\mathcal D\setminus\mathbf U$ in the first
mixed product (recall that $\mathbf K(\mathbf U)^{\mathrm c}
=\mathbf M^{\mathrm c}\sqcup(\mathcal D\setminus\mathbf U)$) likewise
involve two indices in $\mathcal D$.
There is no factor involving $d_j$ and an index in $\mathbf L$ or
$\mathbf M^{\mathrm c}$ in \eqref{eq:allcard-SU-first-expanded}, because
every two-index factor in \eqref{eq:allcard-SU-first-expanded} contains
an index in $\mathcal D\setminus\mathbf U$.

We next collect the factors involving two indices in $\mathcal D$,
including the choices left aside above.
Fix $d_i<d_j$, and write $x=x_i$, $y=x_j$.

If $(\epsilon_i,\epsilon_j)=(0,0)$, both indices belong to
$\mathcal D\setminus\mathbf U$. In the product
\[
\prod_{\substack{a<b\\a,b\in\mathcal D\setminus\mathbf U}}
\frac{1-qw_a^{-1}w_b^{-1}}{1-w_a^{-1}w_b^{-1}}
\]
in \eqref{eq:allcard-SU-first-expanded}, take $a=d_i$ and $b=d_j$.
The resulting factor is
\[
\frac{1-qw_{d_i}^{-1}w_{d_j}^{-1}}{1-w_{d_i}^{-1}w_{d_j}^{-1}}
=\frac{xy-q}{xy-1}=:\mathsf P_{00}(x,y).
\]
No mixed product in \eqref{eq:allcard-SU-first-expanded} contains this
pair, since both indices lie in $\mathcal D\setminus\mathbf U$.
Nor does \eqref{eq:allcard-SU-second-expanded} contain a factor involving
this pair, since neither index belongs to $\mathbf U$.

If $(\epsilon_i,\epsilon_j)=(1,1)$, both indices belong to $\mathbf U$.
In the product
\[
\prod_{\substack{a<b\\a,b\in\mathbf U}}
\frac{1-q^{-1}w_a^{-1}w_b^{-1}}{1-w_a^{-1}w_b^{-1}}
\]
in \eqref{eq:allcard-SU-second-expanded}, take $a=d_i$ and $b=d_j$.
This gives
\[
\frac{1-q^{-1}w_{d_i}^{-1}w_{d_j}^{-1}}{1-w_{d_i}^{-1}w_{d_j}^{-1}}
=\frac{xy-q^{-1}}{xy-1}=:\mathsf P_{11}(x,y).
\]
No mixed product in \eqref{eq:allcard-SU-second-expanded} contains this
pair, and \eqref{eq:allcard-SU-first-expanded} contributes no factor
involving the pair because neither index belongs to
$\mathcal D\setminus\mathbf U$.

If $(\epsilon_i,\epsilon_j)=(1,0)$, take $k=d_i\in\mathbf U$ and
$t=d_j\in\mathcal D\setminus\mathbf U$ in the first mixed product in
\eqref{eq:allcard-SU-first-expanded}.
The condition $k<t$ holds, and the factor is
\[
\frac{1-qw_{d_i}^{-1}w_{d_j}}{1-w_{d_i}^{-1}w_{d_j}}
=\frac{x-qy}{x-y}=:\mathsf P_{10}(x,y).
\]
This is one of the choices $k\in\mathbf U$ left aside when defining
$E_j^{(0)}$. In the first mixed product in
\eqref{eq:allcard-SU-second-expanded}, the pair would require $k=d_j$
and $t=d_i$, which fails $k<t$.
The pair occurs in none of the other products: the two indices are in
different parts of $\mathcal D$, and neither belongs to $\mathbf L$ or
$\mathbf M^{\mathrm c}$.

Finally, if $(\epsilon_i,\epsilon_j)=(0,1)$, take
$k=d_i\in\mathcal D\setminus\mathbf U$ and $t=d_j\in\mathbf U$ in the
first mixed product in \eqref{eq:allcard-SU-second-expanded}.
This gives
\[
\frac{1-q^{-1}w_{d_i}^{-1}w_{d_j}}{1-w_{d_i}^{-1}w_{d_j}}
=\frac{x-q^{-1}y}{x-y}=:\mathsf P_{01}(x,y).
\]
This is one of the choices $k\in\mathcal D\setminus\mathbf U$ left aside
when defining $E_j^{(1)}$.
In the first mixed product in \eqref{eq:allcard-SU-first-expanded},
the pair would require $k=d_j$ and $t=d_i$, again contrary to $k<t$.
As in the preceding case, none of the other products contains this pair.

Every two-index factor in \eqref{eq:allcard-SU-first-expanded} and
\eqref{eq:allcard-SU-second-expanded} has now been included either in
some $E_j^{(\epsilon_j)}$ or in some
$\mathsf P_{\epsilon_i\epsilon_j}(x_i,x_j)$.
By the definition of $S(\mathbf U)$ in
\eqref{eq:allcard-SU-def}, combining the factors computed above gives
\begin{align*}
S(\mathbf U)
={}&q^{\sum_jd_j\epsilon_j}
(1-q)^{m-|\mathbf U|}(1-q^{-1})^{|\mathbf U|}
\prod_{j=1}^{m}\frac{1}{1-x_j^{-2}}\prod_{j=1}^{m}E_j^{(\epsilon_j)}
\prod_{i<j}\mathsf P_{\epsilon_i\epsilon_j}(x_i,x_j).
\end{align*}
For $\mathbf U=\emptyset$, all $\epsilon_j$ are zero, so
\[
S(\emptyset)=(1-q)^m
\prod_{j=1}^{m}\frac{1}{1-x_j^{-2}}
\prod_{j=1}^{m}E_j^{(0)}\prod_{i<j}\mathsf P_{00}(x_i,x_j).
\]
Dividing these expressions, we obtain
\begin{equation}
\frac{S(\mathbf U)}{S(\emptyset)}
=(-1)^{|\mathbf U|}q^{\sum_j(d_j-1)\epsilon_j}
\prod_{j=1}^{m}\frac{E_j^{(\epsilon_j)}}{E_j^{(0)}}
\prod_{i<j}\frac{\mathsf P_{\epsilon_i\epsilon_j}(x_i,x_j)}
{\mathsf P_{00}(x_i,x_j)}.
\label{eq:allcard-SU-factor-ratio}
\end{equation}

We simplify the two products in this quotient separately.
We compare $E_j^{(1)}$ in \eqref{eq:allcard-E1-def} with
$E_j^{(0)}$ in \eqref{eq:allcard-E0-def} in the same way as in the
one-element case above. The product over $\mathbf L$ in
\eqref{eq:allcard-E1-def} is compared with the product over $\mathbf L$
in \eqref{eq:allcard-E0-def}, and likewise for $\mathbf M^{\mathrm c}$.
Each corresponding pair of factors again gives a ratio of $q^{-1}$.
For an element $a\in\mathbf L$ with $a<d_j$, the corresponding factors satisfy
\[
\frac{1-q^{-1}x_j^{-1}w_a}{1-x_j^{-1}w_a}
=q^{-1}\frac{1-qw_a^{-1}x_j}{1-w_a^{-1}x_j}.
\]
For an element $a\in\mathbf M^{\mathrm c}$ with $a<d_j$, they satisfy
\[
\frac{1-q^{-1}w_a^{-1}x_j}{1-w_a^{-1}x_j}
=q^{-1}\frac{1-qx_j^{-1}w_a}{1-x_j^{-1}w_a}.
\]
As in the one-element case, we count the indices in $\mathbf L$ or
$\mathbf M^{\mathrm c}$ smaller than $d_j$.
Of the $d_j-1$ smaller indices, $j-1$ now belong to $\mathcal D$,
namely $d_1,\ldots,d_{j-1}$, leaving $d_j-j$ factors of $q^{-1}$.
Hence
\[
\frac{E_j^{(\epsilon_j)}}{E_j^{(0)}}=q^{-(d_j-j)\epsilon_j}.
\]

From these expressions and the definitions \eqref{eq:allcard-hweights}
of $H_{ab}$, it is easy to check that
\[
\frac{\mathsf P_{\epsilon_i\epsilon_j}(x_i,x_j)}{\mathsf P_{00}(x_i,x_j)}
=q^{-\epsilon_j}
\frac{H_{\epsilon_i,\epsilon_j}(x_i,x_j)}{H_{00}(x_i,x_j)}.
\]
For each fixed $j$, there are $j-1$ choices of $i<j$.
Multiplying over all pairs therefore gives
\[
\prod_{i<j}\frac{\mathsf P_{\epsilon_i\epsilon_j}(x_i,x_j)}{\mathsf P_{00}(x_i,x_j)}
=q^{-\sum_j(j-1)\epsilon_j}
\prod_{i<j}\frac{H_{\epsilon_i,\epsilon_j}(x_i,x_j)}{H_{00}(x_i,x_j)}.
\]
Substituting the two product ratios into
\eqref{eq:allcard-SU-factor-ratio}, the total exponent of $q$ is
\[
\sum_j(d_j-1)\epsilon_j-\sum_j(d_j-j)\epsilon_j
-\sum_j(j-1)\epsilon_j=0.
\]
We conclude that
\[
\frac{S(\mathbf U)}{S(\emptyset)}
=(-1)^{|\mathbf U|}
\prod_{i<j}\frac{H_{\epsilon_i,\epsilon_j}(x_i,x_j)}{H_{00}(x_i,x_j)}.
\]
The denominator is independent of $\mathbf U$, and the correspondence
between $\mathbf U\subseteq\mathcal D$ and $\epsilon\in\{0,1\}^m$ is a
bijection, with $|\mathbf U|=|\epsilon|$. Therefore,
\[
\sum_{\mathbf U\subseteq\mathcal D}S(\mathbf U)
=\frac{S(\emptyset)}{\prod_{i<j}H_{00}(x_i,x_j)}
\sum_{\epsilon\in\{0,1\}^m}(-1)^{|\epsilon|}
\prod_{i<j}H_{\epsilon_i,\epsilon_j}(x_i,x_j).
\]
The final sum vanishes by Lemma~\ref{lem:allcard-universal-cancellation}, since
$(-1)^{|\epsilon|}=(-1)^m(-1)^{m-|\epsilon|}$.
Thus the matrix product \eqref{eq:allcard-boundary-product} has entry
zero when $\mathbf L\subsetneq\mathbf M$.
If $\mathbf L\not\subseteq\mathbf M$, every summand vanishes by
triangularity, while the diagonal entries are $1$, as shown above.
This proves \eqref{eq:allcard-Bhat-inverse} in type $C$.

\end{proof}

\begin{proof}[Proof of Theorem~\ref{thm:BC-all-cardinality-biorthogonality}]
From \eqref{eq:allcard-P-factorization},
\[
\mathsf A^X(\mathbf w,q)^{-1}
=
\widehat\Phi(\mathbf w,q)^{-1}
\widehat{\mathrm{Br}}{}^X(\mathbf w,q)^{-1}.
\]
The first equality below is the $(\mathbf I,\mathbf L)$ entry of the
matrix product. The second follows from Lemma~\ref{lem:allcard-Phihat-inverse}
and Proposition~\ref{prop:allcard-Bhat-inverse}; the last equality uses
\eqref{eq:allcard-P-factorization}:
\begin{align}
\bigl(\mathsf A^X(\mathbf w,q)^{-1}\bigr)_{\mathbf I,\mathbf L}
&=
\sum_{\mathbf K\subset[N]}
\bigl(\widehat\Phi(\mathbf w,q)^{-1}\bigr)_{\mathbf I,\mathbf K}
\bigl(\widehat{\mathrm{Br}}{}^X(\mathbf w,q)^{-1}\bigr)_{\mathbf K,\mathbf L}
\notag\\
&=
q^{|\lambda(\mathbf I)|-|\lambda(\mathbf L)|}
\sum_{\mathbf K\subset[N]}
\widehat{\mathrm{Br}}{}^X_{\mathbf L^{\mathrm c},\mathbf K^{\mathrm c}}
(\mathbf w,q^{-1})
\widehat\Phi_{\mathbf K^{\mathrm c},\mathbf I^{\mathrm c}}
(\mathbf w,q^{-1})
\notag\\
&=
q^{|\lambda(\mathbf I)|-|\lambda(\mathbf L)|}
\left(\widehat{\mathrm{Br}}{}^X(\mathbf w,q^{-1})
\widehat\Phi(\mathbf w,q^{-1})\right)_{\mathbf L^{\mathrm c},\mathbf I^{\mathrm c}}
\notag\\
&=
q^{|\lambda(\mathbf I)|-|\lambda(\mathbf L)|}
\mathsf A^X_{\mathbf L^{\mathrm c},\mathbf I^{\mathrm c}}
(\mathbf w,q^{-1}).
\end{align}
The identity
$\mathsf A^X(\mathbf w,q)\mathsf A^X(\mathbf w,q)^{-1}=\mathrm{Id}$ is exactly
\eqref{eq:BC-all-cardinality-biorthogonality-matrix}; expanding the definition
\eqref{eq:allcard-A-def} gives
\eqref{eq:BC-all-cardinality-biorthogonality}.
\end{proof}

\medskip
\noindent\textbf{Example.}
Take $N=1$ and $X=C$, and write $\mathbf w=(w_1)$.
By \eqref{eq:allcard-Delta-factor}, the diagonal values are
\begin{align*}
\Delta^C_{\emptyset}(\mathbf w,q)
&=\widetilde{WC}_{1,1,0}
(1\,|w_1\,|\emptyset)
=1-qw_1^2,\\
\Delta^C_{\{1\}}(\mathbf w,q)
&=\widetilde{WC}_{1,1,1}
(w_1\,|w_1\,|\{1\})
=1-w_1^{-2}.
\end{align*}
The remaining value needed below is
\[
\widetilde{WC}_{1,1,0}
(w_1\,|w_1\,|\emptyset)=1-q.
\]

Take $\mathbf M=\{1\}$ and $\mathbf L=\emptyset$.
Then
$
\mathbf L^{\mathrm c}=\{1\},
\ |\lambda(\mathbf L)|=0.
$
The two summation subsets in
\eqref{eq:BC-all-cardinality-biorthogonality} are
$\mathbf I=\emptyset$ and $\mathbf I=\{1\}$,
with complements $\{1\}$ and $\emptyset$, respectively,
and the corresponding powers of $q$ are $1$ and $q$.

The contribution from $\mathbf I=\emptyset$ is
\begin{align*}
T_{\emptyset}
&=
\frac{
\widetilde{WC}_{1,1,0}
(w_1\,|w_1\,|\emptyset)
}{
\Delta^C_{\{1\}}(\mathbf w,q)
}
\frac{
\left.
\widetilde{WC}_{1,1,1}
(w_1\,|w_1\,|\{1\})
\right|_{q\mapsto q^{-1}}
}{
\Delta^C_{\{1\}}(\mathbf w,q^{-1})
}=\frac{1-q}{1-w_1^{-2}}.
\end{align*}
The contribution from $\mathbf I=\{1\}$ is
\begin{align*}
T_{\{1\}}
&=
q\,
\frac{
\widetilde{WC}_{1,1,1}
(w_1\,|w_1\,|\{1\})
}{
\Delta^C_{\{1\}}(\mathbf w,q)
}
\frac{
\left.
\widetilde{WC}_{1,1,0}
(w_1\,|w_1\,|\emptyset)
\right|_{q\mapsto q^{-1}}
}{
\Delta^C_{\{1\}}(\mathbf w,q^{-1})
}=
q\,\frac{1-q^{-1}}{1-w_1^{-2}}
=-\frac{1-q}{1-w_1^{-2}}.
\end{align*}
Consequently,
\[
T_{\emptyset}+T_{\{1\}}
=0=\delta_{\{1\},\emptyset}.
\]
The two summands are generically nonzero and belong to
different cardinalities.

\medskip
\noindent\textbf{Example.}
Take $N=2$ and $X=C$.  By \eqref{eq:allcard-Delta-factor},
Theorem~\ref{RTVtheorem} $\mathrm{(ii)}$, and
\eqref{general-diagonal-factor},
\begin{align}
\Delta^C_{\{1\}}(\mathbf w,q)
&=(1-w_1^{-2})(1-qw_2^2)
\left(1-q\frac{w_2}{w_1}\right),
 \notag\\
\Delta^C_{\{2\}}(\mathbf w,q)
&=(1-w_2^{-2})(1-qw_1^2)
\left(1-\frac{w_1}{w_2}\right),
 \notag\\
\Delta^C_{\{2\}}(\mathbf w,q^{-1})
&=(1-w_2^{-2})(1-q^{-1}w_1^2)
\left(1-\frac{w_1}{w_2}\right).
 \notag
\end{align}
In terms of the normalized type $C$ functions, these are
\begin{align*}
\Delta^C_{\{1\}}(\mathbf w,q)
&=\widetilde{WC}_{2,2,1}
\bigl(w_1,1\,|\mathbf w\,|\{1\}\bigr),\\
\Delta^C_{\{2\}}(\mathbf w,q)
&=\widetilde{WC}_{2,2,1}
\bigl(w_2,1\,|\mathbf w\,|\{2\}\bigr),\\
\Delta^C_{\{2\}}(\mathbf w,q^{-1})
&=\left.\widetilde{WC}_{2,2,1}
\bigl(w_2,1\,|\mathbf w\,|\{2\}\bigr)
\right|_{q\mapsto q^{-1}}.
\end{align*}
The remaining values needed below are
\begin{align}
\widetilde{WC}_{2,2,1}
\bigl(w_2,1\,|\mathbf w\,|\{1\}\bigr)
&=(1-w_2^{-2})(1-qw_1^2)(1-q),
 \notag\\
\left.\widetilde{WC}_{2,2,1}
\bigl(w_2,1\,|\mathbf w\,|\{1\}\bigr)
\right|_{q\mapsto q^{-1}}
&=(1-w_2^{-2})(1-q^{-1}w_1^2)(1-q^{-1}).
 \notag
\end{align}
The support and lower-cardinality vanishing relations give
\begin{align}
\widetilde{WC}_{2,2,1}
\bigl(w_1,1\,|\mathbf w\,|\{2\}\bigr)&=0,
\label{eq:allcard-N2-WC12zero}\\
\widetilde{WC}_{2,2,2}
\bigl(w_i,1\,|\mathbf w\,|\{1,2\}\bigr)&=0
\qquad(i=1,2).
\label{eq:allcard-N2-WC2zero}
\end{align}

First take $\mathbf M=\mathbf L=\{1\}$.  Then
$\mathbf L^{\mathrm c}=\{2\}$ and
$|\lambda(\mathbf L)|=1$.  The four summation subsets are
$
\emptyset,\ \{1\},\ \{2\},\ \{1,2\},
$
with complements
$
\{1,2\},\ \{2\},\ \{1\},\ \emptyset
$
and $q$-powers $q^{-1},1,q,q^2$, respectively.  The
$\mathbf I=\emptyset$ term is zero because its second factor is
\[
\left.\widetilde{WC}_{2,2,2}
\bigl(w_2,1\,|\mathbf w\,|\{1,2\}\bigr)
\right|_{q\mapsto q^{-1}}=0.
\]
For $\mathbf I=\{1\}$, both quotients in
\eqref{eq:BC-all-cardinality-biorthogonality} are equal to $1$, so this
term is $1$.  The $\mathbf I=\{2\}$ term is zero by
\eqref{eq:allcard-N2-WC12zero}, and the $\mathbf I=\{1,2\}$ term is zero by
\eqref{eq:allcard-N2-WC2zero}.  Hence
\[
0+1+0+0=1=\delta_{\{1\},\{1\}}.
\]

Next, keep $\mathbf L=\{1\}$ and take $\mathbf M=\{2\}$.  Again the $\mathbf I=\emptyset$ and
$\mathbf I=\{1,2\}$ terms vanish.  The $\mathbf I=\{1\}$ term is
\begin{align}
T_{\{1\}}
&=
\frac{
\widetilde{WC}_{2,2,1}
\bigl(w_2,1\,|\mathbf w\,|\{1\}\bigr)
}{\Delta^C_{\{2\}}(\mathbf w,q)}
\frac{\Delta^C_{\{2\}}(\mathbf w,q^{-1})}
{\Delta^C_{\{2\}}(\mathbf w,q^{-1})}
=\frac{1-q}{1-w_1/w_2},
 \notag
\end{align}
whereas the $\mathbf I=\{2\}$ term is
\begin{align}
T_{\{2\}}
&=q
\frac{\Delta^C_{\{2\}}(\mathbf w,q)}
{\Delta^C_{\{2\}}(\mathbf w,q)}
\frac{
\left.\widetilde{WC}_{2,2,1}
\bigl(w_2,1\,|\mathbf w\,|\{1\}\bigr)
\right|_{q\mapsto q^{-1}}
}{\Delta^C_{\{2\}}(\mathbf w,q^{-1})}
=q\frac{1-q^{-1}}{1-w_1/w_2}
=-\frac{1-q}{1-w_1/w_2}.
 \notag
\end{align}
Consequently
\[
T_{\{1\}}+T_{\{2\}}=0
=\delta_{\{2\},\{1\}}.
\]

For the following inversion formula, we return to arbitrary
$n$ with $r\le n$. The specialization $\mathbf w_{\mathbf M}$
is again padded with $n-|\mathbf M|$ entries equal to $1$.

\begin{corollary}
\label{cor:BC-branching-formula-inversion}
Let $X\in\{B,C\}$ and $0\le r\le\min(n,N)$.
Let $\mathbf K\subset[N]$ be an $r$-element subset, and let
$\mathbf M\subset[N]$ have cardinality
$0\le s=|\mathbf M|\le\min(n,N)$.
Then
\begin{align}
&\sum_{\substack{\mathbf I\subset[N]\\|\mathbf I|=r}}
\widetilde{WX}_{N,n,r}
(\mathbf w_{\mathbf M}\,|\mathbf w\,|\mathbf I)
\bigl(\Phi^{(r)}(\mathbf w,q)^{-1}\bigr)_{\mathbf I,\mathbf K}
=
\begin{cases}
\operatorname{Br}^X_{\mathbf M,\mathbf K}(\mathbf w,q),
&s\ge r\text{ and }\mathbf K\subset\mathbf M,\\[2mm]
0,&\text{otherwise}.
\end{cases}
\label{eq:allcard-branching-inversion}
\end{align}
\end{corollary}

\begin{proof}
If $s<r$, the normalized type $C$ and type $B$ functions vanish by
\eqref{eq:typeC-lower-cardinality-vanishing} and
\eqref{eq:typeB-lower-cardinality-vanishing}.  If $s\ge r$, substitute
\eqref{eq:allcard-branching-factorization}, interchange the two
finite sums, and use
$\Phi^{(r)}(\mathbf w,q)\Phi^{(r)}(\mathbf w,q)^{-1}=\mathrm{Id}$.
\end{proof}

\section{Discussion}
In this paper, we introduced type $B$ and type $C$ trigonometric weight functions
and derived their algebraic properties using integrability.
The weight functions are $q$-deformations of the $GB$- and $GQ$-functions, respectively,
and their normalized versions satisfy properties
that can be viewed as type $B$ and type $C$ analogues
of the properties for type $A$ in \cite{RTV}.
The weight functions introduced in this paper are closely related to
geometric objects such as motivic Chern classes and stable bases
\cite{MO,Okounkov,RTV,RTV2,SZZ,FRW,AMSS,KonckiWeber,MNS}.
In some cases they agree with the corresponding geometric classes,
while in other cases a correction is required.
For example, comparison with the twisted motivic Chern classes
for the Lagrangian Grassmannian $LG(2,4)$ computed in
\cite[Chapter~6]{KonckiThesis} shows that our normalized
type $C$ weight functions agree with the corresponding
fixed-point restrictions in fifteen of the sixteen entries.
We plan to discuss these relations in detail in a subsequent paper.

\section*{Acknowledgements}
We thank Takeshi Ikeda for many helpful discussions
and suggestions for this work.
This work was partially supported by Grant-in-Aid for Scientific Research (C) 26K06839.

\section*{Appendix A: Trigonometric weight functions for the type $A$ Grassmannian}
\setcounter{equation}{0}
\renewcommand{\theequation}{A.\arabic{equation}}
\renewcommand{\theHequation}{A.\arabic{equation}}
\setcounter{thm}{0}
\renewcommand{\thethm}{A.\arabic{thm}}
\renewcommand{\theHthm}{A.\arabic{thm}}

In this appendix, we recall the trigonometric weight functions of
Rim\'anyi--Tarasov--Varchenko and specialize their results to the
Grassmannian $\operatorname{Gr}(r,N)$.
The notation used in this appendix is chosen so as not to conflict with
the notation used in the main body of the paper.
In particular, the fixed-point label used by
Rim\'anyi--Tarasov--Varchenko will be denoted by
$\mathcal I=(\mathbf{I},\mathbf{I}^{\mathrm{c}})$
rather than simply by $I$. Similarly, another fixed point will be
denoted by
$\mathcal M=(\mathbf{M},\mathbf{M}^{\mathrm{c}})$.
Here,
\[
\mathbf{I},\mathbf{M}\subset\{1,\dots,N\},
\qquad
|\mathbf{I}|=|\mathbf{M}|=r,
\]
and
\[
\mathbf{I}^{\mathrm{c}} =
\{1,\dots,N\}\setminus\mathbf{I},
\qquad
\mathbf{M}^{\mathrm{c}} =
\{1,\dots,N\}\setminus\mathbf{M}.
\]
The notation $N$ in the present paper corresponds to the number of
equivariant variables in the Grassmannian. In the general notation of
Rim\'anyi--Tarasov--Varchenko, the relevant specialization is
\[
N_{\mathrm{RTV}}=2,
\qquad
n_{\mathrm{RTV}}=N,
\qquad
\lambda=(r,N-r).
\]
Let $\sigma\in S_N$ be an arbitrary permutation and
write
$
\sigma^{-1}(\mathbf{I})=\{j_1<j_2<\cdots<j_r\}.
$
Let
$
\mathbf{t}=(t_1,\dots,t_r),
\
\mathbf{z}=(z_1,\dots,z_N),
$
and let $h$ be a complex variable.
The type $A$ Grassmannian
trigonometric weight function is
\begin{align}
&W_{\sigma,\mathcal I}^{\mathrm{RTV}}
(\mathbf{t},\mathbf{z},h)
:=
W_{\sigma^{-1}(\mathcal I)}^{\mathrm{RTV}}
\bigl(\mathbf{t},z_{\sigma(1)},\dots,z_{\sigma(N)},h\bigr)
\notag \\
&=
(1-h)^r
\operatorname{Sym}_{\boldsymbol t}
\Biggl[
\prod_{a=1}^r
\left\{
\prod_{c<j_a}
\left(1-h\frac{z_{\sigma(c)}}{t_a}\right)
\prod_{c>j_a}
\left(1-\frac{z_{\sigma(c)}}{t_a}\right)
\right\} \prod_{1\le a<b\le r}
\frac{1-h t_b/t_a}{1-t_b/t_a}
\Biggr].
\label{eq:general-sigma-W}
\end{align}
Here,
\[
\operatorname{Sym}_{\boldsymbol t}f(t_1,\dots,t_r)
:=
\sum_{\tau\in S_r}
f(t_{\tau(1)},\dots,t_{\tau(r)}).
\]
Define the normalization factor
\begin{equation}
E(\mathbf{t},h)
:=
\prod_{a,b=1}^r
\left(1-h t_a^{-1} t_b \right).
\label{eq:E-normalization}
\end{equation}
The normalized weight
function is
\begin{equation}
\widetilde W_{\sigma,\mathcal I}^{\mathrm{RTV}}
(\mathbf{t},\mathbf{z},h)
:=
\frac{
W_{\sigma,\mathcal I}^{\mathrm{RTV}}
(\mathbf{ t},\mathbf{ z},h)
}{
E(\mathbf{ t},h)
}.
\label{eq:general-Wtilde-def}
\end{equation}
For
\[
\mathcal M=(\mathbf{M},\mathbf{M}^{\mathrm{c}}),
\qquad
\mathbf{M}=\{m_1<\cdots<m_r\},
\]
the fixed-point specialization is obtained by setting
\[
\mathbf{t}=\mathbf{z}_{\mathbf{M}}
:=
(z_{m_1},\dots,z_{m_r}),
\]
and we write
$
W_{\sigma,\mathcal I}^{\mathrm{RTV}}
(\mathbf{ z}_{\mathbf{M}},\mathbf{ z},h)
$
and
$
\widetilde W_{\sigma,\mathcal I}^{\mathrm{RTV}}
(\mathbf{ z}_{\mathbf{M}},\mathbf{ z},h).
$
Let
\[
\sigma_0(a)=N+1-a
\qquad
(a=1,\dots,N)
\]
be the longest permutation.
The type $A$ function $W_{N,r}(\mathbf{ u}|\mathbf{ w}| \mathbf{I})$
used in this paper is
$W_{\sigma,\mathcal I}^{\mathrm{RTV}}
(\mathbf{ u},\mathbf{ w},q)$
corresponding to $\sigma_0$:
\begin{equation}
W_{N,r}(\mathbf{ u}|\mathbf{ w}| \mathbf{I}) =
W_{\sigma_0,\mathcal I}^{\mathrm{RTV}}
(\mathbf{ u},\mathbf{ w},q),
\qquad
\mathcal I=(\mathbf{I},\mathbf{I}^{\mathrm{c}}),
\label{WNr-RTV}
\end{equation}
hence
\begin{align}
\frac{
W_{N,r}(\mathbf{ u}|\mathbf{ w}| \mathbf{I})
}{
\displaystyle
\prod_{j,k=1}^r
\left(1-q u_j^{-1} u_k \right)
} &=
\widetilde W_{\sigma_0,\mathcal I}^{\mathrm{RTV}}
(\mathbf{ u},\mathbf{ w},q),
\label{normalized-WNr-RTV} \\
\frac{
W_{N,r}(\mathbf{ w}_{\mathbf{M}} |\mathbf{ w}| \mathbf{I})
}{
\displaystyle
\prod_{j,k=1}^r
\left(1-qw_{m_j}^{-1}w_{m_k}\right)
} &=
\widetilde W_{\sigma_0,\mathcal I}^{\mathrm{RTV}}
(\boldsymbol w_{\mathbf{M}},\mathbf{ w},q).
 \notag
\end{align}
For a permutation $\sigma$, write
\[
\sigma^{-1}(\mathbf{I})=\{i_1'<\cdots<i_r'\},
\qquad
\sigma^{-1}(\mathbf{M})=\{m_1'<\cdots<m_r'\}.
\]
The order in \cite{RTV} is defined as
$
\mathcal M\le_\sigma\mathcal I
$
if
$
m_a'\le i_a',
\
a=1,\dots,r.
$
Note that for
\[
\mathbf{I}=\{i_1<\cdots<i_r\},
\qquad
\mathbf{M}=\{m_1<\cdots<m_r\}
\]
and $\sigma=\sigma_0$,
\[
\mathcal M\le_{\sigma_0}\mathcal I
\quad\Longleftrightarrow\quad
m_a\ge i_a
\quad(a=1,\dots,r),
\]
and the order is equivalent to
$
\mathbf{M} \ge \mathbf{I}
$
used in the main text.
Define
\begin{align}
e_{\sigma,\mathcal M,-}^{\mathrm{hor}}
&:=
\prod_{\substack{
\sigma(a)\in \mathbf{M}, \sigma(b)\in \mathbf{M}^{\mathrm{c}} \\
b>a
}}
\left(1-\frac{z_{\sigma(a)}}{z_{\sigma(b)}}\right),
\label{eq:general-ehor} \\
e_{\sigma,\mathcal M,-}^{\mathrm{vert}}
&:=
\prod_{\substack{
\sigma(a)\in \mathbf{M}, \sigma(b)\in \mathbf{M}^{\mathrm{c}} \\
b<a
}}
\left(1-h\frac{z_{\sigma(b)}}{z_{\sigma(a)}}\right),
\label{eq:general-evert}
\end{align}
and
\[
e_{\sigma,\mathcal M}
:=
e_{\sigma,\mathcal M,-}^{\mathrm{hor}}
e_{\sigma,\mathcal M,-}^{\mathrm{vert}}.
\]
The polarization factor is
\begin{equation}
P_{\sigma,\mathcal M}
:=
\prod_{\substack{
\sigma(a)\in \mathbf{M}, \sigma(b)\in \mathbf{M}^{\mathrm{c}} \\
b>a
}}
\left(-\frac{z_{\sigma(b)}}{z_{\sigma(a)}}\right).
\label{eq:general-polarization}
\end{equation}
The product of $P_{\sigma,\mathcal M}$ and $e_{\sigma,\mathcal M}$
can be written as
\begin{align}
P_{\sigma,\mathcal M}e_{\sigma,\mathcal M}
&=
\prod_{\substack{
\sigma(a)\in \mathbf{M}, \sigma(b)\in \mathbf{M}^{\mathrm{c}} \\
b<a
}}
\left(1-h\frac{z_{\sigma(b)}}{z_{\sigma(a)}}\right)
\prod_{\substack{
\sigma(a)\in \mathbf{M}, \sigma(b)\in \mathbf{M}^{\mathrm{c}} \\
b>a
}}
\left(1-\frac{z_{\sigma(b)}}{z_{\sigma(a)}}\right).
\label{general-diagonal-factor}
\end{align}

For a Laurent polynomial
$\displaystyle f(\mathbf{z}) = \sum_{\nu \in \mathbb{Z}^N} c_{\nu}(h) z_1^{\nu_1} \cdots z_N^{\nu_N} $
(where $c_{\nu}(h) \in \mathbb{C}(h)$), we define its Newton polytope with respect to the variables $z_1, \dots, z_N$ by
\[ \mathcal{N}(f) = \operatorname{Conv}\{ \nu \in \mathbb{Z}^N \mid c_{\nu}(h) \neq 0 \} \subset \mathbb{R}^N. \]
We set $\mathcal N(0)=\emptyset$.
For $\mathcal{M} = (\mathbf{M}, \mathbf{M}^{\mathrm{c}})$, define the linear functional
$ \varphi_{\mathcal{M}} : \mathbb{R}^N \longrightarrow \mathbb{R} $
by
\[
\varphi_{\mathcal{M}}(\nu)
= \sum_{j \in \mathbf{M}} \nu_j + 2 \sum_{j \in \mathbf{M}^{\mathrm{c}}} \nu_j,
\qquad
\nu=(\nu_1,\dots,\nu_N)\in\mathbb{R}^N.
\]
For $f\neq0$, the image $\varphi_{\mathcal M}(\mathcal N(f))$ is a closed interval. We call it the \emph{degree interval} of $f$ with respect to $\varphi_{\mathcal M}$, and call its lower and upper endpoints the \emph{lower Newton degree} and \emph{upper Newton degree} of $f$, respectively.

A Laurent polynomial $f$ is called \emph{$\mathcal M$-small} if
\begin{equation}
\varphi_{\mathcal M}\bigl(\mathcal{N}(f)\bigr)
\subset
[0,r(N-r)-1].
\label{M-small}
\end{equation}
Note
\begin{equation*}
\varphi_{\mathcal M}
\left(
\mathcal{N}\left(P_{\sigma,\mathcal M}e_{\sigma,\mathcal M}\right)
\right) =
[0,r(N-r)].
\end{equation*}

\begin{thm}\cite[Theorems 3.1 and 6.7, Grassmannian specialization]{RTV}
\label{RTVtheorem}
Let
\[
T=(\mathbb C^\times)^N\times\mathbb C^\times
\]
act on $T^*\operatorname{Gr}(r,N)$, where the last factor scales the cotangent
fibers with character $h$.
Let
\[
\Gamma_1=\{\gamma_{1,1},\dots,\gamma_{1,r}\},
\qquad
\Gamma_2=\{\gamma_{2,1},\dots,\gamma_{2,N-r}\}.
\]
The equivariant $K$-theory ring is
\begin{align}
K_T\bigl(T^*\operatorname{Gr}(r,N)\bigr)
\simeq
\frac{
\mathbb C[\Gamma_1^{\pm1},\Gamma_2^{\pm1}]^{S_r\times S_{N-r}}
\otimes
\mathbb C[z_1^{\pm1},\dots,z_N^{\pm1},h^{\pm1}]
}{
\left\langle
f(\Gamma_1,\Gamma_2)-f(z_1,\dots,z_N)
\ \middle|
f\in
\mathbb C[z_1^{\pm1},\dots,z_N^{\pm1}]^{S_N}
\right\rangle
}.
\label{eq:KT-Grassmannian}
\end{align}
Here $\Gamma_1$ and $\Gamma_2$ are the Chern roots of the tautological
subbundle and the tautological quotient bundle, respectively.
For $\sigma\in S_N$ and
$\mathcal I=(\mathbf{I},\mathbf{I}^{\mathrm{c}})$,
there exists a unique class
$
\kappa_{\sigma,\mathcal I}
\in
K_T\bigl(T^*\operatorname{Gr}(r,N)\bigr)
$
satisfying the following properties.

$\mathrm{(i)}$
For every fixed point
$
\mathcal M=(\mathbf{M},\mathbf{M}^{\mathrm{c}}),
$
the restriction
$
\kappa_{\sigma,\mathcal I}\big|_{\mathcal M}
$
is divisible by $e_{\sigma,\mathcal M,-}^{\mathrm{vert}}$.

$\mathrm{(ii)}$
At the diagonal fixed point,
\[
\kappa_{\sigma,\mathcal I}\big|_{\mathcal I} =
P_{\sigma,\mathcal I}e_{\sigma,\mathcal I}.
\]

$\mathrm{(iii)}$
If
$
\mathcal M\ne\mathcal I,
$
then
$
\kappa_{\sigma,\mathcal I}\big|_{\mathcal M}
$
is $\mathcal M$-small, that is,
\[
\varphi_{\mathcal M}
\left(
\mathcal{N}\left(
\kappa_{\sigma,\mathcal I}\big|_{\mathcal M}
\right)
\right)
\subset
[0,r(N-r)-1].
\]
Moreover,
\begin{equation}
\kappa_{\sigma,\mathcal I}\big|_{\mathcal M} =
\widetilde W_{\sigma,\mathcal I}^{\mathrm{RTV}}
(\boldsymbol z_{\mathbf{M}},\boldsymbol z,h).
\label{kappa-Wtilde}
\end{equation}
\end{thm}
For later use, we record the Grassmannian specialization of the orthogonality
formula of \cite{RTV}.  For an $r$-element subset
$\mathbf M\subset\{1,\ldots,N\}$, write
\begin{align}
P_{\mathbf M}(\mathbf z)
&:=
\prod_{a\in\mathbf M}\prod_{b\in\mathbf M^{\mathrm c}}
\left(-\frac{z_b}{z_a}\right),
\label{eq:RTV-Gr-P}\\
R_{\mathbf M}(\mathbf z)
&:=
\prod_{a\in\mathbf M}\prod_{b\in\mathbf M^{\mathrm c}}
\left(1-\frac{z_b}{z_a}\right),
\label{eq:RTV-Gr-R}\\
Q_{\mathbf M}(\mathbf z,h)
&:=
\prod_{a\in\mathbf M}\prod_{b\in\mathbf M^{\mathrm c}}
\left(1-h\frac{z_b}{z_a}\right).
\label{eq:RTV-Gr-Q}
\end{align}
If
$\mathbf K=\{k_1<\cdots<k_r\}$, define
\begin{equation}
p(\mathbf K)
:=
\left|\{(a,b)\in\mathbf K\times\mathbf K^{\mathrm c}:b>a\}\right|
=
\sum_{j=1}^{r}(N-r+j-k_j).
\label{eq:RTV-Gr-p}
\end{equation}
This is the number of inversions in the ordered sequence
$\mathbf K^{\mathrm c},\mathbf K$ appearing in the specialization
$\lambda=(r,N-r)$ of the notation of \cite{RTV}.
For a subset $\mathbf M=\{m_1<\cdots<m_r\}$, put
\[
\mathbf z_{\mathbf M}^{-1}
=(z_{m_1}^{-1},\ldots,z_{m_r}^{-1}),
\qquad
\mathbf z^{-1}=(z_1^{-1},\ldots,z_N^{-1}).
\]

\begin{thm}\cite[Theorem 6.6, Grassmannian specialization]{RTV}
\label{thm:RTV-Gr-orthogonality}
For any two $r$-element subsets
$\mathbf J,\mathbf K\subset\{1,\ldots,N\}$, with
$\mathcal J=(\mathbf J,\mathbf J^{\mathrm c})$ and
$\mathcal K=(\mathbf K,\mathbf K^{\mathrm c})$, one has
\begin{align}
&\sum_{\substack{\mathbf M\subset\{1,\ldots,N\}\\|\mathbf M|=r}}
\frac{
h^{p(\mathbf K)}P_{\mathbf M}(\mathbf z)
}{
R_{\mathbf M}(\mathbf z)
Q_{\mathbf M}(\mathbf z,h)
}
\widetilde W_{\mathrm{id},\mathcal J}^{\mathrm{RTV}}
(\mathbf z_{\mathbf M},\mathbf z,h)
\widetilde W_{\sigma_0,\mathcal K}^{\mathrm{RTV}}
(\mathbf z_{\mathbf M}^{-1},\mathbf z^{-1},h^{-1})
=\delta_{\mathbf J,\mathbf K}.
\label{eq:RTV-Gr-orthogonality}
\end{align}
\end{thm}

\begin{proof}
In the notation of \cite[Theorem~6.6]{RTV}, the ordered set partition has
only two blocks,
$\mathcal M=(\mathbf M,\mathbf M^{\mathrm c})$.
Thus the products $P(\mathbf z_{\mathcal M})$,
$R(\mathbf z_{\mathcal M})$, and $Q(\mathbf z_{\mathcal M},h)$ there reduce
exactly to
\eqref{eq:RTV-Gr-P}--\eqref{eq:RTV-Gr-Q}, and the inversion statistic reduces
to \eqref{eq:RTV-Gr-p}.  Substituting these expressions into
\cite[(6.7)]{RTV} gives \eqref{eq:RTV-Gr-orthogonality}.
\end{proof}

We shall also need to express the weight function for $\sigma=\mathrm{id}$ in terms of that for $\sigma=\sigma_0$.
Let
\[
\rho(a)=N+1-a,
\qquad
\mathbf z^{\rho}=(z_N,z_{N-1},\ldots,z_1),
\qquad
\rho(\mathbf I)=\{N+1-i:i\in\mathbf I\}.
\]
Since the normalized weight function is symmetric in the $t$-variables,
the defining formula \eqref{eq:general-sigma-W} gives
\begin{equation}
\widetilde W_{\mathrm{id},\mathcal I}^{\mathrm{RTV}}
(\mathbf z_{\mathbf M},\mathbf z,h)
=
\widetilde W_{\sigma_0,\rho(\mathcal I)}^{\mathrm{RTV}}
\bigl((\mathbf z^{\rho})_{\rho(\mathbf M)},\mathbf z^{\rho},h\bigr),
\label{eq:RTV-id-via-reversal}
\end{equation}
where
$\rho(\mathcal I)=(\rho(\mathbf I),\rho(\mathbf I)^{\mathrm c})$.
Indeed, after applying $\sigma_0$ to the reversed equivariant variables, the
ordered list of equivariant parameters is again
$(z_1,\ldots,z_N)$; the possible reversal of the specialized $t$-variables is
irrelevant because of their symmetry.

\bigskip
\noindent
(K. Motegi) Faculty of Marine Technology, Tokyo University of Marine Science and Technology,\\
Etchujima 2--1--6, Koto-Ku, Tokyo, 135-8533, Japan

\noindent
\textit{Email address:} \texttt{kmoteg0@kaiyodai.ac.jp}

\end{document}